\documentclass[11pt, letterpaper,reqno]{amsart}
\usepackage[left=1in,right=1in,bottom=0.5in,top=0.8in]{geometry}
\usepackage[all]{xy}
\usepackage{ upgreek }
\usepackage{tabularx}
\usepackage{xcolor}
\usepackage{caption} 
\usepackage{relsize}
\usepackage{mathtools}
\usepackage{graphicx}
\usepackage{subcaption}

\usepackage{amsmath, amsthm, amsfonts, amssymb, bbm}
\usepackage{graphicx, psfrag, color}
\usepackage{cite}
\usepackage{hyperref}
\usepackage{color}
\usepackage{tikz} 
\usetikzlibrary{patterns} 
\usetikzlibrary{arrows}
\usetikzlibrary{decorations.markings} 
\usetikzlibrary{shapes}  
\usepackage{placeins}
\usepackage{float}

\usepackage{pgfplots}
\pgfplotsset{compat=1.18}

\numberwithin{equation}{section}

\newcommand{\Pb}{\mathbb{P}}

\newcommand{\Id}{\mathbbm{1}}

\newcommand{\R}{\mathbb{R}}
\newcommand{\N}{\mathbb{N}}
\newcommand{\Z}{\mathbb{Z}}
\newcommand{\I}{{\rm i}}

\DeclareMathOperator{\sgn}{sgn}

\newcommand{\bra}[1]{\left\langle #1 \right|}
\newcommand{\ket}[1]{\left| #1 \right\rangle}
\newcommand{\braket}[2]{\left\langle #1  \left| #2  \right\rangle \right.}
\newcommand{\brabarket}[3]{\left\langle #1 \left| #2 \right| #3 \right\rangle}
\newcommand{\ketbra}[2]{\left| #1 \right\rangle \left\langle #2 \right|}

\newcommand{\zcd}{\raisebox{-0.5\height}{\includegraphics[scale=0.04]{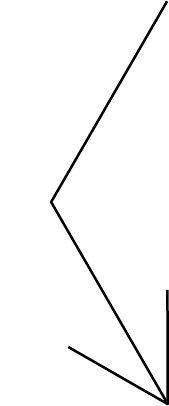}}}
\newcommand{\wcu}{\raisebox{-0.5\height}{\includegraphics[scale=0.04]{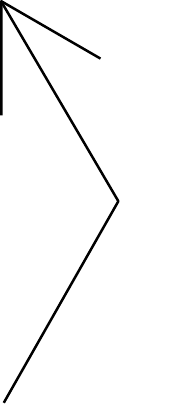}}}

\DeclareMathAlphabet{\mathpzc}{OT1}{pzc}{m}{it}

\newtheorem{prop}{Proposition}[section]
\newtheorem{thm}[prop]{Theorem}
\newtheorem{lem}[prop]{Lemma}
\newtheorem{defin}[prop]{Definition}

\newtheorem{con}[prop]{Conjecture}
\newtheorem{cla}[prop]{Claim}

\newtheorem{remark}[prop]{Remark}

\title{Stationary Half-space Geometric Last Passage Percolation}
\author{
    Jiyue Zeng\\
}
\thanks{Email: jz3524@columbia.edu} 
\date{\today}
\begin{document}

\begin{abstract}
We consider the half-space geometric Last Passage Percolation model starting with stationary measures. We obtain exact formulas for LPP value along the diagonal $(N,N)$ across the entire phase diagram. We also obtain the limits of these distributions under critical scaling which should yield the one-point distribution of the half-space KPZ fixed point starting from stationary initial conditions.
\end{abstract}

\maketitle

\setcounter{tocdepth}{1}

\tableofcontents

\section{Introduction}
The Kardar-Parisi-Zhang (KPZ) equation describes the stochastic growth of the height function $h(t,x)$ for a one-dimensional interface that undergoes a microscopic growth process driven by white noise. Several discrete models belong to the KPZ universality class \cite{KPZ}, including the directed random polymer model, Last Passage Percolation (LPP) model (which can be viewed as the zero-temperature limit of the polymer model), interacting particle systems such as the exclusion process, and vertex models. These models exhibit universal limit shape and fluctuation behavior under the typical KPZ scaling, where fluctuations occur on a $t^{1/3}$ scale, and non-trivial correlations persist over a $t^{2/3}$ scale. 

In the half-space environment, the boundary parameter plays a crucial role. Several half-space models and the half-line KPZ equation have been studied or conjectured with the characterization of their asymptotic distributions, see \cite{Baik_2018, MacdonaldProcess, Betea_2020, borodin2016directed, gueudre2012,Krajenbrink_2020,Sasamoto_2004}. More recently, authors of \cite{imamura2022solvable} have solved models related to the half-space KPZ equation with derivations of exact Fredholm Pfaffian formulas.
The stationary measures of these models have become a particularly active area of research. It has been found that for a fixed boundary parameter, there exists a one-parameter family of stationary measures determined by the drift of the initial conditions. These stationary measures have been explicitly constructed across various models.
For example, \cite{OpenKPZstationary,Barraquand_2023}  characterize the stationary measure for the open KPZ equation; \cite{stripStationary} determines the stationary measure for LPP and polymer models on a strip; \cite{BI23} proves the stationary measure for the log-gamma polymer in half-space; \cite{BI23,KPZBarraquand} prove the stationary measure for the half-space KPZ equation; and \cite{das2025convergence} proves the convergence of half-space log-gamma polymer to the stationary measure along anti-diagonal path; and there are extensive works on stationary open ASEP, for example \cite{nestoridi2024approximating, hegde2024large, yang2025limits, wang2025asymmetric}.

In this paper, we concentrate on one such model—the stationary geometric LPP and analyze the distribution of the height function (or last passage time) at the origin over the entire phase diagram. We identify the KPZ scaling asymptotics and characterize the fluctuations in the system. The distributions we obtain in the limit are expected to be the one-point distribution formulas for the half-space KPZ fixed point with stationary initial condition. Analogues to the full-space KPZ fixed point found by \cite{FixedPoint}, the half-space KPZ fixed point has recently been solved in \cite{Xincheng2024}.

\subsection{Half-space geometric last passage percolation model with initial condition}
Given any initial condition $(G(N,1))_{N\in \Z_{\geq 1}}$ and random variables $\omega_{N,M}$ for all $N\geq M$, $N,M \in \Z_{\geq 2}$, we define the LPP recurrence relation for $G(N,M)$ as follows:
\begin{equation}\label{recurrence}
    \begin{cases}
        G(N,M) = \omega_{N,M} + \max(G(N-1,M),\, G(N,M-1)), &\text{ for }N>M\geq 2,\\
        G(N,N) = \omega_{N,N} + G(N,N-1), &\text{ for } N\geq 2.
    \end{cases}
\end{equation}
We say a random variable $\omega$ is distributed as $\mathrm{Geom}(q)$ if $\Pb(\omega = k) = (1-q)q^k,$ for all $k \in \Z_{\geq 0}.$
Let $\sqrt{q}\in (0,1)$ and $r\in (0,1/\sqrt{q})$ be two parameters, where $\sqrt{q}$ is called the bulk parameter and $r$ is called the boundary parameter. We assume the weights are all independent and distributed as follows:
\begin{equation}\label{weights}
\begin{cases}
    \omega_{N,M} \sim \text{Geom}(q), &\text{ for } N> M\geq 2,\\
    \omega_{N,N} \sim \text{Geom}(r\sqrt{q}), &\text{ for } N\geq 2.
\end{cases}
\end{equation}
This recurrence relation can be solved using LPP, which involves tracing paths backward from $(M,N)$ to the initial data and maximizing over both the paths and the starting points. The precise definition will be provided in section \ref{LPPdefinition}.
If we choose the initial condition such that $G(1,1) = 0$ and $G(N,1) - G(N-1,1)$ are independent $\text{Geom}(q)$ random variables, then the solution to $G(N,M)$ will be the point to point LPP time.

The following definition of stationarity is taken from \cite[section 1.2.1]{BI23}.
\begin{defin}
    We say that the law of a process $(\mathcal{I}(x))_{x\in \Z_{\geq 0}}$ is a stationary measure for the half-space geometric LPP model if the solution to \eqref{recurrence} with initial condition satisfying $G(1+\cdot , 1) - G(1,1) \stackrel{(d)}{=} \mathcal{I}(\cdot)$ has the property that $(G(M+x, M) - G(M,M))_{x\in \Z_{\geq 0}}$ is independent of $M$ and is equal to $(\mathcal{I}(x))_{x\in \Z_{\geq 0}}$ in distribution.
\end{defin}
The above definition of stationary measure is only stated for the horizontal line, but can be generalized to any down-right path. We define a $\text{Geom}(x)$ random walk to be a discrete random process whose increments are independent, identically distributed geometric random variables with rate $x$. In recent decades, there has been active research on the stationary measures of models within the KPZ universality class.
In particular, recent works \cite{KPZBarraquand, BI23} have identified the stationary measure for the half-space KPZ equation as the Hariya-Yor process, see \cite[Definition 1.2]{BI23}, \cite[(5)]{KPZBarraquand}. This result was originally discovered in \cite{KPZBarraquand} through the limit of the interval KPZ equation based on the work \cite{OpenKPZstationary} and was later derived in \cite{BI23} as the intermediate disorder limit of the stationary measure for the half-space log-gamma polymer model. Similarly, a two-parameter family of stationary measures for half-space geometric and exponential LPP was proved in \cite[Proposition 3.2]{BI23}; however, the explicit form of the stationary measure for half-space geomeric LPP was not provided there. Therefore, we present it in Definition \ref{stationaryMeasure}. 
\begin{defin}\label{stationaryMeasure}
    For any fixed $\sqrt{q} \in (0,1),$ $s\in (\sqrt{q},1]$, $r \in (0,1/s],$ we define
    \begin{equation}
        \mathcal{I}_{r,s}(x) := \max\bigg\{R_2(x),\,\, \max_{1\leq k \leq x, k\in \Z} \{R_1(k)+ R_2(x-k+1) - Y\}\bigg\}, \quad \text{ for } x \in \Z_{\geq 1},
    \end{equation}
    where $R_1$ and $R_2$ are independent $\text{Geom}(\sqrt{q}/s)$ and $\text{Geom}(\sqrt{q}s)$ random walks with $R_1(0) = 0$ and $R_2(0) = 0$, and $Y \sim$ $\text{Geom}(rs)$ is independent. We let $\mathcal{I}_{r,s}(0) = 0$. If $rs=1$, we define $Y = \infty$.
\end{defin}

\begin{prop}\label{stationaryMeasureProp}
    For any fixed $\sqrt{q} \in (0,1),$ $s\in (\sqrt{q},1),$ $r \in (0,1/s],$ the process $(\mathcal{I}_{r,s}(x))_{x\in\Z_{\geq 0}}$ is stationary for the half-space geometric LPP model $\eqref{recurrence}$.
\end{prop}

Proposition \ref{stationaryMeasureProp}  is taken from \cite[Proposition 3.2]{BI23}. Its proof is explained in Proposition \ref{I_r,sStationary} and the subsequent paragraph. We call $\mathcal{I}_{r,s}$ the \emph{two-parameter stationary measure} because of its dependence on $r$ (boundary parameter) and $s$ (related to the drift of the initial condition at infinity). For any fixed boundary parameter $r$, there is a family of stationary measures $\mathcal{I}_{r,s}$ depending on $s$. By definition, $\mathcal{I}_{r,1/r}(\cdot)$ is a $\text{Geom}(\sqrt{q}/r)$ random walk, which is the  \emph{product stationary measure}. Under the special case $r=s$, $\mathcal{I}_{r,r}$ is also a $\text{Geom}(\sqrt{q}/r)$ random walk. This will be explained later in \eqref{permutationInvariance} and \eqref{equalr,r}.
The product stationary measure for the half-space exponential LPP model was established in \cite[Lemma 2.1]{Betea_2020}.

\subsection{Phase Diagram}
We propose the following conjecture, similar to the conjecture for the half-space log-gamma polymer in \cite[Conjecture 1.9]{BI23}.
We make it clear that $\{(r,s)|r \in (0,1/\sqrt{q})$, $s\in (\sqrt{q},\min(1,1/r)]\}=\{(r,s)| s\in (\sqrt{q},1], r\in (0,1/s]\}.$
\begin{con}\label{conjecture}
    For given $\sqrt{q} \in (0,1),$ $r \in (0,1/\sqrt{q}),$ $\{\mathcal{I}_{r,s}\}_{s \in (\sqrt{q}, \min\{1,1/r\}]}$ constitutes all extremal stationary measures for $\eqref{recurrence}$. For any initial condition $G(N,1)$ such that $\lim_{N\rightarrow\infty} G(N,1)/N = \rho \in R_{>0}$, we have the following. Let \begin{equation}\label{relation}
        s = \frac{(1+\rho)\sqrt{q}}{\rho}.
    \end{equation}
    \begin{itemize}
        \item(High density) For $\log(s) \leq -\log(r)$ and $\log(s) \leq 0,$ the process $\left(G(x+N, N) - G(N,N)\right)_{x\in \Z_{\geq 0}}$ converges weakly as $N \rightarrow \infty$ to $(\mathcal{I}_{r,s}(x))_{x\in \Z_{\geq 0}}$.
        \item(Maximal current) For $-\log(r) \geq 0$ and $\log(s) \geq 0$, the process $\left(G(x+N, N) - G(N,N)\right)_{x\in \Z_{\geq 0}}$ 
        converges weakly as $N \rightarrow \infty$ to $(\mathcal{I}_{r,1}(x))_{x\in \Z_{\geq0}}$.
        \item(Low density) For $-\log(r) \leq 0$ and $\log(s) \geq -\log(r),$ the process $\left(G(x+N, N) - G(N,N)\right)_{x\in \Z_{\geq 0}}$ converges weakly as $N \rightarrow \infty$ to $(\mathcal{I}_{r,1/r}(x))_{x\in \Z_{\geq 0}}$, i.e., a Geom$(\sqrt{q}/r)$ random walk.
    \end{itemize}
\end{con}
As shown in \eqref{relation}, $s$ is not the slope of the initial condition; rather, its connection to the slope is explicitly given by the formula.
We have created the phase diagram that illustrates the above conjecture.
We remark that it is natural to assume $\rho >0$ for the LPP path to make use of the initial condition, for example, if $\rho \leq 0,$ then $\max(G(N,1),G(N-1,2)) = G(N-1,2)$ for $N$ large. The requirement $\rho > 0$ translates into $s >\sqrt{q}$ in the following figure. Additionally, we note that the High density regime can access the full phase diagram: it reaches the Maximal current phase through the line $0<r<1, s=1$ and can access the Low density phase through the line $rs=1$, $s<1$.
\begin{figure}
\centering
    \begin{tikzpicture}[scale=1.35]
        \fill[green!20] (0,0) -- (1.9,0) -- (1.9,1.9) -- (0,1.9) -- (0,0) -- cycle;
        \fill[red!20,opacity=0.3] (0,0) -- (0,1.9) -- (-1.9,1.9) -- (-1.9,-1.9) -- (0,0) -- cycle;
        \fill[blue!15] (0,0) -- (1.9,0) -- (1.9,-1.9) -- (0,0) -- cycle;
        \fill[blue!15] (0,0) -- (1.9,-1.9) -- (-1.9,-1.9) -- (0,0) -- cycle;
        
        \draw[thin,->] (-2,0) -- (2,0)node[right] {\footnotesize{$\mathsmaller{-\log(r)}$}};
        \draw[thin,->] (0,-2.2) -- (0,2)node[right] 
        {\footnotesize{$\mathsmaller{\log(s)}$}};

        \draw[thick] (0,0) -- (0,2);
        \draw[thick] (0,0) -- (2,0);
        
        \draw[thick] (-1.9,-1.9) -- (0,0);
        \draw[thin, dashed] (0,0) -- (1.9,1.9) node[right] {\footnotesize{$\mathsmaller{rs=1}$}};
        \draw[thin,dashed] (1.9,-1.9) -- (-1.9,1.9) node[left] {\footnotesize{$\mathsmaller{r=s}$}};
        \draw[thin, dashed] (-1.9,-1.9) -- (1.9,-1.9); \node at (1.2,-2) {\footnotesize{$\mathsmaller{s=\sqrt{q}}$}};
        \node at (-0.1,-0.15) {\footnotesize{$0$}};
        
        \node[text = {red}] at (0.6,-1){\footnotesize{HD, $\mathcal{I}_{r,s}$}};
        \node[text = {red}] at (1,1){\footnotesize{MC, $\mathcal{I}_{r,1}$}};
        \node[text = {red}] at (-1.1,0.5){\footnotesize{LD, $\mathcal{I}_{r,1/r}$}};

        \draw[thin,->]
        (2.5,0.8) node[above]{\footnotesize $F^{\mathrm{MC}}_{r,1}$}to[out=180,in=80,looseness=1.2] (1,0.05);
    
      \draw[thin,->]
        (-2.6,-1.05) -- (-1.1,-1.05)
        node[near start,above]{\footnotesize $F^{LD}_r$};
    
      \draw[thin,->]
        (2.8,-1) node[above]{\footnotesize $F^{\mathrm{LD}}_{r}$}to[out=180,in=55] (1.35,-1.3);

    \node[text = {black}] at (1.4,-0.5){\footnotesize{ $F_{r,s}^{HD}$}};
    \node[text = {black}] at (0,-1.55){\footnotesize{ $F_{r,s}^{HD}$}};
        
    \end{tikzpicture}
    \captionof{figure}{This picture describes the phase diagram of the geometric LPP model given boundary parameter and drift parameter. Depending on where $(-\log(r), \log(s))$ lies in the diagram, the process $G(\cdot,N) - G(N,N)$ converges to one of the three spatial processes, $\mathcal{I}_{r,s},$ $\mathcal{I}_{r,1/r},$ or $\mathcal{I}_{r,1},$ which is claimed in Conjecture $\ref{conjecture}$. On the full line $r=s$ and the half line $rs=1, r\geq 1, s\leq 1$, the process converges to the \text{Geom}$(\sqrt{q}/r)$ random walk. \textbf{MC}, \textbf{HD}, and \textbf{LD} represent Maximal current (green), High density (blue), and Low density (red) phases. We characterize the distribution of $G(N,N)$ in the High density phase, i.e., $F_{r,s}^{HD}$, and on the two boundaries where HD meets the other phases: along the MC boundary we obtain $F_{r,1}^{MC}$, and along the LD boundary we obtain $F_{r}^{LD}$. In particular, we have $F_{r}^{LD}$ on the line $rs=1$, $s<1$ in HD, which will be explained further in \ref{degenerate1}.}
    \label{phaseDiagram}
\end{figure}

This phase diagram $\ref{phaseDiagram}$ is derived by matching the parameters of the geometric LPP model with those of the log-gamma polymer model, as outlined in \cite[Definition 2.5]{BI23}. In the log-gamma polymer model, $u$ is the boundary parameter and $v$ is the first row parameter. We found that such correspondence is 
\begin{equation}
\begin{aligned}
    u = -\log(r),&\quad v = \log(s).\\
\end{aligned}
\end{equation}

The concept of the phase diagram is taken from that in the half-space interacting particle systems, such as, ASEP.
We explain the meaning of each phase in the LPP setting in a very heuristic way:
\begin{itemize}
    \item In the High density regime, the LPP paths mainly gather weights along the initial condition (the first row). The conditions $\log(s) \leq - \log(r)$ and $\log(s)\leq 0$ imply that $s < 1$ and $r \leq 1/s$, which means that the initial condition dominates the weights on the diagonal.
    \item Due to symmetry, a similar situation should occur in the Low density phase, where LPP paths
    primarily gather their weights along the diagonal and do not venture much into the bulk.
    \item In the Maximal current regime, the fluctuations are qualitatively different. There is a competition between the diagonal's attractiveness and the initial condition. The LPP paths tend to stay within the bulk for most of their trajectory, balancing between these influences.
\end{itemize}

\subsection{Main results}
We find the distribution of $G(N,N)$ under the stationary initial condition $\mathcal{I}_{r,s}$ with a random shift at $(1,1)$ and also compute its asymptotic limit under the critical scaling. We use $\stackrel{(d)}{=}$ to denote equality in distribution.

Fix any $\sqrt{q}\in (0,1).$ Let $r,s\in \R$, $r\neq \sqrt{q}$, $s\in (\sqrt{q},1/\sqrt{q})$. Let $R_1^s$ and $R_2^s$ be independent $\text{Geom}(\sqrt{q}/s)$ and $\text{Geom}(\sqrt{q}s)$ random walks starting from $0$. Let $Y^s$ be a $\text{Geom}(rs)$ random variable independent of $R_1^s,R_2^s$. Let $\mathcal{I}_{r,s}$ be the stationary process introduced in Definition \ref{stationaryMeasure}, constructed from $R_1^s,R_2^s,Y^s.$ Let $G(N,N)$ be the LPP model defined by \eqref{recurrence} and \eqref{weights}. Let $c_0 = \frac{2\sqrt{q}(1+\sqrt{q})}{(1-\sqrt{q})^3}.$ We introduce a few notations of CDFs for different phases. Variables $\tilde{r}$ and $\tilde{s}$ in the subscript will be explained in Theorem \ref{theorem:limit}.
\begin{itemize}
    \item Consider the High density phase, i.e., $s\in (\sqrt{q},1)$ and $r \in (0,s)\cup(s,1/s)$. We start with the two-parameter stationary initial condition, i.e., $G(1+\cdot,1) - G(1,1)=\mathcal{I}_{r,s}(\cdot)$ and $G(1,1) = Y^s.$ Then we define
    \begin{equation}\label{HighCDF}
        \begin{aligned}
            F_{r,s}^{HD}(d,N) := \Pb(G(N,N)\leq d), \quad \widetilde{F}_{\tilde{r},\tilde{s}}^{HD}(\tilde{d}) := \lim_{N\rightarrow \infty}\Pb\left(\frac{G(N,N) - \frac{2\sqrt{q}}{1-\sqrt{q}}N }{(c_0N/2)^{1/3}} \leq \tilde{d}\right).
        \end{aligned}
    \end{equation}
    \item Consider the Maximal current phase, i.e., $ s = 1$ and $r\in (0,1).$ We start with the special two-parameter stationary initial condition, i.e., $G(1+\cdot,1) - G(1,1)=\mathcal{I}_{r,1}(\cdot)$ and $G(1,1) = Y^1.$ 
    Then we define 
    \begin{equation}\label{MaxCDF}
        \begin{aligned}
            F_{r,1}^{MC}(d,N) := \Pb(G(N,N)\leq d), \quad \widetilde{F}_{\tilde{r},0}^{MC}(\tilde{d}) := \lim_{N\rightarrow \infty}\Pb\left(\frac{G(N,N) - \frac{2\sqrt{q}}{1-\sqrt{q}}N }{(c_0N/2)^{1/3}} \leq \tilde{d}\right).
        \end{aligned}
    \end{equation}
    \item Consider the Low density phase, i.e., $r = 1/s$ or $r = s$. We start with the product stationary initial condition (or Brownian initial condition), i.e., $G(1+\cdot,1) \stackrel{(d)}{=} \mathcal{I}_{r,r}(\cdot)\stackrel{(d)}{=} \mathcal{I}_{r,1/r}(\cdot)$. Then we define
    \begin{equation}\label{LowCDF}
        \begin{aligned}
            F_{r}^{LD}(d,N) := \Pb(G(N,N)\leq d), \quad \widetilde{F}_{\tilde{r}}^{LD}(\tilde{d}) := \lim_{N\rightarrow \infty}\Pb\left(\frac{G(N,N) - \frac{2\sqrt{q}}{1-\sqrt{q}}N }{(c_0N/2)^{1/3}} \leq \tilde{d}\right).
        \end{aligned}
    \end{equation}
\end{itemize}

All these functions are indicated in the phase diagram \ref{phaseDiagram}.

\begin{thm}\label{FiniteTimeFormulaTheorem}

We characterize the distribution of $G(N,N)$ under three phases:
\begin{itemize}
    \item[$(1)$]  (High density phase) For any $N\in \Z_{\geq 1}$, $d\in \Z_{\geq 0}$, and any pair of $(r,s)$ satisfying $r \in (0,s) \cup (s,1/s)$, $s\in (\sqrt{q},1)$, we have \begin{equation}\label{finiteTimeFormulaTheorem}
        F_{r,s}^{HD}(d,N) = 
        \frac{s}{s-r} {\psi(d,N+1)} - \frac{r+s}{s-r}{\psi(d-1,N+1)} + \frac{r}{s-r}{\psi(d-2,N+1)},
    \end{equation}
    where $\psi(d,N)$ is defined in Definition \ref{FiniteTimeFormula}.
    \item[$(2)$] (Maximal current phase) Fix $s = 1.$ For any $N\in \Z_{\geq 1}$, $d\in \Z_{\geq 0}$, and any $r \in (0,1)$, we have
    \begin{equation}
        F_{r,1}^{MC}(d,N) = 
        \frac{1}{1-r}\eta(d,N+1) - \frac{1+r}{1-r}\eta(d-1,N+1) + \frac{r}{1-r}\eta(d-2,N+1),
    \end{equation}
    where $\eta(d,N)$ is defined in Definition \ref{MaximalFinite}.
    \item[$(3)$] (Low density phase) For any $N \in \Z_{\geq 1}$, $d\in \Z_{\geq 0}$ and any $r$ satisfying $r\in (\sqrt{q},1/\sqrt{q})$, we have
    \begin{equation}
        F_{r}^{LD}(d,N) = 
        \theta(d,N) - \theta(d-1,N),
    \end{equation}
    where $\theta(d,N)$ is defined in Definition \ref{oneParamResult_Finite}.
\end{itemize}
\end{thm}

In Theorem \ref{FiniteTimeFormulaTheorem} (1) and (2), the shift of a geometric random variable, $G(1,1) = Y,$ is essential and necessary for the exact computations of $G(N,N)$, which will become clear when we introduce the half-space geometric LPP model with special weights in section $6$.
We isolate the case $r=1/s$ in Theorem \ref{FiniteTimeFormulaTheorem} (3) because this is the coexistence line where the High density phase and Low density phase overlap, which corresponds to having the special product stationary initial condition $\mathcal{I}_{r,1/r}$. We also separate the case $r = s$ because this is the special case in the High density regime where the initial condition degenerates to the product stationary initial condition. For Low density phase, Theorem \ref{FiniteTimeFormulaTheorem} (3), there is no random shift at $(1,1).$

\begin{thm}\label{theorem:limit}
    We characterize the asymptotic limit of the rescaled $G(N,N)$ under three phases:
    \begin{itemize}
    \item[$(1)$] (High density phase) Fix any $\tilde{s} <0$ and $\tilde{r} \in (-\infty, \tilde{s})\cup (\tilde{s}, -\tilde{s})$. Consider the critical scaling
    \begin{equation}
        s = 1 + (c_0N/2)^{-1/3}\tilde{s}, \quad r = 1+ (c_0N/2)^{-1/3}\tilde{r}.
    \end{equation} Then we have
    \begin{equation}\label{limit,L}
        \widetilde{F}_{\tilde{r},\tilde{s}}^{HD}(\tilde{d}) = \frac{1}{\tilde{s} - \tilde{r}}\partial^2_{\tilde{d}} \mathbf{\mathcal{L}}(\tilde{d}) + \partial_{\tilde{d}}\mathcal{L}(\tilde{d}), 
    \end{equation}
    where $\mathcal{L}(\tilde{d})$ is a twice-differentiable function defined in section \ref{HighAsymptotic}.
   \item[$(2)$] (Maximal current phase) Fix any $\tilde{r}<0$. Let $s = 1$, i.e., $\tilde{s}= 0.$ Consider the scaling $$r = 1+(c_0N/2)^{-1/3}\tilde{r}.$$ Then we have
    \begin{equation}
        \widetilde{F}_{\tilde{r},0}^{MC}(\tilde{d}) = -\frac{1}{ \tilde{r}}\partial^2_{\tilde{d}} \Phi(\tilde{d}) + \partial_{\tilde{d}}\Phi(\tilde{d}), 
    \end{equation}
    where $\Phi(\tilde{d})$ is a twice-differentiable function defined in section \ref{MaximalAsymptotic}.
    \item[$(3)$](Low density phase) Fix $\tilde{r}\in \R$. Consider the scaling $$r = 1+(c_0N/2)^{-1/3}\tilde{r}.$$ Then we have
    \begin{equation}
        \widetilde{F}_{\tilde{r}}^{LD}(\tilde{d}) = \partial_{\tilde{d}}\Xi(\tilde{d}),
    \end{equation}
    where $\Xi$ is a differentiable function defined in section \ref{LowAsymptotic}.
    \end{itemize}
\end{thm}

Under the critical scaling, the random shift of the geometric random variable at $(1,1)$ has a nontrivial limiting effect, which is offset by moving away from the diagonal on the scale of $\mathcal{O}(N^{2/3}).$ Alternatively, one can consider different parameter scalings to obtain some Gaussian-type fluctuations. However, we found that taking critical scaling was easier. The condition $\tilde{s}<0$ in Theorem \ref{theorem:limit} (1)
arises from the fact that $s<1$.

Our limiting distribution should characterize the one-point distribution for the half-space KPZ fixed point under the two-parameter stationary initial condition. In \cite[(97)]{KPZbeyondBrownian}, the authors predicted the stationary measure for the half-space KPZ fixed point, and it can be obtained as the scaling limit of $I_{r,s}.$
We note that deriving the off-diagonal distribution ($G(N, M)$ in the High density phase) from our formula ($F_{r,s}^{HD}$) is straightforward, and a similar asymptotic limit can be obtained in that case. It would also be interesting to derive a formula for the multi-point joint distribution, which would connect to the half-space two-parameter stationary Airy process as predicted in \cite[(93)]{KPZbeyondBrownian}. By contrast, our maximal–current formula $F_{r,1}^{MC}$ does not extend to off–diagonal points, since its derivation relies on a symmetry that holds only on the diagonal.

Before outlining our proof strategy, we briefly review the relevant literature. In \cite[Theorem 2.4]{Betea_2020}, the diagonal distribution for the half-space product stationary exponential LPP at finite time is obtained (i.e., Low density phase), along with its asymptotic distribution under critical scaling. Authors of \cite{Betea_2020} also constructed the half-space product stationary Airy process, see \cite{StatAiryprocess}. In related work \cite{StationaryKPZ}, the distribution of the height function $h(0,t)$ for the half-space stationary KPZ equation with Brownian initial condition was found, providing an alternative representation of the distribution presented in \cite[Theorem 2.7]{Betea_2020}. Additionally, in the physics paper \cite{KPZbeyondBrownian}, the distribution of $h(t,0)$ for the half-space stationary KPZ equation with Hariya–Yor initial condition was obtained. However, the authors were unable to characterize the distribution in the region $u+v<0$ and $u>v$, which in our variables corresponds to $rs<1$, $r>s$, whereas in our setting we obtain the formula $\widetilde{F}_{\tilde{r},\tilde{s}}^{HD}$.

\subsection{Our strategy}
We explain how we derived these distribution formulas.
For the two-parameter stationary model (High density phase), we start with the introduction of path-wise definition of LPP and then we consider a version of the LPP model with inhomogeneous geometric weights, which has a nice Pfaffian structure. After specializing the parameters, the inhomogeneous model transforms into the two-parameter stationary model as in \eqref{approxModel}, which was originally constructed in \cite{BI23}. Then we apply the shift argument outlined in \eqref{eq:shift} to get a prelimiting Fredholm Pfaffian formula for the distribution of the stationary model. Our next step is to rewrite the Fredholm Pfaffian and reorganize terms in the kernel to enable taking the limit. Then we rely heavily on analytic continuation to find the distribution formula and Hadamard's bound to ensure the convergence of the Fredholm Pfaffian. Our approach is inspired by \cite{Betea_2020}, which handles the special case of the product stationary measure (Low density phase) in the exponential case. As seen from our formula, implementing a similar idea in the two-parameter case is significantly harder due to the presence of many more terms that need to be renormalized and many more poles that need to be taken out of the contour integrals.
The details of difficulties that arise solely in the two-parameter case are explained and compared with the product stationary case in section \ref{difference}. In the Maximal current phase, we exploit a symmetry of the geometric LPP model that swaps the diagonal parameter with the second-row parameter, reducing the problem to the product stationary case. Consequently, the techniques of \cite{Betea_2020} apply directly in this regime.

We use the finite-time distribution as an intermediate tool and rely heavily on the steepest descent method for finding the asymptotic limit of the diagonal distribution under the critical scaling for all three phases.

\subsection{Outline.} 
In section $2$, we introduce basic notations used throughout the paper. Section 3 provides definition of distribution functions for High density phase and its corresponding asymptotic limits. Section 4 provides distribution functions for Maximal current phase and its corresponding asymptotic limits. Section 5 provides finite time distribution functions for Low current phase and its corresponding asymptotic limits. Section 6 introduces the inhomogeneous geometric LPP model and its Fredholm Pfaffian structure. We also explain the shift argument used to obtain a pre-limiting formula.
In section 7, we apply analytic continuation to the pre-limiting formula to obtain the diagonal distribution of the stationary LPP model. Theorem \ref{FiniteTimeFormulaTheorem} (1) is proved at the end of this section. In section 8, we use the steepest descent method to find the asymptotic limit of the kernel and prove the convergence of the Fredholm Pfaffian. We also compute the asymptotic limits of all terms in the formula. We prove Theorem \ref{theorem:limit} (1) at the end of this section. 

In section 9, we derive a general approach to find finite time diagonal distribution formulas for both Maximal current phase and Low density phase. Theorem \ref{FiniteTimeFormulaTheorem} (2) and (3) are proved at the end of this section. Using a symmetry of the half-space geometric LPP model, we reduce the two-parameter stationary case to the product stationary analysis, which covers the special case $s=1$, and then prove convergence to the corresponding asymptotic distributions under critical scaling, that is Theorem \ref{theorem:limit} (2) and (3), in sections 10 and 11.

Appendix A provides verifications of our two-parameter finite time distribution formula under High density phase. Appendix B provides Matlab estimations of our finite time formula for the Maximal current phase.

\subsection{Acknowledgements.}
The author sincerely thanks her advisor, Ivan Corwin, for suggesting this intriguing problem and for his invaluable guidance throughout the project. The author wants to thank Patrik Ferrari, Matteo Mucciconi, Guillaume Barraquand, Milind Hegde, Zhengye Zhou, Xinyi Zhang, Min Liu for helpful discussions. This research is supported by Ivan Corwin's NSF grant DMS-2246576 and his Simons Investigator grant 929852.

\bigskip
\noindent\textbf{Conflict of interest and data availability statement}

There is no conflict of interest to declare.  
This manuscript does not contain or use any data.

\section{Basic notations}

\label{section:Finite}
In this paper, we will have two kinds of limits. We consistently use $\widehat{F}$, ${F}^M$, ${F}^L$(respectively) to denote the limit $\lim_{t\rightarrow 1/s}F(t)$ for High density, Maximal current and Low density regimes, respectively, and we consistently use $\widetilde{F}$, $\widetilde{F}^M$, $\widetilde{F}^L$ (for the asymptotic limit) to denote the limit of $\widehat{F}$, ${F}^M$, ${F}^L$ under the critical scaling in the High density, Maximal current, and Low density regimes, respectively. 
We use the bracket notation to denote the scalar product on $\ell^2(\{d+1,d+2,\dots\})$, that is
\begin{equation}\label{ellInnerProduct}
\begin{aligned}
    &\braket{f}{g}=\braket{f(k)}{g(k)} = \sum_{k = d+1}^{\infty} f(k)g(k),\\
    &\brabarket{f}{h}{g} = \brabarket{f(k)}{h(k,\ell)}{g(\ell)} = \sum_{k = d+1}^{\infty}\sum_{\ell = d+1}^{\infty}f(k)h(k,\ell)g(\ell).
\end{aligned}
\end{equation}
When the function has a hat or nothing on top of it, the bracket is the scalar product on $\ell^2(\{d+1,d+2,\dots\})$.
We use $\ketbra{f}{g}$ to denote the outer product kernel
\begin{equation}
    \ketbra{f}{g}(k,\ell) = f(k)g(\ell).
\end{equation}
We will remove the parameters $(k,\ell)$ when the context is clear.
Let $\Gamma_{I}$ denote any simple counter-clockwise contour around the set of points $I$. This notation is adopted from \cite{Betea_2020}.

We use the same bracket notation as in the discrete case to denote the scalar product on $L^2((\tilde{d}, \infty))$. 
\begin{equation}\label{LInnerProduct}
\begin{aligned}
    &\braket{\widetilde{f}}{\widetilde{g}} = \int_{\tilde{d}}^{\infty} \widetilde{f}(X)\widetilde{g}(X) dX,\\
    &\brabarket{\widetilde{f}}{\widetilde{h}}{\widetilde{g}} = \int_{\tilde{d}}^{\infty}\int_{\tilde{d}}^{\infty} \widetilde{f}(X)\widetilde{h}(X,Y)\widetilde{g}(Y) dXdY.\\
\end{aligned}
\end{equation}
In the case when the function has a tilde on top of it, the bracket is the scalar product on $L^2((\tilde{d}, \infty))$.
As in \cite[Figure 3]{Betea_2020}, we use ${}_{a}\zcd\, {}_{b}$ to denote any down-oriented path coming in a straight line from $e^{i\pi/3}\infty$ to a point $x \in \R$, which lies on the right of $a$ and on the left of $b$, and leaving in a straight line to $e^{-\I\pi/3}\infty$. We use ${}_{c}\wcu\, {}_{d}$ to denote any up-oriented path coming in a straight line from $e^{-2\pi\I/3}\infty$ to a point $y \in \R,$ which lies on the right of $c$ and on the left of $d$, and leaving in a straight line to $e^{2\pi\I/3}\infty$.

\section{Definition of distribution functions for High density phase}

\subsection{Finite time diagonal distribution for High density phase}\label{HighFinite}
\begin{defin}
\label{FiniteTimeFormula}
    Let $s,\sqrt{q},r$ be real numbers such that $\sqrt{q} \in (0,1)$, $s \in (\sqrt{q},1)$, $r \in (0,s)\cup (s,1/s),$ and $r\neq \sqrt{q}$. We define
    \begin{equation}\label{finiteDis}
    \begin{aligned}
    &\psi(d,N) :=
    \mathrm{Pf}(J - \widehat{K})\Bigg(\widehat{\mu}_d + \widehat{\nu}_{d} + \widehat{\mathcal{A}}_d + \widehat{\mathcal{B}}_d + \widehat{\mathcal{C}}_d + \widehat{\mathcal{D}}_d + \widehat{\mathcal{E}}_d
    + (d+2) - (N-2)\frac{\sqrt{q}(1/s+s-2\sqrt{q})}{(1-\sqrt{q}/s)(1-\sqrt{q}s)}\\
    &- \frac{s(1-r^2)}{(1-sr)(s-r)}\Bigg)- \mathrm{Pf}(J - \widehat{K})
    + 
    \mathrm{Pf}\left(J - \widehat{K} - \ket{\begin{array}{c}
         g_1 \\
         -d_2
    \end{array}}\bra{\widehat{V}_1 \quad \widehat{V}_2} - \ket{\begin{array}{c}
         \widehat{V}_1 \\
         \widehat{V}_2
    \end{array}}\bra{-g_1 \quad d_2}\right)\\
    \end{aligned}
    \end{equation}
    for $d\in \Z_{\geq 0}$ and $\psi(d,N) = 0$ for $d< 0$.
    The Fredholm Pfaffian is taken over $\ell^2(\{ d+1,d+2,\dots\} ).$
\end{defin}

We need the following auxiliary functions to define all terms used in the formula \eqref{finiteDis}. All brakets $\braket{\cdot}{\cdot}$ and $\brabarket{\cdot}{\cdot}{\cdot}$ in this section are inner products over $\ell^2(\{d+1,\dots\})$ as defined in \eqref{ellInnerProduct}. The following functions take inputs $k,\ell \in \Z_{\geq 0}$ and the input $x\in\R$.
\begin{equation}\label{defOfF}
    H(x) = \frac{(1-\sqrt{q}/x)^{N-2}}{(1-\sqrt{q}x)^{N-2}}, \quad f^x(k) = \frac{x^{k+1}}{H(x)}, \quad E(k,\ell) = \begin{cases}
            -r^{k-\ell-1} &\text{ if } k > \ell,\\
            0 &\text{ if } k = \ell,\\
             r^{\ell-k-1} &\text{ if } k < \ell.
    \end{cases}
\end{equation}
\begin{equation}\label{defOfG&R}
    \begin{aligned}
        G_{x}(k) = \oint \limits_{\Gamma_{1/\sqrt{q}}} \frac{dz}{2\pi\I} \frac{H(z)}{z^{k+2}}\frac{(z-x)(z-r)}{(1-xz)(z^2-1)},& \quad R_{x}(k) = -\oint \limits_{\Gamma_{\sqrt{q}}}\frac{dw}{2\pi\I}\frac{w^{k+1}}{H(w)}\frac{(1-xw)}{(w-x)(w-r)}.\\
    \end{aligned}
\end{equation}

\begin{equation}
    \begin{aligned}
        &\widehat{P}(\ell) = \oint\limits_{\Gamma_{\sqrt{q}}} \frac{dw}{2\pi\I} \frac{w^{\ell+1}}{H(w)}\frac{1}{(1-wr)}, \quad \widehat{Q}(k) = -\oint \limits_{\Gamma_{1/\sqrt{q}}} \frac{dz}{2\pi\I}\frac{H(z)}{z^{k+2}}\frac{(1-rz)}{(1-z^2)},\\
        &\widehat{B}(k,\ell) = -\oint \limits_{\Gamma_{\sqrt{q}}} \frac{dw}{2\pi\I}\oint \limits_{\Gamma_{1/\sqrt{q}}} \frac{dz}{2\pi\I} \frac{w^{\ell+1}}{z^{k+2}}\frac{H(z)}{H(w)}\frac{(zw-1)}{(z-w)(1-zr)(w-r)}. 
    \end{aligned}
\end{equation}

\begin{equation}\label{Ahat}
    \begin{aligned}
        \widehat{A}_{11}(k,\ell) =& -\oint \limits_{\Gamma_{\sqrt{q}}}\frac{dw}{2\pi\I} \oint \limits_{\Gamma_{1/\sqrt{q}}}\frac{dz}{2\pi\I}\frac{w^{\ell+1}}{z^{k+2}}\frac{H(z)}{H(w)}\frac{(zw-1)(z-r)(1-wr)}{(z^2-1)(1-w^2)(z-w)},\\
        \widehat{A}_{12}(k,\ell) =& -\widehat{A}_{21}(\ell,k) = - \oint \limits_{\Gamma_{\sqrt{q}}} \frac{dw}{2\pi\I} \oint \limits_{\Gamma_{1/\sqrt{q}}}\frac{dz}{2\pi\I}\frac{w^{\ell+1}}{z^{k+2}}\frac{H(z)}{H(w)}\frac{(zw-1)(z-r)}{(z^2-1)(w-r)(z-w)} + \ketbra{\widehat{Q}}{f^r}(k,\ell),\\
        \widehat{A}_{22}(k,\ell) =&\ketbra{f^r}{\widehat{P}}(k,\ell) -\ketbra{\widehat{P}}{f^r}(k,\ell)+\widehat{B}(k,\ell).\\
    \end{aligned}
\end{equation}

\begin{equation}\label{Khat}
    \widehat{G} = J^{-1}\widehat{K}(k,\ell)= \!\begin{pmatrix}
        -\widehat{K}_{21}(k,\ell) & -\widehat{K}_{22}(k,\ell)\\
        \widehat{K}_{11}(k,\ell) & \widehat{K}_{12}(k,\ell)
    \end{pmatrix},\,\,\,
    \widehat{K} = \!\begin{pmatrix}
        \widehat{K}_{11}(k,\ell)  & \widehat{K}_{12}(k,\ell)\\
        \widehat{K}_{21}(k,\ell) & \widehat{K}_{22}(k,\ell)
    \end{pmatrix},
\end{equation}
where
\begin{equation}
    \begin{aligned}
        \widehat{K}_{11}(k,\ell) = \widehat{A}_{11}(k,\ell),\quad
        \widehat{K}_{12}(k,\ell) = -\widehat{K}_{21}(\ell,k) = \widehat{A}_{12}(k,\ell),\quad
        \widehat{K}_{22}(k,\ell) = \widehat{A}_{22}(k,\ell)+ E(k,\ell).\\
    \end{aligned}
\end{equation}

\begin{equation}
\begin{aligned}
    &\widehat{\mu}_d = \frac{H(s)}{(1-sr)s^{d+1}}\!\!\oint \limits_{\Gamma_{\sqrt{q}}} \!\!\frac{dw}{2\pi\I}\frac{w^{d+2}}{H(w)}\frac{(1-ws)(1-wr)}{(w-s)^2(1-w^2)},\\
    &\widehat{\nu}_d = \frac{(s-r)}{(1-s^2)}\frac{H(s)}{s^{d+1}}\!\!\oint \limits_{\Gamma_{\sqrt{q}}}\!\!\frac{dw}{2\pi\I}\frac{w^{d+2}}{H(w)}\frac{(1-sw)}{(w-r)(w-s)^2} + \frac{(1-sr)}{(s-r)(1-s^2)}\frac{r^{d+2}}{s^{d+1}}\frac{H(s)}{H(r)}.
\end{aligned} 
\end{equation}

\begin{equation}\label{defOfd_2,g_1}
    \begin{aligned}
        &g_1(k) =  \!\!\!\oint \limits_{\Gamma_{1/\sqrt{q},1/s}}\!\!\! \frac{dz}{2\pi\I} \frac{H(z)(z-s)(z-r)}{z^{k+2}(1-sz)(z^2-1)} =G_s(k) - (1-sr)f^s(k),\\
        &d_2(k) =\!\!\! \oint \limits_{\Gamma_{\sqrt{q},s,r}}\!\!\! \frac{dw}{2\pi\I} \frac{w^{k+1}(1-sw)}{H(w)(w-s)(w-r)} = \frac{(1-s^2)}{(s-r)}f^s(k) - \frac{(1-sr)}{(s-r)}f^r(k) - R_{s}(k),\\
        &\widehat{\mathsf{J}}(k) = -\oint \limits_{\Gamma_{\sqrt{q}}}\!\! \frac{dw}{2\pi\I}\!\! \oint \limits_{\Gamma_{1/\sqrt{q}}} \!\!\frac{dz}{2\pi\I} \frac{w^{d+2}H(z)}{z^{k+2}H(w)}\frac{(zw-1)(z-r)}{(z^2-1)(w-s)(w-r)(z-w)}.
    \end{aligned}
\end{equation}

\begin{equation}
    \begin{aligned}
        \widehat{\mathcal{A}}_d = &\frac{(s-r)}{(1-s^2)}\frac{H(s)}{s^{d+1}}\braket{\widehat{\mathsf{J}}}{d_2}
        -\frac{1}{(1-s^2)}\frac{H(s)r^{d+2}}{H(r)s^{d+1}}\braket{\widehat{Q}}{d_2} + \frac{s}{(1-s^2)}\braket{G_{1/s}}{d_2}.\\
        \widehat{\mathcal{B}}_d =&\frac{s}{(1-s^2)}\braket{R_{1/s}}{g_1} + \frac{s(s-r)}{(1-s^2)(1-sr)}\braket{f^r}{g_1}
        + \frac{1}{(1-s^2)}\frac{H(s)r^{d+2}}{H(r)s^{d+1}}\braket{\widehat{P}}{g_1}\\
        &+\frac{(s-r)}{(1-s^2)}\brabarket{f^{1/s}}{ \widehat{B}}{g_1}
        -\frac{(s-r)}{(1-s^2)}\braket{f^{1/s}}{\widehat{P}}\braket{f^r}{g_1}.\\
        \widehat{\mathcal{C}}_d =& \frac{1}{(1-s^2)}\frac{H(s)}{s^{d+1}}\oint \limits_{\Gamma_{1/\sqrt{q}}} \frac{dz}{2\pi\I}\frac{H(z)}{z^{d+2}} \frac{(z-s)^2}{(1-sz)^2(z^2-1)}.\\
        \end{aligned}
    \end{equation}
    \begin{equation}
        \begin{aligned}
        \widehat{\mathcal{D}}_{d} 
        =& \frac{1}{(1-sr)}\brabarket{f^{1/s}}{\widehat{A}_{11}}{d_2} + \frac{s}{(1-s^2)}\braket{G_{1/s}}{d_2}.\\
        \widehat{\mathcal{E}}_d = & \frac{1}{(1-sr)}\brabarket{f^{1/s}}{ \widehat{A}_{12}}{g_1}
        +\frac{(s-r)s}{(1-s^2)(1-sr)}\braket{f^r}{g_1} + \frac{s}{(1-s^2)}\braket{R_{1/s}}{g_1}.
    \end{aligned}
\end{equation}

For the next two functions $\widehat{V}_1$ and $\widehat{V}_2$, we need the following definitions:
\begin{equation}\label{cdot}
\begin{aligned}
    &\bra{f(k)}\mathcal{F}(k,\ell) =\bra{f}\mathcal{F}(\ell) := \sum_{k = d+1}^{\infty} f(k)\mathcal{F}(k,\ell),\\
    &\mathcal{H}(k,\ell) = \mathcal{F}(k,u)\cdot \mathcal{G}(u,\ell) := \sum_{u = d+1}^{\infty} \mathcal{F}(k,u)\mathcal{G}(u,\ell),\\
    &\bra{f}\mathcal{F}\cdot \mathcal{G}(\ell) := \sum_{k = d+1}^{\infty} f(k)\mathcal{H}(k,\ell).
\end{aligned}
\end{equation}

\begin{equation}
    \begin{aligned}
        &\widehat{V}_1(\ell) = \frac{(s-r)}{(1-s^2)}\frac{H(s)}{s^{d+1}}\bra{\widehat{\mathsf{J}}\, }\widehat{A}_{21}(\ell)
        -\frac{1}{(1-s^2)}\frac{r^{d+2}H(s)}{s^{d+1}H(r)}\bra{ \widehat{Q}}\widehat{A}_{21}(\ell)
        -\frac{1}{(1-s^2)}\frac{r^{d+2}H(s)}{s^{d+1}H(r)} \bra{ \widehat{P}}\widehat{A}_{11}(\ell)\\
        &+\frac{(s-r)}{(1-s^2)} \braket{f^{1/s}}{\widehat{P}} \bra{f^{r}}\widehat{A}_{11}(\ell) - \frac{(s-r)}{(1-s^2)} \bra{f^{1/s}}\widehat{B}\cdot\widehat{A}_{11}(\ell)\\
        &+\frac{s}{(1-s^2)}G_{1/s}(\ell)- \frac{2s}{(1-s^2)}\bra{R_{1/s}}\widehat{A}_{11}(\ell)
        -\frac{2s(s-r)}{(1-s^2)(1-sr)}\bra{f^r}\widehat{A}_{11}(\ell)\\
        &+\frac{(s-r)H(s)}{(1-s^2)(1-sr)}\bra{\left(\frac{r^{k-d-1}}{s^{d+1}} + \frac{(1-s^2)}{(s-r)s^{k+1}}\right)}\widehat{A}_{11}(k,\ell)+\frac{2s}{(1-s^2)}\bra{G_{1/s}}\widehat{A}_{21}(\ell)\\
        &+\frac{1}{(1-sr)}\bra{f^{1/s}}\left(\widehat{A}_{11}\cdot\widehat{A}_{21}\right)(\ell) -\frac{1}{(1-sr)}\bra{f^{1/s}}\left(\widehat{A}_{12}\cdot\widehat{A}_{11}\right)(\ell).\\
    \end{aligned}
\end{equation}

\begin{equation}
    \begin{aligned}
        &\widehat{V}_2(\ell) = \frac{(s-r)}{(1-s^2)}\frac{H(s)}{s^{d+1}}\bra{ \widehat{\mathsf{J}}\,} \left(\widehat{A}_{22} + E\right)(\ell)
        -\frac{1}{(1-s^2)}\frac{r^{d+2}H(s)}{s^{d+1}H(r)}\bra{\widehat{Q}} \left(\widehat{A}_{22} + E\right)(\ell)\\
        &-\frac{1}{(1-s^2)} \frac{r^{d+2}H(s)}{s^{d+1}H(r)}\bra{\widehat{P}}\widehat{A}_{12}(\ell)+\frac{(s-r)}{(1-s^2)} \braket{f^{1/s}}{\widehat{P}}\bra{f^r}\widehat{A}_{12}(\ell) - \frac{(s-r)}{(1-s^2)} \bra{f^{1/s}}\widehat{B}\cdot\widehat{A}_{12}(\ell)\\
        &-\frac{2s}{1-s^2}\bra{R_{1/s}}\widehat{A}_{12}(\ell) - \frac{2s(s-r)}{(1-s^2)(1-sr)}\bra{f^r}\widehat{A}_{12}(\ell) +\frac{2s}{(1-s^2)}\bra{G_{1/s}}\left(\widehat{A}_{22} + E\right)(\ell)\\
        &+\frac{(s-r)H(s)}{(1-sr)(1-s^2)}\bra{\left(\frac{r^{k-d-1}}{s^{d+1}} + \frac{(1-s^2)}{(s-r)s^{k+1}}\right)}\widehat{A}_{12}(k,\ell)
        +\frac{(s-r)s}{(1-sr)(1-s^2)}f^r(\ell) + \frac{s}{(1-s^2)}R_{1/s}(\ell)\\
        &+\frac{1}{(1-sr)}\bra{f^{1/s}}\left(\widehat{A}_{11}\cdot \widehat{A}_{22}\right)(\ell) + \frac{1}{(1-sr)}\bra{f^{1/s}}\left(\widehat{A}_{11}\cdot E\right)(\ell)
        - \frac{1}{(1-sr)}\bra{f^{1/s}}\left(\widehat{A}_{12}\cdot\widehat{A}_{12}\right)(\ell).
    \end{aligned}
\end{equation}

\begin{remark}\label{StoRlimit}
    When $r=s<1$, the two-parameter stationary LPP model is still well-defined. It would be interesting to find the $s\rightarrow r$ limit of our formula. Our formula does not apply when $r = 1/s$. Likewise, the case $s = 1$ (Maximal current) cannot be easily obtained as limits of $F_{r,s}^{HD}$ and will be treated by a separate argument in section 9.
\end{remark}

\subsection{Asymptotic limits under High density phase}\label{HighAsymptotic}
\begin{defin}\label{thm:limit}
    Let $\tilde{s}, \tilde{r}, \tilde{d} \in \R$ be parameters such that $\tilde{s} <0$ and $\tilde{r} \in (-\infty, \tilde{s}) \cup (\tilde{s}, -\tilde{s})$. We define \begin{equation}\label{scalingDis}
    \begin{aligned}
        \mathbf{\mathcal{L}}(\tilde{d}) := \mathrm{Pf}(J - \widetilde{K})\Bigg(\widetilde{\mu}_{\tilde{d}} + \widetilde{\nu}_{\tilde{d}} +& \widetilde{\mathcal{A}}_{\tilde{d}} + \widetilde{\mathcal{B}}_{\tilde{d}} + \widetilde{\mathcal{C}}_{\tilde{d}} + \widetilde{\mathcal{D}}_{\tilde{d}} + \widetilde{\mathcal{E}}_{\tilde{d}}
        +\tilde{d} -  2\tilde{s}^2
        - \frac{2\tilde{r}}{\tilde{s}^2 - \tilde{r}^2} \Bigg)\\
        &- \mathrm{Pf}(J - \widetilde{K})
    + 
    \mathrm{Pf}\left(J - \widetilde{K} - \ket{\begin{array}{c}
         \widetilde{g}_1 \\
         -\widetilde{d}_2
    \end{array}}\bra{\widetilde{V}_1 \quad \widetilde{V}_2} - \ket{\begin{array}{c}
         \widetilde{V}_1 \\
         \widetilde{V}_2
    \end{array}}\bra{-\widetilde{g}_1 \quad \widetilde{d}_2}\right)
    \end{aligned}
    \end{equation}
    and the Fredholm Pfaffian is taken over $L^2((\tilde{d},\infty)).$
\end{defin}

We explain all the function used in $\eqref{scalingDis}.$ All brakets $\braket{\cdot}{\cdot}$ and $\brabarket{\cdot}{\cdot}{\cdot}$ are inner products over $L^2((\tilde{d},\infty))$ as defined in \eqref{LInnerProduct}.
The following functions take inputs $X,Y,U$ in $\R$ and $\tilde{y}, \tilde{d} \in \R.$
\begin{equation}\label{fandE}
    \widetilde{f}^{\tilde{y}}(U) = e^{-\frac{\tilde{y}^3}{3} + \tilde{y}U}, \quad \widetilde{E}(X,Y) = \begin{cases}
            -e^{\tilde{r}(X-Y)} &\text{ if } X > Y,\\
            0 &\text{ if } X = Y,\\
             e^{\tilde{r}(Y-X)} &\text{ if } X < Y.
    \end{cases}  
\end{equation}

\begin{equation}
    \begin{aligned}
        \widetilde{G}_{\tilde{y}}(X) = -\int\limits_{{}_{0,\pm\tilde{y}}\zcd\, {}_{}}\! \frac{d\zeta}{2\pi\I}e^{\frac{\zeta^3}{3} - X\zeta}\frac{(\zeta - \tilde{y})(\zeta - \tilde{r})}{(\zeta + \tilde{y})(2\zeta)}, \quad \widetilde{R}_{\tilde{y}}(X) = \int\limits_{{}_{}\wcu\, {}_{0,\tilde{r},\tilde{y}}}\! \frac{d\omega}{2\pi\I} e^{-\frac{\omega^3}{3} + X\omega}\frac{(\omega + \tilde{y})}{(\omega - \tilde{y})(\omega - \tilde{r})}.
    \end{aligned}
\end{equation}

\begin{equation}
    \begin{aligned}
        &\widetilde{P}(X) = -\int\limits_{{}_{}\wcu\, {}_{0,\tilde{s}}}\! \frac{d\omega}{2\pi\I} \frac{e^{-\frac{\omega^3}{3} + X\omega}}{(\omega + \tilde{r})}, \quad \widetilde{Q}(X) = -\int\limits_{{}_{0,-\tilde{s}}\zcd\, {}_{}}\! \frac{d\zeta}{2\pi\I} e^{\frac{\zeta^3}{3} - X\zeta}\frac{(\zeta + \tilde{r})}{2\zeta},\\
        &\widetilde{B}(X,Y) = \int\limits_{{}_{}\wcu\, {}_{\tilde{r},\zeta,0}}\! \frac{d\omega}{2\pi\I}\!\int\limits_{{ }_{0,-\tilde{r}, -\tilde{s}}\zcd } \!\frac{d\zeta}{2\pi\I} e^{\frac{\zeta^3}{3} - \frac{\omega^3}{3}  - X\zeta + Y\omega}\frac{(\zeta + \omega)}{(\zeta - \omega)(\zeta + \tilde{r})(\omega - \tilde{r})}.
    \end{aligned}
\end{equation}

\begin{equation}
\begin{aligned}
    \widetilde{A}_{11}(X,Y) = &-  \!\int\limits_{{}_{}\wcu\, {}_{0,\zeta,\tilde{s}}}\! \frac{d\omega}{2\pi\I}\!\int\limits_{{ }_{0,-\tilde{s}}\zcd } \!\frac{d\zeta}{2\pi\I}e^{\frac{\zeta^3}{3} - \frac{\omega^3}{3}  - X\zeta + Y\omega} \frac{(\zeta + \omega)(\zeta - \tilde{r})(\omega + \tilde{r})}{4\zeta\omega(\zeta - \omega)}, \\
    \widetilde{A}_{12}(X,Y) =& -\widetilde{A}_{21}(Y,X) = -  \!\int\limits_{{}_{}\wcu\, {}_{0,\zeta,\tilde{r}}}\! \frac{d\omega}{2\pi\I}\!\int\limits_{{ }_{0, -\tilde{s}}\zcd } \!\frac{d\zeta}{2\pi\I} e^{\frac{\zeta^3}{3} - \frac{\omega^3}{3}  - X\zeta + Y\omega} \frac{(\zeta + \omega)(\zeta - \tilde{r})}{(2\zeta)(\omega - \tilde{r})(\zeta - \omega)} + \ketbra{\widetilde{Q}}{\widetilde{f}^{\tilde{r}}}(X,Y),\\
    \widetilde{A}_{22}(X,Y) = & \ketbra{\widetilde{f}^{\tilde{r}}}{\widetilde{P}}(X,Y) - \ketbra{\widetilde{P}}{\widetilde{f}^{\tilde{r}}}(X,Y) + \widetilde{B}(X,Y).
    \end{aligned}
\end{equation}

\begin{equation}
    \begin{aligned}
        \widetilde{\mathcal{G}} = J^{-1}\widetilde{K} = \begin{pmatrix}
            -\widetilde{K}_{21} & \widetilde{K}_{22}\\
            \widetilde{K}_{11} & \widetilde{K}_{12}
        \end{pmatrix}, \quad \widetilde{K} = \begin{pmatrix}
            \widetilde{K}_{11} & \widetilde{K}_{12}\\
            \widetilde{K}_{21} & \widetilde{K}_{22}
        \end{pmatrix},
    \end{aligned}
\end{equation}
where
\begin{equation}
    \begin{aligned}
        \widetilde{K}_{11}(X,Y) = \widetilde{A}_{11}(X,Y),\quad \widetilde{K}_{12}(X,Y) = -\widetilde{K}_{21}(Y,X) =\widetilde{A}_{12}(X,Y), \quad \widetilde{K}_{22}(X,Y) = \widetilde{A}_{22}(X,Y) + \widetilde{E}(X,Y).
    \end{aligned}
\end{equation}

\begin{equation}
    \begin{aligned}
    &\widetilde{\mu}_{\tilde{d}} = \frac{e^{\frac{\tilde{s}^3}{3}- \tilde{d}\tilde{s} }}{(\tilde{s} + \tilde{r})}\!\int\limits_{{}_{}\wcu\, {}_{\tilde{s},0}} \!\frac{d\omega}{2\pi\I} \, e^{- \frac{\omega^3}{3} + \tilde{d}\omega} \frac{(\omega+\tilde{s})(\omega+\tilde{r})}{(\omega - \tilde{s})^2(2\omega)},\\
    &\widetilde{\nu}_{\tilde{d}} = \frac{(\tilde{s} - \tilde{r})}{(2\tilde{s})} e^{\frac{\tilde{s}^3}{3}- \tilde{d}\tilde{s} }\int\limits_{{}_{}\wcu\, {}_{\tilde{s},\tilde{r}}} \!\frac{d\omega}{2\pi\I}\,
    e^{ - \frac{\omega^3}{3} +\tilde{d}\omega}\frac{(\omega+\tilde{s})}{(\omega-\tilde{s})^2(\omega-\tilde{r})} + \frac{(\tilde{s}+ \tilde{r})}{2\tilde{s}(\tilde{s} - \tilde{r})}e^{\frac{\tilde{s}^3}{3} - \frac{\tilde{r}^3}{3} - \tilde{d}\tilde{s} + \tilde{d}\tilde{r}}.\\
    \end{aligned}
\end{equation}

\begin{equation}
\begin{aligned}
        &\widetilde{g}_1(X) = -\int\limits_{{}_{0}\zcd\, {}_{-\tilde{s}}} \!\frac{d\zeta}{2\pi\I} e^{\frac{\zeta^3}{3} - X\zeta}\frac{(\zeta - \tilde{s})(\zeta - \tilde{r})}{2\zeta (\zeta + \tilde{s})} = \widetilde{G}_{\tilde{s}}(X) + (\tilde{s}+\tilde{r})\widetilde{f}^{\tilde{s}}(X),\\
        &\widetilde{d}_2(X) = -\int\limits_{{}_{\tilde{r},\tilde{s}}\wcu\, {}_{}} \!\frac{d\omega}{2\pi\I} e^{-\frac{\omega^3}{3} + X\omega} \frac{(\omega + \tilde{s})}{(\omega - \tilde{s})(\omega - \tilde{r})} = -\frac{2\tilde{s}}{(\tilde{s}- \tilde{r})}\widetilde{f}^{\tilde{s}}(X) + \frac{(\tilde{s} + \tilde{r})}{(\tilde{s} - \tilde{r})}\widetilde{f}^{\tilde{r}}(X) - \widetilde{R}_{\tilde{s}}(X).\\
        &\widetilde{\mathsf{J}}(X) = -\int\limits_{{}_{}\wcu\, {}_{\tilde{s},\tilde{r},0}}\! \frac{d\omega}{2\pi\I}\!\int\limits_{{ }_{0,\omega,-\tilde{s}}\zcd } \!\frac{d\zeta}{2\pi\I} e^{\frac{\zeta^3}{3} - \frac{\omega^3}{3} - X\zeta + \tilde{d}\omega}\frac{(\zeta + \omega)(\zeta - \tilde{r})}{(2\zeta)(\omega - \tilde{s})(\omega - \tilde{r})(\zeta - \omega)}.
    \end{aligned}
\end{equation}

\begin{equation}
\begin{aligned}
    \widetilde{\mathcal{A}}_{\tilde{d}} = &-\frac{(\tilde{s}-\tilde{r})}{2\tilde{s}}e^{\frac{\tilde{s}^3}{3} - \tilde{d}\tilde{s}}\braket{\widetilde{\mathsf{J}}}{\widetilde{d}_2}
    + \frac{e^{\frac{\tilde{s}^3}{3} - \frac{\tilde{r}^3}{3} - \tilde{d}\tilde{s} + \tilde{d}\tilde{r}}}{2\tilde{s}}\braket{\widetilde{Q}}{\widetilde{d}_2} -\frac{1}{2\tilde{s}}\braket{\widetilde{G}_{-\tilde{s}}}{\widetilde{d}_2}.\\
    \widetilde{\mathcal{B}}_{\tilde{d}} =&-\frac{1}{2\tilde{s}}\braket{\widetilde{R}_{-\tilde{s}}}{\widetilde{g}_1} + \frac{(\tilde{s} - \tilde{r})}{(2\tilde{s})(\tilde{s} + \tilde{r})}\braket{\widetilde{f}^{\tilde{r}}}{\widetilde{g}_1}
    -\frac{1}{2\tilde{s}}e^{\frac{\tilde{s}^3}{3} - \frac{\tilde{r}^3}{3} - \tilde{d}\tilde{s} + \tilde{d}\tilde{r}}\braket{\widetilde{P}}{\widetilde{g}_1}\\
    &-\frac{(\tilde{s} - \tilde{r})}{(2\tilde{s})}\brabarket{\widetilde{f}^{-\tilde{s}}}{ \widetilde{B}}{\widetilde{g}_1}
    +\frac{(\tilde{s}- \tilde{r})}{2\tilde{s}}\braket{\widetilde{f}^{-\tilde{s}}}{\widetilde{P}}\braket{\widetilde{f}^{\tilde{r}}}{\widetilde{g}_1}.\\
    \widetilde{\mathcal{C}}_{\tilde{d}} =& -\frac{e^{\frac{\tilde{s}^3}{3} - \tilde{d}\tilde{s}}}{(2\tilde{s})}\int\limits_{{ }_{0, -\tilde{s}}\zcd } \!\frac{d\zeta}{2\pi\I} e^{\frac{\zeta^3}{3}  - \tilde{d}\zeta}\frac{(\zeta-\tilde{s})^2}{(\zeta + \tilde{s})^2(2\zeta)}.\\
            \end{aligned}
    \end{equation}
    \begin{equation}
        \begin{aligned}
    \widetilde{\mathcal{D}}_{\tilde{d}}
    =& -\frac{1}{(\tilde{s} + \tilde{r})}\brabarket{\widetilde{f}^{-\tilde{s}}}{\widetilde{A}_{11}}{\widetilde{d}_2} - \frac{1}{2\tilde{s}}\braket{\widetilde{G}_{-\tilde{s}}}{\widetilde{d}_2}.\\
    \widetilde{\mathcal{E}}_{\tilde{d}} = & -\frac{1}{(\tilde{s} + \tilde{r})}\brabarket{\widetilde{f}^{-\tilde{s}}}{ \widetilde{A}_{12}}{\widetilde{g}_1}
    +\frac{(\tilde{s} - \tilde{r})}{2\tilde{s}(\tilde{s} + \tilde{r})}\braket{\widetilde{f}^{\tilde{r}}}{\widetilde{g}_1} - \frac{1}{2\tilde{s}}\braket{\widetilde{R}_{-\tilde{s}}}{\widetilde{g}_1}.
    \end{aligned}
\end{equation}

For the next two functions $\widetilde{V}_1$ and $\widetilde{V}_2$, we need the following definitions:
\begin{equation}
    \begin{aligned}
        &\bra{\widetilde{f}(X)}{\widetilde{\mathcal{F}}}(X,Y) = \bra{\widetilde{f}}{\widetilde{\mathcal{F}}}(Y) := \int_{\tilde{d}}^{\infty} \widetilde{f}(X) \widetilde{\mathcal{F}}(X,Y) dX,\\
        &\widetilde{\mathcal{H}}(X,Y) = \widetilde{\mathcal{F}}(X,U)\times \widetilde{\mathcal{G}}(U,Y) := \int_{\tilde{d}}^{\infty} \widetilde{\mathcal{F}}(X,U) \widetilde{\mathcal{G}}(U,Y) dU,\\
        &\bra{\widetilde{f}}\widetilde{\mathcal{F}}\times \widetilde{\mathcal{G}}(Y) := \int_{\tilde{d}}^{\infty} \widetilde{f}(X)\widetilde{\mathcal{H}}(X,Y) dX.
    \end{aligned}
\end{equation}

\begin{equation}
    \begin{aligned}
        &\widetilde{V}_1(Y) = -\frac{(\tilde{s}-\tilde{r})}{(2\tilde{s})}e^{\frac{\tilde{s}^3}{3} - \tilde{d}\tilde{s}}\bra{\widetilde{\mathsf{J}}\,}\widetilde{A}_{21}(Y)
        +\frac{e^{\frac{\tilde{s}^3}{3} - \frac{\tilde{r}^3}{3} - \tilde{d}\tilde{s} + \tilde{d}\tilde{r}}}{2\tilde{s}}\bra{\widetilde{Q}}\widetilde{A}_{21}(Y)
        +\frac{e^{\frac{\tilde{s}^3}{3} - \frac{\tilde{r}^3}{3} - \tilde{d}\tilde{s} + \tilde{d}\tilde{r}}}{2\tilde{s}} \bra{\widetilde{P}}\widetilde{A}_{11}(Y)\\
        &-\frac{(\tilde{s}-\tilde{r})}{2\tilde{s}} \braket{\widetilde{f}^{-\tilde{s}}}{\widetilde{P}}\bra{\widetilde{f}^{\tilde{r}}}\widetilde{A}_{11}(Y) + \frac{(\tilde{s}-\tilde{r})}{2\tilde{s}} \bra{\widetilde{f}^{-\tilde{s}}} \widetilde{B}\times\widetilde{A}_{11}(Y)
        -\frac{1}{2\tilde{s}}\widetilde{G}_{-\tilde{s}}(Y)
        + \frac{1}{\tilde{s}}\bra{\widetilde{R}_{-\tilde{s}}}\widetilde{A}_{11}(Y)\\
        &-\frac{(\tilde{s}-\tilde{r})}{\tilde{s}(\tilde{s}+\tilde{r})}\bra{\widetilde{f}^{\tilde{r}}}\widetilde{A}_{11}(Y)
        +\frac{(\tilde{s}-\tilde{r})}{(2\tilde{s})(\tilde{s}+\tilde{r})}\bra{\left(e^{\tilde{r}(X - \tilde{d})}e^{\frac{\tilde{s}^3}{3} - \tilde{d}\tilde{s}} - \frac{2\tilde{s}}{(\tilde{s}-\tilde{r})}e^{\frac{\tilde{s}^3}{3} - X\tilde{s}}\right)}\widetilde{A}_{11}(X,Y)-\frac{1}{\tilde{s}}\bra{\widetilde{G}_{-\tilde{s}}}\widetilde{A}_{21}(Y)\\
        &-\frac{1}{(\tilde{s}+ \tilde{r})}\bra{\widetilde{f}^{-\tilde{s}}}\left(\widetilde{A}_{11}\times\widetilde{A}_{21}\right)(Y) +\frac{1}{(\tilde{s}+\tilde{r})}\bra{\widetilde{f}^{-\tilde{s}}}\left(\widetilde{A}_{12}\times\widetilde{A}_{11}\right)(Y).\\
    \end{aligned}
\end{equation}

\begin{equation}
    \begin{aligned}
        &\widetilde{V}_2(Y) = -\frac{(\tilde{s}-\tilde{r})}{2\tilde{s}}e^{\frac{\tilde{s}^3}{3} - \tilde{d}\tilde{s}}\bra{\widetilde{\mathsf{J}}\,} \left(\widetilde{A}_{22} + \widetilde{E}\right)(Y)
        +\frac{e^{\frac{\tilde{s}^3}{3} - \frac{\tilde{r}^3}{3} -\tilde{d}\tilde{s} + \tilde{d}\tilde{r}}}{2\tilde{s}}\bra{\widetilde{Q}} \left(\widetilde{A}_{22}
        + \widetilde{E}\right)(Y)\\
        &+\frac{e^{\frac{\tilde{s}^3}{3} - \frac{\tilde{r}^3}{3} - \tilde{d}\tilde{s} + \tilde{d}\tilde{r}}}{2\tilde{s}}\bra{ \widetilde{P}}\widetilde{A}_{12}(Y)
        - \frac{(\tilde{s} - \tilde{r})}{2\tilde{s}}\braket{\widetilde{f}^{-\tilde{s}}}{\widetilde{P}}\bra{\widetilde{f}^{\tilde{r}}}\widetilde{A}_{12}(Y) + \frac{(\tilde{s} - \tilde{r})}{2\tilde{s}}\bra{\widetilde{f}^{-\tilde{s}}}\widetilde{B}\times \widetilde{A}_{12}(Y)\\
        &+\frac{1}{\tilde{s}}\bra{\widetilde{R}_{-\tilde{s}}}\widetilde{A}_{12}(Y) - \frac{(\tilde{s} - \tilde{r})}{\tilde{s}(\tilde{s} + \tilde{r})}\bra{\widetilde{f}^{\tilde{r}}}\widetilde{A}_{12}(Y) -\frac{1}{\tilde{s}}\bra{\widetilde{G}_{-\tilde{s}}}\left(\widetilde{A}_{22} + \widetilde{E}\right)(Y)\\
        &+\frac{(\tilde{s} - \tilde{r})}{(\tilde{s} + \tilde{r})(2\tilde{s})}\bra{\left(e^{\tilde{r}(X-\tilde{d})}e^{\frac{\tilde{s}^3}{3} - \tilde{d}\tilde{s}}- \frac{2\tilde{s}}{(\tilde{s}-\tilde{r})} e^{\frac{\tilde{s}^3}{3} - X\tilde{s}}\right)}\widetilde{A}_{12}(X,Y)
        +\frac{(\tilde{s}-\tilde{r})}{(\tilde{s} + \tilde{r})(2\tilde{s})}\widetilde{f}^{\tilde{r}}(Y) - \frac{1}{2\tilde{s}}\widetilde{R}_{-\tilde{s}}(Y)\\
        &-\frac{1}{(\tilde{s}+\tilde{r})}\bra{\widetilde{f}^{-\tilde{s}}}\left(\widetilde{A}_{11} \times\widetilde{A}_{22}\right)(Y) - \frac{1}{(\tilde{s} + \tilde{r})}\bra{\widetilde{f}^{-\tilde{s}}}\left(\widetilde{A}_{11}\times \widetilde{E}\right)(Y) 
        + \frac{1}{(\tilde{s} + \tilde{r})}\bra{\widetilde{f}^{-\tilde{s}}}\left(\widetilde{A}_{12}\times\widetilde{A}_{12}\right)(Y).
    \end{aligned}
\end{equation}

\section{Definition of distribution functions for Maximal current phase}
\subsection{Finite time diagonal distribution for Maximal current phase}
\begin{defin}\label{MaximalFinite}
    Let $\sqrt{q}\in (0,1)$, $r\in (0,1).$ Let $N \in \Z_{\geq 2}$ and $d\in \Z$. Define
    \begin{equation}
        \eta(d,N) := \mathrm{Pf}\left(J - \overline{\mathsf{K}}^{M}\right)\left(e_{M}^{1,r}(d) - 1\right) + \mathrm{Pf}\left(J - \overline{\mathsf{K}}^{M} - \ketbra{\begin{array}{c}
            \phi_2^{M}\\\phi_1^{M}
        \end{array}}{-g_1^{M} \quad g_2^{M}} - \ketbra{\begin{array}{c}
            -g_1^{M}\\
            g_2^{M}
        \end{array}}{-\phi_2^{M} \,\,\, -\phi_1^{M}}\right),
    \end{equation}
    for $d \geq 0$ and $\eta(d,N) = 0$ for $d<0$. The Pfaffian is taken over ${\ell^{2}(\{d+1,\dots\}}).$
\end{defin}

We now define all functions used in the above definition. The following functions take inputs $k,\ell \in \Z_{\geq 0}$ and the input $x\in\R$.
\begin{equation}
     H(x) = \left(\frac{1-\sqrt{q}/x}{1-\sqrt{q}x}\right)^{N-2},\quad e_{M}^{1,r}(d) =\oint \limits_{\Gamma_{\sqrt{q},1,r}} \frac{dw}{2\pi\I} \frac{w^{d+2}(1-wr)}{H(w)(w-1)^2(w-r)}, \quad \mathsf{h}_{M}(d,k) = 
        2k-2d-1.
\end{equation}

\begin{equation}
    \begin{aligned}
         &\mathsf{g}_4^{M}(k) = \!\!\!\!\oint \limits_{\Gamma_{1/\sqrt{q}, 1/r,1} }\!\!\!\!\!\!\frac{dz}{2\pi\I}\frac{H(z)(z-r)(z+1)}{z^{k+2}(1-zr)(z-1)^2}, \quad \mathsf{g}_3^{M}(k) = \oint \limits_{\Gamma_{1/\sqrt{q},1/r}}
        \frac{dz}{2\pi\I}\frac{(z-r)H(z)}{z^{k+2}(1-zr)(z-1)}.
    \end{aligned}
\end{equation}
\begin{equation}
    \begin{aligned}
        &g_1^{M}(k) = \oint \limits_{\Gamma_{1/\sqrt{q},1/r}} \frac{dz}{2\pi\I} \frac{H(z)}{z^{k+2}}\frac{(z-r)}{(1-rz)(z+1)},\quad
        g_2^{M}(k) =  \oint \limits_{\Gamma_{\sqrt{q},1,r}} \frac{dw}{2\pi\I} \frac{w^{k+1}}{H(w)}\frac{(1-wr)}{(w-r)(w-1)},
    \end{aligned}
\end{equation}

\begin{equation}
    \begin{aligned}
        &{\Omega}_1^{M}(k) = \frac{1}{(2\pi\I)^2} \oint \limits_{\Gamma_{\sqrt{q},r}} dw \oint \limits_{\Gamma_{1/\sqrt{q},1/r,1}} dz\frac{w^{d+1}}{z^{k+1}}\frac{H(z)R(z)}{H(w)R(w)}\frac{(zw-1)}{(1-z)(w-1)^2(z-w)},\\
        &{\Omega}_2^{M}(k) = \frac{1}{(2\pi\I)^2}\oint \limits_{\Gamma_{\sqrt{q},r}} dw \oint \limits_{\Gamma_{1/\sqrt{q},1/r}}dz\frac{w^{d+1}}{z^{k+1}}\frac{H(z)R(z)}{H(w)R(w)}\frac{(zw-1)}{(w-1)^2(z+1)(z-w)}.
    \end{aligned}
\end{equation}

\begin{equation}
        \begin{aligned}
            \phi_1^{M}(k) = {\Omega}_1^{M}(k) - \mathsf{g}_4^{M}(k) - \mathsf{h}_{M}(d,k), \quad\quad
            \phi_2^{M}(k) = {\Omega}_2^{M}(k) + \mathsf{g}_3^{M}(k).
        \end{aligned}
\end{equation}

\begin{equation}\label{one-param}
    \begin{aligned}
            \overline{\mathsf{K}}_{11}^{M}(k,\ell) &:= \frac{-1}{(2\pi\I)^2} \oint \limits_{\Gamma_{\sqrt{q},r}} dw \oint \limits_{\Gamma_{1/\sqrt{q},1/r}} dz \frac{w^{\ell}}{z^{k+1}} \frac{H(z)R(z)}{H(w)R(w)} \frac{(zw-1)}{(z+1)(1+w)(z-w)},\\
            \overline{\mathsf{K}}_{12}^{M}(k,\ell) &:= \frac{-1}{(2\pi\I)^2} \oint\limits_{\Gamma_{\sqrt{q},r,1}}\!\!\! dw \!\!\!\oint\limits_{\Gamma_{1/\sqrt{q},1/r}}\!\!\!dz \frac{w^{\ell}}{z^{k+1}} \frac{H(z)R(z)}{H(w)R(w)}\frac{(zw-1)}{(w-1)(z+1)(z-w)},\\
                \end{aligned}
\end{equation}
\begin{equation}
    \begin{aligned}
\overline{\mathsf{K}}_{22}^{M}(k,\ell) &:= \frac{-1}{(2\pi\I)^2} \oint \limits_{\Gamma_{\sqrt{q}}} dw \oint \limits_{\Gamma_{1/\sqrt{q},1/r,1}} dz \,\overline{h}_{22}^{M}(z,w)
    +\frac{-1}{(2\pi\I)^2} \oint \limits_{\Gamma_{r}} dw \oint \limits_{\Gamma_{1/\sqrt{q},1}} dz \, \overline{h}_{22}^{M}(z,w)\\
    &
    + \frac{-1}{(2\pi\I)^2} \oint \limits_{\Gamma_{1}} dw \oint \limits_{\Gamma_{1/\sqrt{q},1/r}} dz\, \overline{h}_{22}^{M}(z,w) -\sgn(k-\ell),
        \end{aligned}
    \end{equation}
    where $$\overline{h}_{22}^{M}(z,w) = \frac{w^{\ell}H(z)R(z)(zw-1)}{z^{k+1}H(w)R(w)(1-z)(w-1)(z-w)}, \quad R(z) = \frac{(1-r/z)}{(1-rz)}.$$

\subsection{Asymptotic limits under Maximal current phase}
\begin{defin}\label{MaximalAsymptotic}
    Let $\sqrt{q}\in (0,1)$, $\tilde{r}< 0,$ and $\tilde{d}\in \R$ be parameters. Consider the scaling
    \begin{equation}
        r = 1 + (c_0N/2)^{-1/3}\tilde{r}.
    \end{equation}
    \begin{equation}
    \begin{aligned}
        \Phi(\tilde{d}) := &\mathrm{Pf}(J - \widetilde{\mathpzc{K}}^{M})\left(\widetilde{\mathpzc{e}}^{\tilde{r}}_{M}(\tilde{d}) - 1\right)
        +\mathrm{Pf}\left(J - \widetilde{\mathpzc{K}}^M - \ketbra{\begin{array}{c} \widetilde{\phi}_2^{M} \\ \widetilde{\phi}_1^{M} \end{array}}{-\widetilde{g}_1^{M}\,\,\,\widetilde{g}_2^{M}} - \ketbra{\begin{array}{c}
            -\widetilde{g}_1^{M} \\
            \widetilde{g}_2^{M}
        \end{array}}{-\widetilde{\phi}_2^{M}\,\,\, -\widetilde{\phi}_1^{M}} \right).
    \end{aligned}
    \end{equation}
    The Fredholm Pfaffian is taken over $L^{2}((\tilde{d}, \infty))$.
\end{defin}

We define all functions used in the above definition. The following functions take inputs $X,Y,U$ in $\R$ and $\tilde{y}, \tilde{d} \in \R.$
\begin{equation}
    \begin{aligned}
         \widetilde{f}_{M}^{\tilde{y}}(U) = e^{-\frac{\tilde{y}^3}{3} + \tilde{y}U}, \quad \widetilde{\mathpzc{e}}_{M}^{\tilde{r}}(\tilde{d}) = \int\limits_{{}_{\tilde{r},0}\wcu\, {}_{}}\! \frac{d\omega}{2\pi\I} e^{-\frac{\omega^3}{3} + \tilde{d}\omega}\frac{-(\omega+\tilde{r})}{\omega^2(\omega - \tilde{r})}, \quad\widetilde{\mathsf{h}}_{M}(\tilde{d}, X) = 2X- 2\tilde{d}.
    \end{aligned}
\end{equation}

\begin{equation}
    \begin{aligned}
        \widetilde{\mathsf{g}}_4^{M}(X) = -\!\!\!\int\limits_{{ }_{}\zcd {}_{0, -\tilde{r}}} \!\!\!\frac{d\zeta}{2\pi\I} e^{\frac{\zeta^3}{3} - X\zeta} \frac{2(\zeta - \tilde{r})}{(\zeta + \tilde{r})\zeta^2}, \quad \widetilde{\mathsf{g}}_3^{M}(X) = -\!\int\limits_{{ }_{0}\zcd {}_{-\tilde{r}}} \frac{d\zeta}{2\pi\I} e^{\frac{\zeta^3}{3} - X\zeta} \frac{(\zeta - \tilde{r})}{(\zeta + \tilde{r})\zeta}.
    \end{aligned}
\end{equation}

\begin{equation}
    \begin{aligned}
        \widetilde{g}_1^{M}(X) = -\int\limits_{{ }_{0}\zcd {}_{ -\tilde{r}}} \!\frac{d\zeta}{2\pi\I} e^{\frac{\zeta^3}{3} - X\zeta}\frac{(\zeta - \tilde{r})}{2(\tilde{r}+\zeta)}, \quad \widetilde{g}_2^{M}(X) = -\int\limits_{{ }_{\tilde{r}, 0}\wcu {}_{}} \!\frac{d\omega}{2\pi\I} e^{-\frac{\omega^3}{3} + X\omega}\frac{(\omega + \tilde{r})}{(\omega - \tilde{r})\omega}.
    \end{aligned}
\end{equation}

\begin{equation}
\begin{aligned}
    &\widetilde{\phi}_1^{M}(X) = -\!\int\limits_{{}_{\tilde{r}}\wcu\, {}_{0,\zeta}}\! \frac{d\omega}{2\pi\I}\!\int\limits_{{ }_{\omega}\zcd { }_{0,-\tilde{r}}} \!\frac{d\zeta}{2\pi\I} e^{\frac{\zeta^3}{3} - \frac{\omega^3}{3} - X\zeta + \tilde{d}\omega} \frac{(\zeta - \tilde{r})}{(\zeta + \tilde{r})}\frac{(\omega+ \tilde{r})}{(\omega - \tilde{r})} \frac{(\zeta + \omega)}{\zeta\omega^2(\zeta - \omega)} - \widetilde{\mathsf{g}}_4^{M}(X) - \widetilde{\mathsf{h}}_{M}(\tilde{d},X),\\
    &\widetilde{\phi}_2^{M}(X) = \!\int\limits_{{}_{\tilde{r}}\wcu\, {}_{0,\zeta}}\! \frac{d\omega}{2\pi\I}\!\int\limits_{{ }_{\omega,0}\zcd {}_{-\tilde{r}}} \!\frac{d\zeta}{2\pi\I} e^{\frac{\zeta^3}{3} - \frac{\omega^3}{3} - X\zeta + Y\omega} \frac{(\zeta - \tilde{r})}{(\zeta + \tilde{r})}\frac{(\omega+ \tilde{r})}{(\omega - \tilde{r})} \frac{(\zeta + \omega)}{2\omega^2(\zeta - \omega)} + \widetilde{\mathsf{g}}_3^{M}(X).
\end{aligned}
\end{equation}

\begin{equation}
    \begin{aligned}
        \widetilde{\mathpzc{K}}_{11}^M(X,Y) &= -\int\limits_{{}_{\tilde{r}}\wcu\, {}_{0}}\! \frac{d\omega}{2\pi\I}\!\int\limits_{{ }_{0}\zcd {}_{-\tilde{r}}} \!\frac{d\zeta}{2\pi\I} e^{\frac{\zeta^3}{3} - \frac{\omega^3}{3} - X\zeta + Y\omega} \frac{(\zeta - \tilde{r})(\omega + \tilde{r})}{(\zeta + \tilde{r})(\omega - \tilde{r})}\frac{(\zeta + \omega)}{4(\zeta - \omega)}\\
        \widetilde{\mathpzc{K}}_{12}^M(X,Y) &= -\int\limits_{{}_{\tilde{r}, 0}\wcu\, {}_{\zeta}}\! \frac{d\omega}{2\pi\I}\!\int\limits_{{ }_{0,\omega}\zcd {}_{-\tilde{r}}} \!\frac{d\zeta}{2\pi\I} e^{\frac{\zeta^3}{3} - \frac{\omega^3}{3} - X\zeta + Y\omega} \frac{(\zeta - \tilde{r})(\omega + \tilde{r})}{(\zeta + \tilde{r})(\omega - \tilde{r})}\frac{(\zeta + \omega)}{2\omega(\zeta - \omega)} \\
\end{aligned}
\end{equation}

\begin{equation}
    \begin{aligned}
        &\widetilde{\mathpzc{K}}_{22}^M(X,Y) = \int\limits_{{}_{}\wcu\, {}_{\zeta,0}}\! \frac{d\omega}{2\pi\I}\!\int\limits_{{ }_{\omega}\zcd {}_{-\tilde{r},0}} \!\frac{d\zeta}{2\pi\I}\, e^{\frac{\zeta^3}{3} - \frac{\omega^3}{3} - X\zeta + Y\omega} \frac{(\zeta - \tilde{r})(\omega + \tilde{r})}{(\zeta + \tilde{r})(\omega - \tilde{r})}\frac{(\zeta + \omega)}{\zeta\omega(\zeta - \omega)}-\int\limits_{{ }_{0}\zcd {}_{-\tilde{r}}} \!\frac{d\zeta}{2\pi\I}\, e^{\frac{\zeta^3}{3} - X\zeta}\frac{(\zeta - \tilde{r})}{\zeta(\zeta + \tilde{r})} \\ &+\int\limits_{{}_{\tilde{r}}\wcu\, {}_{0,\zeta}}\! \frac{d\omega}{2\pi\I}\!\int\limits_{{ }_{\omega,0}\zcd {}_{}} \!\frac{d\zeta}{2\pi\I}\, e^{\frac{\zeta^3}{3} - \frac{\omega^3}{3} - X\zeta + Y\omega} \frac{(\zeta - \tilde{r})(\omega + \tilde{r})}{(\zeta + \tilde{r})(\omega - \tilde{r})}\frac{(\zeta + \omega)}{\zeta\omega(\zeta - \omega)}+2e^{-\frac{\tilde{r}^3}{3} + Y\tilde{r}}
         - \sgn(X - Y).
    \end{aligned}
\end{equation}

\section{Definitions of distribution functions for Low density phase}
\subsection{Finite time diagonal distribution for Low density phase}
\begin{defin}\label{oneParamResult_Finite}
    Let $\sqrt{q}\in (0,1)$, $r\in (\sqrt{q},1/\sqrt{q}).$ Let $N \in \Z_{\geq 1}$. We define
    \begin{equation}
    \begin{aligned}
        \theta(d,N) := \mathrm{Pf}\left(J - \overline{\mathsf{K}}^{L}\right)\left(e_{L}^{r}(d) - 1\right)+ \mathrm{Pf}\left(J - \overline{\mathsf{K}}^{L}
        - \ketbra{\begin{array}{c}
            \phi_2^{\text{L}}\\\phi_1^{\text{L}}
        \end{array}}{-g_1^{\text{L}} \quad g_2^{\text{L}}} - \ketbra{\begin{array}{c}
            -g_1^{\text{L}}\\
            g_2^{\text{L}}
        \end{array}}{-\phi_2^{\text{L}} \,\,\, -\phi_1^{\text{L}}}\right),
    \end{aligned}
    \end{equation}
    for $d\in \Z_{\geq 0}$ and $\theta(d,N) = 0$ for $d<0$. The Fredholm Pfaffian is taken over ${\ell^{2}(\{d+1,\dots\}}).$
\end{defin}

We now define all functions used in the above definition. The following functions take inputs $k,\ell \in \Z_{\geq 0}$ and the input $x\in\R$.
\begin{equation}
     \overline{H}(x) = \left(\frac{1-\sqrt{q}/x}{1-\sqrt{q}x}\right)^{N-1},\quad f_{L}^x(k) = \frac{x^{k}}{\overline{H}(x)},\quad e_{L}^r(d) =  \frac{\overline{H}(r)}{r^{d}}\oint \limits_{\Gamma_{\sqrt{q},r}}\frac{dw}{2\pi\I} \frac{w^{d+1}}{\overline{H}(w)(w-r)^2},
\end{equation}
\begin{equation}\begin{aligned}
{\Omega}_1^{L}(k) &= \frac{-1}{(2\pi\I)^2}\frac{\overline{H}(r)}{r^{d}} \oint \limits_{\Gamma_{\sqrt{q}}} dw \oint \limits_{\Gamma_{1/\sqrt{q},r}} dz \frac{w^{d+1}}{z^{k+1}} \frac{\overline{H}(z)}{\overline{H}(w)}\frac{(zw-1)}{(r-w)(r-z)(wr-1)(z-w)},\\
{\Omega}_2^{L}(k) &= \frac{-1}{(2\pi\I)^2} \frac{\overline{H}(r)}{r^d}\oint \limits_{\Gamma_{\sqrt{q}}} dw \oint \limits_{\Gamma_{1/\sqrt{q}}} dz \frac{w^{d+1}}{z^{k+1}} \frac{\overline{H}(z)}{\overline{H}(w)}\frac{(zw-1)(zr-1)}{(r-w)(z^2-1)(wr-1)(z-w)}.\\
    \end{aligned}
\end{equation}
\begin{equation}
    \begin{aligned}
         &\mathsf{g}_{4}^{L}(k) = -\oint \limits_{\Gamma_{1/\sqrt{q}, 1/r,r} }\frac{dz}{2\pi\I}\frac{r\overline{H}(z)(z^2-1)}{z^{k+1}(z-r)^2(1-rz)}, \quad \mathsf{g}_{3}^{L}(k) = \oint \limits_{\Gamma_{1/\sqrt{q}}}
        \frac{dz}{2\pi\I}\frac{r\overline{H}(z)}{z^{k+1}(z-r)},\\
        &\mathsf{h}_{L}(d,k) = \begin{cases}
            \overline{H}(r)\frac{r^{k-2d-1} - r^{-k+1}}{(r^2-1)} + r^{-k-1}(k-d)\overline{H}(r) &\text{if }  r\neq 1,\\
            2k-2d-1 &\text{if } r = 1.
        \end{cases}
    \end{aligned}
\end{equation}
\begin{equation}
    \begin{aligned}
        g_1^{L}(k) = \frac{1}{2\pi\I} \oint \limits_{\Gamma_{1/\sqrt{q}}} \frac{\overline{H}(z)(z-r)}{z^{k+1}(z^2-1)} dz,\quad
    g_2^{L}(k) = \frac{1}{2\pi\I} \oint \limits_{\Gamma_{\sqrt{q},r}} \frac{w^k}{\overline{H}(w)(w-r)} dw,
    \end{aligned}
\end{equation}
\begin{equation}
        \begin{aligned}
            \phi_1^{L}(k) &=  \Omega_1^{L}(k) - \mathsf{g}_{4}^{L}(k) - \mathsf{h}_{L}(d,k), \quad\quad
            \phi_2^{L}(k) = \Omega_2^{L}(k) + \mathsf{g}_{3}^{L}(k).\\
        \end{aligned}
\end{equation}

        \begin{equation}
        \begin{aligned}
            \overline{\mathsf{K}}_{11}^{L}(k,\ell) &:= \frac{-1}{(2\pi\I)^2} \oint \limits_{\Gamma_{\sqrt{q}}} dw \oint \limits_{\Gamma_{1/\sqrt{q}}} dz \frac{w^{\ell}}{z^{k+1}} \frac{\overline{H}(z)}{\overline{H}(w)} \frac{(rz-1)(r-w)(zw-1)}{(z^2-1)(1-w^2)(z-w)},\\
            \overline{\mathsf{K}}_{12}^{L}(k,\ell) &:= \frac{-1}{(2\pi\I)^2} \oint\limits_{\Gamma_{\sqrt{q},1/r}}\!\!\! dw \!\!\!\oint\limits_{\Gamma_{1/\sqrt{q}}}\!\!\!dz \frac{w^{\ell}}{z^{k+1}} \frac{\overline{H}(z)}{\overline{H}(w)}\frac{(zr-1)(zw-1)}{(wr-1)(z^2-1)(z-w)},\\
            \end{aligned}
    \end{equation}   
    \begin{equation}
    \begin{aligned}
    \overline{\mathsf{K}}_{22}^{L}(k,\ell) &:= \frac{-1}{(2\pi\I)^2} \oint \limits_{\Gamma_{\sqrt{q}}} dw \oint \limits_{\Gamma_{r,1/\sqrt{q}}} dz \frac{w^{\ell}\overline{H}(z)(zw-1)}{z^{k+1}\overline{H}(w)(r-z)(wr-1)(z-w)}\\ &
            + \frac{-1}{(2\pi\I)^2} \oint \limits_{\Gamma_{1/r}} dw \oint \limits_{\Gamma_{1/\sqrt{q}}} dz \frac{w^{\ell}\overline{H}(z)(zw-1)}{z^{k+1}\overline{H}(w)(r-z)(wr-1)(z-w)} -\sgn(k-\ell) r^{-|k-\ell|-1}.
        \end{aligned}
    \end{equation}

\subsection{Asymptotic limits under Low density phase}\label{LowAsymptotic}
\begin{defin}
    Let $\sqrt{q}\in (0,1)$, $\tilde{r}, \tilde{d}\in \R$ be parameters. Consider the scaling
    \begin{equation}
        r = 1 + (c_0N/2)^{-1/3}\tilde{r}.
    \end{equation}
    \begin{equation}
    \begin{aligned}
        \Xi(\tilde{d}) := &\mathrm{Pf}(J - \overline{\mathcal{A}})\mathpzc{e}^{-\tilde{r},0}(\tilde{d}) - \mathrm{Pf}(J - \overline{\mathcal{A}})\\
        &+\mathrm{Pf}\left(J - \overline{\mathcal{A}} - \ketbra{\begin{array}{c} \mathpzc{h}_2^{-\tilde{r},0} \\ \widetilde{\mathsf{h}}_{1,L}^{-\tilde{r},0} \end{array}}{-\mathpzc{g}_1^{-\tilde{r},0} \,\,\,\mathpzc{g}_2^{-\tilde{r},0}} - \ketbra{\begin{array}{c}
            -\mathpzc{g}_1^{-\tilde{r},0} \\\mathpzc{g}_2^{-\tilde{r},0}
        \end{array}}{-\mathpzc{h}_2^{-\tilde{r},0}\,\,\, -\widetilde{\mathsf{h}}_{1,L}^{-\tilde{r},0}} \right).
    \end{aligned}
    \end{equation}
    The Pfaffian is taken over $L^{2}((\tilde{d}, \infty))$.
\end{defin}
In the above definition, we set $\overline{\mathcal{A}} = \begin{pmatrix}
        \overline{\mathcal{A}}_{11} & \overline{\mathcal{A}}_{12}\\
        \overline{\mathcal{A}}_{21} &\overline{\mathsf{A}}_{22}^{L}
    \end{pmatrix}$. All functions used in the definition, $\overline{\mathcal{A}}_{11}$, $\overline{\mathcal{A}}_{12}$, $\overline{\mathcal{A}}_{21}$, $\mathpzc{e}^{-\tilde{r},0}$, $\mathpzc{g}_1^{-\tilde{r},0}$, $\mathpzc{g}_2^{-\tilde{r},0},$ $\mathpzc{h}_2^{-\tilde{r},0}$, are identical to \cite[(2.25), (2.26), (2.27), (2.30)]{Betea_2020} with $\delta = -\tilde{r}$, $u=0$ and $S = \tilde{d},$ where $\delta, u, S$ are variables used there. However, two functions differ in our setting, which we define as
\begin{equation}
\begin{aligned}
\overline{\mathsf{A}}_{22}^{L}(X,Y) &:=  \int\limits_{{}_{}\wcu\, {}_{-\tilde{r},\zeta,0}}\! \frac{d\omega}{2\pi\I}\!\int\limits_{{ }_{}\zcd { }_{\tilde{r}}} \!\frac{d\zeta}{2\pi\I} e^{\frac{\zeta^3}{3} - \frac{\omega^3}{3}  - X\zeta + Y\omega} \frac{(\zeta + \omega)}{(\zeta - \tilde{r})(\omega + \tilde{r})(\zeta - \omega)}\\
&+\int\limits_{{}_{-\tilde{r}}\wcu\, {}_{}}\! \frac{d\omega}{2\pi\I}\!\int\limits_{{ }_{0,\tilde{r}, \omega}\zcd } \!\frac{d\zeta}{2\pi\I} e^{\frac{\zeta^3}{3} - \frac{\omega^3}{3}  - X\zeta + Y\omega} \frac{(\zeta + \omega)}{(\zeta - \tilde{r})(\omega + \tilde{r})(\zeta - \omega)} -\sgn(X-Y)e^{-\tilde{r}|X-Y|},\\
\widetilde{\mathsf{h}}_{1,L}^{-\tilde{r},0}(X) &:= \int_{\tilde{d}}^{\infty} \widetilde{\mathcal{A}}_{22}(X,V)\mathpzc{f}^{\tilde{r},0}(V) dV - \mathpzc{g}_4^{-\tilde{r},0}(X) - \mathpzc{j}^{-\tilde{r},0}(\tilde{d},X),
\end{aligned}
\end{equation}
where $\mathpzc{g}_4^{-\tilde{r},0},$ $\mathpzc{f}^{\tilde{r},0},$ $\mathpzc{j}^{-\tilde{r},0}$ are defined in \cite[(2.26)]{Betea_2020} with the same parameter choices: $\delta = -\tilde{r},$ $u = 0$ and $S = \tilde{d}.$ 
The reasons for these differences will be explained in Lemma \ref{KernelPtwiseLimit} and Lemma \ref{FunctionPtwiseLimit}. We make it clear that under the special case $\tilde{r} = 0,$ $\mathpzc{j}^{0,0} = 2X  -2\tilde{d}$.

\section{Finite time diagonal distribution of the stationary LPP under High density phase}

\subsection{Inhomogeneous half-space geometric last passage percolation model}\label{LPPdefinition}
We provide another definition (path version) of the half-space geometric last passage percolation (LPP) model.
Let $q_0, q_1,q_2,
\dots$ be positive real parameters such that $q_0q_i \in (0,1)$ for all $i \geq 1$ and $q_iq_j \in (0,1)$ for all $i\neq j \geq 1.$ We define a set $\mathcal{D} = \{ (i,j) \in \Z^2 | 1 \leq j \leq i\}.$ Let $(\omega_{i,j})_{(i,j)\in \mathcal{D}}$ be independent random variables with $\omega_{i,j} \stackrel{(d)}{=} \text{Geo}(q_iq_j).$ We call $q_0$ the diagonal parameter and $q_i$ the ith-row parameter. An \emph{up-right path} $\pi$ from a point $A$ to a point $E$ is a sequence of points $(\pi(0),\pi(1),\ldots,\pi(n))$ in $\Z^2$ such that \mbox{$\pi(k+1)-\pi(k)\in \{(0,1),(1,0)\}$}, with $\pi(0)=A$ and $\pi(n)=E$, and where $n$ is called the length $\ell(\pi)$ of $\pi$.
We define the half-space last passage time from $(1,1)$ to the point $(N,M)$ for $N\geq M$, denoted as $\mathbf{G}(N,M)$, to be
\begin{equation}
\mathbf{G}(N,M) = \max_{\pi:(1,1)\to (N,M)} \sum_{(i,j)\in\pi} \omega_{i,j}
\end{equation}
where the maximum is taken over all up-right paths in $\mathcal{D}$ from $(1,1)$ to $(N,M)$. 
To relate this definition to the recurrence relation in the first chapter, $\mathbf{G}(N,M)$ is the solution to \ref{recurrence} with initial condition $G(1,1) = \omega_{1,1}$ and $G(N,1) - G(N-1,1) = \omega_{N,1}$ for $N \geq 2.$

\subsection{Half-space stationary geometric LPP model}
In this section, we analyze the special case within the high-density regime in which the two-parameter stationary model reduces to the product stationary model.

\subsubsection{High density phase with two-parameter stationary initial condition}
We set $q_0 = r$, $q_1 = 1/s$, $q_2 = s$ and $q_i = \sqrt{q}$ for $i\geq 3.$ Notice that $q_1q_2 =1$ would cause $\omega_{2,1}$ to explode. To avoid this, we set $\omega_{2,1} = 0.$ Since all LPP paths pass through $(1,1),$ the weight $\omega_{1,1}$
is not relevant when considering the difference in last passage times. Therefore, we set $\omega_{1,1} =0$, which removes the requirement of $r/s<1.$ This is why we can consider $r\in (s,1/s).$ Then the two-parameter stationary LPP model has the following weights:
\begin{equation} \label{eq:stat_wts}
  \omega_{i, j} = \begin{cases}
      \mathrm{Geo}\left( rs \right), & \textrm{if }i=j=2,\\
      \mathrm{Geo}\left( r\sqrt{q} \right), & \textrm{if }i=j\geq 3,\\
    \mathrm{Geo}\left( \sqrt{q}/s \right), & \textrm{if }j=1, i\geq 3, \\
    \mathrm{Geo}\left( s\sqrt{q} \right), & \textrm{if }j=2, i\geq 3, \\
    0, & \textrm{if}\ i=j=1, \\
    0, & \textrm{if}\ i=2,j=1, \\
    \mathrm{Geo}(q), &\textrm{otherwise},
  \end{cases}
\end{equation}
where $\sqrt{q} \in (0,1)$, $s \in (\sqrt{q}, 1)$, and $r\in (0,1/s)$. This model is shown in Figure $\ref{fig:2paramStationaryModel}.$ We use $G_{r,s}^{\mathrm{stat}}(N,M)$ to denote the last passage time from $(1,1)$ to $(N,M)$. We exclude $s=1$ from the present discussion because our formula $F_{r,s}^{HD}$ does not cover this case. We handle $s=1$ by a separate method in section \ref{Max&LowPhase}.

\begin{figure}
\centering
\begin{tikzpicture}
\draw[very thin, gray] (-0.5,-0.5) grid (4.5,4.5);
    \foreach \x in {1,...,4} {
        \foreach \y in {0,1,...,\x} {
            \fill (\x,\y) circle (2pt);
        }
    }
    \fill (0,0) circle (2pt);
    \draw[thick,-] (0,0) -- (4,0);
    \draw[thick,-] (0,0) -- (4,4);
    \draw[thick,-] (1,1) -- (4,1);
    \fill (0,0) circle (1pt);
    \node at (0,-0.25) {$0$};
     
    \node at (1,-0.25) {$0$};
    \node at (-0.5,0) {\footnotesize{$(1,1)$}};
    \node at (0.5,1) {\footnotesize{$(2,2)$}};
    \node at (1.2,0.82) {\small{$rs$}};
    \fill (1,1) circle (1pt);
    \node at (3,-0.28) {\small{$\sqrt{q}/s$}};
    \node at (3,0.72) {\small{$\sqrt{q}s$}};
    \node at  (1.55,2.1) {\small{$r\sqrt{q}$}};
    \node at (3.1,1.75) {\small{$q$}};
\end{tikzpicture}
\caption{Geometric LPP model as described in $\eqref{eq:stat_wts}$.}
\label{fig:2paramStationaryModel}
\end{figure}

The last passage time $G_{r,s}^{\text{stat}}$ is stationary in the following sense.
\begin{prop}[\textnormal{Proposition 3.2 \cite{BI23}}]\label{I_r,sStationary}
    Let $k \in \Z_{\geq 1}$ and consider points $\mathbf{P}_{1} = (n_1,m_1),\dots, \mathbf{P}_{k} = (n_k,m_k)$ along a down-right path in the octant. Assume that $r,s,\sqrt{q}$ satisfies the relations $ \sqrt{q}\in (0,1),$ $s\in (\sqrt{q},1/\sqrt{q}),$ $r\in (0,1/s)$. Then given $N\geq 2,$ we have $G_{r,s}^{\mathrm{stat}}$ is stationary in the sense that the law of $\left(G_{r,s}^{\mathrm{stat}}{((N,N) + \mathbf{P}_{i})} - G_{r,s}^{\mathrm{stat}}{((N,N) + \mathbf{P}_{1})} \right)_{1\leq i\leq k}$ does not depend on $N$.
\end{prop}

We use the stationary LPP model $\eqref{eq:stat_wts}$ to explain how $\mathcal{I}_{r,s}$ is constructed. Intuitively speaking, $R_2(x)$ refers to the case when the up-right path jumps from $(2,1)$ to $(2,2)$ to the second row and $\max_{1\leq k\leq x} R_1(k) + R_2(x-k+1)-Y$ refers to up-right paths that stay on the first row for $k$ steps and jump from $(k+2,1)$ to $(k+2,2)$ to the second row. Taking the maximum over $k$ corresponds to maximization over all up-right paths that that do not make an immediate jump at $(2,1)$. Therefore, we notice that $\left(G_{r,s}^{\text{stat}}(2+N,2) - G_{r,s}^{\text{stat}}(2,2)\right)_{N \in \Z_{\geq 0}} \stackrel{(d)}{=} \left(\mathcal{I}_{r,s}(N)\right)_{N \in \Z_{\geq 0}}$. Since two LPP models \eqref{weights} and \eqref{eq:stat_wts} have the same recurrence relation $\eqref{recurrence}$, we get the following lemma:

\begin{lem}\label{twoparamEqual}
Fix any pair of $(r,s)$ that satisfies $s\in (\sqrt{q},1)$ and $r \in (0,s)\cup (s,1/s)$. Let $N\in \Z_{\geq 1}$, $d\geq 0$. Recall the definition of $F_{r,s}^{HD}(d,N)$ in \eqref{HighCDF}. Then we have $F_{r,s}^{HD}(d,N) = \Pb\left(G_{r,s}^{\mathrm{stat}}(N+1,N+1)\leq d\right)$.
\end{lem}
This shift to $(N+1,N+1)$ for $G_{r,s}^{\mathrm{stat}}$ explains why we take $N+1$ for the function $\psi$ in Theorem \ref{FiniteTimeFormulaTheorem} (1).

\subsubsection{Low density phase with product stationary initial condition}
We define the product stationary LPP model as follows:
\begin{equation} \label{one-paramStatModel}
  \omega_{i, j} = \begin{cases}
      0, & \textrm{if }i=j=1,\\
      \mathrm{Geo}\left( r\sqrt{q} \right), & \textrm{if }i=j\geq 2,\\
    \mathrm{Geo}\left( \sqrt{q}/r \right), & \textrm{if }j=1, i\geq 2, \\
    \mathrm{Geo}(q), &\textrm{otherwise}.
  \end{cases}
\end{equation}
\begin{figure}
\centering
\begin{tikzpicture}
\draw[very thin, gray] (-0.5,-0.5) grid (3.5,3.5);
    \foreach \x in {1,...,3} {
        \foreach \y in {0,1,...,\x} {
            \fill (\x,\y) circle (2pt);
        }
    }
    \fill (0,0) circle (2pt);
    \draw[thick,-] (0,0) -- (3,0);
    \draw[thick,-] (0,0) -- (3,3);
    \draw[thick,-] (1,1) -- (3,1);
    \fill (0,0) circle (1pt);
    \node at (0,-0.25) {$0$};
     
    \node at (1,-0.25) {$\sqrt{q}/r$};
    \node at (-0.5,0) {\footnotesize{$(1,1)$}};
    \node at (0.5,1) {\footnotesize{$(2,2)$}};
    \node at (1.2,0.8) {\small{$r\sqrt{q}$}};
    \fill (1,1) circle (1pt);
    \node at (3,0.72) {\small{$q$}};
    \node at  (1.55,2.1) {\small{$r\sqrt{q}$}};
    \node at (3.1,1.75) {\small{$q$}};
\end{tikzpicture}
\caption{Geometric LPP model as described in $\eqref{one-paramStatModel}$.}
\label{fig:StationaryModel}
\end{figure}
Let $G_{r}^{Low}(N,M)$ denote the last passage time from $(1,1)$ to $(N,M)$ under the model \eqref{one-paramStatModel} as shown in Figure \ref{fig:StationaryModel}. The following lemma is an easy observation.
\begin{lem}\label{oneparamEqual}
    Fix any $r$ satisfying $r\in (\sqrt{q},1/\sqrt{q})$. Let $N \in \Z_{\geq 1}$, $d\geq 0$. Consider the last passage time $G(N,N)$ that satisfies \eqref{recurrence} with $G(1,1) = 0$ and $G(1+\cdot, 1) \stackrel{(d)}{=}\mathcal{I}_{r,1/r}(\cdot)$. Then we have that $G(N,N) \stackrel{(d)}{=} G_{r}^{Low}(N,N)$, i.e., $F_r^{LD} (d,N) = \Pb (G_r^{Low}(N,N) \leq d).$
\end{lem}
We explain how two-parameter stationary LPP model $\eqref{eq:stat_wts}$ degenerates to the product stationary LPP model $\eqref{one-paramStatModel}$ under two special cases $rs=1$ and $r=s.$ 

When $s=1/r,$ the two-parameter stationary model $\eqref{eq:stat_wts}$ is not well-defined due to the explosion of a geometric random variable at $(2,2)$. Such an explosion would force all LPP paths to go through $(1,1)-(2,1)-(2,2)$. If we further subtract the variable $\omega_{2,2}$, then we see that as $s\rightarrow 1/r$\begin{equation}\label{degenerate1}
\begin{aligned}
&G_{r,s}^{\text{stat}}(N+1,M+1) - G_{r,s}^{\text{stat}}(2,2) \xRightarrow{(d)} G_{r}^{Low}(N,M),\\
&G_{r,s}^{\text{stat}}(N+2,2) - G_{r,s}^{\text{stat}}(2,2) \xRightarrow{(d)} G_{r}^{Low}(N+1,1) \stackrel{(d)}{=} \mathcal{I}_{r,1/r}(N).
\end{aligned}
\end{equation}

Consider the case $s=r$. Let $G_{r,1/s}^{\text{stat}}(N,M)$ denote the last passage time with $q_1,q_2$ swapped, i.e., the first row parameter is $s$ and the second row parameter is $1/s$. Under such a swap, $\omega_{2,2}\overset{(d)}{=}\text{Geo}(r/s)$ and will explode when $s\rightarrow r$, which forces the LPP paths to go through $(2,2).$ The non-trivial symmetry property of the Pfaffian Schur process (see the proof of \cite[Proposition 3.2]{BI23} and \cite[Proposition 3.10]{Baik_2018}) tells that for any integers $N\in\Z_{\geq 0}$ and $M\in \Z_{\geq 2}$, $(G_{r,s}^{\text{stat}}(M,M), \dots, G_{r,s}^{\text{stat}}(M+N,M))$ stays invariant under the exchange of $q_1,q_2$. This tells that for all $N\in \Z_{\geq 0},$ \begin{equation}\label{permutationInvariance}
G_{r,s}^{\text{stat}}(2+N,2) - G_{r,s}^{\text{stat}}(2,2)\stackrel{(d)}{=} G_{r,1/s}^{\text{stat}}(2+N,2) - G_{r,1/s}^{\text{stat}}(2,2).
\end{equation} Since both sides of \eqref{permutationInvariance} are well-defined as $s\rightarrow r$, we get
\begin{equation}\label{equalr,r}
\begin{aligned}
&G_{r,s}^{\text{stat}}(2+N,2) - G_{r,s}^{\text{stat}}(2,2)\stackrel{(d)}{=} \mathcal{I}_{r,s}(N) \xRightarrow{(d)} \mathcal{I}_{r,r}(N) \text{ as }s\rightarrow r,\\
&G_{r,1/s}^{\text{stat}}(2+N,2) - G_{r,1/s}^{\text{stat}}(2,2)\stackrel{(d)}{=} \mathcal{I}_{r,1/s}(N) \xRightarrow{(d)} \mathcal{I}_{r,1/r}(N) \text{ as }s\rightarrow r.
\end{aligned}
\end{equation}
Hence, we get $(\mathcal{I}_{r,r}(N))_{N\in \Z_{\geq 0}} \stackrel{(d)}{=} (\mathcal{I}_{r,1/r}(N))_{N\in \Z_{\geq 0}}$ by the uniqueness of the limit of weak convergence. As a result, we have that for $N,M\in \Z_{\geq 1}$, $N\geq M,$
\begin{equation}\label{degenerate2}
\begin{aligned}
    &G_{r,r}^{\text{stat}}(N+1,M+1) - G_{r,r}^{\text{stat}}(2,2) \stackrel{(d)}{=}G_{r}^{Low}(N,M)\\
\end{aligned}
\end{equation}
since $G_{r,r}^{\text{stat}}$ and $G_{r}^{Low}$ satisfy the same recurrence relation \eqref{recurrence} and we just proved that they have the same initial condition.
Hence, to get distribution of $G(N,N),$ it suffices to compute distributions of $G_{r,s}^{\text{stat}}$ for $r \in (0,s) \cup (s,1/s)$ and $G_{r}^{Low}.$ Thanks to \cite{Betea_2020}, there is already a method available for analyzing the product stationary case, the details of which are deferred to section \ref{Max&LowPhase}.

Next, we discuss a prelimiting integrable model for the two-parameter stationary LPP model.

\subsection{From inhomogeneous to stationary}
We define a modified LPP model with special parameters that is related to the two-parameter stationary LPP model. Let $r$ denote the diagonal parameter, $t$ denote the first row parameter, $s$ denote the second row parameter, and $\sqrt{q}$ denote the bulk parameter.
We define the following weights:
\begin{equation} \label{approxModel}
  \omega_{i, j} = \begin{cases}
      \mathrm{Geo}\left( rs \right), & \textrm{if } i=j=2,\\
      \mathrm{Geo}\left( r\sqrt{q} \right), & \textrm{if }i=j\geq 3,\\
      \mathrm{Geo}\left( st \right), & \textrm{if }j=1, i= 2,\\
    \mathrm{Geo}\left( t\sqrt{q} \right), & \textrm{if }j=1, i\geq 3, \\
    \mathrm{Geo}\left( s\sqrt{q} \right), &\textrm{if } j=2, i\geq 3, \\
    \mathrm{Geo}\left( rt \right), & \textrm{if}\ i=j=1, \\
    \mathrm{Geo}(q), &\textrm{otherwise},
  \end{cases}
\end{equation}
where parameters satisfy that $s \in (\sqrt{q}, 1)$, $t\in (0,1)$, $r\in (0,\min(1/s,1/t,1/\sqrt{}q))$ and $\sqrt{q} \in (0,1).$ We assume that $r\neq \sqrt{q},s$ (the case $r=\sqrt{q}$ is simpler) and $t\neq s,\sqrt{q},r$ (because we will take $t\rightarrow 1/s$ in the end). For the stationary model, $s$ can take value in $(\sqrt{q},1)\cup(1,1/\sqrt{q})$. The reason for considering only half of this interval is a symmetry property which was mentioned before and will also be explained in Remark~\ref{Remark:Symmetry}.
Let $L_{N,N}$ be the last passage time from $(1,1)$ to $(N,N)$ under the model $\eqref{approxModel}$.

\begin{figure}[htbp]
\centering
\begin{tikzpicture}
\draw[very thin, gray] (-0.5,-0.5) grid (4.5,4.5);
    \foreach \x in {1,...,4} {
        \foreach \y in {0,1,...,\x} {
            \fill (\x,\y) circle (2pt);
        }
    }
    \fill (0,0) circle (2pt);
    \draw[thick,-] (0,0) -- (4,0);
    \draw[thick,-] (0,0) -- (4,4);
    \draw[thick,-] (1,1) -- (4,1);
    \fill (0,0) circle (1pt);
    \node at (0,-0.2) {$rt$};
     \fill (1,0) circle (1pt);
    \node at (1,-0.2) {$st$};
    \node at (-0.5,0) {\footnotesize{$(1,1)$}};
    \node at (0.5,1) {\footnotesize{$(2,2)$}};
    \node at (1.2,0.82) {\small{$rs$}};
    \fill (1,1) circle (1pt);
    \node at (3,-0.28) {\small{$\sqrt{q}t$}};
    \node at (3,0.72) {\small{$\sqrt{q}s$}};
    \node at  (1.55,2.1) {\small{$r\sqrt{q}$}};
    \node at (3.1,1.75) {\small{$q$}};
\end{tikzpicture}
\caption{Geometric LPP model as described in $\ref{approxModel}$.}
\label{fig:StationaryModelApprox}
\end{figure}

\subsubsection{\textbf{Shift argument: connection to the stationary model}}
    Define the new last passage percolation 
    \begin{equation}\label{Gtilde}
    \widetilde{G}_{N,N} = L_{N,N} - \omega_{1,1} - \omega_{2,1}.
    \end{equation}
    We notice that ${G}_{r,s}^{\text{stat}}(N,N)$ is the $t\rightarrow 1/s$ limit of $\widetilde{G}_{N,N}.$ Based on Lemma \ref{twoparamEqual}, it is obvious that $\lim_{t\rightarrow 1/s} \widetilde{G}_{N+1,N+1} = G(N,N)$ with the initial condition $\mathcal{I}_{r,s}$ and $Y^s$ as specified in \eqref{HighCDF}. The shift argument provides a way to express the distribution of $\widetilde{G}(N,N)$ in terms of the distribution of $L_{N,N}.$ We recall that $\omega_{1,1} \sim \mathrm{Geo}(rt)$ and $\omega_{1,2} \sim \mathrm{Geo}(st).$ Assume $N \geq 2.$
    
    Notice that $(\widetilde{G}_{N,N}, \omega_{1,1}, \omega_{1,2})$ are independent. We perform a discrete Laplace transform and a change of variable $u = d-k$ in the following computation.
    \begin{equation}
    \begin{aligned}
        \sum_{d = 0}^\infty \Pb (\widetilde{G}_{N,N} + \omega_{1,1} + \omega_{2,1} \leq d) e^{-dx} &= \sum_{d=0}^\infty \sum_{k = 0}^{d}\Pb(\widetilde{G}_{N,N} + \omega_{1,1} \leq d-k) (st)^k (1-st) e^{-dx}\\
        &=\sum_{k=0}^\infty \sum_{u = 0}^{\infty}\Pb(\widetilde{G}_{N,N} + \omega_{1,1} \leq u) (st)^k (1-st) e^{-(u+k)x}\\
        &= \frac{(1-st)}{(1-e^{-x}st)} \sum_{u = 0}^\infty \Pb(\widetilde{G}_{N,N} + \omega_{1,1} \leq u)e^{-ux}\\
        &= \frac{(1-st)}{(1-e^{-x}st)}\frac{(1-rt)}{(1-e^{-x}rt)}\sum_{w=0}^{\infty}\Pb(\widetilde{G}_{N,N} \leq w)e^{-wx},
    \end{aligned}
    \end{equation}
    where the factor $\frac{(1-rt)}{(1-e^{-x}rt)}$ arises from removing $\omega_{1,1}.$
    Moving the coefficients to the left gives
    \begin{equation}
        \begin{aligned}
            & \sum_{d = 0}^\infty \Pb(\widetilde{G}_{N,N} \leq d)e^{-dx} = \frac{(1-e^{-x}st)}{(1-st)}\frac{(1-e^{-x}rt)}{(1-rt)} \sum_{d= 0}^\infty \Pb(L_{N,N} \leq d) e^{-dx} \\
            &=\frac{1}{(1-st)(1-rt)}\big(\Pb(L_{N,N} \leq 0) + \Pb(L_{N,N} \leq 1) e^{-x}\big)- \frac{(rt + st)}{(1-st)(1-rt)} \Pb(L_{N,N} \leq 0)e^{-x}\\
            &+ \sum_{d = 2}^\infty \left(\frac{\Pb(L_{N,N} \leq d) }{(1-st)(1-rt)}+ \frac{-(rt + st)\Pb(L_{N,N} \leq d-1)}{(1-st)(1-rt)}  + \frac{(rst^2)\Pb(L_{N,N} \leq d-2)}{(1-st)(1-rt)}\right)e^{-xd}.
        \end{aligned}
    \end{equation}
    Therefore, for $d \geq 2,$ we have
    \begin{equation}\label{eq:shift}
    \Pb(\widetilde{G}_{N,N} \leq d) =\frac{\Pb(L_{N,N} \leq d) }{(1-st)(1-rt)} - \frac{(rt + st)\Pb(L_{N,N} \leq d-1)}{(1-st)(1-rt)} + \frac{(rst^2)\Pb(L_{N,N} \leq d-2)}{(1-st)(1-rt)}.
    \end{equation}

    We also check that
    \begin{equation}
        \Pb(\widetilde{G}_{N,N} = 0)(1-st)(1-rt) = \Pb(L_{N,N} = 0),
    \end{equation}
    and 
    \begin{equation}
    \Pb(\widetilde{G}_{N,N} \leq 1) = \frac{\Pb(L_{N,N} \leq 1)}{(1-rt)(1-st)} - \frac{(rt+st)\Pb(L_{N,N} \leq 0)}{(1-rt)(1-st)}.
    \end{equation}

We remark that this shift argument has been used several times in other papers, for example, \cite[Section 2]{Baik_2010}, \cite[Section 3.2.1]{Betea_2020}, and \cite{Ferrari_2006TASEP}.
\begin{remark}\label{Remark:Symmetry}
    Due to the nontrivial symmetry property that the distribution of $L_{N,N}$ is invariant under permutations of the first-row and second-row parameters, i.e. $s$ and $t$, we can restrict to the case when the second row has parameter $s <1$ and the first row has parameter $1/s$ in the limit. The essential difficulty is to find the limit of $(1-st)^{-1} \Pb(L_{N,N} \leq d)$ as $t\rightarrow 1/s.$
\end{remark}

\subsection{Fredholm Pfaffian}
The following definition of Fredholm Pfaffian is taken from \cite[Definition 2.3]{Baik_2018}.
\begin{defin}
    For a $2\times 2$-matrix valued  skew-symmetric kernel,
	$$ K(x,y) = \begin{pmatrix}
	K_{11}(x,y) & K_{12}(x,y)\\
	K_{21}(x, y) & K_{22}(x,y)
	\end{pmatrix},\ \ x,y\in \mathbb{X},$$
	we define its Fredholm Pfaffian  by the series expansion
	\begin{equation}
    \mathrm{Pf}\big(J+K\big)_{\mathbb{L}^2(\mathbb{X},\mu)} = 1+\sum_{k=1}^{\infty} \frac{1}{k!}
	\int_{\mathbb{X}} \dots \int_{\mathbb{X}}  \mathrm{Pf}\Big( K(x_i, x_j)\Big)_{i,j=1}^k \mathrm{d}\mu^{\otimes k}(x_1 \dots x_k),
	\label{eq:defFredholmPfaffian}
	\end{equation}
	provided the series converges, and we recall that for a skew-symmetric  $2k\times 2k$ matrix $A$, its Pfaffian is defined by
	\begin{equation}
	\mathrm{Pf}(A) = \frac{1}{2^k k!} \sum_{\sigma\in\mathcal{S}_{2k}} \mathrm{sign}(\sigma) a_{\sigma(1)\sigma(2)}a_{\sigma(3)\sigma(4)} \dots a_{\sigma(2k-1)\sigma(2k)}.
	\label{def:pfaffian}
	\end{equation}
	The kernel $J$ is defined by
	$$ J(x,y)  = \Id_{x=y}
	\begin{pmatrix}
	0 & 1\\-1 & 0
	\end{pmatrix} .$$
\end{defin}
For the finite $N$ distribution, we consider Fredholm Pfaffians over $\ell^2(\{d+1,d+2,\dots\})$ for $d \in \Z_{\geq 0}$ with $\mu$ being the counting measure. For the asymptotic distribution, we consider Fredholm Pfaffians over $L^2((\tilde{d}, \infty))$ for $\tilde{d} \in \R$ with $\mu$ being the Lebesgue measure.
We will remove the subscript of the Fredholm Pfaffian to simplify the notation when the definition is clear from the context.

\subsection{Decomposition of the kernel}

We introduce several functions to simplify the notation. Let
\begin{equation}
    H(z) = \left(\frac{1-\sqrt{q}/z}{1-\sqrt{q}z}\right)^{N-2}, \quad S(z) = \frac{1-s/z}{1-sz}, \quad T(z) = \frac{1-t/z}{1-tz}, \quad F(z) = H(z)S(z)T(z).
\end{equation}
It is clear that $H(z^{-1}) = (H(z))^{-1},$ $S(z^{-1}) = (S(z))^{-1},$ and $T(z^{-1})=(T(z))^{-1}.$ 

The correlation kernel and diagonal distribution of the model $\eqref{approxModel}$ was originally constructed in \cite[Section 4.2]{Baik_2018} and also stated in \cite[Theorem C.1]{Betea_2020}. We rewrite the kernel in our setting as follows.

\begin{prop}\label{PfKernel}
    The distribution of $L_{N,N}$ is a Fredholm Pfaffian
    \begin{equation}
        \Pb(L_{N,N} \leq d) = \mathrm{Pf}(J - K)_{\ell^2(\{ d+1, d+2,\dots \})}
    \end{equation}
    with the $2\times 2$ matrix correlation kernel $K : \Z^2 \rightarrow \mathrm{Mat}_2(\R)$ given by
    \begin{equation}\label{Kernel}
    \begin{aligned}
        K_{11}(k,\ell) &= \frac{1}{(2\pi\I)^2} \int \frac{dz}{z^k} \int \frac{dw}{w^{\ell}} F(z)F(w) \frac{(z-w)(z-r)(w-r)}{(z^2-1)(w^2-1)(zw-1)},\\
        K_{12}(k,\ell) = -K_{21}(\ell,k) &= \frac{1}{(2\pi\I)^2} \int \frac{dz}{z^k} \int \frac{dw}{w^{-\ell+1}} \frac{F(z)}{F(w)}\frac{(zw-1)(z-r)}{(z^2-1)(z-w)(w-r)},\\
        K_{22}(k,\ell) &= \frac{1}{(2\pi\I)^2} \int \frac{dz}{z^{-k+1}} \int \frac{dw}{w^{-\ell+1}}\frac{1}{F(z)F(w)}\frac{(z-w)}{(zw-1)(z-r)(w-r)},
    \end{aligned}
    \end{equation}
    where the contours are positively oriented circles around the origin satisfying the following conditions:
    \begin{itemize}
   \item for $K_{11}$, $1 < |z|, |w| < \min \{1/t,1/s,1/\sqrt{q}\}$;
   \item for $K_{12}$, $\max \{ r,t,s,\sqrt{q}\} < |w| < |z|$ and $1 < |z| < \min\{1/t,1/s,1/\sqrt{q}\}$;
   \item for $K_{22}$, $\max \{r,t,s,\sqrt{q}\} < |w|, |z|$ and $1 < |zw|$.
  \end{itemize} 
  All the integration contours are positively oriented.
\end{prop}
In the following, we compute each entry of the kernel to prepare for evaluating this limit.
We define some functions for further decomposition of the kernel.

\begin{equation}\label{PQB}
    \begin{aligned}
        &P(\ell) = \oint\limits_{\Gamma_{\sqrt{q}}} \frac{dw}{2\pi\I} \frac{w^{\ell+1}}{H(w)}\frac{(1-sw)(1-tw)}{(w-s)(w-t)(1-wr)}, \quad
        Q(k) = -\oint \limits_{\Gamma_{1/\sqrt{q}}} \frac{dz}{2\pi\I}\frac{H(z)}{z^{k+2}}\frac{(z-s)(z-t)(1-rz)}{(1-sz)(1-tz)(1-z^2)},\\
        &B(k,\ell) = -\oint \limits_{\Gamma_{\sqrt{q}}} \frac{dw}{2\pi\I}\oint \limits_{\Gamma_{1/\sqrt{q}}} \frac{dz}{2\pi\I} \frac{w^{\ell-1}}{z^{k}}\frac{F(z)}{F(w)}\frac{(zw-1)}{(z-w)(1-zr)(w-r)}, \quad f^x(k) = \frac{x^{k+1}}{H(x)}.
    \end{aligned}
\end{equation}

\begin{equation}
    \begin{aligned}
        &{A}_{11}(k,\ell) = -\oint \limits_{\Gamma_{\sqrt{q}}}\frac{dw}{2\pi\I} \oint \limits_{\Gamma_{1/\sqrt{q}}}\frac{dz}{2\pi\I}\frac{w^{\ell-1}}{z^{k}}\frac{H(z)S(z)T(z)}{H(w)S(w)T(w)}\frac{(zw-1)(z-r)(1-wr)}{(z^2-1)(1-w^2)(z-w)},\\
        &{A}_{12}(k,\ell) = -{A}_{21}(\ell,k) = - \!\oint \limits_{\Gamma_{\sqrt{q}}} \!\!\frac{dw}{2\pi\I}\!\! \oint \limits_{\Gamma_{1/\sqrt{q}}}\!\!\frac{dz}{2\pi\I}\frac{w^{\ell-1}}{z^{k}}\frac{H(z)S(z)T(z)}{H(w)S(w)T(w)}\frac{(zw-1)(z-r)}{(z^2-1)(w-r)(z-w)}\\
        &+ \frac{(1-sr)(1-tr)}{(s-r)(t-r)}\ketbra{Q}{f^r}(k,\ell),\\
        &{A}_{22}(k,\ell) =\frac{(1-sr)(1-tr)}{(s-r)(t-r)}\left(\ketbra{f^r}{P}(k,\ell) -\ketbra{P}{f^r}(k,\ell)\right)+B(k,\ell).\\
    \end{aligned}
\end{equation}
See definitions of $g_1,d_2,G_t,G_s,R_t,R_s$ in \eqref{defOfG&R} and \eqref{defOfd_2,g_1}. Then we modify the kernel following the same arguments as in \cite[Lemma C.2, C.3]{Betea_2020}.
\begin{lem}\label{lem:rewrite11}
    The kernel entry $K_{11}$ can be expressed as
    \begin{equation}
        \begin{aligned}
            K_{11}(k,\ell) &= \overline{K}_{11}(k,\ell)
            +(1-st)\frac{(1-tr)}{(t-s)}\left(\ketbra{g_1}{f^t} - \ketbra{f^t}{g_1}\right).
        \end{aligned}
    \end{equation}
    where
    \begin{equation}
        \overline{K}_{11}(k,\ell) = A_{11}(k,\ell) + (1-st)\frac{(1-sr)}{(t-s)}\left(\ketbra{f^s}{G_t} - \ketbra{G_t}{f^s}\right).   
    \end{equation}
\end{lem}
\begin{proof}
    We first do a change of variable $w \rightarrow w^{-1}$.
    Then we see that the integrand of $K_{11}$ behaves like $w^{N+\ell-1}$ as $w \rightarrow 0$ and has no pole at $w = 0$ for any $\ell \geq 1.$ The integrand behaves like $z^{-(N+k+1)}$ as $z \rightarrow \infty$ and has no pole at $z = \infty.$ So we can deform the contour of $w$ to include $\sqrt{q}, s, t$ and deform the contour of $z$ to include $1/\sqrt{q},1/s,1/t.$ Because of the term $(zw-1)$, there is no contribution from poles at $(z,w) = (1/s,s)$ and $(z,w) = (1/t,t).$ Then
\begin{equation}
        \begin{aligned}
        K_{11} (k,\ell) &= \frac{-1}{(2\pi\I)^2} 
        \oint\limits_{\Gamma_{\sqrt{q}}}\, dw \!\!\!\oint\limits_{\Gamma_{1/s,1/t,1/\sqrt{q}}} \!\!\!\!\!\!dz \frac{w^{\ell-1}}{z^k}\frac{h_{11}(z,w)}{z-w} + \frac{-1}{(2\pi\I)^2} 
        \oint\limits_{\Gamma_{s}}\, dw \!\!\!\oint\limits_{\Gamma_{1/t,1/\sqrt{q}}} \!\!\!\!\!\!dz \frac{w^{\ell-1}}{z^k}\frac{h_{11}(z,w)}{z-w}\\
        &+\frac{-1}{(2\pi\I)^2} 
        \oint\limits_{\Gamma_{t}} dw \!\!\!\oint\limits_{\Gamma_{1/s,1/\sqrt{q}}} \!\!\!\!\!\!dz\frac{w^{\ell-1}}{z^k}\frac{h_{11}(z,w)}{z-w},\\
        \end{aligned}
    \end{equation}
    where $$h_{11}(z,w) = \frac{H(z)S(z)T(z)}{H(w)S(w)T(w)}\frac{(zw-1)(z-r)(1-wr)}{(z^2-1)(1-w^2)}.$$
    In the following, we use $\cdots$ as an abbreviation for the integrand $\frac{w^{\ell-1}}{z^k}\frac{h_{11}(z,w)}{z-w}$. The integrand has simple poles at $z = 1/s,1/t$ and $w = t,s,$ so we compute the contributions from each pair of these poles. The pair $(z,w) = (1/\sqrt{q},\sqrt{q})$ is intentionally omitted and retained in the form of a double contour integral (denoted as $A_{11}$), as these are not simple poles.  Then we have
    \begin{equation}
        \begin{aligned}
            \oint\limits_{\Gamma_{\sqrt{q}}} dw \!\!\oint\limits_{\Gamma_{1/s}} dz \cdots= (1-st)\frac{(1-sr)}{(t-s)}\ketbra{f^s}{G_t}, &\quad \oint\limits_{\Gamma_{s}} dw \!\!\oint\limits_{\Gamma_{1/\sqrt{q}}} dz \cdots = -(1-st)\frac{(1-sr)}{(t-s)}\ketbra{G_t}{f^s},\\
            \oint\limits_{\Gamma_{\sqrt{q},s}} dw \!\!\oint\limits_{\Gamma_{1/t}} dz \cdots= -(1-st)\frac{(1-tr)}{(t-s)}\ketbra{f^t}{g_1}, &\quad \oint\limits_{\Gamma_{t}} dw \!\!\oint\limits_{\Gamma_{1/s,1/\sqrt{q}}} dz \cdots = (1-st)\frac{(1-tr)}{(t-s)}\ketbra{g_1}{f^t}.\\
        \end{aligned}
    \end{equation}
\end{proof}

For the following two lemmas about $K_{12}$ and $K_{22}$, we first assume $r<1$ and then analytically extend to $r>1$.
\begin{lem}\label{lem:rewrite12}
    The kernel entry $K_{12}$ can be expressed as 
    \begin{equation}
        \begin{aligned}
            K_{12}(k,\ell) &= \overline{K}_{12}(k,\ell) 
            + (1-st)\frac{(1-t^2)}{(t-s)(t-r)}\ketbra{g_1}{f^t} + (1-st)\frac{(1-tr)}{(t-s)}\ketbra{f^t}{d_2},
        \end{aligned}
    \end{equation}
    where
    \begin{equation}\label{overlineK_12}
    \begin{aligned}
        \overline{K}_{12}(k,\ell) &= A_{12}(k,\ell) - (1-st)\frac{(1-s^2)}{(t-s)(s-r)}\ketbra{G_t}{f^s}\\ 
        &+ (1-st)\frac{(1-tr)(1-sr)}{(t-r)(t-s)}\ketbra{f^s}{f^r} + (1-st)\frac{(1-sr)}{(t-s)}\ketbra{f^s}{R_t}.\\
    \end{aligned}
    \end{equation}
    By anti-symmetry, we have that $K_{21}(k,\ell) = - K_{12}(\ell,k)$ and let $\overline{K}_{21}(k,\ell) = - \overline{K}_{12}(\ell,k)$.
\end{lem}

\begin{proof}
    Since the integrand of $K_{12}$ behaves like $w^{N+\ell-1}$ as $w \rightarrow 0$, there is no pole at $w=0.$ Similarly, there is no pole at $z= \infty$ because the integrand behaves like $z^{-(k+N+1)}$ as $z \rightarrow \infty$. We can deform the contour of $w$ to include only $r,s,t,\sqrt{q}$ and deform the contour of $z$ to include only $1/s,1/t,1/\sqrt{q}.$ Because of the term $(zw-1),$ there is no contribution from poles at $(z,w) = (1/s,s)$ and $(z,w) = (1/t,t)$. Then
        \begin{equation}
        \begin{aligned}
        K_{12} (k,\ell) &= \frac{-1}{(2\pi\I)^2} 
        \oint\limits_{\Gamma_{\sqrt{q},r, s, t}}\!\!\! dw \oint\limits_{\Gamma_{1/\sqrt{q}}}\!\!\!dz \frac{w^{\ell-1}}{z^k}\frac{h_{12}(z,w)}{z-w} + \frac{-1}{(2\pi\I)^2} 
        \oint\limits_{\Gamma_{r,t,\sqrt{q}}}\!\!\! dw \oint\limits_{\Gamma_{1/s}}\!\!dz \frac{w^{\ell-1}}{z^k}\frac{h_{12}(z,w)}{z-w}\\
        &+ \frac{-1}{(2\pi\I)^2} 
        \oint\limits_{\Gamma_{r,s,\sqrt{q}}}\!\!\! dw \oint\limits_{\Gamma_{1/t}}\!\!dz \frac{w^{\ell-1}}{z^k}\frac{h_{12}(z,w)}{z-w},\\
    \end{aligned}
\end{equation}
where
$$h_{12}(z,w) = \frac{H(z)S(z)T(z)}{H(w)S(w)T(w)}\frac{(zw-1)(z-r)}{(z^2-1)(w-r)}.$$
In the following, $\cdots$ serves as shorthand for the integrand $\frac{w^{\ell-1}}{z^{k}}\frac{h_{12}(z,w)}{z-w}.$
The pair $(z,w) = (1/\sqrt{q},\sqrt{q})$ is preserved in the form of a double contour integral and appears as one of the terms in $A_{12}.$ Then we compute the contributions from all other pairs of simple poles:
    \begin{equation}
    \begin{aligned}
        &\oint \limits_{\Gamma_{t}} dw \!\!\oint \limits_{\Gamma_{1/\sqrt{q},1/s}} \!\! dz \cdots
        =(1-st)\frac{(1-t^2)}{(t-s)(t-r)} \ketbra{g_1}{f^t},\quad\oint \limits_{\Gamma_{s,r,\sqrt{q}}} dw \!\oint \limits_{\Gamma_{1/t}} \!dz \cdots 
        = (1-st)\frac{(1-tr)}{(t-s)}\ketbra{f^t}{d_2},\\
        &\oint \limits_{\Gamma_{r}} \!dw \!\!\oint \limits_{\Gamma_{1/s}} \!\! dz \cdots
        =(1-st)\frac{(1-tr)(1-sr)}{(t-r)(t-s)} \ketbra{f^s}{f^r},\quad
        \oint \limits_{\Gamma_{\sqrt{q}}} dw \!\oint \limits_{\Gamma_{1/s}} \! dz \cdots
        = (1-st) \frac{(1-sr)}{(t-s)} \ketbra{f^s}{R_t},\\
        &\oint \limits_{\Gamma_{s}} dw \!\!\!\oint \limits_{\Gamma_{1/\sqrt{q}}} \!\!\! dz \cdots
        = -(1-st) \frac{(1-s^2)}{(t-s)(s-r)}\ketbra{G_t}{f^s}, \quad \oint \limits_{\Gamma_{r}} dw \!\!\!\oint \limits_{\Gamma_{1/\sqrt{q}}} \!\!\! dz \cdots
        = \frac{(1-sr)(1-tr)}{(s-r)(t-r)}\ketbra{Q}{f^r}.
    \end{aligned}
    \end{equation}
\end{proof}

\begin{lem}\label{lem:rewrite22}
    The kernel entry $K_{22}$ can be expressed as 
    \begin{equation}
        \begin{aligned}
            K_{22}(k,\ell) &= \overline{K}_{22}(k,\ell) 
            + (1-st)\frac{(1-t^2)}{(t-s)(t-r)}\left(\ketbra{f^t}{d_2} - \ketbra{d_2}{f^t}\right),
        \end{aligned}
    \end{equation}
    where
    \begin{equation}
        \begin{aligned}
            \overline{K}_{22}(k,\ell) &= A_{22}(k,\ell) + E(k,\ell) + (1-st)\frac{(1-s^2)}{(t-s)(s-r)}\left(\ketbra{f^s}{R_t} - \ketbra{R_t}{f^s}\right)\\
            &+ (1-st)\frac{(1-tr)(1-s^2)}{(t-s)(s-r)(t-r)}\left(\ketbra{f^s}{f^r} - \ketbra{f^r}{f^s}\right).\\
        \end{aligned}
    \end{equation}
\end{lem}

\begin{proof}
    A change of variable $z\rightarrow z^{-1}$  performed on $\eqref{Kernel}$ yields
    $$K_{22} = \frac{1}{(2\pi\I)^2} \oint dw \oint dz \frac{w^{\ell-1}}{z^k} \frac{H(z)S(z)T(z)}{H(w)S(w)T(w)}\frac{(zw-1)}{(z-w)(1-rz)(w-r)}.$$
    Before the change of variable, the contours can be taken as circles with $|z| = |w| > \max\{\sqrt{q}, r,s,t,1\}$.
    After the change of variable, we take $\max\{r,s,t,\sqrt{q}\}<|z|<|w| < \min \{1/s,1/t,1/\sqrt{q}, 1/r\}$.
    Because of the term $(z-w)$, when shrinking the $w$ contour around $\{\sqrt{q}, r, s, t\}$, we collect the residue at $w = z$, which gives
    $$E(k,\ell) = \frac{1}{2\pi\I}\oint \limits_{\Gamma_{0,r}}dz \frac{z^{\ell-k-1}(1-z^2)}{(z-r)(1-rz)}.$$
    We can compute
    \begin{equation}
        E(k,\ell)=
        \begin{cases}
            \frac{1}{-r} + \frac{1-r^2}{r(1-r^2)} = 0 &\text{ if } k = \ell,\\
            -\frac{1}{2\pi\I}\oint \limits_{\Gamma_{1/r}}\!\!dz \frac{z^{\ell-k-1}(1-z^2)}{(z-r)(1-rz)} = -r^{k-\ell-1} &\text{ if } k > \ell,\\
            \frac{1}{2\pi\I}\oint \limits_{\Gamma_{r}}\!dz \frac{z^{\ell-k-1}(1-z^2)}{(z-r)(1-rz)} = r^{\ell-k-1} &\text{ if } k < \ell.
        \end{cases}
    \end{equation}
    For the case $k> \ell$, we consider the outside poles and there is no pole at $z = \infty.$ The case $k < \ell$ follows from the fact that $0$ is no longer a pole.
    In conclusion, we have $E(k,\ell) = -\sgn{(k-l)}r^{|k-\ell|-1}.$
    Then restricting the contour of $w$ to $\{\sqrt{q},r,s,t\}$ and the contour of $z$ to $\{1/s,1/t,1/\sqrt{q},1/r\}$ gives
    \begin{equation}
        \begin{aligned}
            K_{22}(k,\ell) &= \frac{-1}{(2\pi\I)^2} \oint \limits_{\Gamma_{\sqrt{q}}} dw \!\!\!\oint\limits_{\Gamma_{1/r,1/\sqrt{q},1/s,1/t}}\!\!\!\!\!\! dz \frac{w^{\ell-1}}{z^k} \frac{H(z)S(z)T(z)}{H(w)S(w)T(w)}\frac{(zw-1)}{(z-w)(1-rz)(w-r)}+E(k,\ell)\\
            &+\frac{-1}{(2\pi\I)^2} \oint \limits_{\Gamma_{s}} dw \!\!\!\oint\limits_{\Gamma_{1/r,1/\sqrt{q},1/t}}\!\!\!\!\!\! dz \cdots
            +\frac{-1}{(2\pi\I)^2} \oint \limits_{\Gamma_{t}} dw \!\!\!\oint\limits_{\Gamma_{1/r,1/\sqrt{q},1/s}}\!\!\!\!\!\! dz \cdots
            +\frac{-1}{(2\pi\I)^2} \oint \limits_{\Gamma_{r}} dw \!\!\!\oint\limits_{\Gamma_{1/\sqrt{q},1/s,1/t}}\!\!\!\!\!\! dz \cdots
            ,\\
        \end{aligned}
    \end{equation}
    where $\cdots$ is a shorthand for the integrand $\mathsmaller{\frac{w^{\ell-1}}{z^k} \frac{H(z)S(z)T(z)}{H(w)S(w)T(w)}\frac{(zw-1)}{(z-w)(1-rz)(w-r)}}$.
    We compute the contributions from all pairs of simple poles. The pair $(z,w) = (1/\sqrt{q},\sqrt{q})$ is denoted as $B.$ 
    \begin{equation}
        \begin{aligned}
            &\frac{-1}{(2\pi\I)^2} \oint \limits_{\Gamma_{\sqrt{q},s,r}}\!\!\! dw \!\oint\limits_{\Gamma_{1/t}} dz \cdots + \frac{-1}{(2\pi\I)^2} \oint \limits_{\Gamma_{t}} dw \!\!\!\!\!\oint\limits_{\Gamma_{1/r,1/\sqrt{q},1/s}}\!\!\!\!\! dz \cdots = (1-st)\frac{(1-t^2)}{(t-s)(t-r)}\left(\ketbra{f^t}{d_2} - \ketbra{d_2}{f^t}\right),\\
            &\frac{-1}{(2\pi\I)^2} \oint \limits_{\Gamma_{\sqrt{q}}}\! dw \!\!\oint\limits_{\Gamma_{1/s}} dz \cdots + \frac{-1}{(2\pi\I)^2}\oint \limits_{\Gamma_{s}}\! dw \!\!\oint\limits_{\Gamma_{1/\sqrt{q}}} dz \cdots = (1-st)\frac{(1-s^2)}{(t-s)(s-r)}\left(\ketbra{f^s}{R_t} - \ketbra{R_t}{f^s} \right),\\
            &\frac{-1}{(2\pi\I)^2} \oint \limits_{\Gamma_{r}}\! dw \!\!\oint\limits_{\Gamma_{1/s}} dz \cdots + \frac{-1}{(2\pi\I)^2}\oint \limits_{\Gamma_{s}}\! dw \!\!\oint\limits_{\Gamma_{1/r}} dz \cdots = (1-st)\frac{(1-tr)(1-s^2)}{(t-s)(s-r)(t-r)}\left(\ketbra{f^s}{f^r} - \ketbra{f^r}{f^s}\right),\\
            &\frac{-1}{(2\pi\I)^2} \oint \limits_{\Gamma_{\sqrt{q}}}\! dw \!\!\oint\limits_{\Gamma_{1/r}} dz \cdots + \frac{-1}{(2\pi\I)^2}\oint \limits_{\Gamma_{r}}\! dw \!\!\oint\limits_{\Gamma_{1/\sqrt{q}}} dz \cdots = \frac{(1-sr)(1-tr)}{(s-r)(t-r)}\left(\ketbra{f^r}{P} - \ketbra{P}{f^r}\right).
        \end{aligned}
    \end{equation}
\end{proof}

\subsection{Reformulation of the Fredholm Pfaffian}
We define a few more notations.
Let
\begin{equation}\label{constant,a,b}
    \begin{aligned}
            a = \frac{(1-tr)}{(t-s)}, &\quad b = \frac{(1-t^2)}{(t-s)(t-r)},\\
            X_1 = \begin{pmatrix}
                d_2\\
                g_1
            \end{pmatrix}, \quad Y_1 = \begin{pmatrix}
                af^t & bf^t
            \end{pmatrix},&\quad 
            X_2 = \begin{pmatrix}
                -bf^t\\
                af^t
            \end{pmatrix}, \quad Y_2 = \begin{pmatrix}
                -g_1 & d_2
            \end{pmatrix}.
    \end{aligned}
\end{equation}
Let
\begin{equation}
    \overline{G} = J^{-1}\overline{K}(k,\ell)= \!\begin{pmatrix}
        -\overline{K}_{21}(k,\ell) & -\overline{K}_{22}(k,\ell)\\
        \overline{K}_{11}(k,\ell) & \overline{K}_{12}(k,\ell)
    \end{pmatrix},\quad
    \overline{K}(k,\ell) = \!\begin{pmatrix}
        \overline{K}_{11}(k,\ell)  & \overline{K}_{12}(k,\ell)\\
        \overline{K}_{21}(k,\ell) & \overline{K}_{22}(k,\ell)
    \end{pmatrix}.
\end{equation}
Then
\begin{equation}
\begin{aligned}
    J^{-1}\overline{K} &= \begin{pmatrix}
        -\overline{K}_{21} & -\overline{K}_{22}\\
        \overline{K}_{11} & \overline{K}_{12}
    \end{pmatrix} + (1-st) \begin{pmatrix}
        d_2\\
        g_1
    \end{pmatrix}\begin{pmatrix}
        af^t & bf^t
    \end{pmatrix} + (1-st) \begin{pmatrix}
        -bf^t\\
        af^t
    \end{pmatrix}\begin{pmatrix}
        -g_1 & d_2
    \end{pmatrix}\\
    &= \overline{G} + (1-st) \ketbra{X_1}{Y_1} + (1-st)\ketbra{X_2}{Y_2}.
\end{aligned}
\end{equation}
Recall the definition of $d_2$ and $g_1$ in \eqref{defOfd_2,g_1}.
Given the shift argument, our goal becomes finding the $t\rightarrow 1/s$ limit of 
$\frac{\mathrm{Pf}(J- K)_{\ell^2\{d+1,\dots\}}}{(1-ts)}.$ 
We reformulate the Fredholm Pfaffian so that we can explicitly find where the explosion occurs. We omit the subscript $\ell^2(\{d+1,\dots\})$ in the following formula.
\begin{equation}\label{keyReformulation}
    \begin{aligned}
        \mathrm{Pf}(J-K)^2 &= \det(\Id - J^{-1}K)\\
        &= \det\left(\Id - \overline{G} -(1-st)\ketbra{X_1}{Y_1} - (1-st)\ketbra{X_2}{Y_2}\right)\\
        &= \det(\Id - \overline{G})\det\left(\Id - (1-st)(\Id - \overline{G})^{-1}\ketbra{X_1}{Y_1} - (1-st)(\Id - \overline{G})^{-1}\ketbra{X_2}{Y_2}\right)\\
        &= \det(\Id - \overline{G})\det\left(\Id - (1-st)\begin{pmatrix}
           (\Id -\overline{G})^{-1}\ket{X_1} & (\Id -\overline{G})^{-1}\ket{X_2}
        \end{pmatrix}\begin{pmatrix}
            Y_1\\
            Y_2
        \end{pmatrix}\right)\\
        &= \det(\Id - \overline{G})\det\left(\Id - (1-st)\begin{pmatrix}
            Y_1\\
            Y_2
        \end{pmatrix}\begin{pmatrix}
           (\Id -\overline{G})^{-1}\ket{X_1} & (\Id -\overline{G})^{-1}\ket{X_2}
        \end{pmatrix}\right)\\
        &= \det(\Id - \overline{G})\det\left(\Id - (1-st)\begin{pmatrix}
            \bra{Y_1}(\Id -\overline{G})^{-1}\ket{X_1} & \bra{Y_1}(\Id -\overline{G})^{-1}\ket{X_2}  \\
            \bra{Y_2}(\Id -\overline{G})^{-1}\ket{X_1} & \bra{Y_2}(\Id -\overline{G})^{-1}\ket{X_2} 
        \end{pmatrix}\right)\\
        &= \det(\Id - \overline{G})(1- (1-st)\bra{Y_1}(\Id - \overline{G})^{-1}\ket{X_1})^2.
    \end{aligned}
\end{equation}
The fifth equality in $\eqref{keyReformulation}$ comes from the equality $\det(\Id - AB) = \det(\Id - BA).$ Since $\overline{K}$ is anti-symmetric, by a similar argument as in \cite[Proposition B.1]{Betea_2020}, we see that
\begin{equation}\label{antiSymmetryRelation}
\begin{aligned}
    &(\Id - \overline{G})_{12}^{-1}(k,\ell) = -(\Id - \overline{G})_{12}^{-1}(\ell,k), \quad (\Id - \overline{G})_{21}^{-1}(k,\ell) = -(\Id - \overline{G})_{21}^{-1}(\ell,k),\\
    &(\Id - \overline{G})_{11}^{-1}(k,\ell) = (\Id - \overline{G})_{22}^{-1}(\ell,k).
\end{aligned}
\end{equation}
From this, we get that
\begin{equation}
\begin{aligned}
    \bra{f}(\Id - \overline{G})_{12}^{-1}\ket{g} = -\bra{g}(\Id - \overline{G})_{12}^{-1}\ket{f}, &\quad \bra{f}(\Id - \overline{G})_{21}^{-1}\ket{g} = -\bra{g}(\Id - \overline{G})_{21}^{-1}\ket{f},\\
    \bra{f}(\Id - \overline{G})_{12}^{-1}\ket{f} = \bra{f}(\Id - \overline{G})_{21}^{-1}\ket{f} = 0,&\quad 
    \bra{f}(\Id - \overline{G})_{11}^{-1}\ket{g} =\bra{g}(\Id - \overline{G})_{22}^{-1}\ket{f}. 
\end{aligned}   
\end{equation}
Hence, we have
\begin{equation}
    \begin{aligned}
        \bra{Y_1}(\Id - \overline{G})^{-1}\ket{X_1} = \bra{Y_2}(\Id - \overline{G})^{-1}\ket{X_2}, \quad \bra{Y_1}(\Id - \overline{G})^{-1}\ket{X_2} = \bra{Y_2}(\Id - \overline{G})^{-1}\ket{X_1} = 0,
    \end{aligned}
\end{equation}
which justifies the last equality in \eqref{keyReformulation}.
Taking the square root of \eqref{keyReformulation} and using the identity $\mathrm{Pf}(J-K)^2 = \det{(\Id - J^{-1}K)},$ we get the formula:
\begin{equation}\label{KeyFormula}
    \begin{aligned}
        \frac{1}{(1-st)}\mathrm{Pf}(J-K) &= \mathrm{Pf}(J-\overline{K})\left( \frac{1}{1-st} - \bra{Y_1}(\Id - \overline{G})^{-1}\ket{X_1}
        \right)\\
        &= \mathrm{Pf}(J-\overline{K})\left( \frac{1}{1-st} -\braket{Y_1}{X_1} - \brabarket{Y_1}{\overline{G}}{X_1} -\bra{Y_1}\overline{G}^{^2}(\Id - \overline{G})^{-1}\ket{X_1}
        \right)\\
        &= \mathrm{Pf}(J-\overline{K})\left(\frac{1}{1-st} - \braket{Y_1}{X_1} - \brabarket{Y_1}{\overline{G}}{X_1}\right) - \mathrm{Pf}(J-\overline{K})\\
        &+ \mathrm{Pf}\left(J - \overline{K} - \ketbra{\begin{array}{c}
            g_1\\
            -d_2 
        \end{array}}
        {(Y_1\overline{G}^2)_1 \quad (Y_1\overline{G}^2)_2} - \ketbra{\begin{array}{c}
            (Y_1\overline{G}^2)_1\\
            (Y_1\overline{G}^2)_2 
        \end{array}}{-g_1 \quad d_2}\right).
    \end{aligned}
\end{equation}
The last equality in $\eqref{KeyFormula}$ is justified by the following identity, which is a detailed version of \cite[Lemma 3.18]{Betea_2020}.

\begin{lem}\label{lem:keychange}
    Let $K$ be an anti-symmetric kernel and $J(k,\ell) = \delta_{k,\ell}$\scalebox{0.7}{${\begin{pmatrix}
        0 & 1 \\ -1 & 0
    \end{pmatrix}}$}.
    Let $a,b,c,d$ be functions such that the scalar products in the following formulas are well-defined. Then
    \begin{equation}\label{keyIden}
        \begin{aligned}
            \mathrm{Pf}(J - K)&\braket{c \quad d}{\left(\Id - J^{-1}K\right)^{-1}\begin{pmatrix}
                a\\b
            \end{pmatrix}}\\
            &=
            \mathrm{Pf}(J-K) - \mathrm{Pf}\left(J - K - \ket{\begin{array}{c}
                 b \\
                 -a
            \end{array}}\bra{c \quad d} - \ket{\begin{array}{c}
                 c \\
                 d
            \end{array}}\bra{-b \quad a}\right).
        \end{aligned}
    \end{equation}
\end{lem}

\begin{proof} We simply compute:
    \begin{equation}
        \begin{aligned}
            &\mathrm{Pf}(J - K)\braket{c \quad d}{\left(\Id - J^{-1}K\right)^{-1}\begin{pmatrix}
                a \\
                b 
            \end{pmatrix}}\\
            &= \mathrm{Pf}(J - K) -\mathrm{Pf}(J - K)\left(1-\braket{c \quad d}{\left(\Id - J^{-1}K\right)^{-1}\begin{pmatrix}
                a \\
                b 
            \end{pmatrix}}\right)\\
            &=\mathrm{Pf}(J-K)\\
            &- \sqrt{\det(\Id - J^{-1}K)}\sqrt{\det\left( \Id - \begin{pmatrix}
                \braket{c \quad d}{\left(\Id - J^{-1}K\right)^{-1}\begin{pmatrix} a \\ b \end{pmatrix}} & \braket{c \quad d}{\left(\Id - J^{-1}K\right)^{-1}\begin{pmatrix}
                -d\\
                c
            \end{pmatrix}}\\
            \braket{-b \quad a}{\left(\Id - J^{-1}K\right)^{-1}\begin{pmatrix}
                a \\
                b 
            \end{pmatrix}} & \braket{-b \quad a}{\left(\Id - J^{-1}K\right)^{-1}\begin{pmatrix}
                -d\\
                c
            \end{pmatrix}}
            \end{pmatrix} \right)}\\
            &=\mathrm{Pf}(J-K) - \sqrt{\det\left(\Id - J^{-1}K - \ket{\begin{array}{c}
                a \\
                b
            \end{array}}\bra{c\quad d}  -  \ket{\begin{array}{c}
                -d \\
                c
            \end{array}}\bra{-b\quad a}\right) }\\
            &=\mathrm{Pf}(J-K) - \sqrt{\det\left(\Id - J^{-1}K - J^{-1}\ket{\begin{array}{c}
                b \\
                -a
            \end{array}}\bra{c\quad d}  -  J^{-1}\ket{\begin{array}{c}
                c \\
                d
            \end{array}}\bra{-b\quad a}\right) }\\
            &= \mathrm{Pf}(J-K) -\mathrm{Pf}\left(J - K - \ket{\begin{array}{c}
                 b \\
                 -a
            \end{array}}\bra{c \quad d} - \ket{\begin{array}{c}
                 c \\
                 d
            \end{array}}\bra{-b \quad a}\right).
        \end{aligned}
    \end{equation}
    where the second equality is due to $\eqref{antiSymmetryRelation},$ which holds as long as $\overline{K}$ is anti-symmetric, and the third equality uses $\det(\Id- AB) = \det(\Id- BA)$.
\end{proof}

\begin{remark}
    In our case, the inverse $(\Id - \overline{G})^{-1}$ can be understood as the series expansion when $\sqrt{q},s,t,r$ lie in $(0,\epsilon)$ for some small $\epsilon.$ However, the Fredholm Pfaffian \begin{equation}\label{PfInverse}
        \mathrm{Pf}(J-\overline{K})-\mathrm{Pf}\left(J - \overline{K} - \ketbra{\begin{array}{c}
            g_1\\
            -d_2 
        \end{array}}
        {(Y_1\overline{G}^2)_1 \quad (Y_1\overline{G}^2)_2} - \ketbra{\begin{array}{c}
            (Y_1\overline{G}^2)_1\\
            (Y_1\overline{G}^2)_2 
        \end{array}}{-g_1 \quad d_2}\right)
        \end{equation}
        is well defined for $\sqrt{q},s,t,r$ satisfying $\sqrt{q}\in(0,1)$, $s\in(\sqrt{q},1)$, $r\in(0,s)\cup (s,1/s)$, $t\in (\max(r,1),1/\sqrt{q})$, $t \neq 1/r$ as we will show later. Hence, we see $\eqref{PfInverse}$ as an analytic continuation.
\end{remark}

\begin{remark}
    Initially, we tried the method in \cite[Appendix B.4]{KPZbeyondBrownian} because our kernel has a similar structure to that in \cite[(162)]{KPZbeyondBrownian}. However, we later checked that the equality $\mathrm{Pf}(J- K) = \sqrt{\det{(\Id - K_{12})}}$, which could have significantly simplified our formula, does not hold. We think that in our case, $E$ in $K_{22}$ is very important and cannot be dropped, but it disappears in $\sqrt{\det(\Id - K_{12})}.$ To give an example, when $N = 2$ and $r= 0$,
    \begin{equation}
        \begin{aligned}
            &\Pb(\text{Geom}(st)\leq d)= 1-(st)^{d+1} =\mathrm{Pf}(J-K)_{\ell^2(\{d+2,\dots\})}\\
            &= \mathrm{Pf}(J - \overline{K})_{\ell^2(\{d+2,\dots\})}\left( 1 - (1-st)\braket{Y_1}{X_1} -  (1-st)\brabarket{Y_1}{\overline{G}}{X_1}\right).\\
        \end{aligned}
    \end{equation}
    We do not expand to the second order term $Y_1\overline{G}^2$ because $\overline{G} = \begin{pmatrix}
        0 & E\\
        0 & 0
    \end{pmatrix}$ and $\overline{G}^2=0$. We checked that 
    \begin{equation}
    \begin{aligned}
        &\mathrm{Pf}(J - \overline{K})_{\ell^2(\{d+2,\dots\})} = 1,\quad -(1-st)\brabarket{Y_1}{\overline{G}}{X_1} = (ts)^{d+2},\\
        &1-(1-st)\braket{Y_1}{X_1} = \det(\Id - K_{12}) = 1-(1+st)(st)^{d+1},\\
        &\mathrm{Pf}(J-K) = 1+ (st)^{d+2}-(1+st)(st)^{d+1}  = 1-(st)^{d+1}.
    \end{aligned}
    \end{equation}
    Hence, we see that $\sqrt{\det(\Id - K_{12})} \neq \mathrm{Pf}(J-K).$
\end{remark}

\begin{remark}\label{modifed}
    We have a slightly modified version of \ref{keyIden}. For any constant $\beta$, we have
    \begin{equation}
        \begin{aligned}
            \mathrm{Pf}(J - K)&\beta \braket{c \quad  d}{\left(\Id - J^{-1}K\right)^{-1}\begin{pmatrix}
                a\\b
            \end{pmatrix}}\\
            &=
            \mathrm{Pf}(J-K) - \mathrm{Pf}\left(J - K - \beta\ket{\begin{array}{c}
                 b \\
                 -a
            \end{array}}\bra{c \quad d} - \beta\ket{\begin{array}{c}
                 c \\
                 d
            \end{array}}\bra{-b \quad a}\right).
        \end{aligned}
    \end{equation}
\end{remark}

\subsection{Difference between two-parameter and product stationary models}\label{difference}
The main challenge of this problem is to group all poles at $t,1/t$ in a way such that multiplying by $(1-st)^{-1}$ produces a finite limit while still maintaining control over the remaining kernel. We follow a similar strategy to decompose the Fredholm Pfaffian as in \cite[Section 3.2.3]{Betea_2020}. However, our finite-time distribution formula is considerably more complex than the one in \cite[Theorem 2.4]{Betea_2020}, as it involves many additional $t$-dependent terms.

In the product (one-parameter) stationary case, some 
$t$-dependent terms were left in the kernel $\overline{K}$ without issue, (for example, in \cite[(3.13)]{Betea_2020}, $\overline{K}_{12}^{\text{exp}}$ has a pole at $w=-\beta$ and similarly in our case \eqref{oneParamKernel}, $\overline{K}_{12}^{\text{geo}}$ has a pole at $w=s.$) This is because $\overline{K}_{11}^{\text{exp}}$ contains only poles at $z=1/2,w=-1/2$ (similarly $\overline{K}_{11}^{\text{geo}}$ only contains poles at $z=1/\sqrt{q},w=\sqrt{q}$) which is a dominating term that guarantees the convergence of Fredholm Pfaffian. 
In contrast, in our two-parameter case, regardless of the poles at $t,1/t$, the kernel entry $\overline{K}_{11}$ still contains poles at $s,1/s$. This means that no $t-$dependent terms can be left in $\overline{K}$ if we want the Pfaffian to converge under the limit. Consequently, all $t-$dependent terms must be extracted into $X_i,Y_i$. Since $X_i,Y_i$ are more complicated, they introduce many more terms that diverge as $t\rightarrow 1/s$.

While explaining the difficulties, we give a preview of what will be proved in the next section.
To distinguish notation, we use $\mathcal{Y}_1$ and $\mathcal{X}_1$ to denote functions used in the product stationary exponential LPP, as given in \cite[(3.18), (3.19)]{Betea_2020}. In the product stationary setting, it suffices to expand up to the first term, that is $\lim_{t \rightarrow  1/s} (1 - st)^{-1} - \braket {\mathcal{Y}_1}{\mathcal{X}_1}$ exists, because $\braket{\mathcal{Y}_1}{\mathcal{X}_1}$ contains the only explosive term $\mathsmaller{\braket{f_{-}^{\alpha}}{f_{+}^{\beta}}}$ as in \cite[Lemma 3.9]{Betea_2020}. Similar things happen in the product stationary geometric case as in \eqref{eq:S_fact} and Theorem $\ref{explosion-one}$. In the two-parameter case, however, this expression diverges, and we need to expand further to obtain the limit: $\lim_{t\rightarrow 1/s} (1-st)^{-1} - \braket{Y_1}{X_1} - \brabarket{Y_1}{\overline{G}}{X_1}.$ This is essentially because there are two types of terms that cause explosions: $\braket{f^t}{f^s}$ and $\brabarket{f^t}{E}{f^s}$. We notice that $\braket{X_1}{Y_1}$ includes terms like $\braket{f^t}{f^s}$ and $\brabarket{Y_1}{\overline{G}}{X_1}$ contains terms like $\brabarket{f^t}{E}{f^s}$. As a result, after extracting terms that explode, $\widehat{\mathcal{A}}_{d},\widehat{\mathcal{B}}_{d},\widehat{\mathcal{C}}_{d},\widehat{\mathcal{D}}_{d},\widehat{\mathcal{E}}_{d}$ appear in the limit of $\brabarket{Y_1}{\overline{G}}{X_1}$ and $\widehat{\nu}_{d},\widehat{\mu}_{d}$ appear in the limit of $\braket{X_1}{Y_1}$.
The term $\bra{Y_1}\overline{G}^2$ becomes $\widehat{V}_1,\widehat{V}_2$ in the limit. All these functions are provided in section \ref{FiniteTimeFormula}. Since the kernel itself involves many terms, when we square $\overline{G}$ and then multiply by $Y_1$, the complexity increases even further. We need to be careful about which terms survive in the limit and at the same time keep track of their upper bounds, which are essential ingredients for the convergence of the Fredholm Pfaffian.

\section{Analytic continuation}
We started with the kernel $K$ in $\eqref{Kernel}$ that is defined for $t\in (0,1)$. We want to find another formula that is equal to $\eqref{KeyFormula}$ on the already defined domain and is also defined for $t \in (1/s-\epsilon,1/s+\epsilon)$ for some small $\epsilon >0.$
We make it clear that the function is analytic for $t \in (x,y),$ if for any $0 <\epsilon \ll 1$, the function is analytic in $ t\in [x+\epsilon, y-\epsilon].$
Since we will heavily rely on Hadamard's inequality to prove the convergence of Fredholm Pfaffian, we write it as an auxiliary lemma here.
\begin{lem}\label{Hadamard}
    Let $\mathsf{K}(k,\ell)$ be a $2\times 2$ matrix valued anti-symmetric kernel. If there exist constants $C>0$ and $\alpha>\beta \geq 0$ such that 
    \begin{equation}
        \begin{aligned}
            |\mathsf{K}_{11}(k,\ell)| \leq Ce^{-\alpha k - \alpha \ell}, \quad |\mathsf{K}_{12}(k,\ell)| \leq Ce^{-\alpha k + \beta \ell}, \quad |\mathsf{K}_{22}(k,\ell)| \leq Ce^{\beta k + \beta \ell},
        \end{aligned}
    \end{equation}
    then for all $n \in \Z_{\geq 1},$ $$\big|\mathrm{Pf}(\mathsf{K}(x_i,x_j))_{i,j = 1}^n\big| \leq (2n)^{n/2}C^n \prod_{i = 1}^n e^{-(\alpha-\beta)x_i}.$$
\end{lem}

\subsection{Analyticity of the kernel and the Fredholm Pfaffian}
We will prove the point-wise limit of each entry of the kernel $\overline{K}$. See \eqref{Ahat} and \eqref{Khat} for the definition of $\widehat{K}$.
\begin{lem}\label{lem:limitQ}
    For any fixed $\sqrt{q}\in (0,1),$ $s\in (\sqrt{q},1)$ and $r\in (0,s)\cup (s,1/s)$, the kernel $\overline{K}$ is analytic for $t\in (\max(r,1),1/\sqrt{q})$ and has a limit $\widehat{K} = \lim_{t\rightarrow 1/s} \overline{K}.$
\end{lem}

\begin{proof}
We discuss the limit of each entry separately.
$A_{11}$ is analytic because we can take the contour for $z$ to be as close to $1/\sqrt{q}$ and the contour for $w$ to be as close to $\sqrt{q}$ as desired. Since $\lim_{t\rightarrow 1/s}\frac{S(z)T(z)}{S(w)T(w)} = 1$, we get $\lim_{t\rightarrow 1/s}{A}_{11} = \widehat{A}_{11}.$ The other two rank-one terms have limit zero due to the factor $(1-st)$. Hence, $\lim_{t\rightarrow 1/s} \overline{K}_{11} = \widehat{K}_{11}.$

The analyticity of $A_{12}$ and $Q$ follows from the same reason. The factor $\frac{(1-sr)(1-tr)}{(s-r)(t-r)} \rightarrow 1$. The remaining three rank-one terms have limit zero due to the factor $(1-st)$. Hence, $\lim_{t\rightarrow 1/s} \overline{K}_{12} = \widehat{K}_{12}.$
For $A_{22}$, the analyticity of $P$ and $B$ follow by choosing the contours close enough to the poles. The term $E$ remains unchanged under the limit and the other four rank-one terms vanish. Hence, $\lim_{t\rightarrow 1/s} \overline{K}_{22} = \widehat{K}_{22}.$
\end{proof}

\begin{prop}\label{Fredconv}
    For any fixed $\sqrt{q}\in (0,1),$ $s\in (\sqrt{q},1)$, $r\in (0,s)\cup(s,1/s)$, $\mathrm{Pf}(J - \overline{K})$ is analytic for $t \in (\max(r,1),1/\sqrt{q})$, and there is a well-defined limit
    \begin{equation}
        \lim_{t\rightarrow 1/s} \mathrm{Pf}(J - \overline{K})_{\ell^2(\{d+1,d+2,\dots\})} = \mathrm{Pf}(J - \widehat{K})_{\ell^2(\{d+1,d+2,\dots\})}.
    \end{equation}
\end{prop}

\begin{proof}
Fix $0<\epsilon \ll 1$ such that   $1/s < 1/\sqrt{q}- \epsilon/2$ and $\sqrt{q} + \epsilon/2 < s.$ We take the contour for $z$ to be $|z-1/\sqrt{q}| = \epsilon/2$ and the contours for $w$ to be $|w-\sqrt{q}| = \epsilon/2$. Then we know that there exists a constant $C$ independent of $k,\ell$ such that the following bounds hold uniformly in $t\in [\max(r,1)+\epsilon,1/\sqrt{q}-\epsilon]$.
For $11$-entry, we have the upper bound \begin{equation}\label{11UpperBound}
    |\overline{K}_{11}(k,\ell)| \leq Ce^{(k+\ell)\log(s)},
\end{equation}
since
    \begin{equation} \label{UpperBoundonQ}
        \begin{aligned}
            &|A_{11}(k,\ell)| \leq Ce^{-k\log(1/\sqrt{q} - \epsilon/2) + \ell\log(\sqrt{q}+\epsilon/2)},\\
            &\big|\mathsmaller{(1-st)\frac{(1-tr)}{(t-s)}}\ketbra{f^s}{G_t}\big| \leq Ce^{k\log(s) -\ell \log(1/\sqrt{q} - \epsilon/2)}, \\
            &\big|\mathsmaller{(1-st)\frac{(1-tr)}{(t-s)}}\ketbra{G_t}{f^s}\big| \leq Ce^{ -k \log(1/\sqrt{q} - \epsilon/2) + \ell\log(s)}.\\
        \end{aligned}
    \end{equation}
    For $12$-entry, we have the upper bound
    \begin{equation}\label{12UpperBound}
        |\overline{K}_{12}(k,\ell)| \leq Ce^{k\log(s) + \ell\log(\max(r,s))},
    \end{equation}
    since
    \begin{equation}
        \begin{aligned}
            &|A_{12}(k,\ell)| \leq Ce^{-k\log(1/\sqrt{q}- \epsilon/2) + \ell\log(\sqrt{q}+\epsilon/2)} + Ce^{ - k \log(1/\sqrt{q} - \epsilon/2) + \ell\log(r)},\\
            &\big|(1-st)\left(\mathsmaller{\frac{(1-tr)(1-sr)}{(t-r)(t-s)}}\ketbra{f^s}{f^r} + \mathsmaller{\frac{(1-sr)}{(t-s)}}\ketbra{f^s}{R_t}\right)\big| \leq Ce^{k\log(s) + \ell\log(\max(r, \sqrt{q}+\epsilon/2))},\\
            &\big|(1-st)\mathsmaller{\frac{(1-s^2)}{(t-s)(s-r)}}\ketbra{G_t}{f^s}\big| \leq Ce^{-k\log(1/\sqrt{q} - \epsilon/2) + \ell\log(s)}.
        \end{aligned}
    \end{equation}
    For $22$-entry, we have the upper bound
    \begin{equation}\label{22UpperBound}
        |\overline{K}_{22}(k,\ell)| \leq \begin{cases}
            C &\text{ if } r\in (0,s)\cup(s,1],\\
            Ce^{(k+\ell)\log(r)}  &\text{ if } r\in (1,1/s),
        \end{cases}
    \end{equation}
    since
    \begin{equation}
        \begin{aligned}
            &|E(k,\ell)| \leq \mathsmaller{\begin{cases}
                C &\text{if } r \in (0,s)\cup(s,1], \\
                Ce^{(k+\ell) \log(r)} &\text{if } r \in (1,1/s),
            \end{cases}}\\
            &|A_{22}(k,\ell)| \leq Ce^{-k\log(1/\sqrt{q}- \epsilon/2) + \ell\log(\sqrt{q}+\epsilon/2)} + Ce^{k\log(r) + \ell\log(\sqrt{q}+\epsilon/2)}+Ce^{k\log(\sqrt{q}+\epsilon/2) + \ell\log(r)},\\
            &\big|(1-st)\mathsmaller{\frac{(1-s^2)}{(t-s)(t-r)}}(\ketbra{f^s}{R_t} - \ketbra{R_t}{f^s})\big| \leq Ce^{k\log(s) + \ell\log(\sqrt{q}+\epsilon/2)} + Ce^{k\log(\sqrt{q}+\epsilon/2) + \ell\log(s)},\\
            &\big|(1-st)\mathsmaller{\frac{(1-tr)(1-s^2)}{(t-s)(s-r)(t-r)}}(\ketbra{f^s}{f^r} - \ketbra{f^r}{f^s})\big| \leq Ce^{k\log(s) + \ell\log(r)} + Ce^{k\log(r) + \ell\log(s)}.
        \end{aligned}
    \end{equation}
    Then Hadamard's bound, Lemma \ref{Hadamard}, gives
    \begin{equation}
        \begin{aligned}
            |\mathrm{Pf} ({K(k_i,k_j)})_{i,j=1}^n| \leq \begin{cases}
                (2n)^{n/2}C^n\prod_{i=1}^n e^{\log(s)k_i} &\text{if } r\in (0,s)\cup(s,1],\\
                (2n)^{n/2}C^n\prod_{i=1}^n e^{(\log(s) + \log(r))k_i} &\text{ if }r\in (1,1/s).
            \end{cases}
        \end{aligned}
    \end{equation}
    Then we know that
    \begin{equation}
        \begin{aligned}
            \mathrm{Pf}(J-\overline{K})_{\ell^2(\{d+1,d+2,\dots\})} &= 1+ \sum_{n=1}^{\infty} \frac{(-1)^n}{n!} \sum_{k_1 = d+1}^{\infty} \cdots \sum_{k_n = d+1}^{\infty} \mathrm{Pf} ({\overline{K}(k_i,k_j)})_{i,j=1}^n\\
        \end{aligned}
    \end{equation}
    is absolutely convergent and its limit is obtained by Lemma $\ref{lem:limitQ}$ and dominated convergence theorem.
\end{proof}

\subsection{Analyticity of the term $(1-st)^{-1} - \braket{Y_1}{X_1} - \brabarket{Y_1}{\overline{G}}{X_1}$}
We further expand two terms to see where the explosion occurs. Let $\mathcal{K}_{22} = \overline{K}_{22}(k,\ell) - E(k,\ell).$ Recall the definition of $a,b$ in $\eqref{constant,a,b}$.
\begin{equation}\label{firstTerm}
\begin{aligned}
    \brabarket{Y_1}{\overline{G}}{X_1} = &-a\brabarket{f^t}{\overline{K}_{21}}{d_2} + b\brabarket{f^t}{\overline{K}_{11}}{d_2} - a\brabarket{f^t}{\mathcal{K}_{22}}{g_1} + b\brabarket{f^t}{\overline{K}_{12}}{g_1}\\
    &-a\brabarket{f^t}{E}{G_s} + a(1-sr)\brabarket{f^t}{E}{f^s}.
\end{aligned}
\end{equation}
\begin{equation}\label{secondTerm}
\begin{aligned}
    \braket{Y_1}{X_1} &= a\braket{f^t}{d_2} + b\braket{f^t}{g_1}\\
    &= a\frac{(1-s^2)}{(s-r)}\braket{f^t}{f^s} - b(1-sr)\braket{f^t}{f^s} + b\braket{f^t}{G_s} - a\frac{(1-sr)}{(s-r)}\braket{f^t}{f^r} - a\braket{f^t}{R_s}.
\end{aligned}
\end{equation}
Combining \eqref{firstTerm} and \eqref{secondTerm}, we see that
 \begin{equation}
    \begin{aligned}
        &(1-st)^{-1} - \braket{Y_1}{X_1} - \brabarket{Y_1}{\overline{G}}{X_1}\\
        &= (1-st)^{-1} + \left(b(1-sr) - a\frac{(1-s^2)}{(s-r)} \right)\braket{f^t}{f^s} - a(1-sr)\brabarket{f^t}{E}{f^s}\\
        &+a\brabarket{f^t}{\overline{K}_{21}}{d_2} - b\brabarket{f^t}{\overline{K}_{11}}{d_2} + a\brabarket{f^t}{\mathcal{K}_{22}}{g_1} - b\brabarket{f^t}{\overline{K}_{12}}{g_1}\\
        &+a\brabarket{f^t}{E}{G_s} - b\braket{f^t}{G_s} + a\frac{(1-sr)}{(s-r)}\braket{f^t}{f^r} + a\braket{f^t}{R_s}.
    \end{aligned}
\end{equation}

\begin{lem}\label{lem:analytic}
    For any fixed $\sqrt{q}\in (0,1)$, $s\in (\sqrt{q},1),$ $r\in (0,s)\cup (s,1/s)$, the following term has its $t\rightarrow 1/s$ limit
    \begin{equation}
    \begin{aligned}
         &\lim_{t\rightarrow 1/s} (1-st)^{-1} + \left(b(1-sr) - a\frac{(1-s^2)}{(s-r)} \right)\braket{f^t}{f^s} - a(1-sr)\brabarket{f^t}{E}{f^s}\\
         &= (d+2) - (N-2)\frac{\sqrt{q}(1/s+s-2\sqrt{q})}{(1-\sqrt{q}/s)(1-\sqrt{q}s)}-\frac{(1-r^2)s}{(s-r)(1-sr)}.
    \end{aligned}
    \end{equation}
\end{lem}

\begin{proof}
    For $t\in (0,\min(1/s,1/r))$, $t\neq s,r$, we have that
    \begin{equation}\label{ft,E,fs}
        \begin{aligned}
            &-a(1-sr)\brabarket{f^t}{E}{f^s} = \frac{1}{(1-st)}\frac{t^{d+2}s^{d+2}}{H(t)H(s)},\\
        \end{aligned}
    \end{equation}
    \begin{equation}
        \begin{aligned}\label{ft,fs}
            &\left(b(1-sr)-a\frac{(1-s^2)}{(s-r)}\right)\braket{f^t}{f^s} = \frac{1}{(1-st)}\left(\frac{(1-t^2)(1-sr)}{(t-s)(t-r)} - \frac{(1-tr)(1-s^2)}{(t-s)(s-r)}\right)\frac{t^{d+2}s^{d+2}}{H(t)H(s)}.
        \end{aligned}
    \end{equation}
    Notice that ${\frac{(1-t^2)(1-sr)}{(t-s)(t-r)}} \rightarrow -1$ and ${-\frac{(1-tr)(1-s^2)}{(t-s)(s-r)}}\rightarrow -1$ as $t\rightarrow 1/s.$ Then taking the $t\rightarrow 1/s$ limit via L'Hôpital's rule gives the 
    \begin{equation}
        \begin{aligned}
            &\lim_{t\rightarrow 1/s}\frac{1}{(1-st)}\left(1+\left(1+ \frac{(1-t^2)(1-sr)}{(t-s)(t-r)} - \frac{(1-tr)(1-s^2)}{(t-s)(s-r)} \right)\frac{(ts)^{d+2}}{H(t)H(s)}\right)\\
            &=\frac{1}{s}\frac{d}{dt}\left(\frac{(ts)^{d+2}}{H(t)H(s)}\right)\bigg|_{t = 1/s} + \frac{1}{(-s)}\frac{d}{dt}\left(\frac{(1-t^2)(1-sr)}{(t-s)(t-r)} - \frac{(1-tr)(1-s^2)}{(t-s)(s-r)}\right)\bigg|_{t=1/s}\\
            &=(d+2) - (N-2)\frac{\sqrt{q}(1/s+s-2\sqrt{q})}{(1-\sqrt{q}/s)(1-\sqrt{q}s)} + \frac{(s-r)s}{(1-s^2)(1-sr)} + \frac{(sr-1)s}{(1-s^2)(s-r)}\\
            &= (d+2) - (N-2)\frac{\sqrt{q}(1/s+s-2\sqrt{q})}{(1-\sqrt{q}/s)(1-\sqrt{q}s)} - \frac{(1-r^2)s}{(s-r)(1-sr)}.
        \end{aligned}
    \end{equation}
\end{proof}

\begin{lem}\label{lem:analytic1}
    For any fixed $\sqrt{q}\in (0,1),$ $s\in (\sqrt{q},1),$ $r\in (0,s)\cup (s,1/s)$, the following terms are analytic for $t \in (\max(r,1),1/\sqrt{q}),$ $t\neq 1/r$ with their $t\rightarrow 1/s$ limits being
    \begin{equation}
        \begin{aligned}
            & \lim_{t\rightarrow 1/s} -b\braket{f^t}{G_s}=\widehat{\mu}_d, \quad   \lim_{t\rightarrow 1/s} a\frac{(1-sr)}{(s-r)}\braket{f^t}{f^r} + a\braket{f^t}{R_s} = \widehat{\nu}_d,\\
            &\lim_{t\rightarrow 1/s} a\brabarket{f^t}{\overline{K}_{21}}{d_2} = \widehat{\mathcal{A}}_d, \quad \lim_{t\rightarrow 1/s} a\brabarket{f^t}{\mathcal{K}_{22}}{g_1} = \widehat{\mathcal{B}}_d, \quad \lim_{t\rightarrow 1/s} a\brabarket{f^t}{E}{G_s} = \widehat{\mathcal{C}}_d,\\
            &\lim_{t\rightarrow 1/s} -b\brabarket{f^t}{\overline{K}_{11}}{d_2} = \widehat{\mathcal{D}}_d, \quad \lim_{t\rightarrow 1/s} -b\brabarket{f^t}{\overline{K}_{12}}{g_1} = \widehat{\mathcal{E}}_d.
        \end{aligned}
    \end{equation}
\end{lem}

\begin{proof} Recall the definition of $a,b$ in \eqref{constant,a,b}.
    The arguments for $\widehat{\mu}_d$ and $\widehat{\nu}_d$ are similar. We give the computation for $\widehat{\nu}_d$. For $t \in (0,1/r),$ we have that 
    \begin{equation}
        \begin{aligned}
            &a\frac{(1-sr)}{(s-r)}\braket{f^t}{f^r} + a\braket{f^t}{R_s}\\
            =& a\frac{(1-sr)}{(s-r)}\frac{(tr)^{d+2}}{(1-tr)H(t)H(s)} - a\oint \limits_{\Gamma_{\sqrt{q}}} \frac{dw}{2\pi\I}\frac{(tw)^{d+2}}{(1-tw)H(t)H(w)}\frac{(1-sw)}{(w-s)(w-r)},
        \end{aligned}
    \end{equation}
    which is analytic for $t \in (\max(r,1),1/\sqrt{q})$, $t \neq 1/r$. Then taking $t\rightarrow 1/s$ limit gives $\widehat{\nu}_d$. To get $\widehat{\mu}_d$, we need a change of variable in the end.

    The arguments for $\widehat{\mathcal{A}}_{d}, \widehat{\mathcal{B}}_{d}, \widehat{\mathcal{D}}_{d}, \widehat{\mathcal{E}}_{d}$ are similar. We give the proof for the most complicated term $\widehat{\mathcal{B}}_{d}$. For $t \in (0,\min(1/r,1/s)),$ $t \neq s,r,$ we have that
    \begin{equation}\label{Bformula}
        \begin{aligned}
            a\brabarket{f^t}{\mathcal{K}_{22}}{g_1} &= a\brabarket{f^t}{A_{22}}{g_1} + a(1-st)\frac{(1-s^2)}{(t-s)(s-r)}\left(\braket{f^t}{f^s}\braket{R_t}{g_1} - \braket{f^t}{R_t}\braket{f^s}{g_1}\right)\\
            &+a(1-st)\frac{(1-tr)(1-s^2)}{(t-s)(s-r)(t-r)}\left(\braket{f^t}{f^s}\braket{f^r}{g_1} - \braket{f^t}{f^r}\braket{f^s}{g_1}\right)\\
            &= a\brabarket{f^t}{B}{g_1} + a\frac{(1-sr)}{(s-r)(t-r)}\left(\frac{(tr)^{d+2}\braket{P}{g_1}}{H(t)H(r)} - (1-tr)\braket{f^t}{P}\braket{f^r}{g_1}\right) +\xi+\tau,
        \end{aligned}
    \end{equation}
    where we used the definition of $A_{22}$ in the second equality and
    \begin{equation}
    \begin{aligned}
        \xi &= a\frac{(1-s^2)}{(t-s)(s-r)}\frac{(ts)^{d+2}}{H(t)H(s)}\braket{R_t}{g_1}
        +a\frac{(1-tr)(1-s^2)}{(t-s)(s-r)(t-r)}\frac{(ts)^{d+2}}{H(t)H(s)}\braket{f^r}{g_1},\\
        \tau &= - a(1-st)\frac{(1-s^2)}{(t-s)(s-r)}\braket{f^t}{R_t}\braket{f^s}{g_1} - a(1-st)\frac{(1-s^2)}{(t-s)(s-r)(t-r)}\frac{(tr)^{d+2}\braket{f^s}{g_1}}{H(t)H(r)}.
    \end{aligned}
    \end{equation}
    The right hand side of $\eqref{Bformula}$ is analytic for $t\in (\max(r,1),1/\sqrt{q})$, $t\neq 1/r,$ and we take the $t\rightarrow 1/s$ limit of each term:
    \begin{equation}\label{Bpart1}
        \begin{aligned}
            &\lim_{t\rightarrow 1/s} \tau = 0, \quad \lim_{t\rightarrow 1/s} \xi = \frac{s}{(1-s^2)}\braket{R_{1/s}}{g_1} + \frac{s(s-r)}{(1-s^2)(1-sr)}\braket{f^r}{g_1},
        \end{aligned}
    \end{equation}
    \begin{equation}\label{Bpart2}
        \begin{aligned}
            &\lim_{t\rightarrow 1/s} a\brabarket{f^t}{B}{g_1} + a\frac{(1-sr)}{(s-r)(t-r)}\left(\frac{(tr)^{d+2}\braket{P}{g_1}}{H(t)H(r)} - (1-tr)\braket{f^t}{P}\braket{f^r}{g_1}\right)\\
            &= \frac{(s-r)}{(1-s^2)}\brabarket{f^{1/s}}{\widehat{B}}{g_1} + \frac{1}{(1-s^2)} \frac{r^{d+2}H(s)}{s^{d+1}H(r)}\braket{\widehat{P}}{g_1} - \frac{(s-r)}{(1-s^2)} \braket{f^{1/s}}{\widehat{P}}\braket{f^r}{g_1}.
        \end{aligned}
    \end{equation}
    To justify the convergence of bracket, we use the term $\brabarket{f^t}{B}{g_1}$ to give an example. Other bracket terms follow a similar argument. Choose $\epsilon \ll 1$ such that $\sqrt{q} + \epsilon/2 < 1/s <1/\sqrt{q} -3\epsilon/2$. For $B$ defined in \eqref{PQB}, we take the contour for $z$ to be $|z-1/\sqrt{q}| = \epsilon/2$ and the contour for $w$ to be $|w=\sqrt{q}| = \epsilon/2.$ Then for $t \leq 1/s+\epsilon,$ there exists a constant $C$ independent of $k,\ell$ such that
    \begin{equation}
        \brabarket{f^t}{B}{g_1} = \sum_{k= d+1}^{\infty}\sum_{\ell = d+1}^{\infty} f^t(k)B(k,\ell)g_1(l) \leq C\sum_{k= d+1}^{\infty}\sum_{\ell = d+1}^{\infty} \left(\frac{(1/s + \epsilon)}{(1/\sqrt{q} - \epsilon/2)}\right)^{k}\left((\sqrt{q}+\epsilon/2)s\right)^{\ell} <\infty.
    \end{equation}
    Then by dominated convergence theorem, we can take the limit inside the summation. It is easy to see that $\lim_{t\rightarrow 1/s}B = \widehat{B}$. Combining $\eqref{Bpart1}$ and $\eqref{Bpart2}$ gives
    $\lim_{t\rightarrow 1/s} a\brabarket{f^t}{\mathcal{K}_{22}}{g_1} = \widehat{\mathcal{B}}_{d}.$ To get $\widehat{\mathcal{A}}_{d}$, we need to use $\lim_{t\rightarrow 1/s}P = \widehat{P}$, $\lim_{t\rightarrow 1/s}Q = \widehat{Q}$ and also the following limit:
    \begin{equation}
        \begin{aligned}
            &a\bra{f^t}{A}_{21}(\ell) = a\oint\limits_{\Gamma_{\sqrt{q}}}\frac{dw}{2\pi\I}\oint \limits_{\Gamma_{1/\sqrt{q}}}\frac{dz}{2\pi\I} \frac{t^2(tw)^d}{H(t)(1-tw)z^{\ell}}\frac{H(z)S(z)T(z)}{H(w)S(w)T(w)}\frac{(zw-1)(z-r)}{(z^2-1)(w-r)(z-w)}\\
            &- a\frac{(1-sr)}{(s-r)(t-r)}\frac{(tr)^{d+2}}{H(t)H(r)}Q(\ell) \rightarrow \frac{(s-r)}{(1-s^2)}\frac{H(s)}{s^{d+1}}\widehat{\mathsf{J}}(\ell) - \frac{1}{(1-s^2)}\frac{r^{d+2}H(s)}{s^{d+1}H(r)}\widehat{Q}(\ell).
        \end{aligned}
    \end{equation}

    Lastly, for $t \in (0,1/r)$, $t\neq s$, direct computation gives
    \begin{equation}
        \begin{aligned}
            a\brabarket{f^t}{E}{G_s} &= \frac{t^{d+2}}{(t-s)H(t)}\oint \limits_{\Gamma_{1/\sqrt{q}}} \frac{dz}{2\pi\I} \frac{H(z)}{z^{d+2}}\frac{(z-s)(1-tz)}{(1-sz)(z^2-1)(z-t)}\\
            &\rightarrow \frac{1}{(1-s^2)}\frac{H(s)}{s^{d+1}}\oint \limits_{\Gamma_{1/\sqrt{q}}}\frac{dz}{2\pi\I}\frac{H(z)}{z^{d+2}}\frac{(z-s)^2}{(1-sz)^2(z^2-1)} = \widehat{\mathcal{C}}_{d}.
        \end{aligned}
    \end{equation}
\end{proof}

\subsection{Analyticity of the difference of Pfaffians}
We first get the expansion of all terms in $Y_1\overline{G}^2.$ See $\eqref{cdot}$ for the definition of $\cdot$ operator. Recall definitions of $a,b$ in \eqref{constant,a,b}.
\begin{equation}
    \begin{aligned}\label{V_1V_2prelimit}
       \begin{pmatrix}
       (Y_1\overline{G}^2)_1 &(Y_1\overline{G}^2)_2\end{pmatrix}= \begin{pmatrix}
            a\bra{f^t}\overline{K}_{21}\cdot\overline{K}_{21} - a\bra{f^t}\overline{K}_{22}\cdot\overline{K}_{11} - b\bra{f^t}\overline{K}_{11}\cdot\overline{K}_{21} + b\bra{f^t}\overline{K}_{12}\cdot\overline{K}_{11}\\
            a\bra{f^t}\overline{K}_{21}\cdot\overline{K}_{22} - a\bra{f^t}\overline{K}_{22}\cdot\overline{K}_{12} - b\bra{f^t}\overline{K}_{11}\cdot\overline{K}_{22} + b\bra{f^t}\overline{K}_{12}\cdot\overline{K}_{12}
        \end{pmatrix}^{T}.
    \end{aligned}
\end{equation}

\begin{lem}\label{V_1,limit}
    For any fixed $r,s,\sqrt{q}$ that that satisfies $\sqrt{q}\in(0,1), s\in (\sqrt{q},1), r \in (0,s)\cup (s,1/s)$, $(Y_1\overline{G}^2)_1$ is analytic for $t\in (\max(1,r),1/\sqrt{q}),$ $t\neq 1/r$. Its $t\rightarrow 1/s$ limit is
    $$\lim_{t\rightarrow 1/s} (Y_1\overline{G}^2)_1 = \widehat{V}_1.$$
\end{lem}

\begin{proof}
All four terms in $\eqref{V_1V_2prelimit}$ follow a similar proof. We provide detailed computations for the most complicated term $a\bra{f^t}\overline{K}_{22}\cdot\overline{K}_{11}$. For the other three terms, we explain which terms, after expansion, contribute to the limit and which terms vanish as $t\rightarrow 1/s$.
For $t \in (0,\min(1/s,1/r)),$ $t\neq r,s,$ we have
\begin{equation}\label{expand22,11}
\begin{aligned}
&\bra{f^t}\overline{K}_{22}\cdot\overline{K}_{11} \\
&= \bra{f^t}A_{22}\cdot A_{11} + \bra{f^t}E\cdot A_{11} + (1-st)\mathsmaller{\frac{(1-s^2)}{(t-s)(s-r)}}\left(\braket{f^t}{f^s}\bra{R_t}A_{11} - \braket{f^t}{R_t} \bra{f^s}A_{11}\right)\\
&+(1-st)\mathsmaller{\frac{(1-tr)(1-s^2)}{(t-s)(s-r)(t-r)}}\left(\braket{f^t}{f^s}\bra{f^r}A_{11} - \braket{f^t}{f^r}\bra{f^s}A_{11}\right)\\
&+ (1-st)\mathsmaller{\frac{(1-sr)}{(t-s)}}\left(\brabarket{f^t}{E}{f^s}G_t - \brabarket{f^t}{E}{G_t}f^s\right)\\
&+(1-st)^2\mathsmaller{\frac{(1-sr)}{(t-s)}}\mathsmaller{\frac{(1-s^2)}{(t-s)(s-r)}}\left(\braket{f^t}{f^s}\braket{R_t}{f^s}G_t - \braket{f^t}{R_t}\braket{f^s}{f^s}G_t\right)\\
&+(1-st)^2\mathsmaller{\frac{(1-sr)}{(t-s)}}\mathsmaller{\frac{(1-s^2)}{(t-s)(s-r)}}\left(-\braket{f^t}{f^s}\braket{R_t}{G_t}f^s +\braket{f^t}{R_t}\braket{f^s}{G_t}f^s \right)\\
&+(1-st)^2\mathsmaller{\frac{(1-tr)(1-s^2)}{(t-s)(s-r)(t-r)}}\mathsmaller{\frac{(1-sr)}{(t-s)}}\left( \braket{f^t}{f^s}\braket{f^r}{f^s}G_t - \braket{f^t}{f^r}\braket{f^s}{f^s}G_t \right)\\
&+(1-st)^2\mathsmaller{\frac{(1-tr)(1-s^2)}{(t-s)(s-r)(t-r)}}\mathsmaller{\frac{(1-sr)}{(t-s)}}\left( -\braket{f^t}{f^s}\braket{f^r}{G_t}f^s + \braket{f^t}{f^r} \braket{f^s}{G_t}f^s\right),
\end{aligned}
\end{equation}
where $\braket{f^t}{f^s}$ is an abbreviation for $\frac{(ts)^{d+2}}{(1-ts)H(t)H(s)}$ and $\braket{f^t}{f^r}$ is an abbreviation for $\frac{(tr)^{d+2}}{(1-tr)H(t)H(r)}$. After canceling the factor $(1-st)^{-1}$, the above formula is analytic for $t\in (\max(1,r),1/\sqrt{q}),$ $t\neq 1/r$. If a term does not contain a factor of $(1-st)$, or if its $(1-st)$ factor is canceled by terms such as $\braket{f^t}{f^s}$ or $\brabarket{f^t}{E}{f^s}$, then this term survives in the limit. Hence, only the following terms remain in the limit. Let $*$ denote the usual multiplication.
\begin{equation}\label{limit22,11}
\begin{aligned}
    &\lim_{t\rightarrow 1/s}(-a)*\eqref{expand22,11}\\ &= \lim_{t\rightarrow 1/s} -a\bra{f^t}A_{22}\cdot A_{11} - a\bra{f^t}E\cdot A_{11} - a{\frac{(1-s^2)}{(t-s)(s-r)}}\frac{(ts)^{d+2}}{H(t)H(s)}\bra{R_t}A_{11}\\
    &-a{\frac{(1-tr)(1-s^2)}{(t-s)(s-r)(t-r)}}\frac{(ts)^{d+2}}{H(t)H(s)}\bra{f^r}A_{11} -a(1-st)\frac{(1-sr)}{(t-s)}\brabarket{f^t}{E}{f^s}G_t\\
    &= -\frac{(s-r)}{(1-s^2)}\bra{f^{1/s}}\widehat{B}\cdot \widehat{A}_{11} -\frac{1}{(1-s^2)}\frac{r^{d+2}H(s)}{s^{d+1}H(r)}\bra{\widehat{P}}\widehat{A}_{11} + \frac{(s-r)}{(1-s^2)}\braket{f^{1/s}}{\widehat{P}}\bra{f^r}\widehat{A}_{11}\\
    &+\frac{(s-r)H(s)}{(1-s^2)(1-sr)}\bra{\left(\frac{r^{k-d-1}}{s^{d+1}} + \frac{(1-s^2)}{(s-r)s^{k+1}}\right)}\widehat{A}_{11}-\frac{s}{(1-s^2)}\bra{R_{1/s}}\widehat{A}_{11}\\&-\frac{s(s-r)}{(1-sr)(1-s^2)}\bra{f^r}\widehat{A}_{11} + \frac{s}{(1-s^2)}G_{1/s},
\end{aligned}
\end{equation}
where
\begin{equation}\label{ftE}
    \begin{aligned}
        \bra{f^t}E(k) = \frac{t^{d+2}r^{k-d-1} - t^{k+1}}{(r-t)H(t)} - \frac{t^{k+2}}{(1-tr)H(t)}\rightarrow -\frac{H(s)}{(1-sr)}\left(\frac{r^{k-d-1}}{s^{d+1}} + \frac{(1-s^2)}{(s-r)s^{k+1}}\right),\\
    \end{aligned}
\end{equation}
and $\brabarket{f^t}{E}{f^s}$ is computed in $\eqref{ft,E,fs}$. The limits of the bracket terms follow from their upper bounds and the application of the dominated convergence theorem.
All terms in $\eqref{limit22,11}$ can be found as partial components of $\widehat{V}_{1}$.

For $a\bra{f^t}\overline{K}_{21}\cdot\overline{K}_{21},$ only the following terms contribute to the limit.
\begin{equation}\label{limit21,21}
    \begin{aligned}
        \lim_{t\rightarrow 1/s} a\bra{f^t}\overline{K}_{21}\cdot\overline{K}_{21} &= \lim_{t\rightarrow 1/s} a\bra{f^t}A_{21}\cdot A_{21} + a(1-st)\frac{(1-s^2)}{(t-s)(s-r)}\braket{f^t}{f^s}\bra{G_t}A_{21}\\
        &=\frac{(s-r)}{(1-s^2)}\frac{H(s)}{s^{d+1}}\bra{\widehat{\mathsf{J}}}\widehat{A}_{21} -\frac{1}{(1-s^2)}\frac{r^{d+2}H(s)}{s^{d+1}H(r)}\bra{\widehat{Q}}\widehat{A}_{21} + \frac{s}{(1-s^2)}\bra{G_{1/s}}\widehat{A}_{21}.
    \end{aligned}
\end{equation}
For $-b\bra{f^t}\overline{K}_{11}\cdot \overline{K}_{21},$ only the following terms contribute to the limit:
\begin{equation}\label{limit11,21}
    \begin{aligned}
        \lim_{t\rightarrow 1/s} -b\bra{f^t}\overline{K}_{11}\cdot \overline{K}_{21} &=\lim_{t\rightarrow 1/s} -b\bra{f^t}A_{11}\cdot A_{21} -b (1-st)\frac{(1-sr)}{(t-s)}\braket{f^t}{f^s}\bra{G_t}A_{21} \\
        &=\frac{1}{(1-sr)}\bra{f^{1/s}}\widehat{A}_{11}\cdot \widehat{A}_{21} + \frac{s}{(1-s^2)}\bra{G_{1/s}}\widehat{A}_{21}.
    \end{aligned}
\end{equation}
For $b\bra{f^t}\overline{K}_{12}\cdot \overline{K}_{11}$, only the following terms contribute to the limit:
\begin{equation}\label{limit12,11}
    \begin{aligned}
        &\lim_{t\rightarrow 1/s}b\bra{f^t}\overline{K}_{12}\cdot \overline{K}_{11}\\
        &= \lim_{t\rightarrow 1/s} b\bra{f^t}A_{12}\cdot A_{11} + b(1-st)\frac{(1-sr)}{(t-s)}\braket{f^t}{f^s}\bra{R_t}A_{11} + b(1-st)\frac{(1-tr)(1-sr)}{(t-r)(t-s)}\braket{f^t}{f^s}\bra{f^r}A_{11}\\
        &= -\frac{1}{(1-sr)}\bra{f^{1/s}}\widehat{A}_{12}\cdot\widehat{A}_{11} -\frac{s}{(1-s^2)}\bra{R_{1/s}}\widehat{A}_{11} - \frac{s(s-r)}{(1-sr)(1-s^2)}\bra{f^r}\widehat{A}_{11}.
    \end{aligned}
\end{equation}
Combining $\eqref{limit22,11}, \eqref{limit21,21}, \eqref{limit11,21},\eqref{limit12,11}$ gives the desired limit $\widehat{V}_1.$
\end{proof}

\begin{remark}
    In the following lemmas,
we say that some upper bound holds uniformly for $t\in (\max(1,r),1/\sqrt{q}),$ $t\neq 1/r$ if for any $0<\epsilon \ll 1,$ we can find a constant $C$ that depends on $\epsilon$ but independent of $\ell$ such that the upper bound holds for

\begin{equation}\notag
\begin{cases}
    t\in [1+\epsilon, 1/r - \epsilon] \cup [1/r+\epsilon, 1/\sqrt{q}-\epsilon] &\text{if } r\in (\sqrt{q},1],\\
    t\in [1+\epsilon, 1/\sqrt{q}-\epsilon] &\text{if } r\in (0,\sqrt{q}),\\
    t\in [r+\epsilon, 1/\sqrt{q}-\epsilon] &\text{if }r\in (1,1/s).
\end{cases}
\end{equation}
\end{remark}

\begin{lem}\label{V_1UpperBoundLemma}
    For any fixed $r,s,\sqrt{q}$ that satisfies $\sqrt{q}\in(0,1), s\in (\sqrt{q},1), r \in (0,s)\cup (s,1/s)$, there exists a constant $C$ independent of $\ell$ such that the following bound holds uniformly for $t \in (\max(r,1), 1/\sqrt{q}),$ $t\neq 1/r,$
    \begin{equation}\label{V_1UpperBound}
        \begin{aligned}
            |(Y_1\overline{G}^2)_1(\ell)| \leq Ce^{\ell\log(s)}.
        \end{aligned}
    \end{equation}
\end{lem}

\begin{proof}
    Fix $0<\epsilon \ll 1$ such that   $1/s < 1/\sqrt{q}- \epsilon/2$ and $\sqrt{q} + \epsilon/2 < s.$ We take the contour for $z$ to be $|z-1/\sqrt{q}| = \epsilon/2$ and the contours for $w$ to be $|w-\sqrt{q}| = \epsilon/2$ whenever the pole is present. There are four terms and we discuss one term carefully and the other three terms follow from a similar argument. 
    
    For the first term $\bra{f^t}\overline{K}_{22}\cdot\overline{K}_{11}$, based on $\eqref{expand22,11},$ we see that all terms that involve $A_{11}$ multiplied on the right (e.g. terms like $\bra{f^t}{A_{22}\cdot A_{11}}$) have an upper bound $Ce^{\ell\log(\sqrt{q}+ \epsilon/2)}$ for some $C$ independent of $\ell$ because the asymptotics of the second variable (i.e., $\ell$) of $A_{11}(k,\ell)$ comes from the pole at $w=\sqrt{q}$. All terms that involve $G_t$ multiplied on the right (e.g. terms like $\brabarket{f^t}{E}{f^s}G_t$) have an upper bound $Ce^{-\ell\log(1/\sqrt{q} - \epsilon/2)}$ because $G_t$ has a pole at $z = 1/\sqrt{q}.$ All other terms that involve $f^s$ multiplied on the right(e.g. $\brabarket{f^t}{E}{G_t}f^s$) have an upper bound $Ce^{\ell \log(s)}$.
    
    For the second term $\bra{f^t}\overline{K}_{21}\overline{K}_{21}$, the leading asymptotics is determined by the behavior of the second variable of $\overline{K}_{21}(k,\ell)$, which based on $\eqref{overlineK_12},$ is given by $Ce^{\ell\log(s)}.$ 
    We use the same argument for the third term $\bra{f^t}\overline{K}_{11}\cdot \overline{K}_{21}$ and get the same upper bound $Ce^{\ell\log(s)}.$
    Lastly, for the fourth term, $\bra{f^t}\overline{K}_{12}\cdot \overline{K}_{11},$ its leading asymptotics is determined by the second variable of $\overline{K}_{11}$, which is $Ce^{\ell\log(\sqrt{q}+\epsilon/2)}$ because of the pole at $w= \sqrt{q}.$

    Combining four upper bounds together, we get $\eqref{V_1UpperBound}$.
\end{proof}

\begin{lem}\label{V_2,limit}
    For any fixed $r,s,\sqrt{q}$ satisfying $\sqrt{q}\in(0,1), s\in (\sqrt{q},1), r \in (0,s)\cup (s,1/s)$, $(Y_1\overline{G}^2)_2$ is analytic for $t\in (\max(1,r),1/\sqrt{q}),$ $t\neq 1/r$. Its $t\rightarrow 1/s$ limit is
    $$\lim_{t\rightarrow 1/s} (Y_1\overline{G}^2)_2 = \widehat{V}_2.$$
\end{lem}

\begin{proof}
    The proof idea is similar to that of the limit of $(Y_1\overline{G}^2)_1$. For $t \in (0,\min(1/r,1/s)),$ we can compute $\braket{f^t}{f^s} = \frac{(ts)^{d+2}}{(1-st)H(t)H(s)}$ and $\braket{f^t}{f^r} = \frac{(tr)^{d+2}}{(1-tr)H(t)H(r)}.$ Given that we have replaced all brakets by their values, the new formula for $(Y_1\overline{G}^2)_2$ is analytic for $t\in (\max(1,r), 1/\sqrt{q}),$ $t\neq 1/r$ and we can take the $t\rightarrow 1/s$ limit. In the following computations, we only present terms that contribute to the limit.
    \begin{equation}\label{V_2,21,22}
    \begin{aligned}
        &\lim_{t\rightarrow 1/s}a\bra{f^t}\overline{K}_{21}\cdot \overline{K}_{22}\\
        &= \lim_{t\rightarrow 1/s} a\bra{f^t}A_{21}\cdot \left(A_{22}+E\right) + a(1-st)\frac{(1-s^2)}{(t-s)(s-r)}\braket{f^t}{f^s}\bra{G_t}\left(A_{22}+ E \right)\\
        &=\frac{(s-r)}{(1-s^2)}\frac{H(s)}{s^{d+1}}\bra{\widehat{\mathsf{J}}}\left(\widehat{A}_{22} + E\right)-\frac{1}{(1-s^2)}\frac{r^{d+2}H(s)}{s^{d+1}H(r)}\bra{\widehat{Q}}\left(\widehat{A}_{22} + E\right)+\frac{s}{(1-s^2)}\bra{G_{1/s}}\left(\widehat{A}_{22} + E\right).
    \end{aligned}
    \end{equation}
    \begin{equation}\label{V_2,22,12}
        \begin{aligned}
            &\lim_{t\rightarrow 1/s} -a\bra{f^t}\overline{K}_{22}\cdot \overline{K}_{12}\\
            &= \lim_{t\rightarrow 1/s} -a\bra{f^t}\left(A_{22}+E\right)\cdot A_{12} - a\mathsmaller{\frac{(1-st)(1-tr)(1-sr)}{(t-r)(t-s)}}\brabarket{f^t}{E}{f^s}f^r \\
            &- a\mathsmaller{\frac{(1-st)(1-sr)}{(t-s)}}\brabarket{f^t}{E}{f^s}R_t- a\mathsmaller{\frac{(1-st)(1-s^2)}{(t-s)(s-r)}}\braket{f^t}{f^s}\bra{R_t}A_{12}
            -a\mathsmaller{\frac{(1-st)(1-tr)(1-s^2)}{(t-s)(s-r)(t-r)}}\braket{f^t}{f^s}\bra{f^r}A_{12}\\
            &=-\frac{1}{(1-s^2)}\frac{r^{d+2}H(s)}{s^{d+1}H(r)}\bra{\widehat{P}}\widehat{A}_{12} + \frac{(s-r)}{(1-s^2)}\braket{f^{1/s}}{\widehat{P}}\bra{f^r}\widehat{A}_{12} - \frac{(s-r)}{(1-s^2)}\bra{f^{1/s}}\widehat{B}\cdot \widehat{A}_{12}\\
            &+\frac{(s-r)H(s)}{(1-s^2)(1-sr)}\bra{\left(\frac{r^{k-d-1}}{s^{d+1}} + \frac{(1-s^2)}{(s-r)s^{k+1}}\right)}\widehat{A}_{12} +\frac{s(s-r)}{(1-s^2)(1-sr)}f^{r} +\frac{s}{(1-s^2)}R_{1/s}\\
            &- \frac{s}{(1-s^2)}\bra{R_{1/s}}\widehat{A}_{12} -\frac{s(s-r)}{(1-s^2)(1-sr)}\bra{f^r}\widehat{A}_{12},
        \end{aligned}
    \end{equation}
    where $\bra{f^t}E$ can be found in \eqref{ftE} and $\brabarket{f^t}{E}{f^s}$ can be found in \eqref{ft,E,fs}.
    \begin{equation}\label{V_2,11,22}
        \begin{aligned}
            &\lim_{t\rightarrow 1/s}-b\bra{f^t}\overline{K}_{11}\cdot \overline{K}_{22}
            = \lim_{t\rightarrow 1/s} -b\bra{f^t}A_{11}\cdot \left(A_{22}+E\right) -b(1-st)\frac{(1-sr)}{(t-s)}\braket{f^t}{f^s}\bra{G_t}\left(A_{22}+E\right)\\
            &=\frac{1}{(1-sr)}\bra{f^{1/s}}\widehat{A}_{11}\cdot \left(\widehat{A}_{22}+E\right) + \frac{s}{(1-s^2)}\bra{G_{1/s}}\left(\widehat{A}_{22} + E\right).\\
        \end{aligned}
    \end{equation}
    \begin{equation}\label{V_2,12,12}
        \begin{aligned}
            &\lim_{t\rightarrow 1/s}b\bra{f^t}\overline{K}_{12}\cdot \overline{K}_{12}\\
            &= \lim_{t\rightarrow 1/s} b\bra{f^t}A_{12}\cdot A_{12} + b\mathsmaller{\frac{(1-st)(1-tr)(1-sr)}{(t-r)(t-s)}}\braket{f^t}{f^s}\bra{f^r}A_{12} + b\mathsmaller{\frac{(1-st)(1-sr)}{(t-s)}}\braket{f^t}{f^s}\bra{R_t}A_{12}\\
            &= -\frac{1}{(1-sr)}\bra{f^{1/s}}\widehat{A}_{12}\cdot \widehat{A}_{12} - \frac{s(s-r)}{(1-sr)(1-s^2)}\bra{f^r}\widehat{A}_{12} -\frac{s}{(1-s^2)}\bra{R_{1/s}}\widehat{A}_{12}.
        \end{aligned}
    \end{equation}
    Combining $\eqref{V_2,21,22}, \eqref{V_2,22,12},\eqref{V_2,11,22}, \eqref{V_2,12,12},$ we get $\widehat{V}_2.$
\end{proof}

\begin{lem}\label{V_2UpperBound}
    For any fixed $r,s,q$ that satisfies $q\in(0,1), s\in (\sqrt{q},1), r \in (0,s)\cup (s,1/s)$, there exists a constant $C$ independent of $\ell$ such that the following upper bound holds uniformly for $t \in (\max(r,1), 1/\sqrt{q}),$ $t\neq 1/r,$
    \begin{equation}
        \bigg|(Y_1\overline{G}^2)_2(\ell)\bigg|\leq \begin{cases}
            C &\text{ if } r\in (0,s)\cup(s,1],\\
            Ce^{\ell\log(r)} &\text{ if } r\in (1,1/s).
        \end{cases} 
    \end{equation}
\end{lem}

\begin{proof}
    The leading asymptotics of $(Y_1\overline{G}^2)_2$ is determined by the behavior of the second variable of $\overline{K}_{22}$ and $\overline{K}_{12}.$ Since $s>\sqrt{q}$ and $r\in (0,s)\cup (s,1/s),$ the leading terms are those that contains $s^{\ell}$ or $r^{\ell}.$
    \begin{itemize}
        \item The leading asymptotic for $a\bra{f^t}\overline{K}_{21}\cdot \overline{K}_{22}$  comes from $\bra{f^s}E$, $\brabarket{f^t}{\overline{K}_{21}}{f^s}{f^r}$, $\brabarket{f^t}{\overline{K}_{21}}{P}f^r,$ $\brabarket{f^t}{\overline{K}_{21}}{f^r}f^s,$ and $\brabarket{f^t}{\overline{K}_{21}}{R_t}f^s$
        \item The leading asymptotic for $-a\bra{f^t}\overline{K}_{22}\cdot \overline{K}_{12}$ comes from $\brabarket{f^t}{\overline{K}_{22}}{f^s}{f^r}$, $\brabarket{f^t}{\overline{K}_{22}}{Q}f^r$, and $\brabarket{f^t}{\overline{K}_{22}}{G_t}f^s$.
        \item The leading asymptotic for $-b\bra{f^t}\overline{K}_{11}\cdot \overline{K}_{22}$ comes from $\bra{f^t} \overline{K}_{11}\cdot E$,  $\brabarket{f^t}{\overline{K}_{11}}{f^s}f^r$, $\brabarket{f^t}{\overline{K}_{11}}{f^r}f^s$, $\brabarket{f^t}{\overline{K}_{11}}{P}f^r,$ and $\brabarket{f^t}{\overline{K}_{11}}{R_t}f^s.$
        \item The leading asymptotic for $b\bra{f^t}\overline{K}_{12}\cdot\overline{K}_{12}$ comes from $\brabarket{f^t}{\overline{K}_{12}}{Q}f^r,$ $\brabarket{f^t}{\overline{K}_{12}}{f^s}f^r$, and $\brabarket{f^t}{\overline{K}_{12}}{G_t}f^s$.
    \end{itemize}
    Considering the upper bound of each leading term yields the desired upper bound.
\end{proof}

\begin{lem}\label{d_2,g_1UpperBound}
    For any fixed $r,s,\sqrt{q}$ that satisfies $\sqrt{q}\in(0,1), s\in (\sqrt{q},1), r \in (0,s)\cup (s,1/s)$, there exists a constant $C$ independent of $\ell$ such that the following bound holds uniformly for $t \in (\max(r,1), 1/\sqrt{q}),$
    \begin{equation}
        \begin{aligned}
            |g_1(\ell)| \leq Ce^{\ell\log(s)}, \quad |d_2(\ell)| \leq Ce^{\ell\log(\max(s,r))}.
        \end{aligned}
    \end{equation}
\end{lem}

\begin{proof}
    Since $g_1$ and $d_2$ have no $t$ involved, they remain unchanged under the limit. The dominating term in $g_1$ is $-(1-sr)f^s$, which give an upper bound $e^{\ell\log(s)}$. The dominating term for $d_2$ is either $-\frac{(1-sr)}{(s-r)}f^r$ or $\frac{(1-s^2)}{(s-r)}f^s$ depending on the value of $s,r$, which gives an upper bound $e^{\ell\log(\max(s,r))}.$
\end{proof}

\begin{lem}\label{lem:analytic2}
     For any fixed $r,s,\sqrt{q}$ satisfying $\sqrt{q}\in(0,1), s\in (\sqrt{q},1), r \in (0,s)$, the following term is analytic for $t\in (\max(1,r),1/\sqrt{q}),$ $t\neq 1/r.$ Its $t\rightarrow 1/s$ limit is
     \begin{equation}
     \begin{aligned}
         &\lim_{t\rightarrow 1/s} \mathrm{Pf}\left(J - \overline{K} - \ketbra{\begin{array}{c}
            g_1\\
            -d_2 
        \end{array}}
        {(Y_1\overline{G}^2)_1 \quad (Y_1\overline{G}^2)_2} - \ketbra{\begin{array}{c}
            (Y_1\overline{G}^2)_1\\
            (Y_1\overline{G}^2)_2 
        \end{array}}{-g_1 \quad d_2}\right)_{\ell^{2}(\{d+1,d+2,\dots\})}\\
        &=\mathrm{Pf}\left(J - \widehat{K} - \ketbra{\begin{array}{c}
            g_1\\
            -d_2 
        \end{array}}
        {\widehat{V}_1 \quad \widehat{V}_2} - \ketbra{\begin{array}{c}
            \widehat{V}_1\\
            \widehat{V}_2 
        \end{array}}{-g_1 \quad d_2}\right)_{\ell^{2}(\{d+1,d+2,\dots\})}.
    \end{aligned}
     \end{equation}
\end{lem}

\begin{proof}
    We have three kernels to discuss. For $\overline{K},$ the upper bound for each entry has been analyzed in $\eqref{11UpperBound},\eqref{12UpperBound},\eqref{22UpperBound}.$ 
    We set
    \begin{equation}
        N_1 = \begin{pmatrix}
                \ketbra{g_1}{(Y_1\overline{G}^2)_1} &\ketbra{g_1}{(Y_1\overline{G}^2)_2}\\
                \ketbra{-d_2}{(Y_1\overline{G}^2)_1} & \ketbra{-d_2}{(Y_1\overline{G}^2)_2 }
            \end{pmatrix}, \quad N_2 = \begin{pmatrix}
                \ketbra{(Y_1\overline{G}^2)_1}{-g_1} &\ketbra{(Y_1\overline{G}^2)_1}{d_2}\\
                \ketbra{(Y_1\overline{G}^2)_2}{-g_1} & \ketbra{(Y_1\overline{G}^2)_2 }{d_2}
            \end{pmatrix}.
    \end{equation}
    By using Lemmas $\ref{V_1UpperBoundLemma}$, $\ref{V_2UpperBound}$, $\ref{d_2,g_1UpperBound}$, we get that
    \begin{equation}
        \begin{aligned}
            N_1,N_2\leq 
            \begin{cases}
                \begin{pmatrix}
                    Ce^{(k+\ell) \log(s)} & 
                    Ce^{k \log(s)}\\
                    Ce^{\ell \log(s)} & C
                \end{pmatrix} &\text{ if }r\in(0,s)\cup (s,1),\\
                \begin{pmatrix}
                e^{(k+\ell) \log(s)} & e^{k\log(s) + \ell\log(r)}\\
                e^{k\log(r) + \ell\log(s)} & e^{(k+\ell)\log(r) }
            \end{pmatrix} &\text{ if } r \in (1,1/s).
            \end{cases}
        \end{aligned}
    \end{equation}
    By Lemmas \ref{lem:limitQ}, \ref{V_1,limit} and \ref{V_2,limit}, we get the entry-wise limits of three kernels.
    By Hadamard's bound, Lemma \ref{Hadamard}, and dominated convergence theorem, we get the convergence of Fredholm Pfaffian.
\end{proof}

In conclusion, we have found an analytic continuation of $\eqref{keyReformulation}$ and computed its $t\rightarrow 1/s$ limit.

\subsection{Analytic continuation of the LPP model}
In this section, we show that we can analytically extend the diagonal LPP probability distribution from $t \in (0,1)$ to $t \in (0,1/\sqrt{q})$. Recall the definition of $\widetilde{G}_{N,N}$ in \eqref{Gtilde}.
\begin{prop}\label{ModelAnalytic}
    For any fixed $r,s,q$ that satisfies $q \in (0,1),$ $s\in(\sqrt{q},1),$ $r\in (0,1/s),$ the map $t \rightarrow \Pb(\widetilde{G}_{N,N} \leq d)$ is real analytic for $t \in (1,1/\sqrt{q}).$
\end{prop}

\begin{proof}
    We know that $\widetilde{G}_{N,N}$ is measurable with respect to the sigma algebra generated by $\{{\omega}_{i,j}| i\geq j, i,j \in \Z_{\geq 1}, i,j \leq N, (i,j) \neq (1,1),(2,1)\}.$
    Let 
    \begin{equation}
        V_{d}(a_{3,1}, \dots, a_{N,N}) = \Pb(\{\widetilde{G}_{N,N} \leq d\}\, | \,{\omega}_{3,1} = a_{3,1},\dots, {\omega}_{N,N} = a_{N,N})
    \end{equation}
    as conditional probability. Clearly, $V_d$ does not depend on ${\omega}_{1,1}\!\sim\!\text{Geo}(rt)$ and ${\omega}_{1,2} \!\sim\!\text{Geo}(st)$ and $V_d = 1$ or $0.$
    Then 
    \begin{equation}
    \begin{aligned}
    &\begin{aligned}
        \Pb(\widetilde{G}_{N,N} \leq d\} ) = \sum_{a_{3,1},\dots, a_{N,N} \in \Z_{\geq 0}}
            &(t\sqrt{q})^{a_{3,1} + \dots + a_{N,1}} (1-t\sqrt{q})^{N-2}\times \\
        & (s\sqrt{q})^{a_{3,2} +\dots + a_{N,2}}(1-s\sqrt{q})^{N-2}(rs)^{a_{2,2}}(1-rs)\times\\
        & q^{a_{4,3} + \dots + a_{N,3}}(1-q)^{N-3}(r\sqrt{q})^{a_{3,3}}(1-r\sqrt{q}) \times \cdots \times\\
        & q^{a_{N,N-1}}(1-q)(r\sqrt{q})^{a_{N-1,N-1}}(1-r\sqrt{q})\times\\
        & (r\sqrt{q})^{a_{N,N}}(1-r\sqrt{q}) V_d(a_{3,1}, \dots , a_{N,N})
        \end{aligned}\\
        &= \sum_{a_{3,1},\dots, a_{N,N} \in \Z_{\geq 0}}t^{a_{3,1}+ \dots + a_{N,1}}(1-t\sqrt{q})^{N-2} \Phi(a_{2,2}, a_{3,2},\dots, a_{N,N}, q,s,r),
    \end{aligned}
    \end{equation}
    where $\Phi$ is simply a function that collects all other terms unrelated to $t$. 
    The sum converges and each term in the summation is a polynomial of $t$. Hence, the distribution function is real analytic for $t \in (1,1/\sqrt{q}).$
\end{proof}

\begin{proof}[Proof of Theorem~\ref{FiniteTimeFormulaTheorem} $(1)$] To prove the theorem, we first used shift argument $\eqref{eq:shift}$ to get a formula for the distribution of half space stationary geometric LPP model.
We have shown that the left-hand side of $\eqref{eq:shift}$ can be analytically extended to $t \in (\max(r,1),1/\sqrt{q})$ by Lemma $\ref{ModelAnalytic}$. We have also found the analytic continuation of the right-hand side of $\eqref{eq:shift}$ for $t\in (\max(r,1),1/\sqrt{q})$, which is $\eqref{KeyFormula}$.
We computed the limit of the kernel and the limit of the Fredholm Pfaffian as in Lemmas $\ref{lem:limitQ}$ and $\ref{Fredconv}$. We used analytic continuation to find limits of each term in $\eqref{KeyFormula}$ as proved in Lemmas $\ref{lem:analytic}$, $\ref{lem:analytic1}$, 
$\ref{V_1,limit}$,
$\ref{V_2,limit}$,
$\ref{lem:analytic2}$. Combining things all together, we get the desired finite time diagonal distribution in $\eqref{finiteDis}$ and \eqref{finiteTimeFormulaTheorem}.
\end{proof}

\section{Asymptotic analysis for High density phase}
We know that ${G}_{r,s}^{\text{stat}}(N,N)$ is centered around $2\sqrt{q}/(1-\sqrt{q})N$ as $r,s$ are scaled close to the critical value $1$. We consider fluctuations on the $N^{1/3}$ scale in order to discover some non-Gaussian limiting distribution.
Fix $\sqrt{q}\in (0,1).$ We first define the scaling of each variable.
\begin{equation}\label{scale}
    \begin{aligned}
        k = \frac{2\sqrt{q}}{1-\sqrt{q}}N + (c_0N/2)^{1/3}X, &\quad \ell = \frac{2\sqrt{q}}{1-\sqrt{q}}N + (c_0N/2)^{1/3}Y,\\
s = 1 + (c_0N/2)^{-1/3}\tilde{s},&\quad r = 1 + (c_0N/2)^{-1/3}\tilde{r},\\
z = 1 + (c_0N/2)^{-1/3}\zeta, &\quad w = 1+ (c_0N/2)^{-1/3}\omega,\\
d= \frac{2\sqrt{q}}{1-\sqrt{q}}N &+ (c_0N/2)^{1/3}\tilde{d}.
    \end{aligned}
\end{equation}
Under this critical scaling, the random variable $\text{Geom}(rs)$ at $(2,2)$ plays a nontrivial role in the limit. Alternatively, one could consider keeping $r$ and $s$ fixed and investigate the scaling limit across the entire high-density phase.
By assumption, we require $ \tilde{s} <0$ and $\tilde{r} < - \tilde{s}.$ Recall that $c_0 = \frac{2\sqrt{q}+2q}{(1-\sqrt{q})^3}.$ We will write $f^r(X) = f^r\left(\mathsmaller{\frac{2\sqrt{q}}{1-\sqrt{q}}N + (c_0N/2)^{1/3}X}\right)$ since $X$ is the new variable under the critical scaling. Similarly, we use $f^r(Y)$ to replace $f^r(\ell)$ and use $K(X,Y)$ to replace $k(k,\ell).$ The dependence on $N$ is hidden. We first discuss the asymptotic limits of functions that do not have contour integrals involved.

\begin{lem}
    For any $q \in (0,1),$ $s\in (\sqrt{q},1)$ and $r \in (0,s)\cup (s,1/s)$ such that $s = 1+ (c_0N/2)^{-1/3}\tilde{s}$ and $r = 1 + (c_0N/2)^{-1/3}\tilde{r}$ and any $d\in \Z$ such that $d = \frac{2\sqrt{q}}{1-\sqrt{q}}N + (c_0N/2)^{1/3}\tilde{d}$, we have the following limit:
    \begin{equation}
        \begin{aligned}
            \lim_{N\rightarrow \infty} (c_0N/2)^{-1/3}\left((d+2) - (N-2)\frac{\sqrt{q}(1/s+s-2\sqrt{q})}{(1-\sqrt{q}/s)(1-\sqrt{q}s)} - \frac{s(1-r^2)}{(1-sr)(s-r)}\right) = \tilde{d} - 2\tilde{s}^2 - \frac{2\tilde{r}}{\tilde{s}^2 - \tilde{r}^2}.
        \end{aligned}
    \end{equation}
\end{lem}

\begin{proof}
    We first find the talyor expansion of the following term at $s = 1$:
    \begin{equation}
        \frac{\sqrt{q}(1/s+s-2\sqrt{q})}{(1-\sqrt{q}/s)(1-\sqrt{q}s)} = \frac{2\sqrt{q}}{(1-\sqrt{q})} +\frac{2\sqrt{q}(1+\sqrt{q})}{(1-\sqrt{q})^3}(c_0N/2)^{-2/3}\tilde{s}^2 + \mathcal{O}((c_0N/2)^{-1}\tilde{s}^3).
    \end{equation}
    Then we get
    \begin{equation}\label{1stConstantLimit}
        \begin{aligned}
            &\lim_{N\rightarrow \infty}{c_0N/2}^{-1/3}\left((d+2) - (N-2) \frac{\sqrt{q}(1/s+s-2\sqrt{q})}{(1-\sqrt{q}/s)(1-\sqrt{q}s)}\right)\\
            &= \lim_{N\rightarrow \infty}\tilde{d} - Nc_0(c_0N/2)^{-1}\tilde{s}^2 + \mathcal{O}((c_0N/2)^{-1/3}\tilde{s}^3) = \tilde{d} - 2\tilde{s}^2.
        \end{aligned}
    \end{equation}
    The limit of the remaining term is
    \begin{equation}\label{2ndConstantLimit}
        \lim_{N\rightarrow \infty} -(c_0N/2)^{-1/3} \frac{s(1-r^2)}{(1-sr)(s-r)} = -\frac{2\tilde{r}}{\tilde{s}^2 - \tilde{r}^2}.
    \end{equation}
    Combining $\eqref{1stConstantLimit}$ and $\eqref{2ndConstantLimit}$, we get the desired result.
\end{proof}

Recall definitions of $\widetilde{f}^{\tilde{r}}$, $\widetilde{f}^{\tilde{s}},$ and $\widetilde{E}$ in $\eqref{fandE}$.
\begin{lem}\label{limit,f_r}
    For any given $u > 0,$ the following limit holds uniformly over $X \in [-u,u],$
    \begin{equation}
        \begin{aligned}
            &\lim_{N\rightarrow \infty} f^r(X) = \widetilde{f}^{\tilde{r}}(X), \quad \lim_{N\rightarrow \infty} f^s(X) = \widetilde{f}^{\tilde{s}}(X),\quad \lim_{N\rightarrow \infty} f^{1/s}(k) = \widetilde{f}^{-\tilde{s}}(X),\quad \lim_{N\rightarrow 1/s} E(X,Y) = \widetilde{E}(X,Y).
        \end{aligned}
    \end{equation}
    Moreover, there exist constants $C$,  $N_0 \in \N$, and $0< \sigma<$\scalebox{0.8}{$\begin{cases}  -\tilde{s} &\text{if }\tilde{r} < \tilde{s}\\
\min(-\tilde{s}, |\tilde{r}|, -(\tilde{s}+|\tilde{r}|)/2) &\text{if }\tilde{s}<\tilde{r} < -\tilde{s}\end{cases}$} such that for all $N\geq N_0$ and $X,Y\geq -u$, we have
    \begin{equation}\label{UpperBoundf_r}
        \begin{aligned}
            &|f^r(X)| \leq \begin{cases}
                C &\text{ if } \tilde{r}<\tilde{s}\\
                Ce^{(|\tilde{r}|+\sigma)X} &\text{ if } \tilde{s}<\tilde{r} < -\tilde{s} \\ 
            \end{cases}, \quad |f^s(X)| \leq Ce^{(\tilde{s}+\sigma)X},\\
            &|f^{1/s}(X)| \leq Ce^{(-\tilde{s}+\sigma)X}, \quad |E(X,Y)| \leq \begin{cases}
        C &\text{ if } \tilde{r } < \tilde{s},\\
        Ce^{(|\tilde{r}|+\sigma)(X+Y)} &\text{ if } \tilde{s}<\tilde{r} < -\tilde{s}.
    \end{cases}
        \end{aligned}
    \end{equation}
\end{lem}

\begin{proof}
We have
    \begin{equation}
        \begin{aligned}
            f^r(k) = \frac{(1-\sqrt{q}r)^{N-2}}{(1-\sqrt{q}/r)^{N-2}}r^{k+1} = \exp{\left(-(N-2)h_0(r)+ \frac{6\sqrt{q}}{1-\sqrt{q}}\log(r) +  (c_0N/2)^{1/3}X\log(r)\right)},\\
        \end{aligned}
    \end{equation}
    where $h_0(x) = \log(1-\sqrt{q}/x) - \log(1-\sqrt{q}x) - \frac{2\sqrt{q}}{1-\sqrt{q}}\log(x).$
    We compute derivatives of $h_0:$
    \begin{equation}\label{derivatives,h0}
        \begin{aligned}
            &h_0^{\prime}(x) = \frac{\sqrt{q}}{x^2-\sqrt{q}x} + \frac{\sqrt{q}}{1-\sqrt{q}x} - \frac{2\sqrt{q}}{(1-\sqrt{q})x}, \quad h_{0}^{\prime}(1) = 0,\\
            &h_0^{\prime\prime}(x) = \frac{-\sqrt{q}(2x - \sqrt{q})}{(x^2 - \sqrt{q}x)^2} +\frac{q}{(1-\sqrt{q}x)^2} + \frac{2\sqrt{q}}{(1-\sqrt{q})x^2}, \quad h_0^{\prime\prime}(1) = 0,\\
            &h_0^{\prime\prime\prime}(1) =  \frac{2\sqrt{q}(1+\sqrt{q})}{(1-\sqrt{q})^3} = c_0 > 0, \text{ and } h_0(1) = 0.
        \end{aligned}
    \end{equation}
    By Taylor expansion, we see that
    \begin{equation}
        \begin{aligned}
            f^{r}(k) = \exp\left( -\frac{\tilde{r}^3}{3} + X\tilde{r} + \mathcal{O}\left((c_0N/2)^{-1/3}(\tilde{r} + X\tilde{r}^2 + \tilde{r}^3)\right)\right) \rightarrow \widetilde{f}^{\tilde{r}}(X),
        \end{aligned}
    \end{equation}
    and if $\tilde{s}< \tilde{r} < -\tilde{s}$, there exists some $0<\sigma<\min(-\tilde{s}, |\tilde{r}|)$ and $N_0 \in \N$ such that for $N\geq N_0$, 
    $|f^r(x)| \leq Ce^{(\tilde{r}+ \sigma)X}.$ We note that the requirement of $\sigma < -(\tilde{s} + |\tilde{r}|)/2$ is not necessary in this case but it appears as a requirement in other cases (in Lemma \ref{limit,constants}, \ref{limit,V}), so we want to keep the conditions for $\sigma$ consistent. We use the exact same proof for $f^s$ and $f^{1/s}$. 
    For the same reason, we have $E\rightarrow \widetilde{E}$. If $\tilde{r} < \tilde{s}$, then $|E(X,Y)| \leq C$ for some constant $C>1$. If $\tilde{s}<\tilde{r} < -\tilde{s}$, then there exist $0<\sigma<\min(-\tilde{s}, \tilde{r})$ and $N_0\in \N$ such that for all $N\geq N_0$, we have $|E(X,Y)| \leq Ce^{(|\tilde{r}|+\sigma)(X+Y)}$.
\end{proof}

\subsection{Asymptotic limit of the double integral}
In this section, we provide a complete argument of using steepest descent method to obtain Airy-type asymptotics.

We use the double integral in  $\widehat{A}_{12}$ to give an example of the steepest descent method.
We use the letter $\mathsf{Q}$ to denote the scaled double integral in $\widehat{A}_{12}$:
\begin{equation}
    \begin{aligned}
        \mathsf{Q}(X,Y) = -(c_0N/2)^{1/3}\oint \limits_{\Gamma_{\sqrt{q}}} \frac{dw}{2\pi\I}\oint \limits_{\Gamma_{1/\sqrt{q}}} \frac{dz}{2\pi\I} \frac{w^{\ell+1}H(z)}{z^{k+2}H(w)}\frac{(zw-1)(z-r)}{(z^2-1)(w-r)(z-w)}.
    \end{aligned}
\end{equation}

\begin{lem}\label{steepestDescent}
        For any given $u>0$, the following limit holds uniformly over $X,Y \in [-u,u]$,
    \begin{equation}
    \begin{aligned}
        &\lim_{N\rightarrow \infty} \mathsf{Q}(X,Y) = -\int\limits_{{}_{}\wcu\, {}_{0,\zeta}}\! \frac{d\omega}{2\pi\I}\!\int\limits_{{ }_{0,\omega}\zcd } \!\frac{d\zeta}{2\pi\I} e^{\frac{\zeta^3}{3} - \frac{\omega^3}{3} -X\zeta + Y\omega} \frac{(\zeta + \omega)(\zeta - \tilde{r})}{2\zeta(\omega- \tilde{r})(\zeta - \omega)}.
    \end{aligned}
    \end{equation}
    Furthermore, there exist a constant $\eta>\max(-2\tilde{s},2|\tilde{r}|)$, a positive constant $C$, and $N_0\in \N$ such that for all $N \geq N_0$ and $X,Y > -u,$ we have
    \begin{equation}
        \begin{aligned}
            |(c_0N/2)^{1/3}\mathsf{Q}(X,Y)| < Ce^{-\eta(X+Y)}.
        \end{aligned}
    \end{equation}
\end{lem}
\begin{proof} 
The first step is to rewrite the integrand in the exponential form. Recall the change of variables for $k,\ell$. Let us write
\begin{equation}
\begin{aligned}
    \frac{H(z)w^{\ell}}{H(w)z^{k}} &= \exp \left((N-2)(h_0(z) - h_0(w)) - (c_0N/2)^{1/3}(X\log(z) - Y\log(w)) - \frac{4\sqrt{q}}{1-\sqrt{q}}(\log(z)- \log(w))\right),
\end{aligned}
\end{equation}
where
\begin{equation}
    \begin{aligned}
        h_0(x) &= \log(1-\sqrt{q}/x) - \log (1-\sqrt{q}x) - \frac{2\sqrt{q}}{1-\sqrt{q}}\log (x),\\
    \end{aligned}
\end{equation}
and we know by $\eqref{derivatives,h0},$
\begin{equation}
    \begin{aligned}
        h_0^{\prime}(1) = 0,\quad  h_0^{\prime\prime}(1) = 0,\quad c_0 = h_0^{\prime\prime\prime}(1) = \frac{2\sqrt{q}+2q}{(1-\sqrt{q})^3}.
    \end{aligned}
\end{equation}
We are going to replace $(N-2)$ by $N$ because this change does not affect the limit and would make the formula more concise. As a result, we will remove $\frac{-4\sqrt{q}}{1-\sqrt{q}}(\log(z)-\log(w)).$

\scalebox{0.85}{
\begin{minipage}{0.6\textwidth}  
    \begin{tikzpicture}\label{contours}
    \draw[thick,->] (-1,0) -- (4,0) node[right] {$\mathbb{R}$}; 

    \draw[thick,-] (1,0) -- ++(60:1);  
    \draw[thick,-] (1,0) -- ++(-60:1); 

    \draw[thick] (3/2,0.866) -- ++(90:0.7215);
    \draw[<-,thick] (3/2,1.587) -- ++(90:0.73);
    \draw[thick,->] (3/2,-0.866) -- ++(-90:0.7215);
    \draw[thick] (3/2,-1.587) -- ++(-90:0.73);

    \draw[thick] (3.809,0) arc[start angle=0,end angle=90,radius=2.309cm];
    \draw[thick,->] (3.809,0) arc[start angle=0,end angle=1,radius=0.1cm];
    \draw[thick] (3.809,0) arc[start angle=0,end angle=-90,radius=2.309cm];

    \draw[gray,thick,->] (1,0) -- ++(120:1);  
    \draw[gray,thick,-] (1,0) -- ++(-120:1); 

    \draw[gray,thick] (-0.366,0) arc[start angle=180,end angle=90,radius=0.866cm];
    \draw[gray,thick,->] (-0.366,0) arc[start angle=-180,end angle=-90,radius=0.866cm];

    \fill (1,0) circle (1.3pt);
    \fill (-0.15,0) circle (1.3pt);
    \node at (1,-0.25) {\footnotesize{1}};
    \node at (-0.15,-0.25) {\footnotesize{0}};
    \node at (2.2,-0.3) {\footnotesize{$1/\sqrt{q}$}};
    \fill (2.1,0) circle (1.3pt);
    

\end{tikzpicture}
\captionof{figure}{Dark path for $z$ and grey path for $w$}
\label{fig:contour}
\end{minipage}
\hfill  
\begin{minipage}{0.6\textwidth}
    \begin{tikzpicture}
    \draw[thick,->] (-1.5,0) -- (3.5,0) node[right] {$\mathbb{R}$}; 
    
    \draw[dashed] (1,0) -- ++(60:1);  
    \draw[dashed] (1,0) -- ++(-60:1); 

    \draw[thick] (3/2,0.866) -- ++(60:1); 
    \draw[thick,<-] (2,1.732) -- ++(60:1);
    \draw[thick,->] (3/2,-0.866) -- ++(-60:1);
    \draw[thick] (2,-1.732) -- ++(-60:1);

    \draw[thick] (2,0) arc[start angle=0,end angle=60,radius=1cm];
    \draw[thick] (2,0) arc[start angle=0,end angle=-60,radius=1cm];

    \node at (1.1, 0.55) {\tiny{$\eta (c_0N/2)^{-1/3}$}};
    \node at (1.1, -0.55) {\tiny{$\eta (c_0N/2)^{-1/3}$}};
    \node at (1,-0.25) {$1$};
    \fill (1,0) circle (1.3pt);

    \draw[dashed] (1,0) -- ++(120:1);  
    \draw[dashed] (1,0) -- ++(-120:1); 

    \draw[gray,thick,->] (1/2,0.866) -- ++(120:1); 
    \draw[gray,thick] (0,1.732) -- ++(120:1);
    \draw[gray,thick] (1/2,-0.866) -- ++(-120:1);
    \draw[gray,thick,<-] (0,-1.732) -- ++(-120:1);

    \draw[gray,thick] (0,0) arc[start angle=-180,end angle=-240,radius=1cm];
    \draw[gray,thick] (0,0) arc[start angle=180,end angle=240,radius=1cm];

\end{tikzpicture}

\captionof{figure}{a $\mathcal{O}(N^{-1/3})$ region near 1}
\label{fig:modify1/3}
\end{minipage}}

The second step is to find steepest descent paths. We describe the curve with nonnegative imaginary part. The whole path is obtained by taking union with its conjugation. For a $w$-contour that encloses $\sqrt{q}$, we choose
\begin{equation}
    \begin{aligned}
        &\gamma_1 := \{1 + te^{2\pi\I/3} : t \in [0,4/5) \},\quad \gamma_2 := \bigg\{\frac{3}{5}+ \frac{2\sqrt{3}}{5}e^{\I\theta} : \theta \in [\pi/2, \pi]\bigg\}.
    \end{aligned}
\end{equation}
Along $\gamma_1,$ the real part $\text{Re}(h_0(1 + te^{2\pi\I/3} ))$ first increases from $0$ and then decreases to a positive value. We computed that
\begin{equation}
    \begin{aligned}
        &\frac{d}{dt}\text{Re}(h_0(1+te^{2\pi\I/3}))
        = \scalebox{0.95}{\(\frac{(\sqrt{q} - q - q^{3/2} + q^2)(-2t^2 + 2t^3 + t^4)+ (q + q^{3/2})(2t^5 - t^4)}{2(-1+\sqrt{q})(1-t+t^2)(1-2\sqrt{q}+q - t + \sqrt{q}t + t^2)(1-2\sqrt{q} + q + \sqrt{q}t - qt + qt^2)},\)}
    \end{aligned}
\end{equation}
where the denominator is always negative whereas the numerator first stays negative for $t \in [0,1/2]$ and then is monotonically increasing for $t \in [1/2,1).$ We have that $\text{Re}(h_0(1/2 + \sqrt{3}\I/2)) = 0.$ This means that the real part first increases and then decreases to a positive number at $w = 3/5 + \I 2\sqrt{3}/5$, implying the real part stays positive along $\gamma_1$.
Along $\gamma_2,$ the derivative of the real part $\text{Re}(h_0(3/5 + 2\sqrt{3}e^{\I\theta}/5 ))$ is
\begin{equation}
    \begin{aligned}
        &\frac{d}{d\theta} \text{Re}(h_0(3/5+ 2\sqrt{3}e^{\I\theta}/5) 
        = \scalebox{0.95} {\(\frac{16(1+\sqrt{q})\sqrt{q}(\cos(\theta)(90-84\sqrt{q}+90q) + \sqrt{3}(-2(5+12\sqrt{q}+5q) + 3(15-34\sqrt{q} +15q)\cos(2\theta)))\sin(\theta)}{(-1+\sqrt{q})(7+4\sqrt{3}\cos(\theta))(21-30\sqrt{q}+25q - 4\sqrt{3}(-3+5\sqrt{q})\cos(\theta)(25-30\sqrt{q}+21q +4\sqrt{3}(-5+2\sqrt{q})\sqrt{q}\cos(\theta))}\)},
    \end{aligned}
\end{equation}
where both the numerator and the denominator are negative
for $\theta \in [\pi/2,\pi]$ and $q\in (0,1).$
This implies that the real part increases from $0$ to a positive number. Then, on its conjugate component, we have the real part decreases to a positive value at $w = 3/5-2\sqrt{3}\I/5$.

We choose a $z$-contour that encloses $1/\sqrt{q}$. We describe the path that has negative imaginary part and the whole path is obtained by taking the union with its conjugate.
\begin{equation}
    \begin{aligned}
        &\tau_1 := \{ 1+xe^{-\pi\I/3} : x \in [0,1)\},\quad
        \tau_2 := \{ 3/2-y\I : y \in [\sqrt{3}/2,2/\sqrt{q})\},\\
        &\tau_3 := \bigg\{3/2+\frac{2}{\sqrt{q}}\cos(\theta) +\I \frac{2}{\sqrt{q}}\sin(\theta) : \theta \in [-\pi/2,0]\bigg\}.
    \end{aligned}
\end{equation}
Along $\tau_1,$ the real part is strictly decreasing in $x$ because
\begin{equation}
    \begin{aligned}
        &\frac{d}{dx}\left( \text{Re}(h_0(1+xe^{-\pi \I/3}))\right)
        = \scalebox{0.95}{\(\frac{(1+\sqrt{q})\sqrt{q}x^2(2+2x-x^2 + q(2+2x-x^2) -\sqrt{q}(4+4x-3x^2-2x^3))}{2(-1+\sqrt{q})(1+x+x^2)(1+q+x+x^2 - \sqrt{q}(2+x))(1-\sqrt{q}(2+x) + q(1+x+x^2))},\)}
    \end{aligned}
\end{equation}
where the numerator is positive for $x \in (0,1)$ and the denomiantor is negative for $x \in [0,1).$
Along $\tau_2,$ the real part is strictly decreasing in $y$ because
\begin{equation}
    \begin{aligned}
        \frac{d}{dy}\left(\text{Re}(h_0(3/2-y\I))\right) = \scalebox{0.95}{\(\frac{4(1+\sqrt{q})\sqrt{q}y(8(3+4y^2) + 8q(3+4y^2) + \sqrt{q}(-47 -56y^2 +16y^4))}{(\sqrt{q}-1)(9+4y^2)(9-12\sqrt{q}+4q +4y^2)(4-12\sqrt{q}+q(9+4y^2))},\)}
    \end{aligned}
\end{equation}
where the numerator is positive due to \begin{equation}
\begin{aligned}
    &\scalebox{0.9}{\(8(3+4y^2) + 8q(3+4y^2) + \sqrt{q}(-47 -56 y^2+16y^4)
    =8(3+4y^2)(q-2\sqrt{q}+1) + \sqrt{q} + 8y^2\sqrt{q} >0,\)}
\end{aligned}
\end{equation}
and the denominator is negative because of $(\sqrt{q}-1).$
Hence, the derivative is always negative for $y \in [\sqrt{3}/2,2/\sqrt{q}),$ implying that this is a descent path.
Along $\tau_3,$ depending on $q,$ we have that for $\theta \in [-\pi/2,0]$,
$\text{Re}\left(h_0(\tau_3)\right)$ behaves like one of the three cases . It can strictly decrease, or strictly increase, or first strictly decrease and then strictly increase. The real part for poistive imaginary part, i.e., $\theta \in [0,\pi/2],$ behaves in a symmetric pattern accordingly. The derivative is so lengthy that we do not write it down. However, it is enough to check that three endpoints have negative values for the real part to be strictly negative along the whole curve $\tau_3$. We have for $q\in (0,1)$,
\begin{equation}\notag
\begin{aligned}
    &\scalebox{0.95}{\(\text{Re}(h_0(3/2 - 2\I/\sqrt{q})) = \text{Re}(h_0(3/2 + 2\I/\sqrt{q}))
    = \frac{\sqrt{q}\log(9/4+4/q)}{\sqrt{q}-1} - \frac{\log(5-3\sqrt{q}+9q/4)}{2} + \log\left(1-\frac{12q^{3/2} - 4q^2}{16+9q}\right) <0,\)}\\
    &\scalebox{0.95}{\(\text{Re}(h_0(3/2+2/\sqrt{q})) = \frac{2\sqrt{q}\log(3/2 + 2/\sqrt{q})}{\sqrt{q}-1} +\log\left( 1-\frac{2q}{4+3\sqrt{q}} \right) -\frac{\log(1+3\sqrt{q}+9q/4)}{2} <0.\)}
\end{aligned}
\end{equation}
Hence, the real part of the $h_0$ stays negative on $\tau_1\cup \tau_2\cup \tau_3$. We have that on its conjugate the real part still stays negative.

Therefore, $(\gamma_1\cup\gamma_2)\cup \overline{(\gamma_1\cup\gamma_2)}$ is a steepest descent path for $w$ and $(\tau_1\cup\tau_2\cup\tau_3)\cup \overline{(\tau_1\cup\tau_2\cup\tau_3)}$ is a steepest descent path for $z$.
By Cauchy's theorem, we can freely deform the contours in a $\mathcal{O}(N^{-1/3})$ neighborhood near $1$, provided that the deformed $w$-contour includes $\sqrt{q}$ but excludes $r$, the deformed $z$-contour includes $1/\sqrt{q}$, and the $z$- and $w$-contours do not intersect.

Let $\eta\in \R_{>|\tilde{r}|}$.
We use $\mathcal{U}$ and $\mathcal{V}$ to denote the contours for $z$ and $w$ after modifying a $\eta (c_0N/2)^{-1/3}$ neighborhood near $1$ as shown in the figure $\ref{fig:modify1/3}$. The constant $\eta$ can be any positive constant greater than $|\tilde{r}|$ because modifying the contour in a region of $\mathcal{O}(N^{-1/3})$ does not affects poles at $z =1\sqrt{q}$ and $w = \sqrt{q}$ and $\eta>|\tilde{r}|$ ensures that the pole at $w = r$ is not included. Let $\mathcal{U}_{loc,\delta}$ and $\mathcal{V}_{loc,\delta}$ denote the intersection of the $z-$contour and respectively $w-$contour with a ball of radius $\delta$ centered at $1$. We want to prove that for any small $\delta >0,$ the integral restricted to $\mathcal{U}_{loc,\delta}\cup \mathcal{V}_{loc,\delta}$ has the desired limit and the integral restricted to $(\mathcal{U} \setminus\mathcal{U}_{loc,\delta})\cup (\mathcal{V} \setminus\mathcal{V}_{loc,\delta})$ vanishes as $N \rightarrow \infty.$
Let $\widetilde{\mathcal{U}}_{loc,\delta}$ and $\widetilde{\mathcal{V}}_{loc,\delta}$
denoted the corresponding contours for $\zeta$ and $\omega$ after the change of variable performed on $\mathcal{U}_{loc,\delta}$ and $\mathcal{V}_{loc,\delta}$.

To prove the limit, we make two approximations. Firstly, we approximate $N(h_0(z)-h_0(w))-(c_0N/2)^{1/3}(X\log(z)-Y\log(w))$ by its taylor expansion $\frac{\zeta^3}{3} - \frac{\omega^3}{3} -X\zeta + Y\omega$. Secondly, we approximate $(c_0N/2)^{-1/3}\frac{w(zw-1)(z-r)}{z^2(z^2-1)(w-r)(z-w)}$ by $\frac{(\zeta + \omega)(\zeta - \tilde{r})}{2\zeta(\omega - \tilde{r})(\zeta - \omega)}.$ Let $E_1$ and $E_2$ be errors generated by the first and second approximations respectively. By applying the inequality $|e^x-1|\leq |x|e^{|x|}$ to the first approximation, we have that the total error is bounded by $E_1 + E_2$ where
\begin{equation}
    \begin{aligned}
        E_1 = &\frac{1}{(2\pi)^2} \oint \limits_{\widetilde{\mathcal{U}}_{loc,\delta}} d\zeta \oint\limits_{\widetilde{\mathcal{V}}_{loc,\delta}} d\omega \, |\Delta| e^{|\Delta|}\times e^{\text{Re}(\frac{\zeta^3}{3} - \frac{\omega^3}{3} - X\zeta + Y\omega)}\bigg|\frac{(c_0N/2)^{-1/3}w(zw-1)(z-r)}{z^2(z^2-1)(w-r)(z-w)}\bigg|,\\
        \Delta = &N(h_0(z) - h_0(w)) - (c_0N/2)^{1/3}(X\log(z)- Y\log(w)) - \left(\frac{\zeta^3}{3} - \frac{\omega^3}{3} - X\zeta + Y\omega\right),
    \end{aligned}
\end{equation}
and 
\begin{equation}
    \begin{aligned}
        E_2 = \frac{1}{(2\pi)^2}\oint \limits_{\widetilde{\mathcal{U}}_{loc,\delta}}  d\zeta \oint\limits_{\widetilde{\mathcal{V}}_{loc,\delta}} d\omega \,e^{\text{Re}(\frac{\zeta^3}{3} - \frac{\omega^3}{3} - X\zeta + Y\omega)} &\bigg|\frac{(\zeta - \tilde{r})}{(\omega - \tilde{r})(\zeta - \omega)}\bigg|\bigg|\frac{w(zw-1)}{z^2(z^2-1)}-\frac{(\zeta+\omega)}{2\zeta}\bigg|\\
        = \frac{1}{(2\pi)^2}\!\!\oint \limits_{\widetilde{\mathcal{U}}_{loc,\delta}}\!\!  d\zeta\!\!\! \oint\limits_{\widetilde{\mathcal{V}}_{loc,\delta}}\!\! d\omega \,e^{\text{Re}(\frac{\zeta^3}{3} - \frac{\omega^3}{3} - X\zeta + Y\omega)} &\bigg|\frac{(\zeta - \tilde{r})}{(\omega - \tilde{r})(\zeta - \omega)}\bigg|\\
\bigg(\bigg|\left(\frac{zw-1}{z^2-1} - \frac{\zeta + \omega}{2\zeta}\right)\left(\frac{w}{z^2}\right)\bigg| &+ \bigg| \frac{(\zeta + \omega)}{2\zeta}\bigg(\frac{w}{z^2}-1\bigg)\bigg|\bigg),\\
    \end{aligned}
\end{equation}
where
\begin{equation}
\frac{zw-1}{z^2-1} - \frac{\zeta+\omega}{2\zeta} = (c_0N/2)^{-1/3}\frac{(\omega- \zeta)}{4+2\zeta(c_0N/2)^{-1/3}}, \quad \frac{w}{z^2} -1 = (c_0N/2)^{-1/3}\left(\frac{\omega - 2\zeta - \zeta^2(c_0N/2)^{-1/3}}{(1+\zeta(c_0N/2)^{-1/3})^2}\right).
\end{equation}
Notice that since $|(c_0N/2)^{-1/3}\zeta| = |z-1| \leq \delta$ for $z \in \mathcal{U}_{loc,\delta}$ and respectively $|w-1| \leq \delta$, we know for $E_1$,
\begin{equation}
    \begin{aligned}
        &|\Delta| \leq N \sup_{x\in B(1,\delta)}|h_0^{(4)}(x)|\left(\frac{|z-1|^4}{4!} + \frac{|w-1|^4}{4!}\right) + (c_0N/2)^{1/3}\sup_{y\in B(1,\delta)} |y^{-2}|\left(|X|\frac{|z-1|^2}{2} +|Y|\frac{|w-1|^2}{2}\right)\\ 
        &= \frac{(c_0N/2)^{-1/3}}{12c_0}\sup_{x\in B(1,\delta)}|h_0^{(4)}(x)|(|\zeta|^4 + |\omega|^4) + \frac{(c_0N/2)^{-1/3}}{2}\sup_{y\in B(1,\delta)} |y^{-2}|(|X\zeta^2| + |Y\omega^2|) \\
        &\leq \frac{\delta}{12c_0}\sup_{x\in B(1,\delta)} |h_0^{(4)}(x)|\left( |\zeta|^3 + |\omega|^3\right) + \frac{\delta}{2}\sup_{y\in B(1,\delta)} |y^{-2}| (|X\zeta| + |Y\omega|)\\
        &\leq C(|\zeta|^3+|\omega|^3 + |X\zeta| + |Y\omega|),\\
\end{aligned}
\end{equation}
where $C = \max \big\{\delta\sup_{x\in B(1,\delta)} |h_0^{(4)}(x)|/12c_0,\, \delta\sup_{y\in B(1,\delta)} |y^{-2}|/2\big\}$. 

We also set $C_1 = \max \big\{ \sup_{x\in B(1,\delta)} |h_0^{(4)}(x)|/12c_0,\, \sup_{y\in B(1,\delta)} |y^{-2}|/2 \big\}.$
Therefore, we know that
\begin{equation}
\begin{aligned}
        &|\Delta|e^{|\Delta|} \leq (c_0N/2)^{-1/3} C_1 (|\zeta|^4 + |\omega|^4 + |X\zeta^2| + |Y\omega^2|) e^{C(|\zeta|^3+|\omega|^3 + |X\zeta| + |Y\omega|)}.
    \end{aligned}
\end{equation}
We simply need to choose $\delta$ small enough such that $C < 1/3,$ which ensures that the dominating term in the integrand of $E_1$ is still $e^{\text{Re}(\frac{\zeta^3}{3} - \frac{\omega^3}{3} - X\zeta + Y\omega)}$. We know that $\sup_{x\in B(1,\delta)}|h_0^{(4)}(x)|$ is bounded for $\delta$ small enough because $h_0^{(4)}(1) = \frac{12(\sqrt{q}+q)}{(-1+\sqrt{q})^3}$ and $h_0^{(4)}(z)$ is continuous at $z=1.$
Therefore,
$E_1$ goes to zero by dominated convergence theorem because the leading term $e^{\zeta^3/3 - \omega^3/3}$ has exponential decay and $|\Delta|$ contains a factor of $(c_0N/2)^{-1/3}$ which goes to zero in the limit.
$E_2$ also goes to zero by dominated convergence theorem and the exponential decay due to the leading term $e^{\text{Re}(\zeta^3/3 - \omega^3/3)}$ and the factor $(c_0N/2)^{-1/3}$.
It is obvious that the convergence of $E_1$ and $E_2$ to zero holds uniformly over $X,Y \in [-u,u].$

Finally, we can extend the integration path $\widetilde{\mathcal{U}}_{loc,\delta}$ and $\widetilde{\mathcal{V}}_{loc,\delta}$ to infinity due to the exponential decay of the integrand for $\zeta,\omega$ in direction $\pm \pi/3$ and $\pm 2\pi/3$. By Cauchy's theorem and such exponential decay, we can freely deform the contours. The $\zeta-$contour can pass through any $x\in \R_{>0}$ and the $\omega-$contour can go through any point $y \in \R_{<\min(0,\tilde{r})}$.

Next, we prove that integration over 
$(\mathcal{U} \setminus\mathcal{U}_{loc,\delta})\cup (\mathcal{V} \setminus\mathcal{V}_{loc,\delta})$ vanishes in the limit. Since the contours we chose are steepest descent paths, there exists $\epsilon(\delta) <0$ such that $\text{Re}(h_0(z) - h_0(w)) < \epsilon$ for all $z \in (\mathcal{U} \setminus\mathcal{U}_{loc,\delta})$ and $w \in(\mathcal{V} \setminus\mathcal{V}_{loc,\delta}).$ Since $(\mathcal{U} \setminus\mathcal{U}_{loc,\delta}) \cup(\mathcal{V} \setminus\mathcal{V}_{loc,\delta})$ is a bounded compact set and $X,Y\in[-u,u]$, there exists $M>0$ such that $|X\log(z)|, |Y\log(w)| <M.$ Based on the construction of the contours, there exists a constant $C_0$ such that for $(z,w) \in (\mathcal{U} \setminus\mathcal{U}_{loc,\delta})\cup (\mathcal{V} \setminus\mathcal{V}_{loc,\delta}),$
\begin{equation}
    \begin{aligned}
        \bigg| \frac{w(zw-1)(z-r)}{z^2(z^2-1)(w-r)(z-w)} \bigg| < C_0.
    \end{aligned}
\end{equation}
Hence, we have for some constant $C>0,$
\begin{equation}
    \begin{aligned}
    |(c_0N/2)^{1/3}\mathsf{Q}(X,Y)| \leq C e^{\epsilon N + 2(c_0N)^{1/3}M},
    \end{aligned}
\end{equation}
which goes to zero as $N \rightarrow \infty$ due to the dominating term $e^{\epsilon N}.$

Lastly, we show the upper bound.
When $X,Y$ are in a compact interval $[-u,u],$ we have shown (in the above proof) that there exists $N_0 >0$ such that for all $N \geq N_0$ and for all $X,Y \in [-u,u],$ $|(c_0N/2)^{1/3}\mathsf{Q}(X,Y)|< C$ for some large constant $C$ that depends on $u,N_0$ but independent of $X,Y$. So now we consider $X,Y > u>0.$
Recall that we have modified a $\eta(c_0N/2)^{-1/3}$ neighborhood near $1$. By the construction of the steepest descent path, we have $\text{Re}(h_0(z)- h_0(w))<0$ on $\mathcal{U}\cup \mathcal{V}$. 
Given any fixed $\eta>0$ and any $N\in \N$, we have
\begin{equation}
    \begin{aligned}
        &\text{Re}(\log(z)) \geq \log(1+\eta (c_0N/2)^{-1/3}/2)\text{ for all }z\in \mathcal{U},\\
        &\text{Re}(\log(w)) \leq \log(1-\eta(c_0N/2)^{-1/3} +\eta^2(c_0N/2)^{-2/3})\text{ for all }w\in \mathcal{V}.
    \end{aligned}
\end{equation}
For any $\epsilon<1/2$, there exist constants $c_1,c_2>0$ that depends on $\epsilon$ such that
\begin{equation}
    \begin{aligned}
        &\log(1+y)>c_1y,\quad
        \log(1-y +y^2)< -c_2y \text{  for all } y < \epsilon.
    \end{aligned}
\end{equation}
We fix $\epsilon$ and decide $c_1,c_2>0.$ For some large enough $N_0,$ we can find $\eta>\max(-2\tilde{s},2|\tilde{r}|)$ such that for all $N\geq N_0,$ 
\begin{equation}
\text{Re}(\log(z))-(c_0N/2)^{-1/3}\eta >0, \quad \text{Re}(\log(w)) + (c_0N/2)^{-1/3}\eta <0.
\end{equation} Then for some constant $C_0$, we have
\begin{equation}
\begin{aligned}
    |(c_0N/2)^{1/3}&\mathsf{Q}(k,\ell)|
    \leq  \frac{C_0 e^{-\eta( X+ Y)}}{(2\pi)^2}\\
    &\oint \limits_{\mathcal{U}}\! dz \!\oint \limits_{\mathcal{V}}\! dw \, e^{N\text{Re}(h_0(z)- h_0(w))-(c_0N/2)^{1/3} \left(X(\text{Re}(\log(z))-\eta (c_0N/2)^{-1/3}) - Y(\text{Re}(\log(w))+\eta (c_0N/2)^{-1/3}) \right)} \\
    &\times\bigg|\frac{(c_0N/2)^{1/3}w(zw-1)(z-r)}{z^2(z^2-1)(w-r)(z-w)}\bigg|,\\
\end{aligned}
\end{equation}
where the double integral has a limit due to the exponential decay of the integrand. Hence, we get the desired upper bound $|(c_0N/2)^{1/3}\mathsf{Q}(k,\ell)| \leq Ce^{-\eta(X+Y)}$.
\end{proof}

\subsection{Asymptotic limit of the kernel}
\begin{lem}\label{lem:limit&bound}
    For any given $u>0$, the following limit holds uniformly over $X,Y \in [-u,u]$,
    \begin{equation}
    \begin{aligned}
        &\lim_{N\rightarrow \infty} (c_0N/2)^{2/3}\widehat{K}_{11}(X,Y) = \widetilde{K}_{11}(X,Y), \quad 
        \lim_{N\rightarrow \infty} (c_0N/2)^{1/3}\widehat{K}_{12}(X,Y) = \widetilde{K}_{12}(X,Y),\\
        &\lim_{N\rightarrow \infty}(c_0N/2)^{1/3}\widehat{K}_{21}(X,Y)=\widetilde{K}_{21}(X,Y) ,\quad \lim_{N\rightarrow \infty} \widehat{K}_{22}(X,Y) = \widetilde{K}_{22}(X,Y).
    \end{aligned}
    \end{equation}
    Furthermore, there exist constants $a>\max(-2\tilde{s},2|\tilde{r}|)$, $0<\sigma<$\scalebox{0.8}{$\begin{cases}  -\tilde{s} &\text{if }\tilde{r} < \tilde{s}\\
\min(|\tilde{r}|, -(\tilde{s}+|\tilde{r}|)/2) &\text{if }\tilde{s}<\tilde{r} < -\tilde{s}\end{cases}$}, $C>0$, and $N_0\in \N$ such that for all $N \geq N_0$ and $X,Y > -u,$ we have
    \begin{equation}
        \begin{aligned}
            \widehat{K}^{\text{scaled}} := \begin{pmatrix}
                (c_0N/2)^{2/3}\widehat{K}_{11} & (c_0N/2)^{1/3}\widehat{K}_{12}\\
                (c_0N/2)^{1/3}\widehat{K}_{21} & \widehat{K}_{22}
            \end{pmatrix} \leq \begin{cases}
                \begin{pmatrix}
                    Ce^{-a(X+Y)} & Ce^{-aX}\\
                    Ce^{-aY} & C
                \end{pmatrix} &\text{if } \tilde{r} < \tilde{s},\\
                \begin{pmatrix}
                    Ce^{-a(X+Y)} & Ce^{-aX+(|\tilde{r}|+\sigma)Y}\\
                    Ce^{(|\tilde{r}|+\sigma)X-aY} & Ce^{(|\tilde{r}|+\sigma)(X+Y)} 
                \end{pmatrix}&\text{if }\tilde{s}<\tilde{r}<-\tilde{s}
            \end{cases}
        \end{aligned}
    \end{equation}
\end{lem}
\begin{proof}
    We explain the scaling factor for $11$-entry of $\widehat{K}$. For $\widehat{A}_{11}$, after inserting new variables, $dw,dz$ each contributes $(c_0N/2)^{-1/3}$ and $(zw-1)(z-r)(1-wr)$ contributes $(c_0N/2)^{-1}$ whereas $(z^2-1)^{-1}(1-w^2)^{-1}(z-w)^{-1}$ contributes $(c_0N/2)$. Therefore, we need to rescale $\widehat{A}_{11}$ by $(c_0N/2)^{2/3}$ for it to have a limit.  We use the steepest descent method explained in Lemma $\ref{steepestDescent}$ to find limits of the double integral with poles at $z=1/\sqrt{q}$ and $w=\sqrt{q}$. For single integrals $\widehat{P}, \widehat{Q}$, we can still adopt the steepest descent method but only use the steepest descent path $\gamma_1,\gamma_2$ for $\widehat{P}$ and the steepest descent path $\tau_1,\tau_2,\tau_3$ for $\widehat{Q}$. The upper bounds for contour integrals follow from the same argument as in Lemma $\ref{steepestDescent}$ and the upper bounds for $f^r$ and $E$ are obtained in $\eqref{UpperBoundf_r}$.
\end{proof}

\begin{remark}
   For $\widetilde{A}_{11}$, we choose the contour for $\zeta$ to lie to the right of $-\tilde{s}$.
   While this choice is not strictly necessary for $\widetilde{A}_{11}$
  to be well-defined—since the contour can be freely deformed to cross any point $x>0$ without encountering poles—it becomes essential when considering terms such as $\brabarket{\widetilde{f}^{-\tilde{s}}}{\widetilde{A}_{11}}{\widetilde{d}_2}$. In these cases, we must ensure that the integral $\int_{\tilde{d}}^{\infty}{e^{-\tilde{s}X}}{\widetilde{A}_{11}(X,Y)}dX$ is well-defined. This is the reason for our specific contour choice. Similar considerations apply to $\widetilde{A}_{12}$ and $\widetilde{A}_{22}$.
\end{remark}

\subsection{Asymptotic limit of the Fredholm Pfaffian}

\begin{lem}\label{lem:fredholmConv}
    For any $d \in \Z_{\geq 0}$ such that $d = \frac{2\sqrt{q}}{1-\sqrt{q}}N + (c_0N/2)^{1/3}\tilde{d}$, we have
    \begin{equation}
        \lim_{N\rightarrow \infty} \mathrm{Pf}(J - \widehat{K})_{\ell^2(\{d+1,\dots\})} = \mathrm{Pf}(J - \widetilde{K})_{L^2((\tilde{d},\infty))}.
    \end{equation}
\end{lem}

\begin{proof}
We know that after change of variable
\begin{equation}\label{finiteSum}
\begin{aligned}
    &\mathrm{Pf}(J - \widehat{K})_{\ell^2(\{d+1,\dots\})} = 1 + \sum_{n=1}^{\infty} \frac{(-1)^n}{n!} \sum_{X_1 = \tilde{d}+ (c_0N/2)^{-1/3}}^{\infty}\cdots\sum_{X_n = \tilde{d} + (c_0N/2)^{-1/3}}^{\infty} \mathrm{Pf}[\widehat{K}(X_i,X_j)]_{i,j=1}^{n}.\\
\end{aligned}
\end{equation}
We prove the convergence of each term in the summation. Let $\epsilon = (c_0N/2)^{-1/3}$, $E_{\epsilon} = \{\tilde{d} + i\epsilon \, : \, i\in \Z_{\geq 1} \},$ $E_{\epsilon}^n$ is the n-times product and let $\vec{X} = (X_1,\dots, X_n)$, $\vec{y} = (y_1,\dots,y_n)$. For each $n$, we rewrite the summation as integrals of a function that takes constant value on each cube. By Hadamard's bound Lemma $\ref{Hadamard}$ and Lemma $\ref{lem:limit&bound}$, we can find $a > \max(-2\tilde{s}, 2|\tilde{r}|)$, a constant $0<\sigma<\min(-\tilde{s}, |\tilde{r}|)$, $C>0$, and $N_0\in\N$ such that for all $N\geq N_0$ and $X,Y\geq -u$, we have 
\begin{equation}\label{LimitEachTerm}
    \begin{aligned}
        &\sum_{X_1 = \tilde{d}+ \epsilon}^{\infty}\cdots\sum_{X_n = \tilde{d} + \epsilon}^{\infty} \mathrm{Pf}[\widehat{K}(X_i,X_j)]_{i,j=1}^{n}\\
        &=\int_{\tilde{d}+\epsilon}^{\infty} \cdots \int_{\tilde{d}+\epsilon}^{\infty} \sum_{\vec{y} \in E_{\epsilon}^n}\mathrm{Pf}\left( \begin{pmatrix}
            \epsilon^{-2}\widehat{K}_{11} & \epsilon^{-1}\widehat{K}_{12}\\
            \epsilon^{-1}\widehat{K}_{21} & \widehat{K}_{22}
        \end{pmatrix}(y_i,y_j)\right)_{i,j=1}^n \Id_{[y_1-\epsilon, y_1)\times \dots\times [y_n-\epsilon,y_n)}(\vec{X}) \prod_{i=1}^n dX_i\\
        &\leq \int_{\tilde{d} }^{\infty} \cdots \int_{\tilde{d}}^{\infty} C^{n}(2n)^{n/2}\prod_{i=1}^{n}e^{-(a-|\tilde{r}| - \sigma)X_i} \prod_{i=1}^{n}dX_i < \infty,\\
    \end{aligned}
\end{equation}
which is a dominating function that is also summable over $n$ after multiplied by $\frac{1}{n!}$.
Then we apply the dominated convergence theorem to get
\begin{equation}
    \begin{aligned}
        \lim_{N\rightarrow \infty} \sum_{X_1 = \tilde{d}+ \epsilon}^{\infty}\cdots\sum_{X_n = \tilde{d} + \epsilon}^{\infty} \mathrm{Pf}[\widehat{K}(X_i,X_j)]_{i,j=1}^{n}   = \int_{\tilde{d} }^{\infty} \cdots \int_{\tilde{d}}^{\infty}  \mathrm{Pf}(\widetilde{K}(X_i,X_j))_{i,j = 1}^{n} \prod_{i = 1}^{n} dX_i.
    \end{aligned}
\end{equation}
Summing up over $n$ gives
\begin{equation}
\begin{aligned}
    &\lim_{N\rightarrow \infty} \mathrm{Pf}( J - \widehat{K})_{\ell^2(\{d+1,\dots\})}
    = \mathrm{Pf}(J - \widetilde{K})_{L^2((\tilde{d},\infty))}.
\end{aligned}
\end{equation}
\end{proof}

\subsection{Asymptotic limit of all other functions}
Recall the definitions of the following pre-limit functions (LHS) and limiting functions (RHS) with tilde in section \ref{HighAsymptotic}.

\begin{lem}\label{limit,functions}For any given $u > 0,$ the following limit holds uniformly over $X\in[-u,u],$
    \begin{equation}
        \begin{aligned}
            &\lim_{N\rightarrow \infty} (c_0N/2)^{1/3}{G}_{1/s}(X) = \widetilde{G}_{-\tilde{s}}(X), \quad \lim_{N\rightarrow \infty} (c_0N/2)^{1/3}{G}_{s}(X) = \widetilde{G}_{\tilde{s}}(X),\quad \lim_{N\rightarrow \infty} (c_0N/2)^{1/3}\widehat{Q}(X) = \widetilde{Q}(X),\\
            &\lim_{N\rightarrow \infty} {R}_{1/s}(X) = \widetilde{R}_{-\tilde{s}}(X), \quad \lim_{N\rightarrow \infty} {R}_{s}(X) = \widetilde{R}_{\tilde{s}}(X),\quad \lim_{N\rightarrow \infty} \widehat{P}(X) = \widetilde{P}(X),\quad
            \lim_{N\rightarrow \infty} \widehat{\mathsf{J}}(X) = \widetilde{\mathsf{J}}(X),\\
            &\lim_{N\rightarrow \infty} \widehat{B}(X,Y) = \widetilde{B}(X,Y),\quad
            \lim_{N\rightarrow \infty} (c_0N/2)^{1/3}{g_1}(X) =  \widetilde{g}_1(X), \quad \lim_{N\rightarrow \infty} d_2(X) = \widetilde{d}_2(X).
        \end{aligned}
    \end{equation}
Moreover, there exist constants $a > \max(-2\tilde{s},2|\tilde{r}|)$, $0<\sigma<$\scalebox{0.8}{$\begin{cases}  -\tilde{s} &\text{if }\tilde{r} < \tilde{s}\\
\min(-\tilde{s}, |\tilde{r}|, -(\tilde{s}+|\tilde{r}|)/2) &\text{if }\tilde{s}<\tilde{r} < -\tilde{s}\end{cases}$}, $C>0$, and $N_0 \in \N$ such that for all $N\geq N_0$ and $X,Y\geq -u$,  
    \begin{equation}
        \begin{aligned}
            &|(c_0N/2)^{1/3}{G}_{1/s}(X)| \leq Ce^{-aX}, \quad |(c_0N/2)^{1/3}{G}_{s}(X)| \leq Ce^{-aX},\quad |\widehat{Q}(X)| \leq Ce^{-aX},\\
            &|(c_0N/2)^{1/3}{R}_{1/s}(X)| \leq Ce^{-aX},\quad |(c_0N/2)^{1/3}{R}_{s}(X)| \leq Ce^{-aX},\quad |\widehat{P}(X)| \leq Ce^{-aX}, \quad |\widehat{\mathsf{J}}(X)| \leq Ce^{-aX},\\
            & \quad |\widehat{B}(X,Y)| \leq Ce^{-a(X+Y)}, \quad |g_1(X)| \leq Ce^{(\tilde{s}+ \sigma)X}, \quad |{d}_2(X)| \leq \begin{cases}
            Ce^{(\tilde{s}+\sigma)X}\text{if }\tilde{r} <\tilde{s},\\
            Ce^{(|\tilde{r}|+\sigma)X} \text{if } 
            \tilde{s} <\tilde{r} < - \tilde{s}
            \end{cases}.
        \end{aligned}
    \end{equation}
\end{lem}

\begin{proof}
    We use steepest descent method to find limits of contour integrals and their upper bounds as in Lemma $\ref{steepestDescent}$. Upper bounds for $g_1$ and $d_2$ follow from $\eqref{UpperBoundf_r}.$
\end{proof}

\begin{remark}
For $\widetilde{B}$, we also need to ensure that the $\zeta$ contour lies to the right of $-\tilde{s}$ to guarantee that terms like $\brabarket{\widetilde{f}^{-\tilde{s}}}{\widetilde{B}}{\widetilde{g}_1}$
are well-defined. Similar considerations apply to $ \widetilde{P}, \widetilde{Q}, \widetilde{\mathsf{J}},\widetilde{G}_{-\tilde{s}},\widetilde{G}_{\tilde{s}}.$
\end{remark}

\begin{lem}\label{limit,constants}
For any $d\in\Z_{\geq 0}$ such that $d = \frac{2\sqrt{q}}{1-\sqrt{q}}N + (c_0N/2)^{1/3}\tilde{d}$, we have the following limits:
    \begin{equation}
    \begin{aligned}
        &\lim_{N\rightarrow \infty}(c_0N/2)^{-1/3} {\widehat{\mu}_d} = \widetilde{\mu}_{\tilde{d}}, \quad \lim_{N\rightarrow \infty} (c_0N/2)^{-1/3}{\widehat{\nu}_d} =  \widetilde{\nu}_{\tilde{d}},\quad \lim_{N\rightarrow \infty} (c_0N/2)^{-1/3}\widehat{\mathcal{A}}_{d} = \widetilde{\mathcal{A}}_{d},\\
        & \lim_{N\rightarrow \infty} (c_0N/2)^{-1/3}\widehat{\mathcal{B}}_{d} =  \widetilde{\mathcal{B}}_{\tilde{d}}, \quad \lim_{N\rightarrow \infty}(c_0N/2)^{-1/3} {\widehat{\mathcal{C}}_d} = \widetilde{\mathcal{C}}_{\tilde{d}}, \quad \lim_{N\rightarrow \infty}(c_0N/2)^{-1/3} {\widehat{\mathcal{D}}_d} =  \widetilde{\mathcal{D}}_{\tilde{d}},\\
        &\lim_{N\rightarrow \infty}(c_0N/2)^{-1/3} {\widehat{\mathcal{E}}_d} =  \widetilde{\mathcal{E}}_{\tilde{d}}.
    \end{aligned}
    \end{equation}
\end{lem}

\begin{proof}
We use the same steepest descent method as in Lemma $\ref{lem:limit&bound}$ to compute limits. We justify how the summation over $d$ converges to integral over $\tilde{d}$. We pick the most complicated term to give an example. Consider $\mathsmaller{\brabarket{f^{1/s}}{\widehat{A}_{12}}{g_1}}$ in $\widehat{\mathcal{E}}_{d}$. Let $\epsilon = (c_0N/2)^{-1/3}$.
Then by Lemmas $\ref{limit,f_r},\ref{lem:limit&bound},\ref{limit,functions},$ there exist constants $C>0,$ $0<\sigma<$\scalebox{0.8}{$\begin{cases}  -\tilde{s} &\text{if }\tilde{r} < \tilde{s}\\
\min(|\tilde{r}|, -(\tilde{s}+|\tilde{r}|)/2) &\text{if }\tilde{s}<\tilde{r} < -\tilde{s}\end{cases}$} such that
\begin{equation}
\begin{aligned}
    &\sum_{k,\ell = d+1}^{\infty} f^{1/s}(k)\widehat{A}_{12}(k,\ell)g_1(\ell) = \sum_{X,Y = \tilde{d}+\epsilon}^{\infty} f^{1/s}(X)\widehat{A}_{12}(X,Y)g_1(Y)\\
    &= \int_{\tilde{d}+\epsilon}^{\infty}\int_{\tilde{d}+\epsilon}^{\infty} \sum_{(y_1,y_2) \in E_{\epsilon}^2}f^{1/s}(y_1)\left(\epsilon^{-1}\widehat{A}_{12}(y_1,y_2)\right)\epsilon^{-1} g_1(y_2)\Id_{[y_1-\epsilon,y_1)\times[y_2-\epsilon,y_2)}(X,Y) dX dY\\
    &\leq\begin{cases}
        C^3\int_{\tilde{d}+\epsilon}^{\infty}\int_{\tilde{d}+\epsilon}^{\infty} e^{-(a+\tilde{s}-\sigma)X}e^{(\tilde{s}+\sigma)Y} dX dY &\text{if }\tilde{r} < \tilde{s},\\
        C^3\int_{\tilde{d}+\epsilon}^{\infty}\int_{\tilde{d}+\epsilon}^{\infty} e^{-(a+\tilde{s}-\sigma)X}e^{(|\tilde{r}|+\tilde{s}+2\sigma)Y} dX dY  &\text{if } \tilde{s}< \tilde{r} < -\tilde{s}
    \end{cases},
\end{aligned}
\end{equation}
where the upper bounds are dominating functions.
Since we already have the pointwise limit of each function, we get by dominated convergence theorem
\begin{equation}
    \lim_{N\rightarrow \infty} \brabarket{f^{1/s}}{\widehat{A}_{12}}{g_1} = \brabarket{\widetilde{f}^{\tilde{s}}}{\widetilde{A}_{12}}{\widetilde{g}_1}.
\end{equation}
\end{proof}

\begin{lem}\label{limit,V}
    For any given $u > 0,$ the following limit holds uniformly over $Y\in[-u,u],$ 
    \begin{equation}
        \begin{aligned}
            \lim_{N\rightarrow \infty} \widehat{V}_{1}(Y) = \widetilde{V}_{1}(Y), \quad \lim_{N\rightarrow \infty} (c_0N/2)^{-1/3}\widehat{V}_{2}(Y) = \widetilde{V}_{2}(Y).
        \end{aligned}
    \end{equation}
    Moreover, there exist constants $a > -\tilde{s}$, $0<\sigma<$\scalebox{0.8}{$\begin{cases}  -\tilde{s} &\text{if }\tilde{r} < \tilde{s}\\
\min(-\tilde{s}, |\tilde{r}|, -(\tilde{s}+|\tilde{r}|)/2) &\text{if }\tilde{s}<\tilde{r} < -\tilde{s}\end{cases}$}, $C >0$ independent of $Y$, and $N_0\in \N$ such that for all $N \geq N_0$ and $Y \geq -u$, we have
    \begin{equation}
        |\widehat{V}_{1}(Y)| \leq Ce^{-aY}, \quad |\widehat{V}_{2}(Y)| \leq \begin{cases}
            C &\text{ if } \tilde{r} < \tilde{s},\\
            Ce^{(|\tilde{r}|+\sigma)Y} &\text{ if } \tilde{s}<\tilde{r} < -\tilde{s}.
        \end{cases}
    \end{equation}
\end{lem}

\begin{proof}
    We pick one term in $\widehat{V}_2$ to discuss its scaling factor $(c_0N/2)^{-1/3}$. Consider the term $\mathsmaller{\frac{(s-r)}{(1-s^2)}\braket{f^{1/s}}{\widehat{P}}\bra{f^r}\widehat{A}_{12}(\ell)}$. From Lemma $\ref{limit,f_r}$ and $\ref{limit,functions}$, we know that $f^s$, $f^r,$ $\widehat{P}$ need no scaling. By Lemma 
    $\ref{lem:limit&bound}$, $\widehat{A}_{12}$ need a scaling factor $(c_0N/2)^{1/3}$, which means that $\widehat{A}_{12}$ contributes a factor of $(c_0N/2)^{-1/3}$. We also notice that the inner product $\mathsmaller{\braket{f^{1/s}}{\widehat{P}}}$ generates a factor $(s-w)^{-1}$, which contributes a factor of $(c_0N/2)^{1/3}$ after scaling. Similarly, $\bra{f^r}\widehat{A}_{12}$ contributes a factor of $(c_0N/2)^{1/3}$. Combining things together, we see that $\widehat{V}_{2}$ needs a scaling factor $(c_0N/2)^{-1/3}$.
    
    By Lemma $\ref{limit,f_r}$, $\ref{steepestDescent}$, $\ref{limit,functions}$, we get limits of all functions involved in $\widehat{V}_1$ and $\widehat{V}_2$ except one term which we compute as follows:
    \begin{equation}
        \lim_{N\rightarrow \infty} \left(\frac{r^{k-d-1}}{s^{d+1}} + \frac{(1-s^2)}{(s-r)s^{k+1}}\right)H(s) = e^{\tilde{r}(X-\tilde{d}) +\frac{\tilde{s}^3}{3}- \tilde{s}\tilde{d}} - \frac{2\tilde{s}}{(\tilde{s}-\tilde{r})}e^{\frac{\tilde{s}^3}{3}-\tilde{s}X}. 
    \end{equation}
    Based on the upper bounds found in Lemma $\ref{limit,f_r}$, $\ref{lem:limit&bound}$, $\ref{limit,functions}$, we can use the same argument as in Lemma $\ref{limit,constants}$ to prove the convergence of summation over $\{d+1,\dots\}$ to integration over $(\tilde{d},\infty)$. This explains the convergence of terms like $\lim_{N\rightarrow \infty} \bra{f^{1/s}}\widehat{A}_{11}\cdot \widehat{A}_{21} = \bra{\widetilde{f}^{-\tilde{s}}}\widetilde{A}_{11}\times \widetilde{A}_{21}$.
    
    All terms except $\frac{s}{(1-s^2)}R_{1/s}$ in $\widehat{V}_2$ are dominating terms because the second variable of $\widehat{A}_{12},$ $\widehat{A}_{22},$ and $E$ are of order $r^{\ell}$. If $\tilde{r} < \tilde{s} <0,$ then $r^{\ell} \leq 1$ and there exists a constant $C$ such that $|(c_0N/2)^{-1/3}\widehat{V}_{2}|\leq C$ for all $Y\geq -u.$ However, if $\tilde{s} < \tilde{r} < -\tilde{s}$, then there exist $N_0\in \N$ and $0<\sigma<\min(-\tilde{s},|\tilde{r}|)$ such that $|r^{\ell}|\leq e^{(|\tilde{r}|+ \sigma)Y}$ for $N\geq N_0$ and $Y\geq -u.$ For $\widehat{V}_{1}$, the second variable of $\widehat{A}_{21}$ and $\widehat{A}_{11}$ is of order $\sqrt{q}^{\ell}.$ We also know the upper bound of $|(c_0N/2)^{1/3}G_{1/s}(\ell)|$ from Lemma $\ref{limit,functions}.$ Hence, we get the desired upper bound for $\widehat{V}_{1}$.
\end{proof}

For the following Lemma, recall Remark \ref{modifed} for the reason of placing $(c_0N/2)^{-1/3}$ inside the Pfaffian.
\begin{lem}\label{limit,V_1,V_2}
    For any $\sqrt{q}\in (0,1)$ and any $d\in \Z$ such that $d = \frac{2\sqrt{q}}{1-\sqrt{q}}N + (c_0N/2)^{1/3}\tilde{d},$ we have
    \begin{equation}
        \begin{aligned}
            \lim_{N\rightarrow \infty} \mathrm{Pf}\bigg(J - \widehat{K} - (c_0N/2)^{-1/3}&\ket{\begin{array}{c}
         g_1 \\
         -d_2   \end{array}}\bra{\widehat{V}_1 \quad \widehat{V}_2} - (c_0N/2)^{-1/3}\ket{\begin{array}{c}
         \widehat{V}_1 \\
         \widehat{V}_2
    \end{array}}\bra{-g_1 \quad d_2}\bigg)_{\ell^{2}(\{d+1,\dots\})} \\
    &= \mathrm{Pf}\left(J - \widehat{K} - \ket{\begin{array}{c}
         \widetilde{g}_1 \\
         -\widetilde{d}_2
    \end{array}}\bra{\widetilde{V}_1 \quad \widetilde{V}_2} - \ket{\begin{array}{c}
         \widetilde{V}_1 \\
         \widetilde{V}_2
    \end{array}}\bra{-\widetilde{g}_1 \quad \widetilde{d}_2}\right)_{L^2((\tilde{d}, \infty))}.
         \end{aligned}
    \end{equation}
\end{lem}

\begin{proof}
    Fix any $u >0.$ By Lemma $\ref{limit,functions}$ and Lemma $\ref{limit,V}$, we know that there exist constants $a > \max(-2\tilde{s}, 2|\tilde{r}|)$, $0<\sigma<$\scalebox{0.8}{$\begin{cases}  -\tilde{s} &\text{if }\tilde{r} < \tilde{s}\\
\min(-\tilde{s}, |\tilde{r}|, -(\tilde{s}+|\tilde{r}|)/2) &\text{if }\tilde{s}<\tilde{r} < -\tilde{s}\end{cases}$}, $C>0$ independent of $X,Y$ and $N_0\in \N$ such that for all $N \geq N_0$ and $X,Y\geq -u,$ we have 
    \begin{equation}
        \begin{aligned}
            &M_1 = \\
            &\begin{pmatrix}
                (c_0N/2)^{1/3}\ketbra{g_1}{\widehat{V}_1} & \ketbra{g_1}{\widehat{V}_2}\\
                \ketbra{-d_2}{\widehat{V}_1} & (c_0N/2)^{-1/3}\ketbra{-d_2}{\widehat{V}_2}
            \end{pmatrix}
            \leq \begin{cases}
                \begin{pmatrix}
                Ce^{(\tilde{s}+\sigma)X}e^{-aY} &  Ce^{(\tilde{s}+\sigma)X}\\
                Ce^{(\tilde{s} +\sigma)X}e^{-aY} & Ce^{(\tilde{s} +\sigma)X}
            \end{pmatrix} &\text{ if }\tilde{r} < \tilde{s},\\
            \begin{pmatrix}
            Ce^{(\tilde{s}+\sigma)X}e^{-aY} &  Ce^{(\tilde{s}+\sigma)X}e^{(|\tilde{r}|+\sigma)Y}\\
                Ce^{(|\tilde{r}| +\sigma)X}e^{-aY} & Ce^{(|\tilde{r}| +\sigma)(X+Y)}
            \end{pmatrix} &\text{ if } 
            \tilde{s}<\tilde{r} < -\tilde{s},
            \end{cases}
        \end{aligned}
    \end{equation}
    and
    \begin{equation}
        \begin{aligned}
        &M_2 = \\
        &\begin{pmatrix}(c_0N/2)^{1/3}\ketbra{\widehat{V}_1}{-g_1} & \ketbra{\widehat{V}_1}{d_2}\\\ketbra{\widehat{V}_2} {-g_1}& (c_0N/2)^{-1/3}\ketbra{\widehat{V}_2}{d_2}
            \end{pmatrix} \leq \begin{cases}
                \begin{pmatrix}
                Ce^{-aX}e^{(\tilde{s}+\sigma)Y} &  Ce^{-aX}e^{(\tilde{s}+\sigma)Y}\\
                Ce^{(\tilde{s}+\sigma)Y} & Ce^{(\tilde{s}+\sigma)Y}
            \end{pmatrix} &\text{ if }\tilde{r} < \tilde{s},\\
            \begin{pmatrix}
            Ce^{-aX}e^{(\tilde{s}+\sigma)Y} &  Ce^{-aX}e^{(|\tilde{r}|+\sigma)Y}\\
                Ce^{(|\tilde{r}| +\sigma)X}e^{(\tilde{s}+\sigma)Y} & Ce^{(|\tilde{r}|+\sigma)(X+Y)}
            \end{pmatrix} &\text{ if } 
            \tilde{s}<\tilde{r} < -\tilde{s},
            \end{cases}
        \end{aligned}
    \end{equation}
    Hence, for three kernels, we get the same upper bound:
    \begin{equation}
       \widehat{K}^{\text{scaled}}, M_1, M_2  \leq \begin{cases}
            \begin{pmatrix}
                Ce^{(\tilde{s}+\sigma)X}e^{(\tilde{s}+\sigma)Y} & Ce^{(\tilde{s}+\sigma)X}\\
                Ce^{(\tilde{s}+\sigma)Y} & C
            \end{pmatrix} &\text{ if }\tilde{r} < \tilde{s},\\
            \begin{pmatrix}
                Ce^{(\tilde{s}+\sigma)(X+Y)} & Ce^{(\tilde{s}+\sigma)X}e^{(|\tilde{r}|+\sigma)Y}\\
                Ce^{(|\tilde{r}+\sigma|X)}e^{(\tilde{s}+\sigma)Y} & Ce^{(|\tilde{r}|+\sigma)(X+Y)}
            \end{pmatrix} & \text{ if } \tilde{s} < \tilde{r} < - \tilde{s}.
        \end{cases}
    \end{equation}
    Applying Hadamard's bound Lemma $\ref{Hadamard}$ and dominated convergence theorem as in Lemma $\ref{lem:fredholmConv}$ gives the desired convergence of Fredholm Pfaffian.
\end{proof}

\subsection{Proof of Theorem $\ref{theorem:limit}$ $(1)$}\label{continuousDerivative}
\begin{proof}[Proof of Theorem~\ref{theorem:limit} $(1)$]
    By Lemmas $\ref{limit,f_r},$ $\ref{lem:limit&bound},$  $\ref{limit,functions},$ 
    $\ref{limit,constants},$
    $\ref{limit,V},$ $\ref{limit,V_1,V_2}$, we have that 
    \begin{equation}
        \lim_{N\rightarrow \infty}(c_0N/2)^{-1/3}\psi\left(\frac{2\sqrt{q}}{1-\sqrt{q}}N + (c_0N/2)^{1/3}\tilde{d}\right) = \mathcal{L}(\tilde{d}).
    \end{equation}
    We define $h_N(x) = (c_0N/2)^{-1/3}\psi\left(\frac{2\sqrt{q}}{1-\sqrt{q}}N + (c_0N/2)^{1/3}x\right)$ and $\epsilon = (c_0N/2)^{-1/3}.$
    \begin{equation}\label{finalResult}
        \begin{aligned}
        &\Pb\left(\frac{G_{N,N}^{\text{stat}} - \frac{2\sqrt{q}}{1-\sqrt{q}}N}{(c_0N/2)^{1/3}}\leq \tilde{d}\right)
        = \frac{1+\tilde{s}\epsilon}{\epsilon(\tilde{s}-\tilde{r})} \frac{h_N(\tilde{d})}{\epsilon} - \frac{2+\epsilon(\tilde{s}+ \tilde{r})}{\epsilon(\tilde{s}-\tilde{r})} \frac{h_N(\tilde{d}-\epsilon)}{\epsilon}
        + \frac{1+\tilde{r}\epsilon}{\epsilon(\tilde{s}-\tilde{r})} \frac{h_N(\tilde{d} - 2\epsilon)}{\epsilon}\\
        &= \frac{h_N(\tilde{d}) - 2h_N(\tilde{d} - \epsilon) + h_N(\tilde{d}-2\epsilon)}{(\tilde{s} - \tilde{r})\epsilon^2}
        + \frac{\tilde{s}(h_N(\tilde{d}) - h_N(\tilde{d}-\epsilon))}{(\tilde{s} - \tilde{r})\epsilon} - \frac{\tilde{r}(h_N(\tilde{d}-\epsilon) - h_N(\tilde{d}-2\epsilon))}{(\tilde{s}- \tilde{r})\epsilon}\\
        &\rightarrow 
        \frac{1}{\tilde{s}-\tilde{r}}\partial_{\tilde{d}}^2 \mathcal{L}(\tilde{d}) + \partial_{\tilde{d}}\mathcal{L}(\tilde{d}).
        \end{aligned}
    \end{equation}
    We will only show that $\lim_{N \rightarrow 0} \frac{h_N(\tilde{d}) - h_N(\tilde{d}- \epsilon)}{\epsilon} = \partial_{\tilde{d}} \mathcal{L}(\tilde{d}).$ Second order derivative will follow from a nearly identical argument.
    There are many terms in the function $h_N(\tilde{d})$, but the core difficulty lies in differentiating the Fredholm Pfaffian, which we will prove first. Other terms, such as $\widehat{A}_{d}$ and $\widehat{\mu}_d$, are comparatively easier to differentiate. We use a superscript $g^{\tilde{d}}$ to denote the dependence of $g$ on $\tilde{d}$. Notice that $g_1,d_2, \widehat{K}$ has no dependence on $\tilde{d}.$ 
    To shorten the notation, we define functions $\mathcal{G}$ and $\mathsf{g}_n$ (only used in this theorem). Fix $n \in \N.$ Let
    \begin{equation}
        \begin{aligned}
             \mathsf{g}_n(\tilde{d},X_1,\dots,X_n) &:= \mathrm{Pf}\left( \widehat{K}(X_i,X_j) + \epsilon\ket{\begin{array}{c}
             g_1 \\
             -d_2   \end{array}}\bra{\widehat{V}_1^{\tilde{d}} \quad \widehat{V}_2^{\tilde{d}}}(X_i,X_j) + \epsilon\ket{\begin{array}{c}
             \widehat{V}_1^{\tilde{d}} \\
             \widehat{V}_2^{\tilde{d}}
            \end{array}}\bra{-g_1 \quad d_2}(X_i,X_j)\right)_{i,j = 1}^n,\\
             \mathcal{G}_n(\tilde{d}) &:= \sum_{X_1 = \tilde{d}+\epsilon}^{\infty}\cdots \sum_{X_n = \tilde{d}+\epsilon}^{\infty} \mathsf{g}_n(\tilde{d},X_1,\dots,X_n).
        \end{aligned}
    \end{equation}
    For some constant $0<\sigma <$ \scalebox{0.8}{$\begin{cases}
        -\tilde{s} &\text{ if } \tilde{r} < \tilde{s}<0\\
        \min(-(\tilde{s} + |\tilde{r}|)/2, |\tilde{r}|)&\text{ if } \tilde{s}<\tilde{r}<-\tilde{s},
    \end{cases}$}, we know that  by Lemma \ref{limit,V_1,V_2}\begin{equation}\label{UpperBounds}
        \begin{aligned}
            &|\epsilon^{-n}\mathsf{g}_n(\tilde{d},X_1,\dots,X_n)| \leq
            \begin{cases}
                C^n (2n)^{n/2}\prod_{i=1}^{n} e^{(\tilde{s}+\sigma)X_i} &\text{ if }\tilde{r} < \tilde{s},\\
                C^n(2n)^{n/2}\prod_{i=1}^{n} e^{(\tilde{s}+|\tilde{r}|+2\sigma)X_i} &\text{ if } \tilde{s}<\tilde{r}<-\tilde{s}
            \end{cases}.\\
           &\lim_{N\rightarrow \infty} \epsilon^{-n}\mathsf{g}_n(\tilde{d},X_1,\dots,X_n) = \mathrm{Pf}\left(\mathcal{Q}_n^{\tilde{d}}\right),\quad
           \lim_{N\rightarrow \infty} \mathcal{G}_n(\tilde{d}) = \int_{\tilde{d}}^{\infty}\!\!\!\cdots \!\!\int_{\tilde{d}}^{\infty}\!\!\!\mathrm{Pf}(\mathcal{Q}_n^{\tilde{d}}) \prod_{i=1}^{n}d X_i,\\
        \end{aligned}
    \end{equation}
    where $\mathcal{Q}_n^{\tilde{d}}= \mathsmaller{\left(\widetilde{K}(X_i,X_j) + \ket{\begin{array}{c}
             \widetilde{g}_1 \\
             -\widetilde{d}_2   \end{array}}\bra{\widetilde{V}_1^{\tilde{d}} \quad \widetilde{V}_2^{\tilde{d}}}(X_i,X_j) +\ket{\begin{array}{c}
             \widetilde{V}_1^{\tilde{d}} \\
             \widetilde{V}_2^{\tilde{d}}
            \end{array}}\bra{-\widetilde{g}_1 \quad \widetilde{d}_2}(X_i,X_j)\right)_{i,j=1}^n}$.
    Then we rewrite
    \begin{equation}\label{derivative}
    \begin{aligned}
        &\mathcal{G}_n(\tilde{d} - \epsilon) - \mathcal{G}_n(\tilde{d}) = \sum_{X_1 = \tilde{d}}^{\infty}\cdots \sum_{X_n = \tilde{d}}^{\infty} \mathsf{g}_n(\tilde{d}-\epsilon,X_1,\dots,X_n) - \sum_{X_1 = \tilde{d}+\epsilon}^{\infty}\cdots \sum_{X_n = \tilde{d}+\epsilon}^{\infty} \mathsf{g}_n(\tilde{d},X_1,\dots,X_n)\\
        &=\sum_{i = 1}^{n}\mathcal{T}_{i} +\!\! \sum_{X_1 = \tilde{d}+\epsilon}^{\infty}\!\! \cdots\!\! \sum_{X_n = \tilde{d}+\epsilon}^{\infty}\!\! \mathsf{g}_n(\tilde{d}-\epsilon, X_1,\dots, X_n) - \mathsf{g}_n(\tilde{d},X_1,\dots,X_n),
    \end{aligned}
    \end{equation}
    where \begin{equation}
    \begin{aligned}
        &\mathcal{T}_i = \sum_{X_1 = \tilde{d}+\epsilon}^{\infty}\!\!\cdots\!\!\! \sum_{X_{i-1} = \tilde{d}+\epsilon}^{\infty}  \,\sum_{X_{i+1}= \tilde{d}}^{\infty} \cdots \sum_{X_{n}= \tilde{d}}^{\infty}  \mathsf{g}_n(\tilde{d}-\epsilon, X_1,\dots,X_{i-1},\tilde{d}, X_{i+1}, \dots,X_{n}) \text{ for } 2\leq i \leq n-1,\\
    \end{aligned}
    \end{equation}
    and \begin{equation}
    \mathcal{T}_{1} = \sum_{X_2 = \tilde{d}}^{\infty} \cdots \sum_{X_n = \tilde{d}}^{\infty}\mathsf{g}_n(\tilde{d}-\epsilon,\tilde{d},X_2\dots,X_n), \quad \mathcal{T}_{n} = \!\! \sum_{X_1 = \tilde{d}+\epsilon}^{\infty}\!\!\!\cdots\!\!\! \sum_{X_{n-1}= \tilde{d}+\epsilon}^{\infty}\!\!\!\!\mathsf{g}_n(\tilde{d}-\epsilon, X_1,\dots, X_{n-1}, \tilde{d}).
    \end{equation}
    By $\eqref{UpperBounds}$, we see that by dominated convergence theorem, \begin{equation}
    \lim_{N\rightarrow\infty }\frac{\mathcal{T}_{i}}{\epsilon} = \int_{\tilde{d}}^{\infty}\!\!\!\cdots \!\!\int_{\tilde{d}}^{\infty}\!\!\!\mathrm{Pf}\left(\mathcal{Q}_n^{\tilde{d}}|_{X_j = \tilde{d}}\right) \prod_{j=1, j\neq i}^{n}d X_j,
    \end{equation}
     where the notation $\mathcal{Q}_n^{\tilde{d}}|_{X_{i}= \tilde{d}}$ means setting the variable $X_i$ as a constant $\tilde{d}$ and keep other variables $X_j, j\neq i$ unchanged and these variables will be integrated.
    For the last term in \eqref{derivative}, we notice that the difference  $(\mathsf{g}_n(\tilde{d})-\mathsf{g}_n(\tilde{d}-\epsilon))/\epsilon$ only produces polynomial factors, while $\eqref{UpperBounds}$ provides exponential upper bounds. By the dominated convergence theorem, we get
    \begin{equation}
        \epsilon^{-1}\sum_{X_1 = \tilde{d}+\epsilon}^{\infty}\!\! \cdots\!\! \sum_{X_n = \tilde{d}+\epsilon}^{\infty}\!\! \left(\mathsf{g}_n(\tilde{d}-\epsilon, X_1,\dots, X_n) - \mathsf{g}_n(\tilde{d},X_1,\dots,X_n)\right) = - \int_{\tilde{d}}^{\infty}\!\!\!\cdots \!\!\int_{\tilde{d}}^{\infty}\!\!\!\partial_{\tilde{d}}\mathrm{Pf}\left(\mathcal{Q}^{\tilde{d}}_n\right) \prod_{i=1}^{n}d X_i.
    \end{equation}
    Combining the results, we obtain
    \begin{equation}
        \begin{aligned}
            \lim_{N\rightarrow \infty}\left(\frac{\mathcal{G}_n(\tilde{d}-\epsilon) - \mathcal{G}_n(\tilde{d})}{\epsilon}\right) 
            &= \int_{\tilde{d}}^{\infty}\!\!\!\cdots \!\!\int_{\tilde{d}}^{\infty}\!\!\!\mathrm{Pf}\left(\mathcal{Q}_n^{\tilde{d}}|_{X_1 = \tilde{d}}\right) \prod_{i=2}^{n}d X_i + \int_{\tilde{d}}^{\infty}\!\!\!\cdots \!\!\int_{\tilde{d}}^{\infty}\!\!\!\mathrm{Pf}\left(\mathcal{Q}_n^{\tilde{d}}|_{X_2 = \tilde{d}}\right)  \prod_{i=1,i \neq 2}^{n}d X_i\\
            &+\dots + \int_{\tilde{d}}^{\infty}\!\!\!\cdots \!\!\int_{\tilde{d}}^{\infty}\!\!\!\mathrm{Pf}\left(\mathcal{Q}_n^{\tilde{d}}|_{X_{n}= \tilde{d}}\right) \prod_{i=1}^{n-1}d X_i - \int_{\tilde{d}}^{\infty}\!\!\!\cdots \!\!\int_{\tilde{d}}^{\infty}\!\!\!\partial_{\tilde{d}}\mathrm{Pf}\left(\mathcal{Q}^{\tilde{d}}_n\right) \prod_{i=1}^{n}d X_i.
        \end{aligned}
    \end{equation}
    This implies that
    \begin{equation}
    \begin{aligned}
        \lim_{N\rightarrow \infty}\!\sum_{n=1}^{\infty}\!
        \frac{(-1)^n}{n!}\!\left(\!\frac{\mathcal{G}_n(\tilde{d}-\epsilon) - \mathcal{G}_n(\tilde{d})}{\epsilon}\!\right)\! =
        -\partial_{\tilde{d}}\left(\!\mathrm{Pf}\left(\!J - \widetilde{K} + \ket{\begin{array}{c}
             \widetilde{g}_1 \\
             -\widetilde{d}_2   \end{array}}\bra{\widetilde{V}_1^{\tilde{d}} \,\,\, \widetilde{V}_2^{\tilde{d}}} +\ket{\begin{array}{c}
             \widetilde{V}_1^{\tilde{d}} \\
             \widetilde{V}_2^{\tilde{d}}
            \end{array}}\bra{-\widetilde{g}_1 \,\,\, \widetilde{d}_2}\right)_{L^2((\tilde{d}, \infty))}\!\right).
    \end{aligned}
    \end{equation}
    The limit that involves the second order derivative follows similarly. Terms like $\braket{\widehat{\mathsf{J}}}{d_2}$ and $\brabarket{f^{1/s}}{\widehat{A}_{12}}{g_1}$ in $\widehat{\mathcal{A}}_{d},\widehat{\mathcal{B}}_{d}, \widehat{\mathcal{D}}_{d},\widehat{\mathcal{E}}_{d}$ involve single or double summations, and their derivatives can be justified using similar arguments, along with their upper bounds established in previous lemmas. Terms like $\widehat{\mu}_d,\widehat{\nu}_d, \widehat{\mathcal{C}}_d$ are slightly different, so we provide a representative example.
    \begin{equation}
        \begin{aligned}
            \lim_{N\rightarrow \infty}\left(\widehat{\mu}_{d} - \widehat{\mu}_{d-1}\right) = &\frac{\left({s^{-d-1}} - s^{-d}\right)}{\epsilon}\frac{\epsilon H(s)}{(1-sr)}\!\!\oint \limits_{\Gamma_{\sqrt{q}}} \!\!\frac{dw}{2\pi\I}\frac{w^{d+2}}{H(w)}\frac{(1-ws)(1- wr)}{(w-s)^2(1-w^2)} \\
            &+\frac{\epsilon H(s)}{(1-sr)s^{d}}\!\!\oint \limits_{\Gamma_{\sqrt{q}}} \!\!\frac{dw}{2\pi\I}\frac{w^{d+1}}{H(w)}\frac{(w-1)}{\epsilon}\frac{(1-ws)(1- wr)}{(w-s)^2(1-w^2)}.
        \end{aligned}
    \end{equation}
    We discuss each term separately. For the first term,
    \begin{equation}\label{1stTerm}
        \begin{aligned}
            \frac{\left({s^{-d-1}} - s^{-d}\right)H(s)}{\epsilon} \rightarrow -\tilde{s}e^{\frac{\tilde{s}^3}{3} - \tilde{d}\tilde{s}}, \quad \frac{\epsilon}{(1-sr)}\!\!\oint \limits_{\Gamma_{\sqrt{q}}} \!\!\frac{dw}{2\pi\I}\frac{w^{d+2}}{H(w)}\frac{(1-ws)(1- wr)}{(w-s)^2(1-w^2)} \rightarrow \widetilde{\mu}_{\tilde{d}}\\
        \end{aligned}
    \end{equation}
    by the steepest descent method as in Lemma \ref{steepestDescent}.
    For the second term, we have $\frac{\epsilon H(s)}{(1-sr)s^{d+1}} \rightarrow -\frac{e^{\frac{\tilde{s}^3}{3} - \tilde{d}\tilde{s}}}{(\tilde{s}+\tilde{r})},$ and after a change of variables $\omega, \tilde{s}, \tilde{r}, \tilde{d}$, we see that  \begin{equation}\label{2ndTerm}
        \begin{aligned}
            &\oint \limits_{\Gamma_{\sqrt{q}}} \!\!\frac{dw}{2\pi\I}\frac{w^{d+1}}{H(w)}\frac{(w-1)}{\epsilon}\frac{(1-ws)(1- wr)}{(w-s)^2(1-w^2)}\\
            &= \int_{\epsilon^{-1}\widetilde{\mathcal{V}}_{\text{loc},\delta}}  e^{-\frac{\omega^3}{3} + \tilde{d}\omega + \mathcal{O}(\epsilon)}\omega\left(\frac{(\omega+\tilde{s})(\omega+\tilde{r})}{(\omega-\tilde{s})^2(2\omega)} +  \mathcal{O}(\epsilon\omega)\right)\frac{d\omega}{2\pi\I} + \mathcal{O}(\epsilon) \\
            &\rightarrow \int\limits_{{}_{}\wcu\, {}_{\tilde{s},0}} \!\frac{d\omega}{2\pi\I} \, e^{ - \frac{\omega^3}{3} + \tilde{d}\omega} \frac{(\omega+\tilde{s})(\omega+\tilde{r})}{2(\omega - \tilde{s})^2},
        \end{aligned}
    \end{equation}
    where the convergence is justified by the exponential decay of $e^{-\omega^3/3}$ along the contour $\widetilde{\mathcal{V}}_{\text{loc},\delta}$ which was explained in Lemma \ref{steepestDescent}.
    Combining \eqref{1stTerm} and \eqref{2ndTerm} together, we get $\lim_{N\rightarrow \infty}\left(\widehat{\mu}_{d} - \widehat{\mu}_{d-1}\right)= \partial_{\tilde{d}}\widetilde{\mu}_{\tilde{d}}$. Similar arguments can be applied to justify derivatives of $\widehat{\mathcal{C}}_{d}$ and $\widehat{\nu}_d.$
    Putting all pieces together, we get the desired limit, $\lim_{N \rightarrow 0} \frac{h_N(\tilde{d}) - h_N(\tilde{d}- \epsilon)}{\epsilon} = \partial_{\tilde{d}} \mathcal{L}(\tilde{d}).$
\end{proof}

\section{The Maximal current phase and Low density phase}\label{Max&LowPhase}
In this section, we provide derivation of the LPP distribution along the diagonal in the Maximal current phase and Low density phase.

First, we recall the stationary geometric LPP model under the Maximal current regime:
\begin{equation} \label{MaximalModel}
  \omega_{i, j} = \begin{cases}
      \mathrm{Geo}\left( r \right), & \textrm{if } i=j=2,\\
      \mathrm{Geo}\left( r\sqrt{q} \right), & \textrm{if }i=j\geq 3,\\
      0, & \textrm{if }j=1, i= 2,\\
    \mathrm{Geo}\left( \sqrt{q} \right), & \textrm{if }j=1, i\geq 3, \\
    \mathrm{Geo}\left( \sqrt{q} \right), &\textrm{if } j=2, i\geq 3, \\
    0, & \textrm{if}\ i=j=1, \\
    \mathrm{Geo}(q), &\textrm{otherwise}
  \end{cases}
\end{equation}
where $\sqrt{q} \in (0,1)$ and $r\in(0,1).$ Let ${G}_{r,1}^{MC}(N,N)$ denote the last passage percolation from $(0,0)$ to $(N,N)$ under the above model $\eqref{MaximalModel}$. Recall the definition of $\mathcal{I}_{r,s}(\cdot)$ in Definition \ref{stationaryMeasure}.

We use the same prelimiting model as before. Let $r$ be the diagonal parameter, $t$ be the first row parameter, $s$ be the second row parameter, and $\sqrt{q}$ be the bulk parameter.
\begin{equation} \label{MaximalApprox}
  \omega_{i, j} = \begin{cases}
      \mathrm{Geo}\left( rs\right), & \textrm{if } i=j=2,\\
      \mathrm{Geo}\left( r\sqrt{q} \right), & \textrm{if }i=j\geq 3,\\
      \mathrm{Geo}(st), & \textrm{if }j=1, i= 2,\\
    \mathrm{Geo}\left( \sqrt{q}t \right), & \textrm{if }j=1, i\geq 3, \\
    \mathrm{Geo}\left( \sqrt{q}s \right), &\textrm{if } j=2, i\geq 3, \\
    \mathrm{Geo}(rt), & \textrm{if}\ i=j=1, \\
    \mathrm{Geo}(q), &\textrm{otherwise}
  \end{cases}
\end{equation}
where $\sqrt{q} \in (0,1),$ $\sqrt{q}< s \leq 1$, $0<r< s.$ 
We use a new notation $\mathbf{L}_{r,t,s}(N,N)$ to denote the LPP from $(0,0)$ to $(N,N)$ under the above model $\eqref{MaximalApprox}.$

The following proposition states the symmetry between  the diagonal parameter $r$ and the second row parameter $s.$
This symmetry property originates from \cite[(7.56)]{baik2001algebraic} and also appears in other contexts, for example, \cite[Proposition 8.1]{MacdonaldProcess}. We omit the proof.

\begin{prop} \label{symmetryProperty}
    Given any $\sqrt{q} \in (0,1),$ $s\in (\sqrt{q},1],$ $t\in (0,1/s)$, $r <\min(1/s,1/t),$ we have for any $N\geq 2,$
    \begin{equation}
        \mathbf{L}_{r,t,s}(N,N) \stackrel{(d)}{=} \mathbf{L}_{s,t,r}(N,N).
    \end{equation}
\end{prop}

By Proposition $\ref{symmetryProperty}$, the diagonal LPP distribution of the following two models are equal.

\begin{figure}[htbp]
\centering

\begin{subfigure}[t]{0.48\textwidth}
\centering
\begin{tikzpicture}
\draw[very thin, gray] (-0.5,-0.5) grid (3.5,3.5);
\foreach \x in {1,...,3} {
  \foreach \y in {0,1,...,\x} {
    \fill (\x,\y) circle (2pt);
  }
}
\fill (0,0) circle (2pt);
\draw[thick] (0,0) -- (3.5,0);
\draw[thick] (0,0) -- (3.3,3.3);
\draw[thick] (1,1) -- (3.5,1);
\fill (0,0) circle (1pt);
\node at (0,-0.2) {$rt$};
\fill (1,0) circle (1pt);
\node at (1,-0.2) {$st$};
\node at (-0.5,0) {\footnotesize$(1,1)$};
\node at (0.5,1) {\footnotesize$(2,2)$};
\node at (1.2,0.82) {\small$rs$};
\fill (1,1) circle (1pt);
\node at (3,-0.28) {\small$\sqrt{q}\,t$};
\node at (3,0.72) {\small$\sqrt{q}\,s$};
\node at (1.55,2.1) {\small$r\sqrt{q}$};
\node at (3.1,1.75) {\small$q$};
\end{tikzpicture}
\caption{Geometric LPP model $\mathbf{L}_{r,t,s}$ as described in \eqref{MaximalModel}.}
\label{fig:StationaryModelApproxA}
\end{subfigure}
\hfill
\begin{subfigure}[t]{0.48\textwidth}
\centering
\begin{tikzpicture}
\draw[very thin, gray] (-0.5,-0.5) grid (3.5,3.5);
\foreach \x in {1,...,3} {
  \foreach \y in {0,1,...,\x} {
    \fill (\x,\y) circle (2pt);
  }
}
\fill (0,0) circle (2pt);
\draw[thick] (0,0) -- (3.5,0);
\draw[thick] (0,0) -- (3.3,3.3);
\draw[thick] (1,1) -- (3.5,1);
\fill (0,0) circle (1pt);
\node at (0,-0.2) {$st$};
\fill (1,0) circle (1pt);
\node at (1,-0.2) {$rt$};
\node at (-0.5,0) {\footnotesize$(1,1)$};
\node at (0.5,1) {\footnotesize$(2,2)$};
\node at (1.2,0.82) {\small$rs$};
\fill (1,1) circle (1pt);
\node at (3,-0.28) {\small$\sqrt{q}\,t$};
\node at (3,0.72) {\small$\sqrt{q}\,r$};
\node at (1.55,2.1) {\small$s\sqrt{q}$};
\node at (3.1,1.75) {\small$q$};
\end{tikzpicture}
\caption{Geometric LPP model $\mathbf{L}_{s,t,r}$ which is equal to $\mathbf{L}_{r,t,s}$ in distribution.}
\label{fig:StationaryModelApproxB}
\end{subfigure}

\caption{Comparison of two LPP configurations.}
\label{fig:StationaryModelApproxPair}
\end{figure}

Recall the product stationary half space LPP model in \eqref{one-paramStatModel}:
\begin{equation} \label{one-paramModel}
  \omega_{i, j} = \begin{cases}
      0, & i=j=1,\\
      \mathrm{Geo}\left( x\sqrt{q} \right), & i=j>1,\\
    \mathrm{Geo}\left( \sqrt{q}/x \right), & j=1, i>2, \\
    \mathrm{Geo}(q), &\textrm{otherwise}.
  \end{cases}
\end{equation}
Recall the notation ${G}_{x}^{\text{Low}}(N,N)$ for LPP from $(1,1)$ to $(N,N)$.
Setting $r = \sqrt{q},$ it is easy to see that as $t\rightarrow 1/s$
\begin{equation}
    \mathbf{L}_{s,t,\sqrt{q}}(N,N) - \mathbf{L}_{s,t,\sqrt{q}}(1,1) \xRightarrow{(d)} G_{s}^{\text{Low}}(N,N).
\end{equation}

Thus, $\mathbf{L}_{r,t,s}$ degenerates into stationary models in both Maximal current (when $s=t=1$) and Low density phases. Recall the Pfaffian structure of
$\mathbf{L}_{r,t,s}(N,N)$ in Proposition \ref{PfKernel}. We now perform a different kernel decomposition following the idea of \cite[Proposition 3.4]{Betea_2020}.

\subsection{Reformulation of the correlation kernel}
 All symbols and functions introduced in this section are defined exclusively for Maximal current phase and Low density phase, and are independent of those used in the High density phase.
We define 
\begin{equation}
    \begin{aligned}
        &H(x) = \left(\frac{(1-\sqrt{q}/x)}{(1-\sqrt{q}x)}\right)^{N-2},\quad R(x) = \frac{1-s/x}{1-sx},\quad f^{s}(k) = \frac{s^{k+1}}{H(s)},\\
        &E(k,\ell) = \begin{cases}
            -s^{k-\ell-1}, & \textrm{if }k> \ell,\\
        0, & \textrm{if }k = \ell,\\
        s^{\ell-k-1}, & \textrm{if }k < \ell\\
        \end{cases} =-\sgn{(k-\ell)}s^{|k-\ell| -1}.
    \end{aligned}
\end{equation}
In the following proposition, $K^{\text{geo}}$ is the kernel $K$ defined in Proposition \ref{PfKernel} with diagonal parameter $s$, first row parameter $t$ and second row parameter $r.$
Definitions of $g_1$, $g_2$ can be found in \eqref{g1g2def}.

\begin{prop}
Fix any $\sqrt{q}\in (0,1),$ $s \in (\sqrt{q},1)$, $r\in (0,s)$. Let $\mathbf{L}_{s,t,r}$ be the LPP model defined in Figure \ref{fig:StationaryModelApproxB}. For any $N\in \Z_{\geq 2},$ the distribution of $\mathbf{L}_{s,t,r}(N,N)$ is given by a Fredholm Pfaffian
\begin{equation}
    \Pb(\mathbf{L}_{s,t,r}(N,N) \leq d) = \mathrm{Pf}(J- K^{\text{geo}})_{\ell^2(\{d+1,d+2,\dots\})},
\end{equation}
and we can decompose $K^{\text{geo}}$ as follows:
$$K^{\text{geo}} = \overline{K}^{\text{geo}} + (1-st)\frac{(1-tr)}{(t-r)}\mathcal{R} ,\quad \mathcal{R} = \begin{pmatrix}
     \ketbra{g_1}{f^t} - \ketbra{f^t}{g_1} & \ketbra{f^t}{g_2}\\
    -\ketbra{g_2}{f^t} & 0
\end{pmatrix},$$
where
\begin{equation}\label{oneParamKernel}
\begin{aligned}
    \overline{K}_{11}^{\text{geo}}(k,\ell) &= \frac{-1}{(2\pi\I)^2} \oint \limits_{\Gamma_{\sqrt{q},r}} dw \oint \limits_{\Gamma_{1/\sqrt{q} ,1/r}} dz \frac{w^{\ell-1}}{z^k}\frac{H(z)R(z)T(z)}{H(w)R(w)T(w)}\frac{(zw-1)(z-s)(1-sw)}{(z^2-1)(1-w^2)(z-w)},\\
    \overline{K}_{12}^{\text{geo}}(k,\ell) &= \frac{-1}{(2\pi\I)^2} \oint \limits_{\Gamma_{\sqrt{q},r,t,s}} dw \oint \limits_{\Gamma_{1/\sqrt{q},1/r}} dz \frac{w^{\ell-1}}{z^k} \frac{H(z)R(z)T(z)}{H(w)R(w)T(w)}\frac{(zw-1)(z-s)}{(z^2-1)(w-s)(z-w)} = -\overline{K}_{21}^{\text{geo}}(\ell,k),\\
    \overline{K}_{22}^{\text{geo}}(k,\ell) &= \frac{-1}{(2\pi\I)^2} \oint \limits_{\Gamma_{\sqrt{q}}} dw\!\!\!\!\! \oint \limits_{\Gamma_{1/\sqrt{q}, 1/r,1/s,1/t}}\!\!\!\!\!\! dz \frac{w^{\ell-1}}{z^k}\frac{h_{22}^{\text{geo}}(z,w)}{z-w} + \frac{-1}{(2\pi\I)^2} \oint \limits_{\Gamma_{r}} dw \!\!\!\!\!\oint \limits_{\Gamma_{1/\sqrt{q}, 1/s,1/t}}\!\!\!\!\! dz \frac{w^{\ell-1}}{z^k}\frac{h_{22}^{\text{geo}}(z,w)}{z-w} \\
    &+\frac{-1}{(2\pi\I)^2} \oint \limits_{\Gamma_{s}} dw\!\!\!\!\! \oint \limits_{\Gamma_{1/\sqrt{q},1/r,1/t}}\!\!\!\!\! dz \frac{w^{\ell-1}}{z^k}\frac{h_{22}^{\text{geo}}(z,w)}{z-w} + \frac{-1}{(2\pi\I)^2} \oint \limits_{\Gamma_{t}} dw \!\!\!\!\!\oint \limits_{\Gamma_{1/\sqrt{q},1/r,1/s}}\!\!\!\!\! dz \frac{w^{\ell-1}}{z^k}\frac{h_{22}^{\text{geo}}(z,w)}{z-w}+ E(k,\ell)\\
\end{aligned}
\end{equation}
with
\begin{equation}
    \begin{aligned}
        h_{22}^{\text{geo}}(z,w) &:= \frac{H(z)R(z)T(z)}{H(w)R(w)T(w)}\frac{(zw-1)}{(w-s)(1-sz)}.\\
    \end{aligned}
\end{equation}
\end{prop}
\begin{proof}
    This reformulation of the kernel is very similar to the two-parameter case, as presented in Lemmas~\ref{lem:rewrite11}, \ref{lem:rewrite12}, and \ref{lem:rewrite22}. We omit the detailed derivation here. The only difference is that we do not isolate all the $t$-poles.
\end{proof}

For the next lemma, we recall definitions of $F_{r,1}^{MC}$ in \eqref{MaxCDF} and $F_{r}^{LD}$ in \eqref{LowCDF}.
\begin{lem}\label{DeriveMax&Low}
    Fix any $\sqrt{q} \in (0,1).$ Let $N \in \Z_{\geq 2}$. Assume that the limit $${F}_{s,r}(d,N) := \lim_{t\rightarrow 1/s} \frac{1}{1-st}\Pb(\mathbf{L}_{s,t,r}(N,N) \leq d)$$ exists for any $s\in (\sqrt{q},1),$ $r \in (0,s),$ $d\in Z_{\geq 0},$ and let ${F}_{s,r}(d,N) = 0$ for $d < 0.$ If the limit ${F}_{s,r}(d,N)$ is analytic in $s \in (\sqrt{q}, 1+\epsilon)$ for some $\epsilon>0$ when $r \in (0,1)$ is fixed, then for $N \in \Z_{\geq 1},$
    \begin{equation}\label{formulaMax}
        \begin{aligned}
            &F_{r,1}^{MC}(d,N)=\Pb({G}_{r,1}^{MC}(N+1,N+1) \leq d) =  \\
            &\frac{1}{1-r} {F}_{1,r}(d,N+1) - \frac{1+r}{1-r}{F}_{1,r}(d-1,N+1) + \frac{r}{1-r}{F}_{1,r}(d-2,N+1).\\
            \end{aligned}\end{equation}
            If the limit ${F}_{s,r}(d,N)$ is analytic in $r\in (0,s)$ when $s \in (\sqrt{q},1)$ is fixed, then for $N\in \Z_{\geq 2},$
                \begin{equation}
        \begin{aligned}\label{formulaLow}
        &F_{\overline{r}}^{LD}(d,N)=\Pb({G}_{\overline{r}}^{\text{Low}}(N,N) \leq d) =   
            {F}_{\overline{r},\sqrt{q}}(d,N) - {F}_{\overline{r},\sqrt{q}}(d-1,N). 
        \end{aligned}
    \end{equation}
\end{lem}

\begin{proof}
    The Maximal current phase formula is obtained by applying the shift argument \eqref{eq:shift} to remove geometric random variables at $(1,1)$ and $(2,1)$ from $\mathbf{L}_{r,t,s}$, taking the limit of $t\rightarrow 1/s$ and then setting the first and second row parameters as $1$, i.e., $s = 1/s = 1$. By Proposition \ref{symmetryProperty}, we have that $\mathbf{L}_{r,t,s} \stackrel{(d)}{=} \mathbf{L}_{s,t,r}$. Hence, \eqref{formulaMax} gives the desired distribution formula for Maximal current phase. 
    The Low density phase formula \eqref{formulaLow} is obtained by applying the shift argument to remove the geometric random variable at $(1,1)$, taking $t \rightarrow 1/s$, and then setting the diagonal parameter to be $\overline{r}$, i.e., $s = \overline{r}$ and the second row parameter to be $\sqrt{q},$ i.e., $r = \sqrt{q}$. We restrict to $N\in \Z_{\geq 2}$ for $F_{\overline{r}}^{LD}$ simply because $G_{\overline{r}}^{Low}(1,1) = 0.$
\end{proof}

Similar as before, we aim to derive the limit $\lim_{t\rightarrow 1/s} \frac{1}{1-st}\mathrm{Pf}(J - K^{\text{geo}}).$
 Recall that $J(x,y) = \delta_{x,y}\begin{pmatrix}
            0 & 1\\ -1 & 0
        \end{pmatrix}.$
Let \begin{equation}
    G = J^{-1}K^{\text{geo}}, \quad \overline{G} = J^{-1}\overline{K}^{\text{geo}}, \quad T = J^{-1}\mathcal{R}.
\end{equation}
Then we know that $T$ can be rewritten as 
\begin{equation}
    \begin{aligned}
        T = \begin{pmatrix}
            \ketbra{g_2}{f^s} & 0\\
            \ketbra{g_1}{f^s}- \ketbra{f^s}{g_1} & \ketbra{f^s}{g_2} 
        \end{pmatrix}= \ketbra{X_1}{Y_1} + \ketbra{X_2}{Y_2} 
    \end{aligned}
\end{equation}
with 
\begin{equation}
\begin{aligned} \label{eq:S_fact}
  X_1 = \ket{ \begin{array}{c}  g_2 \\ g_1 \end{array} }, \quad & X_2 &= \ket{ \begin{array}{c} 0 \\ f^s \end{array} },\quad
  Y_1 = \bra{{f^s} \quad 0}, \quad & Y_2 &= \bra{-g_1 \quad g_2}.
\end{aligned}
\end{equation}
By using exact same argument as in \cite[Section $3.2.3$]{Betea_2020}, we get 
\begin{equation}\label{Newidentity}
    \frac{1}{1-st}\text{Pf}(J - K^{\text{geo}}) = \text{Pf}(J - \overline{K}^{\text{geo}}) \left(\frac{1}{1-st} - \frac{(1-tr)}{(t-r)}\braket{Y_2}{X_2} - \frac{(1-tr)}{(t-r)}\braket{Y_2}{(\Id - \overline{G})^{-1}\overline{G}X_2}\right).
\end{equation}

\begin{thm}\label{Max&LowResult_Finite}
    For any $\sqrt{q}\in (0,1)$, $s \in (\sqrt{q},1)$, $r\in (0,s)$ and $N \in \Z_{\geq 2}$ and $d\in \Z_{\geq 0}$, we have
    \begin{equation}
        \lim_{t\rightarrow 1/s} \frac{1}{1-st}\mathrm{Pf}(J-K^{\text{geo}})_{\ell^{2}(\{d+1,d+2,\dots\})} = \Theta(d),
    \end{equation}
    where
    \begin{equation}
        \Theta(d) := \mathrm{Pf}(J - \overline{\mathsf{K}}^{\text{geo}})e^{s,r}(d) - \mathrm{Pf}(J - \overline{\mathsf{K}}^{\text{geo}}) + \mathrm{Pf}\left(J - \overline{\mathsf{K}}^{\text{geo}} - \ketbra{\begin{array}{c}
            \overline{c}\phi_2\\\overline{c}\phi_1
        \end{array}}{-g_1 \quad g_2} - \ketbra{\begin{array}{c}
            -g_1\\
            g_2
        \end{array}}{-\overline{c}\phi_2 \,\,\, -\overline{c}\phi_1}\right),
    \end{equation}
    and $\Theta(d) = 0$ for $d<0$. The Pfaffian is taken over ${\ell^{2}(\{d+1,\dots\}}).$
\end{thm}

 The proof of this theorem will be presented in detail in later sections. We now define each function used in the theorem. All functions introduced below are used exclusively in the Maximal current phase and Low density phase. It should not be confused with those appearing in the High density phase. The kernel $\overline{\mathsf{K}}^{\text{geo}}$ is defined in Theorem $\ref{MaximalKernel}.$
\begin{equation}
\begin{aligned}
     &H(x) = \left(\frac{1-\sqrt{q}/x}{1-\sqrt{q}x}\right)^{N-2}, \quad\quad  R(x) = \frac{(1-r/x)}{(1-rx)}, \quad\quad f^x(k) = \frac{x^{k+1}}{H(x)},\\
     &e^{s,r}(d) = \frac{(s-r)}{(1-sr)}\frac{H(s)}{s^{d+1}}\oint \limits_{\Gamma_{\sqrt{q},s,r}} \frac{dw}{2\pi\I} \frac{w^{d+2}(1-wr)}{H(w)(w-s)^2(w-r)},
\end{aligned}
\end{equation}

\begin{equation}
    \begin{aligned}
         &\mathsf{g}_4(k) = -\!\!\!\!\oint \limits_{\Gamma_{1/\sqrt{q}, 1/r,1/s,s} }\!\!\!\!\!\!\frac{dz}{2\pi\I}\frac{sH(z)(z-r)(z^2-1)}{z^{k+2}(1-zr)(z-s)^2(1-sz)}, \quad \mathsf{g}_3(k) = \oint \limits_{\Gamma_{1/\sqrt{q},1/r}}
        \frac{dz}{2\pi\I}\frac{s(z-r)H(z)}{z^{k+2}(1-zr)(z-s)},\\
        &\mathsf{h}(d,k) = \begin{cases}
            H(s)\frac{s^{k-2d-2} - s^{-k}}{(s^2-1)} + s^{-k-2}(k-d)H(s) &\text{if }  s\neq 1,\\
            2k-2d-1 &\text{if } s = 1,
        \end{cases} \quad\quad \quad \overline{c} = \frac{(s-r)}{(1-sr)}.
    \end{aligned}
\end{equation}
\begin{equation}\label{g1g2def}
    \begin{aligned}
        &g_1(k) = \oint \limits_{\Gamma_{1/\sqrt{q},1/r}} \frac{dz}{2\pi\I} \frac{H(z)}{z^{k+2}}\frac{(z-r)(z-s)}{(1-rz)(z^2-1)},\quad\quad
        g_2(k) = \oint \limits_{\Gamma_{\sqrt{q},s,r}} \frac{dw}{2\pi\I} \frac{w^{\ell+1}}{H(w)}\frac{(1-wr)}{(w-r)(w-s)},
    \end{aligned}
\end{equation}

\begin{equation}
    \begin{aligned}
        &\overline{\phi_1}(k) = \frac{-1}{(2\pi\I)^2} \frac{H(s)}{s^{d+1}}\oint \limits_{\Gamma_{\sqrt{q},r}}\!\! dw\!\!\!\!\!\! \oint \limits_{\Gamma_{1/\sqrt{q},1/r,s}}\!\!\!\!\!\! dz \frac{w^{d+1}}{z^{k+1}} \frac{H(z)R(z)}{H(w)R(w)}\frac{(zw-1)}{(s-w)(s-z)(ws-1)(z-w)},\\
        &\overline{\phi_2}(k) = \frac{-1}{(2\pi\I)^2} \frac{H(s)}{s^{d+1}}\oint\limits_{\Gamma_{\sqrt{q},r}}\!\!\! dw \!\!\!\oint\limits_{\Gamma_{1/\sqrt{q},1/r}}\!\!\!dz \frac{w^{d+1}}{z^{k+1}} \frac{H(z)R(z)}{H(w)R(w)}\frac{(zs-1)(zw-1)}{(s-w)(ws-1)(z^2-1)(z-w)}.\\
    \end{aligned}
\end{equation}

\begin{equation}
        \begin{aligned}
            \phi_1(k) &= \overline{\phi_1}(k) - \frac{(1-sr)}{(s-r)}\mathsf{g}_4(k) - \mathsf{h}(d,k), \quad\quad
            \phi_2(k) = \overline{\phi_2}(k) + \frac{(1-sr)}{(s-r)}\mathsf{g}_3(k).\\
        \end{aligned}
\end{equation}

\begin{remark}
    The method used to obtain distribution formulas in the Maximal current phase and Low density phase also applies to a portion of the High density phase, namely $s \in (\sqrt{q},1)$ and $0<r<s$. For the remaining High density region, i.e., $s<r<1/s$, one must instead use the approach that yields $F_{r,s}^{HD}.$
\end{remark}

\begin{thm}\label{MaximalKernel}
    The kernel $\overline{K}^{\text{geo}}$ is analytic for $t \in (0,1/r).$ The limiting kernel $\overline{\mathsf{K}}^{\text{geo}} := \lim_{t \rightarrow 1/s} \overline{K}^{\text{geo}}$ has the following entries:
        \begin{equation}
        \begin{aligned}
            \overline{\mathsf{K}}_{11}^{\text{geo}} &:= \frac{-1}{(2\pi\I)^2} \oint \limits_{\Gamma_{\sqrt{q},r}} dw \oint \limits_{\Gamma_{1/\sqrt{q},1/r}} dz \frac{w^{\ell}}{z^{k+1}} \frac{H(z)R(z)}{H(w)R(w)} \frac{(zs-1)(s-w)(zw-1)}{(z^2-1)(1-w^2)(z-w)},\\
            \overline{\mathsf{K}}_{12}^{\text{geo}} &:= \frac{-1}{(2\pi\I)^2} \oint \limits_{\Gamma_{\sqrt{q},r,1/s}} dw \oint \limits_{\Gamma_{1/\sqrt{q},1/r}} dz \frac{w^{\ell}}{z^{k+1}} \frac{H(z)R(z)}{H(w)R(w)} \frac{(zs-1)(zw-1)}{(z^2-1)(ws-1)(z-w)},
            \end{aligned}
    \end{equation}   
    \begin{equation}
    \begin{aligned}
    \overline{\mathsf{K}}_{22}^{\text{geo}} &:= \frac{-1}{(2\pi\I)^2} \oint \limits_{\Gamma_{\sqrt{q}}} dw \oint \limits_{\Gamma_{1/\sqrt{q},1/r,s}} dz \,\overline{h}_{22}^{\text{geo}}(z,w)
    +\frac{-1}{(2\pi\I)^2} \oint \limits_{\Gamma_{r}} dw \oint \limits_{\Gamma_{1/\sqrt{q},s}} dz \, \overline{h}_{22}^{\text{geo}}(z,w)\\
    &
            + \frac{-1}{(2\pi\I)^2} \oint \limits_{\Gamma_{1/s}} dw \oint \limits_{\Gamma_{1/\sqrt{q},1/r}} dz\, \overline{h}_{22}^{\text{geo}}(z,w) -\sgn(k-\ell) s^{-|k-\ell|-1}
        \end{aligned}
    \end{equation}
    where $$\overline{h}_{22}^{\text{geo}} = \frac{w^{\ell}H(z)R(z)(zw-1)}{z^{k+1}H(w)R(w)(s-z)(ws-1)(z-w)}.$$
\end{thm}

\begin{proof}
    $\overline{K}^{\text{geo}}_{11}$ (respectively $\overline{K}^{\text{geo}}_{12}$) is analytic for $t \in (r,1/r)$ , as we can choose the contour for $z$ to be arbitrarily close to $1/\sqrt{q},1/r$ as needed and the contour for $w$ to be arbitrarily close to $r, \sqrt{q}$ (respectively $\sqrt{q},r,s,t$). The pole for $w = r$ in $\overline{K}_{12}^{\text{geo}}$ disappears after taking the limit.
    For $\overline{K}^{\text{geo}}_{22}$, poles at $(z,w)= (1/s,\sqrt{q}), (1/s,r)$ and poles at $(z,w) = (1/\sqrt{q},s), (1/r,s)$ vanish as $t\rightarrow 1/s.$ Poles at $(z,w) = (1/t,s)$ and $(z,w) = (1/s,t)$ are evaluated as (in the following sum, the first term is for $(1/t,s)$ and second term is for $(1/s,t)$)
    \begin{equation}\label{twoPoles}
    \begin{aligned}
        &\frac{1}{H(s)H(t)R(s)R(t)}\left({-s^{\ell}t^{k}}\frac{(1-t^2)}{(s-t)} + {s^{k}t^{\ell}}\frac{(1-t^2)}{(s-t)} \right)\\
        &\rightarrow -s^{\ell-k-1} + s^{k-\ell-1} = \sgn (k-\ell) \left( s^{|k-\ell| -1} - s^{-|k-\ell|-1} \right).
    \end{aligned}
    \end{equation}
    Hence, adding $E(k,\ell)$ and $\eqref{twoPoles}$ yields
    $-\sgn(k-\ell) s^{-|k-\ell|-1}.$
\end{proof}

\begin{prop}
    For fixed $\sqrt{q}\in(0,1),$ $s\in (\sqrt{q},1),$ and $r \in (0,s),$ $\mathrm{Pf}(J-\overline{K}^{\text{geo}})$ is analytic for $t \in (0,1/r)$ and its limit is given by
    \begin{equation}
        \lim_{t \rightarrow 1/s} \mathrm{Pf}(J - \overline{K}^{\text{geo}})_{\ell^2(\{d+1,d+2,\dots\})} = \mathrm{Pf}(J - \overline{\mathsf{K}}^{\text{geo}})_{\ell^2(\{d+1,d+2,\dots\})}.
    \end{equation}
\end{prop}

\begin{proof}
Fix $0<\epsilon$ such that $1/r>1/s+\epsilon$, $r\in (0, s-\epsilon)$ and $t \in (0,1/s + \epsilon)$. Choose the contours for $z,w$ to be small circles around their poles. Without loss of generality, we assume that $r > \sqrt{q} + \epsilon.$
Then we get the following upper bounds on each entry of the kernel:
    \begin{equation}\label{BoundsforK}
    \begin{aligned}
    &|\overline{K}^{\text{geo}}_{11}(k,\ell)| \leq C e^{-(k+\ell)\log(1/r)},\quad
    |\overline{K}^{\text{geo}}_{12}(k,\ell)| \leq Ce^{-k\log(1/r) + \ell\log(1/s+ \epsilon)},\\
    &|\overline{K}^{\text{geo}}_{21}(k,\ell)| \leq Ce^{k\log(1/s+ \epsilon) - \ell\log(1/r)},\quad
    |\overline{K}^{\text{geo}}_{22}(k,\ell)| \leq Ce^{(k+\ell)\log(1/s+ \epsilon)}.\\
    \end{aligned}
    \end{equation}
    Then applying Hadamard's bound and the dominated convergence theorem gives the desired limit.
\end{proof}

\begin{lem}\label{explosion-one}
    The term $\frac{1}{1-ts} - \braket{Y_2}{X_2}$ is analytic for $t \in (0, 1/r)$ with
    \begin{equation}
        e^{s,r}(d) = \lim_{t\rightarrow 1/s} \left(\frac{1}{1-st} - \frac{(1-tr)}{(t-r)}\braket{Y_2}{X_2}\right).
    \end{equation}
\end{lem}
\begin{proof}
    This is similar to \cite[Lemma 3.9]{Betea_2020}. We see that
    \begin{equation}
        \frac{(1-tr)}{(t-r)}\braket{Y_2}{X_2} = \frac{(1-tr)}{(t-r)}\braket{g_2}{f^t} = \frac{(1-tr)}{(t-r)}\oint_{\Gamma_{\sqrt{q},s,r}} \frac{dw}{2\pi\I} \frac{(wt)^{d+2}(1-wr)}{(1-tw)(w-r)(w-s)H(t)H(w)},
    \end{equation}
    which is not analytic as $t \rightarrow 1/s$ due to the pole at $w = s$. We instead fix a larger contour that encloses $\sqrt{q},r,s,1/t$ such that we can freely move $t$ to $1/s$ and then subtract the extra pole at $w = 1/t$ that was intentionally added.
    We found that $\frac{(1-tr)}{(t-r)}\text{Res}(w,1/t) = -1/(1-st)$ cancels the term $(1-st)^{-1}$ and the limit is
    \begin{equation}
    \begin{aligned}
        &\lim_{t\rightarrow 1/s}\frac{1}{1-st} - \frac{(1-tr)}{(t-r)}\braket{Y_2}{X_2}\\
        &= \lim_{t\rightarrow 1/s} -\frac{(1-tr)}{(t-r)}\oint_{\Gamma_{\sqrt{q},s,r,1/t}} \frac{dw}{2\pi\I} \frac{(wt)^{d+2}(1-wr)}{(1-tw)(w-r)(w-s)H(t)H(w)} = e^{s,r}(d).
    \end{aligned}
    \end{equation}
\end{proof}

Next, we need to define a few functions that will be used in the following Proposition.
\begin{equation}
    \begin{aligned}
        &g_3(k) := -\!\!\!\!\!\oint \limits_{\Gamma_{1/\sqrt{q},1/r}}\!\!\!\!\! \frac{dz}{2\pi\I} \frac{H(z)}{z^{k+2}}\frac{(z-r)}{(1-zr)(1-zt)},
        \quad g_4(k) := \!\!\!\!\!\!\!\oint \limits_{\Gamma_{1/\sqrt{q},1/r,/1/s,s,1/t}} \!\!\!\!\!\!\!\!\!\frac{dz}{2\pi\I} \frac{H(z)}{z^{k+2}}\frac{(z-r)(z^2-1)}{(1-zr)(1-zt)(z-s)(1-zs)},\\
        &g_5(k) := \!\!\!\oint \limits_{\Gamma_{1/\sqrt{q},1/r}}\!\!\! \frac{dz}{2\pi\I} \frac{H(z)}{z^{k+2}}\frac{(z-t)(z-r)(1-sz)}{(1-zt)(1-zr)(z^2-1)}, \quad g_6(k) :=-\!\!\!\!\oint \limits_{\Gamma_{1/\sqrt{q},1/r,1/t}}\!\!\!\! \frac{dz}{2\pi\I} \frac{H(z)}{z^{k+2}} \frac{(z-t)(z-r)}{(1-zt)(1-zr)(z-s)} .
    \end{aligned}
\end{equation}

\begin{remark}\label{WhySneq1}
   If we did not want to consider the $s = t =1$ case, then the next proposition is unnecessary. However, the maximal current phase occurs precisely at $s=t=1.$ This proposition provides a technique to remove all terms in the kernel that have a pre-factor $(s-t)^{-1}.$ An analogous issue arises in the two-parameter stationary setting: we have not found a way to eliminate the $(t-s)^{-1}$ terms. As a result, our formula for two-parameter stationary case ($F_{r,s}^{HD}$) cannot handle $s=1$.
\end{remark}

\begin{prop}
    For $\sqrt{q} \in (0,1),$ $s\in (\sqrt{q},1)$, $r \in (0,s)$, the kernel $\overline{K}^{\text{geo}}$ can be decomposed into
    \begin{equation}
    \begin{aligned}
        \overline{K}^{\text{geo}} = \widetilde{K}^{\text{geo}} + \begin{pmatrix}
            0 & 0\\
            0 & E
        \end{pmatrix} + \widetilde{O} + \widetilde{P}
    \end{aligned}
    \end{equation}
    where 
    \begin{equation}
        \widetilde{K}^{\text{geo}}_{11} = \overline{K}^{\text{geo}}_{11}, \quad \widetilde{K}^{\text{geo}}_{21} = \overline{K}^{\text{geo}}_{21},\\
    \end{equation}
    \begin{equation}
    \begin{aligned}
    \widetilde{K}_{12}^{\text{geo}}(k,\ell) &= \frac{-1}{(2\pi\I)^2} \oint \limits_{\Gamma_{\sqrt{q},r}} dw \oint \limits_{\Gamma_{1/\sqrt{q},1/r}} dz \frac{w^{\ell-1}}{z^k} \frac{H(z)R(z)T(z)}{H(w)R(w)T(w)}\frac{(zw-1)(z-s)}{(z^2-1)(w-s)(z-w)},\\
        \widetilde{K}_{22}^{\text{geo}}(k,\ell) &= \frac{-1}{(2\pi\I)^2} \oint \limits_{\Gamma_{\sqrt{q}}} dw \oint \limits_{\Gamma_{1/\sqrt{q}, 1/r,1/s,1/t}} dz \frac{w^{\ell-1}}{z^k}\frac{H(z)R(z)T(z)}{H(w)R(w)T(w)}\frac{(zw-1)}{(w-s)(1-sz)(z-w)}\\
        &+ \frac{-1}{(2\pi\I)^2} \oint \limits_{\Gamma_{r}} dw \oint \limits_{\Gamma_{1/\sqrt{q},1/s,1/t}} dz \frac{w^{\ell-1}}{z^k}\frac{H(z)R(z)T(z)}{H(w)R(w)T(w)}\frac{(zw-1)}{(w-s)(1-sz)(z-w)},
    \end{aligned}
    \end{equation}
    and
    \begin{equation}
    \begin{aligned}
        \widetilde{O} = \ketbra{\begin{array}{c}
             -g_1\\
             g_2
        \end{array}}{ 0 \quad \frac{(1-t^2)(1-tr)}{(s-t)(t-r)}f^t - \frac{(1-st)(1-sr)}{(s-t)(s-r)}f^s},
    \end{aligned}
\end{equation}
and
\begin{equation}
    \begin{aligned}
        \widetilde{P} = \ketbra{\begin{array}{c}
             \frac{(1-st)(1-sr)}{(s-r)}g_3\\
             -\frac{(1-st)(1-sr)}{(s-r)}g_4 - s^{-1}f^{1/s}
        \end{array}}{0 \quad f^s }.
    \end{aligned}
\end{equation}
\end{prop}

\begin{proof}We see that
    \begin{equation}
        (\widetilde{O} + \widetilde{P})_{12} =  \frac{(1-t^2)}{(t-s)}\frac{(1-tr)}{(t-r)}g_1(k)f^t(\ell) + \frac{(1-st)}{(s-t)}\frac{(1-sr)}{(s-r)} g_5(k)f^s(\ell).
    \end{equation}
    Direct computation gives
    \begin{equation}
        g_5(k) - g_1(k) = (s-t)g_3(k).
    \end{equation}
    Therefore,
    \begin{equation}
        (\widetilde{O} + \widetilde{P})_{12} = \frac{(1-t^2)}{(t-s)}\frac{(1-tr)}{(t-r)}\ketbra{g_1}{f^t} + \frac{(1-st)(1-sr)}{(s-t)(s-r)}\ketbra{g_1}{f^s} + \frac{(1-st)(1-sr)}{(s-r)}\ketbra{g_3}{f^s}.
    \end{equation}
    We know that 
    \begin{equation}
    \begin{aligned}
        (\widetilde{O} + \widetilde{P})_{22} &= \oint \limits_{\Gamma_{t}} dw \oint \limits_{\Gamma_{1/\sqrt{q},1/r,1/s}} dz \cdots + \oint \limits_{\Gamma_{s}} dw \oint \limits_{\Gamma_{1/\sqrt{q},1/r,1/t}} dz \cdots\\
        &= \frac{(1-t^2)(1-tr)}{(s-t)(t-r)}\ketbra{g_2}{f^t} - \frac{(1-st)(1-sr)}{(s-t)(s-r)}\ketbra{g_6}{f^s},
    \end{aligned}
    \end{equation}
    where $\cdots$ represents the integrand of $\overline{K}_{22}^{\text{geo}}$.
    We have the identity
    \begin{equation}
        \begin{aligned}
            g_6 - g_2 = (s-t) \left( g_4 + \frac{(s-r)}{(1-sr)(1-st)s}f^{1/s} \right).
        \end{aligned}
    \end{equation} 
    Taking a contour that includes $1/\sqrt{q},1/s,1/r,1/t$ enables us to compute the difference $g_6 - g_2$, which gives $ (r-s)g_4$ minus the pole at $z = s$ in $g_4.$
    Therefore, we get
    \begin{equation}
        \begin{aligned}
            (\widetilde{O} + \widetilde{P})_{22} &= \frac{(1-t^2)(1-tr)}{(t-r)(s-t)} \ketbra{g_2}{f^t} - \frac{(1-st)(1-sr)}{(s-t)(s-r)}\ketbra{g_2}{f^s} - \frac{(1-st)(1-sr)}{(s-r)}\ketbra{g_4}{f^s}\\
            &- \frac{1}{s}\ketbra{f^{1/s}}{f^s}.
        \end{aligned}
    \end{equation}
\end{proof}

Given the decomposition, we have that 
\begin{equation}
    \begin{aligned}
        \overline{G} = \widehat{G} + \begin{pmatrix}
            0 & \ketbra{-g_2}{\frac{(1-t^2)(1-tr)}{(s-t)(t-r)}f^t - \frac{(1-st)(1-sr)}{(s-t)(s-r)}f^s}\\
            0 & \ketbra{-g_1}{\frac{(1-t^2)(1-tr)}{(s-t)(t-r)}f^t - \frac{(1-st)(1-sr)}{(s-t)(s-r)}f^s}
        \end{pmatrix}
    \end{aligned},
\end{equation}
where 
\begin{equation}\label{Ghat}
    \widehat{G} = \begin{pmatrix}
        -\widetilde{K}^{\text{geo}}_{21} & -\widetilde{K}^{\text{geo}}_{22}- E\\
        \widetilde{K}^{\text{geo}}_{11} & \widetilde{K}^{\text{geo}}_{12}
    \end{pmatrix} + \begin{pmatrix}
        0 & \ketbra{\frac{(1-st)(1-sr)}{(s-r)}g_4 + s^{-1}f^{1/s}}{f^s}\\
        0 & \ketbra{\frac{(1-st)(1-sr)}{(s-r)}g_3}{f^s}
    \end{pmatrix}.
\end{equation}

By a similar argument as in \cite[Lemma 3.10]{Betea_2020}, we know that 
\begin{equation}
    \braket{Y_2}{(\Id - \overline{G})^{-1} \overline{G}X_2 } = \braket{Y_2}{(\Id - \overline{G})^{-1} \widehat{G}X_2 }.
\end{equation}

\begin{lem}
    Fix any $\sqrt{q} \in (0,1),$ $s\in (\sqrt{q},1),$ $r\in (0,s)$. For any $t \in (0,1/r),$ we have
    \begin{equation}\label{BoundsForg}
        |g_1(k)| \leq Ce^{k\log(\max(\sqrt{q}+\epsilon/2,r))}, \quad |g_2(k)| \leq Ce^{k\log(s)}
    \end{equation}
    for some constant $C$ independent of $x$.
\end{lem}
\begin{proof}
    For $g_1$, the upper bound comes from poles at $z = 1/\sqrt{q}$ and $z = 1/r$. For $g_2$, the upper bound comes from the pole $w = s.$ We note that $g_1$ and $g_2$ are independent of $t.$
\end{proof}

The following lemma corresponds to \cite[Lemma $3.14$]{Betea_2020}.
\begin{lem}
    We have $\widehat{G}X_2$ is analytic for $t \in (0,1/r)$ with its limit
    \begin{equation}
        \lim_{t \rightarrow 1/s} \widehat{G}X_2 = \ket{ \begin{array}{c}  -\phi_1 \\ \phi_2 \end{array} }.
    \end{equation}
    Moreover, we have the upper bounds:
    \begin{equation}\label{BoundsforG}
        \begin{aligned}
            |(\widehat{G}X_2)_{1}(k)| \leq Ce^{k\log(1/s)}, \quad |(\widehat{G}X_2)_{2}(k)| \leq  Ce^{k \log(\max(\sqrt{q} + \epsilon, r))},
        \end{aligned}
    \end{equation}
    for some constant $C$ independent of $k$.
\end{lem}

\begin{proof}
By \eqref{Ghat}, we know that
    \begin{equation}\label{one-paramGX_2}
        \begin{aligned}
            \widehat{G}X_2 = \ket{\begin{array}{c}  -\left(\widetilde{K}_{22}^{\text{geo}}+E\right)f^t + \frac{(1-st)(1-sr)}{(s-r)}g_4\braket{f^s}{f^t} + s^{-1}f^{1/s}\braket{f^s}{f^t}\\ 
            \widetilde{K}_{12}^{\text{geo}}f^t + \frac{(1-st)(1-sr)}{(s-r)}g_3\braket{f^s}{f^t} \end{array}}.
        \end{aligned}
    \end{equation}
First, we compute $\braket{f^s}{f^t}$ and get
\begin{equation}
    \braket{f^s}{f^t} = \sum_{\ell = d+1}^{\infty} \frac{(st)^{\ell+1}}{H(s)H(t)} = \frac{(st)^{d+2}}{(1-st)H(s)H(t)}.
\end{equation}
Since both $\widetilde{K}_{12}^{\text{geo}}$ and $\widetilde{K}_{22}^{\text{geo}}$ have poles at $w = \sqrt{q}$ and $w = r$ in the $\ell$ variable, the leading order is $(\max(\sqrt{q}+\epsilon/2,r))^{\ell}$. This allows us to take their inner product with $f^t$ and derive limits when $t\rightarrow 1/s$. So we get that
\begin{equation}
    \lim_{t \rightarrow 1/s} \widetilde{K}_{12}^{\text{geo}}f^t = \widetilde{\mathsf{K}}_{12}^{\text{geo}}f^{1/s}, \quad \lim_{t \rightarrow 1/s} \widetilde{K}_{22}^{\text{geo}}f^t = \widetilde{\mathsf{K}}_{22}^{\text{geo}}f^{1/s}.
\end{equation}
There are two terms in \eqref{one-paramGX_2} that contain the prefactor $(1-st)$ which will cancel the factor $1/(1-st)$ generated from $\braket{f^s}{f^t}$. 
Functions $g_4$ and $g_3$ are analytic and we get their limits as $t\rightarrow 1/s,$
\begin{equation}
    \begin{aligned}
        \frac{(1-st)(1-sr)}{(s-r)}g_4 \braket{f^s}{f^t} \rightarrow \frac{(1-sr)}{(s-r)}\mathsf{g}_4,\quad
        \frac{(1-st)(1-rs)}{(s-r)}g_3 \braket{f^r}{f^s} \rightarrow \frac{(1-sr)}{(s-r)}\mathsf{g}_3.
        \end{aligned} 
\end{equation}
Then we have the last term to analyze. Recall that $E(k,k) = 0.$
\begin{equation}\label{Efs}
    \begin{aligned}
        &-Ef^t + s^{-1}f^{1/s}\braket{f^s}{f^t} = \sum_{\ell= d+1}^{\infty}\sgn(k-\ell)s^{|k-\ell|-1} \frac{t^{\ell+1}}{H(t)} + s^{-1}\frac{H(s)}{s^{k+1}}\frac{(st)^{d+2}}{(1-st)H(s)H(t)}\\
        &=\frac{t^{d+2}s^{k-d-1} - t^{k+1}}{(s-t)H(t)} + \frac{s^{d-k}t^{d+2} - t^{k+2}}{(1-st)H(t)}  \rightarrow \mathsf{h}(d,k).
    \end{aligned}
\end{equation}
When $s = t = 1,$ we take the limit of $s \rightarrow 1$ of $\mathsf{h}$ and get $2k-2d-1.$

Next, we analyze the upper bounds of each term. Fix $0<\epsilon \ll 1$ such that $(1/\sqrt{q}+\epsilon/2)^{-1} \leq \sqrt{q}+\epsilon/2.$ Choose the contours for poles at $\sqrt{q}$ and $1/\sqrt{q}$ to be $|w-\sqrt{q}| = \epsilon/2$ and $|z-1/\sqrt{q}| = \epsilon/2.$ Then we have
\begin{equation}\label{boundforK}
    \begin{aligned}
    \big|\widetilde{K}_{12}^{\text{geo}}f^t(k) \big| \leq Ce^{k\log(\max(\sqrt{q}+ \epsilon/2, r))}, \quad  \big|\widetilde{K}_{22}^{\text{geo}}f^t(k) \big| \leq Ce^{k\log(t)},
    \end{aligned}
\end{equation}
because $\widetilde{K}_{12}^{\text{geo}}$ only has poles at $z=1/\sqrt{q},1/r$  and the dominating term for $\widetilde{K}_{22}^{\text{geo}}$ is $t^k$ since $t>s>\max(r,\sqrt{q}).$
We also know that
\begin{equation}\label{boundforg}
    \begin{aligned}
        \bigg|\frac{(1-st)(1-sr)}{(s-r)}g_4 \braket{f^r}{f^s}\bigg| \leq Ce^{k\log(1/s)},\quad \bigg|\frac{(1-st)(1-sr)}{(s-r)}g_3 \braket{f^r}{f^s}\bigg| \leq Ce^{k\log(\max(r,\sqrt{q}+\epsilon/2))}
    \end{aligned}
\end{equation}
by a similar analysis of the poles of $g_4$ and $g_3.$

Lastly, by $\eqref{Efs},$ we see that 
\begin{equation}\label{boundforE}
    \big|-Ef^t + s^{-1}f^{1/s}\braket{f^s}{f^t} \big| \leq Ce^{k\log(1/s)}.
\end{equation}
Hence, combining $\eqref{boundforK}$, $\eqref{boundforg}$, and $\eqref{boundforE}$, we get the desired upper bounds on $\widehat{G}X_2$.
\end{proof}

Now, we prove the theorem of the finite time distribution. 
\begin{proof}[Proof of Theorem~\ref{Max&LowResult_Finite}] For simplicity of the notation, we denote $c_1 = \frac{(1-tr)}{(t-r)}$ which becomes $\frac{(s-r)}{(1-sr)}$ in the limit.
    We use the identity in Lemma \ref{lem:keychange} to rewrite $\eqref{Newidentity}$ into
    \begin{equation}
        \begin{aligned}
            &\mathrm{Pf}(J - \overline{K}^{\text{geo}})\left(\frac{1}{1-st} - c_1\braket{Y_2}{X_2}\right)- \mathrm{Pf}(J - \overline{K}^{\text{geo}})
            + \mathrm{Pf}\left(J - \overline{K}^{\text{geo}} - c_1B - c_1C\right).
        \end{aligned}
    \end{equation}
    where 
    \begin{equation}
        B = \ketbra{\begin{array}{c}
                  (\overline{G}X_2)_2 \\
                 -(\overline{G}X_2)_1 
            \end{array}}{-g_1 \quad g_2}, \quad C = \ketbra{\begin{array}{c}
                  -g_1 \\
                 g_2 
            \end{array}}{-(\overline{G}X_2)_2\quad (\overline{G}X_2)_1}.
    \end{equation}
    Let $\mathpzc{A}(k,\ell) = \begin{pmatrix}
        \mathpzc{A}_{11}(k,\ell) & \mathpzc{A}_{12}(k,\ell)\\
        \mathpzc{A}_{21}(k,\ell) & \mathpzc{A}_{22}(k,\ell)
    \end{pmatrix}$ represent the entries of the kernel $\overline{K}^{\text{geo}}$, $B$, and $C$.
    By using $\eqref{BoundsforK}$, $\eqref{BoundsForg}$, and $\eqref{BoundsforG}$, we see that
    \begin{equation}
    \begin{aligned}
        |\mathpzc{A}_{11}(k,\ell)| &\leq Ce^{(k+\ell)\log(\max(\sqrt{q} + \epsilon/2, r))},\\
        |\mathpzc{A}_{12}(k,\ell)| &\leq Ce^{k\log(\max(\sqrt{q} + \epsilon/2, r))+ \ell \log(1/s)},\\
        |\mathpzc{A}_{21}(k,\ell)| &\leq Ce^{k \log(1/s)+\ell\log(\max(\sqrt{q} + \epsilon/2, r))},\\
        |\mathpzc{A}_{22}(k,\ell)| &\leq Ce^{(k+\ell)\log(1/s)}.
    \end{aligned}     
    \end{equation}
    Hence, we can apply Hadamard's bound to prove convergence of Fredholm Pfaffian due to the fact that $s > \max(\sqrt{q} + \epsilon/2, r).$
\end{proof}

\begin{proof}[Proof of Theorem~\ref{FiniteTimeFormulaTheorem} $(2)$ (Maximal current phase)] Since the formula provided in Theorem \ref{Max&LowResult_Finite} has a well-defined limit as $s \rightarrow 1,$ we see that this is an immediate result of Theorem \ref{Max&LowResult_Finite} and Lemma \ref{DeriveMax&Low}.
\end{proof}

\begin{proof}[Proof of Theorem~\ref{FiniteTimeFormulaTheorem} $(3)$ (Low density phase)] This is an immediate result of Theorem \ref{Max&LowResult_Finite} and Lemma \ref{DeriveMax&Low}. It is also easy to check that $F_{\overline{r}}^{LD}(d,1) = \theta(d,1) - \theta(d-1,1)=1$ for all $d\in \Z_{\geq 0}$. We recall that for fixed $N,$ $\overline{H}(x) = \left(\frac{1-\sqrt{q}/x}{1-\sqrt{q}x}\right)^{N-1}$ and ${H}(x) = \left(\frac{1-\sqrt{q}/x}{1-\sqrt{q}x}\right)^{N-2}.$
We set $(s,r) = (\overline{r},\sqrt{q})$ for the formula in Theorem \ref{Max&LowResult_Finite}.
For example, we have
\begin{equation}
    \begin{aligned}
        \overline{c} \overline{\phi_2}(k)\bigg|_{(s,r) = (\overline{r},\sqrt{q})} &= \frac{-1}{(2\pi\I)^2}\frac{(\overline{r}-\sqrt{q})}{(1-\sqrt{q}\overline{r})}\frac{H(\overline{r})}{\overline{r}^{d+1}}\oint \limits_{\Gamma_{\sqrt{q}}} dw \oint \limits_{\Gamma_{1/\sqrt{q}}} dz \frac{w^{d+1}}{z^{k+1}}\frac{\overline{H}(z)}{\overline{H}(w)}\frac{(z\overline{r}-1)(zw-1)}{(\overline{r}-w)(w\overline{r}-1)(z^2-1)(z-w)}\\
        &= \frac{-1}{(2\pi\I)^2}\frac{\overline{H}(\overline{r})}{\overline{r}^{d}}\oint \limits_{\Gamma_{\sqrt{q}}} dw \oint \limits_{\Gamma_{1/\sqrt{q}}} dz \frac{w^{d+1}}{z^{k+1}}\frac{\overline{H}(z)}{\overline{H}(w)}\frac{(z\overline{r}-1)(zw-1)}{(\overline{r}-w)(w\overline{r}-1)(z^2-1)(z-w)}\\
        &= \Omega_2^{L}(k).
    \end{aligned}
\end{equation}
\end{proof}

\section{Asymptotic analysis for Maximal current phase}
Consider the scaling of $k,l,z,w,d,r$ as in $\eqref{scale}$. Maximal current phase requires that $r < 1,$ i.e., $\tilde{r} < 0.$ We will replace $k,l$ with the corresponding new variables $X,Y$ in the following lemmas. Recall that $c_0 = \frac{2\sqrt{q}(1+\sqrt{q})}{(1-\sqrt{q})^3}.$
\begin{lem}\label{MaxKernelLimit}
    For any given $L>0,$ the following limits hold uniformly for $X,Y \in [-L,L]:$
    \begin{equation}
        \begin{aligned}
            &\lim_{N\rightarrow \infty}(c_0N/2)^{2/3}\overline{\mathsf{K}}_{11}^{M}(X,Y) = \overline{\mathpzc{K}}_{11}^{M}(X,Y),\quad \lim_{N\rightarrow \infty} (c_0N/2)^{1/3}\overline{\mathsf{K}}_{12}^{M}(X,Y) = \overline{\mathpzc{K}}_{12}^{M}(X,Y),\\
            &\lim_{N\rightarrow \infty}(c_0N/2)^{1/3}\overline{\mathsf{K}}_{21}^{M}(X,Y) = \overline{\mathpzc{K}}_{21}^{M}(X,Y),\quad \lim_{N\rightarrow \infty} \overline{\mathsf{K}}_{22}^{M}(X,Y) = \overline{\mathpzc{K}}_{22}^{M}(X,Y).\\ 
        \end{aligned}
    \end{equation}
    Furthermore, there exist constants $0<b<-\tilde{r}$, $C>0$, and $N_0\in \N$ such that for all $N \geq N_0$ and $X,Y > -u,$ we have
    \begin{equation}
        \begin{aligned}
            \overline{K}_{\text{scaled}}^{M}(X,Y) := &\begin{pmatrix}
                (c_0N/2)^{2/3}\overline{K}_{11}^{M} & (c_0N/2)^{1/3}\overline{K}_{12}^{M}\\
                (c_0N/2)^{1/3}\overline{K}_{21}^{M} & \overline{K}_{22}^{M}
            \end{pmatrix}(X,Y) \leq
                \begin{pmatrix}
                    Ce^{-b(X+Y)} & Ce^{-bX}\\
                    Ce^{-bY} & C
                \end{pmatrix}.
        \end{aligned}
    \end{equation}
\end{lem}

\begin{proof}
    The scaling of each kernel entry remains the same as in Lemma~\ref{lem:limit&bound}.  We apply the steepest descent method as in Lemma~\ref{steepestDescent} to get the limit. We use the $\overline{\mathsf{K}}_{12}^M$ to give an example of the upper bound. The pole at $z = 1/r$ is the dominating term for $X$ and gives the upper bound $e^{-bX}$ for some $0<b < -\tilde{r}$. The pole at $w = 1$ is the dominating term for $Y$ and gives the upper bound $C$ for some $C>0$ independent of $X,Y.$ Similar analysis of poles applies to other entries of the kernel.
\end{proof}

\begin{lem}\label{MaxFunctionlimit}
    For any given $L>0,$ the following limits hold uniformly for $X \in [-L,L]:$
    \begin{equation}
        \begin{aligned}
            &\lim_{N\rightarrow \infty}(c_0N/2)^{-1/3}\overline{\phi}_{1}^{M}(X) = \widetilde{\phi}_{1}^{M}(X),\quad \lim_{N\rightarrow \infty}\overline{\phi}_{2}^{M}(X) = \widetilde{\phi}_{2}^{M}(X),\\
            &\lim_{N\rightarrow \infty}(c_0N/2)^{1/3}{g}_{1}^{M}(X) = \widetilde{g}_{1}^{M}(X),\quad \lim_{N\rightarrow \infty}{g}_{2}^{M}(X) = \widetilde{g}_{2}^{M}(X),\\
            &\lim_{N\rightarrow \infty}(c_0N/2)^{-1/3}{\mathsf{g}}_{4}^{M}(X) = \widetilde{\mathsf{g}}_{4}^{M}(X),\quad \lim_{N\rightarrow \infty}{\mathsf{g}}_{3}^{M}(X) = \widetilde{\mathsf{g}}_{3}^{M}(X),\\
            &\lim_{N\rightarrow \infty} (c_0N/2)^{-1/3}\mathsf{h}_{M}(\tilde{d},X) = \widetilde{\mathsf{h}}_{M}(\tilde{d},X), \quad \lim_{N\rightarrow \infty} (c_0N/2)^{-1/3}e_M^{1,r}(\tilde{d}) = \widetilde{\mathpzc{e}}_M^{\tilde{r}}(\tilde{d}).
        \end{aligned}
    \end{equation}
\end{lem}

\begin{proof}
    We use the steepest descent method to get corresponding limits for functions that involve contour integrals. The limit of $(c_0N/2)^{-1/3}\mathsf{h}_M$ is obvious.
\end{proof}

\begin{lem}\label{MaxFunctionBounds}
    Fix any $L>0,$ $\tilde{r} < 0.$ There exist constants $-\tilde{r}>b>0$, $C>0,$ $N_0 \in \N$ such that for all $N \geq N_0$, $X\geq -L$, we have 
        \begin{equation}
        \begin{aligned}
            &|(c_0N/2)^{1/3}g_1^M(X)| \leq Ce^{-b X}, \quad |g_2^{M}(X)| \leq C,\quad |(c_0N/2)^{-1/3}\Omega_1^{M}(X)| \leq C,\quad |\Omega_2^{M}(X)| \leq Ce^{-bX},\\
            &|\mathsf{g}_3^M(X)| \leq Ce^{-b X},\quad |(c_0N/2)^{-1/3}\mathsf{g}_4^M(X)| \leq C, \quad |(c_0N/2)^{-1/3}\mathsf{h}_M(\tilde{d},X)| \leq 2|X|,\\
            & |(c_0N/2)^{-1/3}{\phi}_1^M| \leq C|X|, \quad  |(c_0N/2)^{-1/3}{\phi}_2^M| \leq Ce^{-bX}.
        \end{aligned}
    \end{equation}
\end{lem}

\begin{proof}
    All upper bounds follow from analogous pole analysis. We give one example. For $\mathsf{g}_3^M$, the dominant contribution comes from the pole at $ z = 1/r$, which yields an upper bound $e^{-bX}$ for some $0<b<-\tilde{r}.$
\end{proof}

\begin{proof}[Proof of Theorem~\ref{theorem:limit} $(2)$]
By Lemma $\ref{MaxKernelLimit}$ and $\ref{MaxFunctionlimit}$, we have the point-wise limit of each function and each kernel entry. To show convergence of the Fredholm Pfaffian, we use upper bounds for the kernel as in Lemma $\ref{MaxKernelLimit}$, Lemma $\ref{MaxFunctionBounds}$ and choose any $-\tilde{r}>b>0$ which allow us to apply Hadamard's bound and dominated convergence theorem. The convergence of the discrete difference in shift argument formula $\eqref{formulaMax}$ to partial derivatives can be justified by the same method as in section $\ref{continuousDerivative}.$
\end{proof}

\section{Asymptotic analysis for Low density phase}\label{appendixB} 
We keep the scaling for $k,l,z,w,d,r,s$ the same as in $\eqref{scale}$.
We should expect that the limit under the critical scaling of the product stationary geometric LPP model will coincide with that of the product stationary exponential LPP model. We will use functions already defined in \cite{Betea_2020}. For example, in \cite[(2.25)]{Betea_2020}, a function $\mathpzc{e}^{\delta,u}(S)$ is introduced, which we adapt to our setting by taking $\delta = -\tilde{r}$, $u = 0$ and $S = \tilde{d}.$ All other functions follow the same parameter choices: $\delta = -\tilde{r}$, $u = 0$ and $S = \tilde{d}.$

There are, however, two important differences to note in our usage.
First, the limit of $\overline{\mathsf{K}}_{22}^{L}$ does not match its expected counterpart, namely  $\overline{\mathcal{A}}_{22}$ in \cite[(2.27)]{Betea_2020}. Second, the limit of $(c_0N/2)^{-1/3}\phi_1$ does not coincide with the expected limit, $\mathpzc{h}_{1}^{-\tilde{r}, 0}$ in \cite[(2.30)]{Betea_2020}. We will explain these two limits in the following Lemmas.

We first state the pointwise limit of each entry of the kernel. We use the same functions $\overline{\mathcal{A}}_{11}$, $\overline{\mathcal{A}}_{12}$, $\overline{\mathcal{A}}_{21}$, $\widetilde{\mathcal{A}}_{12}$, $\widetilde{\mathcal{A}}_{22}$ as in \cite[(2.27), (2.29)]{Betea_2020}, all with the same parameter choices: $\delta = -\tilde{r},$ $u=0$. The limit of $\overline{\mathsf{K}}_{22}^{L}$ is defined separately. We will now use $X,Y$ as new variables.
\begin{lem}\label{KernelPtwiseLimit}
    For any given $L >0,$ the following limits hold uniformly for $X,Y \in [-L,L]$:
    \begin{equation}
    \begin{aligned}
        \lim_{N\rightarrow \infty} (c_0N/2)^{2/3}\overline{\mathsf{K}}_{11}^{L}(k,\ell) &= \overline{\mathcal{A}}_{11}(X,Y), \quad \lim_{N\rightarrow \infty} (c_0N/2)^{1/3}\overline{\mathsf{K}}_{12}^{L}(k,\ell) = \overline{\mathcal{A}}_{12}(X,Y),\\
        \lim_{N\rightarrow \infty} (c_0N/2)^{1/3}\overline{\mathsf{K}}_{21}^{L}(k,\ell) &= \overline{\mathcal{A}}_{21}(X,Y), \quad \lim_{N\rightarrow \infty} \overline{\mathsf{K}}_{22}^{L}(k,\ell) = \overline{\mathsf{A}}_{22}^{L}(X,Y),\\
        \lim_{N \rightarrow \infty} (c_0N/2)^{1/3}\widetilde{\mathsf{K}}_{12}^{L}(k,\ell) &= \widetilde{\mathcal{A}}_{12}(X,Y),\quad
        \lim_{N \rightarrow \infty} \widetilde{\mathsf{K}}_{22}^{L}(k,\ell) = \widetilde{\mathcal{A}}_{22}(X,Y).
    \end{aligned} 
    \end{equation}
\end{lem}

\begin{proof}
    The scaling of each kernel entry remains the same as in Lemma~\ref{lem:limit&bound}.  We apply the steepest descent method as in Lemma~\ref{steepestDescent} to get the limit. We observe that $\overline{\mathsf{A}}_{22}^{\text{geo}}$ matches $\overline{\mathcal{A}}_{22}(X,Y)$ with $\delta = - \tilde{r}$ and $u = 0$ , except that one term, namely $\mathcal{E}_1$ in \cite[(2.28)]{Betea_2020}, is removed. This is because $\mathcal{E}_1$ originates from the factor $(1-\sqrt{q}z)^n$ where $n$ represents the shift away from the diagonal. Since we are focusing on the diagonal distribution, this term does not appear in our setting.
\end{proof}

Next, we discuss the scaling limit of all other functions used in the formula. Definitions of all limiting functions, $\mathpzc{g}_1^{-\tilde{r},0}, \mathpzc{g}_2^{-\tilde{r},0}, \mathpzc{g}_3^{-\tilde{r},0}, \mathpzc{g}_4^{-\tilde{r},0}, \mathpzc{f}^{\tilde{r},0}, \mathpzc{j}^{-\tilde{r},0},\mathpzc{e}^{-\tilde{r},0}, \mathpzc{h}^{-\tilde{r},0}$ in the next lemma can be found in \cite[(2.25), (2.26), (2.30)]{Betea_2020}.
\begin{lem}\label{FunctionPtwiseLimit}
    For any given $L >0,$ the following limits hold uniformly for $X \in [-L,L]$:
    \begin{equation}
    \begin{aligned}
        &\lim_{N\rightarrow \infty} (c_0N/2)^{1/3} g_1^L(X) = \mathpzc{g}_1^{-\tilde{r},0} (X)  , \quad \lim_{N\rightarrow \infty} g_2^L(X) = \mathpzc{g}_2^{-\tilde{r},0} (X),\\
        &\lim_{N\rightarrow \infty} \mathsf{g}_3^L(X) = \mathpzc{g}_3^{-\tilde{r},0} (X), \quad \lim_{N\rightarrow \infty} (c_0N/2)^{-1/3}\mathsf{g}_4^L(X) = \mathpzc{g}_4^{-\tilde{r},0} (X),\\
        &\lim_{N\rightarrow \infty} f_L^{1/r}(X) = \mathpzc{f}^{\tilde{r},0}(X), \quad \lim_{N\rightarrow \infty} (c_0N/2)^{-1/3}\mathsf{h}_L(\tilde{d}, X) = \mathpzc{j}^{-\tilde{r},0}(\tilde{d},X), \\
        &\lim_{N\rightarrow \infty} (c_0N/2)^{-1/3}e_L^{r}(\tilde{d}) = \mathpzc{e}^{-\tilde{r},0}(\tilde{d}),\quad \lim_{N \rightarrow \infty} \phi_2^L(X) =  \mathpzc{h}_2^{-\tilde{r},0}(X),\\
        &\lim_{N \rightarrow \infty} (c_0N/2)^{-1/3}\phi_1^L(X) = \widetilde{\mathsf{h}}_{1,\text{geo}}^{-\tilde{r},0}(X).
    \end{aligned}
    \end{equation}
\end{lem}

\begin{proof}
    We use steepest descent method for functions involving contour integrals. To see the limit of $e^r$, we need a change of variable. To study the limit of $\mathsf{h}_L$, we observe that
    \begin{equation}
        \begin{aligned}
            &(c_0N/2)^{-1/3}\mathsf{h}_L(\tilde{d},X) = \frac{\overline{H}(r)}{r^{d}}\frac{r^{k-d-1} - r^{d-k+1}}{(c_0N/2)^{1/3}(r^2-1)} + \frac{\overline{H}(r)}{r^{d}}r^{d-k-1}\left((c_0N/2)^{-1/3}(k-d)\right)\\
            &\rightarrow \left(e^{\frac{\tilde{r}^3}{3} - \tilde{r}\tilde{d}}\right)\left(\frac{e^{\tilde{r}(X-\tilde{d})} - e^{\tilde{r}(\tilde{d}-X)}}{2\tilde{r}} + (X-\tilde{d})e^{\tilde{r}(\tilde{d}-X)} \right) \\
            &= \mathpzc{f}^{\tilde{r},0}(\tilde{d})\left( \frac{\sinh(\tilde{r}(X-\tilde{d}))}{\tilde{r}} + (X-\tilde{d})e^{-\tilde{r}(X-\tilde{d})}\right) = \mathpzc{j}^{-\tilde{r},0}(\tilde{d},X).
        \end{aligned}
    \end{equation}

    For the limit of $(c_0N/2)^{-1/3}\phi_1,$ it is different from $\mathpzc{h}_1^{-\tilde{r},0}$ in \cite[(2.30)]{Betea_2020} because we do not include the term $\int_{\tilde{d}}^{\infty} dV \mathcal{E}_1(Y,V)\mathpzc{f}^{-\delta,u}(V)$. As we explained in the previous lemma, we do not have the term $\mathcal{E}_1$ because this term comes from an off-diagonal shift.
\end{proof}

We have almost the same upper bound for each function as in the stationary exponential LPP model. The follow lemma is a modified version of \cite[Lemma 4.2]{Betea_2020} and \cite[Lemma 4.3]{Betea_2020} for our setting.
\begin{lem}\label{KernelBounds}
    Fix any $L>0$. Fix any given $\kappa>0,$ there exist constants $\sigma>0$, $C>0$ independent of $X,Y$, $N_0 \in \N$ such that for all $N \geq N_0$ the following upper bounds hold for all $X,Y\geq -L$
    \begin{equation}
        \begin{aligned}
            &|(c_0N/2)^{2/3}\overline{\mathsf{K}}_{11}^{L}(X,Y)| \leq Ce^{-\kappa(X+Y)},\quad
            |(c_0N/2)^{1/3}\overline{\mathsf{K}}_{12}^{L}(X,Y)| \leq C\left(e^{-\kappa(X+Y)} + e^{-\kappa X + (|\tilde{r}|+\sigma)Y}\right),\\
            &|(c_0N/2)^{1/3}\overline{\mathsf{K}}_{21}^{L}(X,Y)| \leq C\left(e^{-\kappa(X+Y)} + e^{-\kappa Y + (|\tilde{r}|+\sigma)X}\right), \quad \\
            &|\overline{\mathsf{K}}_{22}^{L}(X,Y)| \leq C\left(e^{-\kappa X + (|\tilde{r}|+\sigma)Y} + e^{-\kappa Y + (|\tilde{r}|+\sigma)X} + e^{(|\tilde{r}|+\sigma)|X-Y|}\right),\\
            &|{\Omega}_{2}^{L}(X,Y)| \leq Ce^{-\kappa(X+Y)},\quad
            |(c_0N/2)^{-1/3}{\Omega}_{1}^{L}(X,Y)| \leq C\left(e^{-\kappa(X+Y)} + e^{(|\tilde{r}|+\sigma)X - \kappa Y}\right).\\
        \end{aligned}
    \end{equation}
\end{lem}
\begin{proof}
    We use similar arguments as in Lemma \ref{lem:limit&bound}.
    For poles at $1/\sqrt{q}$ and $\sqrt{q}$, we get terms like $e^{-kX}$ and $e^{-kY}$. For poles at $r$ and $1/r$, we get terms like $e^{(|\tilde{r}|+\sigma)X}$, where $\sigma$ comes from the error term in the convergence of $H(r)/r^{k} \rightarrow e^{\frac{\tilde{r}^3}{3} - X\tilde{r}}$ which is not present in the exponential case.
\end{proof}

The following lemma is a modified version of \cite[Lemma 4.1]{Betea_2020} and \cite[Corollary 4.5]{Betea_2020} adapted to our setting.
\begin{lem}\label{FunctionBounds}
    Fix any $L>0.$ For any $\kappa>0,$ there exist constants $\sigma>0$, $C>0,$ $N_0 \in \N$ such that for all $N \geq N_0$, $X\geq -L$, we have 
        \begin{equation}
        \begin{aligned}
            &|(c_0N/2)^{1/3}g_1^L(X)| \leq Ce^{-\kappa X}, \quad |g_2^L(X)| \leq C\left(e^{(|\tilde{r}|+\sigma)X} + e^{-\kappa X}\right),\quad |\phi_2^L(X)| \leq Ce^{-\kappa X},\\
            &|\mathsf{g}_3^L(X)| \leq Ce^{-\kappa X},\quad |(c_0N/2)^{-1/3}\mathsf{g}_4^L(X)| \leq C\left(e^{-\kappa X} + e^{(|\tilde{r}|+\sigma)X} + |X|e^{(|\tilde{r}|+\sigma)X}\right).\\
        \end{aligned}
    \end{equation}
    Independent of $\kappa$, there exists $\sigma>0,$ $C>0$, $N_1\in \N$ such that for all $N \geq N_1,$ $X\geq -L$, we have 
    \begin{equation}
    \begin{aligned}
        &|f_L^{1/r}(X)| \leq Ce^{(|\tilde{r}|+\sigma)X}, \quad |(c_0N/2)^{-1/3}\mathsf{h}_L(\tilde{d},X)| \leq C|X|e^{(|\tilde{r}|+\sigma)X},\\
        &|(c_0N/2)^{-1/3}\phi_1^{L}(X)| \leq C|X|e^{(|\tilde{r}|+\sigma)X}.
    \end{aligned}
    \end{equation}
\end{lem}

\begin{proof}
    For $\mathsf{g}_4$, the upper bound for a pole of order two gives $|X|e^{|\tilde{r}X|}.$ Upper bounds for other functions follow a similar argument as in Lemma \ref{KernelBounds}.
\end{proof}

Lastly, we prove the theorem of the large time scaling limit.
\begin{proof}[Proof of Theorem~\ref{theorem:limit} $(3)$]
By Lemma $\ref{KernelPtwiseLimit}$ and $\ref{FunctionPtwiseLimit}$, we have the point-wise limit of each function and each kernel entry. To show convergence of the Fredholm Pfaffian, we use Lemma $\ref{KernelBounds}$, Lemma $\ref{FunctionBounds}$ and choose any $\kappa>|\tilde{r}|+\sigma$ which allow us to apply Hadamard's bound and dominated convergence theorem. The convergence of the discrete difference in shift argument formula $\eqref{formulaLow}$ to partial derivative can be justified by the same method as in section $\ref{continuousDerivative}.$
\end{proof}

\appendix
\section{Numerical evaluations for High density phase}
We first provide two examples ($N=2$ and $N=3, r=0$) of direct calculation that give the exact CDF of the half-space LPP model.

When $N=2,$ $r,s\neq 0,$ we express each function appearing in formula \ref{finiteDis} as follows.
\begin{equation}
    \begin{aligned}
        &H(x) = 1, \quad f^{x}(k) = x^{k+1}, \quad E(k,\ell) = \begin{cases}
            -r^{k-\ell-1}, & \textrm{if }k> \ell,\\
        0, & \textrm{if }k = \ell,\\
        r^{\ell-k-1}, & \textrm{if }k < \ell,\\
        \end{cases}\\
        &\widehat{\mu}_{d} = 0, \quad \widehat{\nu}_d = \frac{(1-sr)}{(s-r)(1-s^2)}\frac{r^{d+2}}{s^{d+1}}, \quad \widehat{A}_{11} = 0, \quad \widehat{A}_{12} = \widehat{A}_{21} = 0,\quad
        \widehat{A}_{22} = 0, \quad \widehat{K} = \begin{pmatrix}
            0 & 0\\
            0 & E
        \end{pmatrix},\\
        &
        g1(k) = -(1-sr)f^{s}(k), \quad d_2(k) = \frac{(1-s^2)}{(s-r)}f^s(k) - \frac{(1-sr)}{(s-r)}f^r(k), \quad \widehat{\mathsf{J}} = 0, \quad \widehat{\mathcal{A}}_d =\widehat{\mathcal{C}}_d = \widehat{D}_{d} = 0,\\
        & \widehat{\mathcal{B}}_d = \widehat{\mathcal{E}}_d = -\frac{s(s-r)}{(1-s^2)(1-sr)}(sr)^{d+2}, \quad
        \widehat{V}_1(\ell) = 0, \quad \widehat{V}_2(\ell) = \frac{(s-r)s}{(1-sr)(1-s^2)}f^r(\ell).\\
    \end{aligned}
\end{equation}
Then we compute all components of the formula $\eqref{finiteDis}$ as follows.
\begin{equation}
\begin{aligned}
    &\mathrm{Pf}(J - \widehat{K}) = 1, \quad \widehat{\mu}_d +\widehat{\nu}_d +\widehat{\mathcal{A}}_d +\widehat{\mathcal{B}}_d+\widehat{\mathcal{C}}_d+\widehat{\mathcal{D}}_d+\widehat{\mathcal{E}}_d =\frac{(1-sr)}{(s-r)(1-s^2)}\frac{r^{d+2}}{s^{d+1}} -\frac{2s(s-r)}{(1-s^2)(1-sr)}(sr)^{d+2}. \\
    &- \mathrm{Pf}(J - \widehat{K})
    + 
    \mathrm{Pf}\left(J - \widehat{K} - \ket{\begin{array}{c}
         g_1 \\
         -d_2
    \end{array}}\bra{\widehat{V}_1 \quad \widehat{V}_2} - \ket{\begin{array}{c}
         \widehat{V}_1 \\
         \widehat{V}_2
    \end{array}}\bra{-g_1 \quad d_2}\right)\\
    &= -\mathrm{Pf}(J-\widehat{K}) \brabarket{\widehat{V}_1 \quad \widehat{V}_2} {(\Id - J^{-1}\widehat{K})^{-1}}{\begin{array}{c}
         d_2 \\
         g_1
    \end{array}} = -\braket{\widehat{V}_2}{g_1} = \frac{s(s-r)}{(1-s^2)(1-sr)}(sr)^{d+2}.
\end{aligned}
\end{equation}
Therefore,
\begin{equation}
   \psi(d,2) = d+2 - \frac{s(1-r^2)}{(1-sr)(s-r)}+ \frac{(1-sr)}{(s-r)(1-s^2)}\frac{r^{d+2}}{s^{d+1}} -\frac{s(s-r)}{(1-s^2)(1-sr)}(sr)^{d+2}.
\end{equation}
Then applying the shift argument \ref{finiteTimeFormulaTheorem} to $\psi$ yields the following.
When $d\geq 2,$
\begin{equation}\label{vanish}
    \begin{aligned}
        \text{RHS of }\ref{finiteTimeFormulaTheorem} &=\frac{(1-sr)}{(1-s^2)(s-r)^2}\left(\frac{r^{d+2}}{s^{d}} - \frac{r^{d+2}}{s^d} - \frac{r^{d+1}}{s^{d-1}}
        + \frac{r^{d+1}}{s^{d-1}}\right)\\
        &- \frac{1}{(1-s^2)(1-sr)}(sr)^{d+1}\left(rs^3 - rs - s^2 + 1\right)\\
        &+ \frac{1}{(s-r)}\left( sd - r(d-1) - s(d-1) + r(d-2) \right)\\
        &= 1- (sr)^{d+1} = \Pb(\text{Geom}(rs) \leq d) = \Pb(G(2,2) \leq d).
    \end{aligned}
\end{equation}
When $d= 1,$
\begin{equation}
    \begin{aligned}
        \text{RHS of }\ref{finiteTimeFormulaTheorem} &= \frac{1}{(s-r)}(s-2r) - \frac{s(1-r^2)}{(1-sr)(s-r)^2}(-r) + \frac{(1-sr)}{(1-s^2)(s-r)^2}\left(-r^2\right)\\
        &- \frac{s}{(1-s^2)(1-sr)}(sr)^2(rs^2-r-s) = 1 - (sr)^2 = \Pb(G(2,2) \leq 1).
    \end{aligned}
\end{equation}
When $d = 0,$
\begin{equation}
    \begin{aligned}
        \text{RHS of }\ref{finiteTimeFormulaTheorem} &= \frac{2s}{(s-r)} - \frac{s^2(1-r^2)}{(1-sr)(s-r)^2} + \frac{(1-sr)r^2}{(s-r)^2(1-s^2)} - \frac{s^2}{(1-s^2)(1-sr)}(sr)^2\\
        &=1-sr =\Pb(G(2,2) = 0).
\end{aligned}
\end{equation}

When 
$N=3$ and 
$r=0,$ the formula remains manually computable. Although the statement of Theorem \ref{finiteTimeFormulaTheorem} excludes the case 
$r=0$, this exclusion is purely due to a lack of boundary interaction, making the case uninteresting from the perspective of the stationary LPP model. Nevertheless, we consider $r = 0$
here because it allows for a straightforward verification of our formula through direct computation.
We express each function appearing in formula \ref{finiteDis} as follows.
\begin{equation}
    \begin{aligned}
        &\widehat{A}_{11} = \widehat{A}_{12} = \widehat{A}_{21} = \widehat{A}_{22} = 0, \quad  \widehat{K} = \begin{pmatrix}
            0 & 0\\
            0 & E
        \end{pmatrix}, \quad E(k,\ell) = \begin{cases}
            -1, & \textrm{if }k= \ell+1,\\
        1, & \textrm{if }k = \ell-1,\\
        0, & \textrm{otherwise },\end{cases}\\
        &g_1(k) = G_s(k) - f^s(k), \quad d_2(k) = \frac{(1-s^2)}{s}f^s(k) - R_s(k),\quad 
        \widehat{\mu}_d = \frac{1}{(s-\sqrt{q})}\frac{\sqrt{q}^{d+3}}{s^{d+2 }},\\ &\widehat{\nu}_d = \frac{(1-q)}{(1-s^2)(s-\sqrt{q})}\frac{\sqrt{q}^{d+2}}{s^{d+1}},\quad 
        \widehat{\mathcal{A}}_d = \widehat{\mathcal{D}}_d = \frac{s}{(1-s^2)}\braket{G_{1/s}}{d_2},\quad
        \widehat{B}_d = \widehat{E}_d = \frac{s}{1-s^2}\braket{R_{1/s}}{g_1},\\
        &\widehat{\mathcal{C}}_d = -\frac{s}{(1-s^2)}\frac{(1-s\sqrt{q})}{(s-\sqrt{q})}\frac{\sqrt{q}^{d+3}}{s^{d+3}},\quad \widehat{V}_1(k) = \frac{s}{(1-s^2)}G_{1/s}(k), \\
        &\widehat{V}_{2}(k) = \frac{s}{(1-s^2)}R_{1/s}(k) + \frac{2s}{(1-s^2)}\bra{G_{1/s}}E, \quad \mathrm{Pf}(J-\widehat{K}) = 1,\\
    \end{aligned}
\end{equation}
Notice that $\braket{G_{s}}{R_{1/s}} = \braket{G_{1/s}}{R_s}$. Immediately, we see that 
\begin{equation}
    \begin{aligned}
        \widehat{\mathcal{A}}_{d} + \widehat{\mathcal{B}}_{d} + \widehat{\mathcal{D}}_{d} + \widehat{\mathcal{E}}_{d}
        =2\braket{G_{1/s}}{f^s} - \frac{2s}{(1-s^2)}\braket{R_{1/s}}{f^s}.
    \end{aligned}
\end{equation}
We also have that
\begin{equation}
    \begin{aligned}
        &- \mathrm{Pf}(J - \widehat{K})
        + 
        \mathrm{Pf}\left(J - \widehat{K} - \ket{\begin{array}{c}
             g_1 \\
             -d_2
        \end{array}}\bra{\widehat{V}_1 \quad \widehat{V}_2} - \ket{\begin{array}{c}
             \widehat{V}_1 \\
             \widehat{V}_2
        \end{array}}\bra{-g_1 \quad d_2}\right)\\
        &= -\mathrm{Pf}(J-\widehat{K})\brabarket{\widehat{V}_1\quad \widehat{V}_2}{(\Id - J^{-1}\widehat{K})^{-1}}{\begin{array}{c}
             d_2 \\
             g_1
        \end{array}}\\
        &= -\brabarket{\widehat{V}_1\quad \widehat{V}_2}{\begin{pmatrix}
            1 & -E\\
            0 & 1
        \end{pmatrix}}{\begin{array}{c}
             d_2 \\
             g_1
        \end{array}}\\
        &= -\braket{G_{1/s}}{f^s} + \frac{s}{(1-s^2)}\braket{R_{1/s}}{f^s} + \frac{s}{(1-s^2)}\brabarket{G_{1/s}}{E}{f^s}.
        \end{aligned}
\end{equation}
Then $\psi$ becomes
\begin{equation}
\begin{aligned}
    &\psi(d,3) = d+1 - \frac{\sqrt{q}(1/s+s-2\sqrt{q})}{(1-\sqrt{q}/s)(1-\sqrt{q}s)}
    +\braket{G_{1/s}}{f^s} - \frac{s}{(1-s^2)}\braket{R_{1/s}}{f^s}\\
    &+ \frac{s}{(1-s^2)}\brabarket{G_{1/s}}{E}{f^s} + \widehat{\mu}_d + \widehat{\nu}_d + \widehat{\mathcal{C}}_d\\
    & =d+1 - \frac{\sqrt{q}(1/s+s-2\sqrt{q})}{(1-\sqrt{q}/s)(1-\sqrt{q}s)}  -\frac{s^2(1-\sqrt{q}/s)}{(1-s\sqrt{q})(1-s^2)}(s\sqrt{q})^{d+2} +  \frac{s(1-\sqrt{q}s)}{(s-\sqrt{q})(1-s^2)}\frac{\sqrt{q}^{d+2}}{s^{d+2}}.
\end{aligned}
\end{equation}
Applying the shift argument to $\psi$ gives for $d \geq 1$
\begin{equation}
\begin{aligned}
    \text{RHS of }\ref{finiteTimeFormulaTheorem} &= d-(d-1) -\frac{s^2(1-\sqrt{q}/s)}{(1-s\sqrt{q})(1-s^2)}(s\sqrt{q})^{d+1} (s\sqrt{q}-1) +\frac{s(1-\sqrt{q}s)}{(s-\sqrt{q})(1-s^2)}\frac{\sqrt{q}^{d+1}}{s^{d+1}}\left(\frac{\sqrt{q}}{s}-1\right)\\
    &= 1 + \frac{s^2(1-\sqrt{q}/s)}{(1-s^2)}(s\sqrt{q})^{d+1} - \frac{(1-s\sqrt{q})}{(1-s^2)}\frac{\sqrt{q}^{d+1}}{s^{d+1}} \\
    &= \Pb(\text{Geom}(s\sqrt{q})+ \text{Geom}(\sqrt{q}/s) \leq d) = \Pb(G(3,3)\leq d).
\end{aligned}
\end{equation}

For $d = 0,$
\begin{equation}
\begin{aligned}
    \text{RHS of }\ref{finiteTimeFormulaTheorem} &= 1 - \frac{\sqrt{q}(1/s+s-2\sqrt{q})}{(1-\sqrt{q}/s)(1-\sqrt{q}s)} -\frac{s^2(1-\sqrt{q}/s)}{(1-s\sqrt{q})(1-s^2)}(s\sqrt{q})^{2} +  \frac{s(1-\sqrt{q}s)}{(s-\sqrt{q})(1-s^2)}\frac{\sqrt{q}^{2}}{s^{2}}\\
    &= 1 + \frac{s^2(1-\sqrt{q}/s)}{(1-s^2)}(s\sqrt{q}) - \frac{(1-s\sqrt{q})}{(1-s^2)}\frac{\sqrt{q}}{s} \\
    &= \Pb(\text{Geom}(s\sqrt{q}) + \text{Geom}(\sqrt{q}/s) = 0).
\end{aligned}
\end{equation}

The case $N = 3, s,r\neq 0$ is already quite difficult to verify manually. Therefore, we present the numerical evaluation of our formula using MATLAB and compare it with the CDF obtained through simulations of the two-parameter stationary half-space geometric LPP in Figure \ref{(3,3)}. The two curves are found to be very close and show strong agreement. We computed each function explicitly. For example, brackets such as $\braket{\widehat{P}}{\widehat{A}_{12}}$ are evaluated explicitly. We found the following code particularly useful for estimating Fredholm Pfaffian, \href{https://michaelwimmer.org/downloads.html}{Click here}.

\begin{figure}[htbp]
    \centering
\begin{subfigure}[b]{0.45\textwidth}
\includegraphics[width=1.04\textwidth]{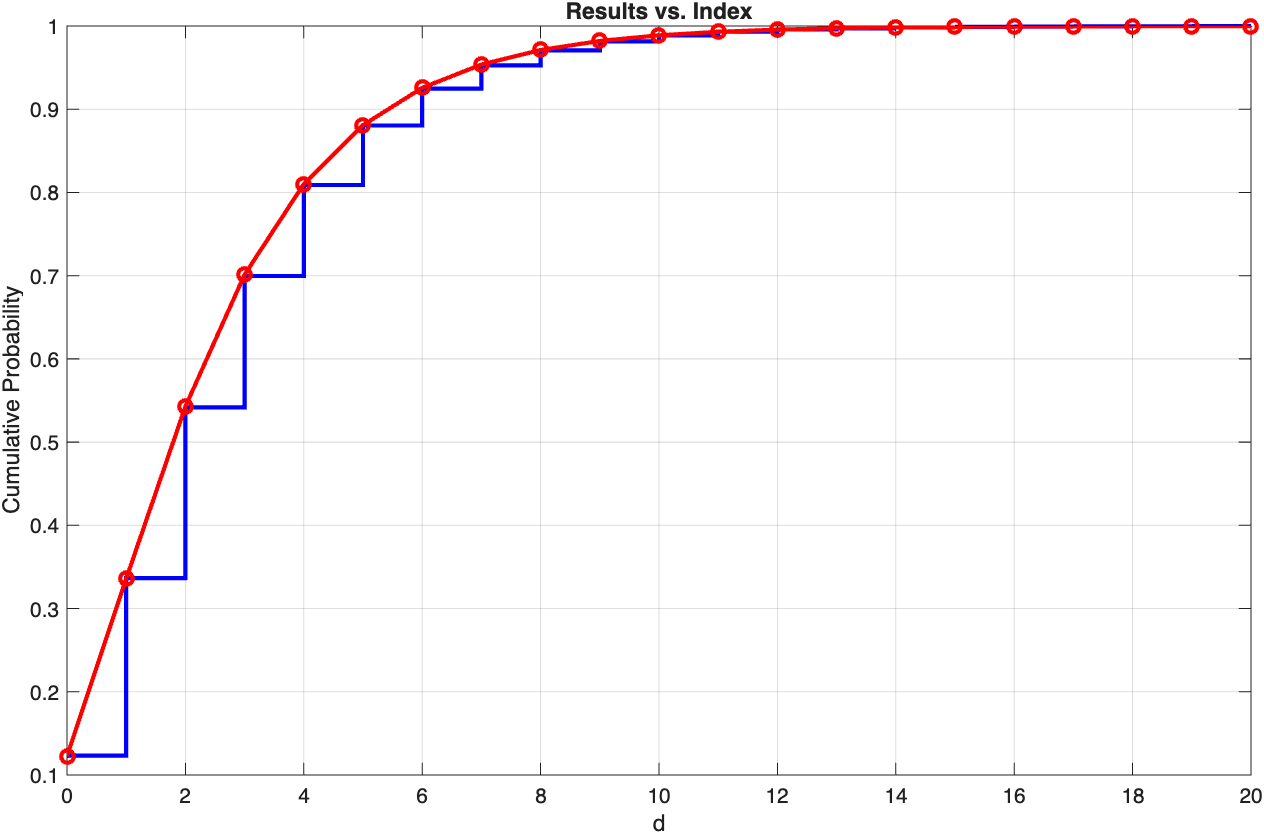}
        \caption{{CDF of $G(3,3)$ with $s = 0.8,$ $r = 0.4,$ $\sqrt{q} = 0.5.$}}
        \label{fig:subfig1}
    \end{subfigure}
    \hspace{0.02\textwidth} 
    \begin{subfigure}[b]{0.45\textwidth}
    \includegraphics[width=1.04\textwidth]{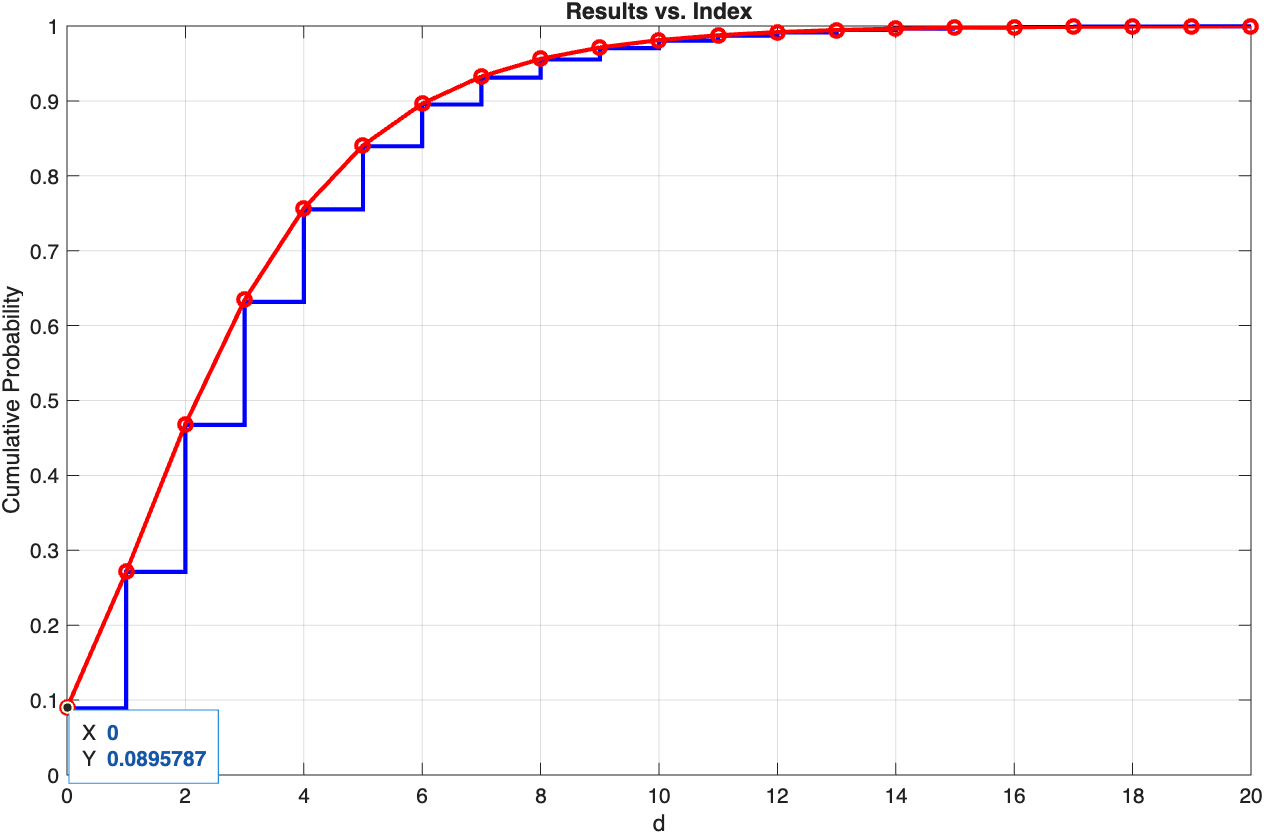}
        \caption{{CDF of $G(3,3)$ with $s = 0.6,$ $r = 0.8,$ $\sqrt{q} = 0.4.$}}
        \label{fig:subfig2}
    \end{subfigure}
    \caption{The red curves in both cases represent the approximate numerical evaluation of our formula for $N=3$. The left figure corresponds to the case $r<s$, while the right figure corresponds to $r>s$. The horizontal axis corresponds to $d$, and the vertical axis shows the exact probability $\Pb(G_{r,s}^{\text{stat}}(3,3) \leq d).$ The blue staircase plots in both cases are the empirical CDFs obtained from simulations of the half-space two-parameter stationary geometric LPP. We generated 20,000 samples of the model and recorded the LPP values at $(3,3)$. The CDF was constructed by counting the proportion of samples less than or equal to each 
    $d$. In Figure (B), the label indicates the probability $\Pb(G(3,3) = 0) = (1-sr)(1-r\sqrt{q})(1-s\sqrt{q})(1-\sqrt{q}/s) = 0.52*0.68*0.76/3 \approx  0.0896.$}
    \label{(3,3)}
\end{figure}

We further present the CDF of our finite time distribution formula for $G_{r,s}^{\text{stat}}(5,5)$ in Figure \ref{(5,5)} and make a few comments. First, we only computed up to $10,$ i.e., $\Pb(G_{r,s}^{\text{stat}}(5,5) \leq d)$ for $d\in \{0,1,2,\dots 10\}$. This is because to get a somewhat precise estimate of the Fredholm Pfaffian, we need at least a $30 \times 30$ matrix (i.e., $K(i,j)$ for $i,j = 1,2,\dots,15$) to make the error within $0.001$. To construct such a large matrix whose entries are as complicated as $\widehat{V}_1$ and $\widehat{V}_2$ is very time consuming. Increasing the size of the approximation matrix can reduce the error of the estimation. Unlike $N=3$, we did not compute each function explicitly, which can cause some error. In particular, all brackets $\braket{\cdot}{\cdot}$ and $\brabarket{\cdot}{\cdot}{\cdot}$ are approximated by finite sums.   

\begin{figure}
    \centering
    \includegraphics[width=0.75\textwidth]{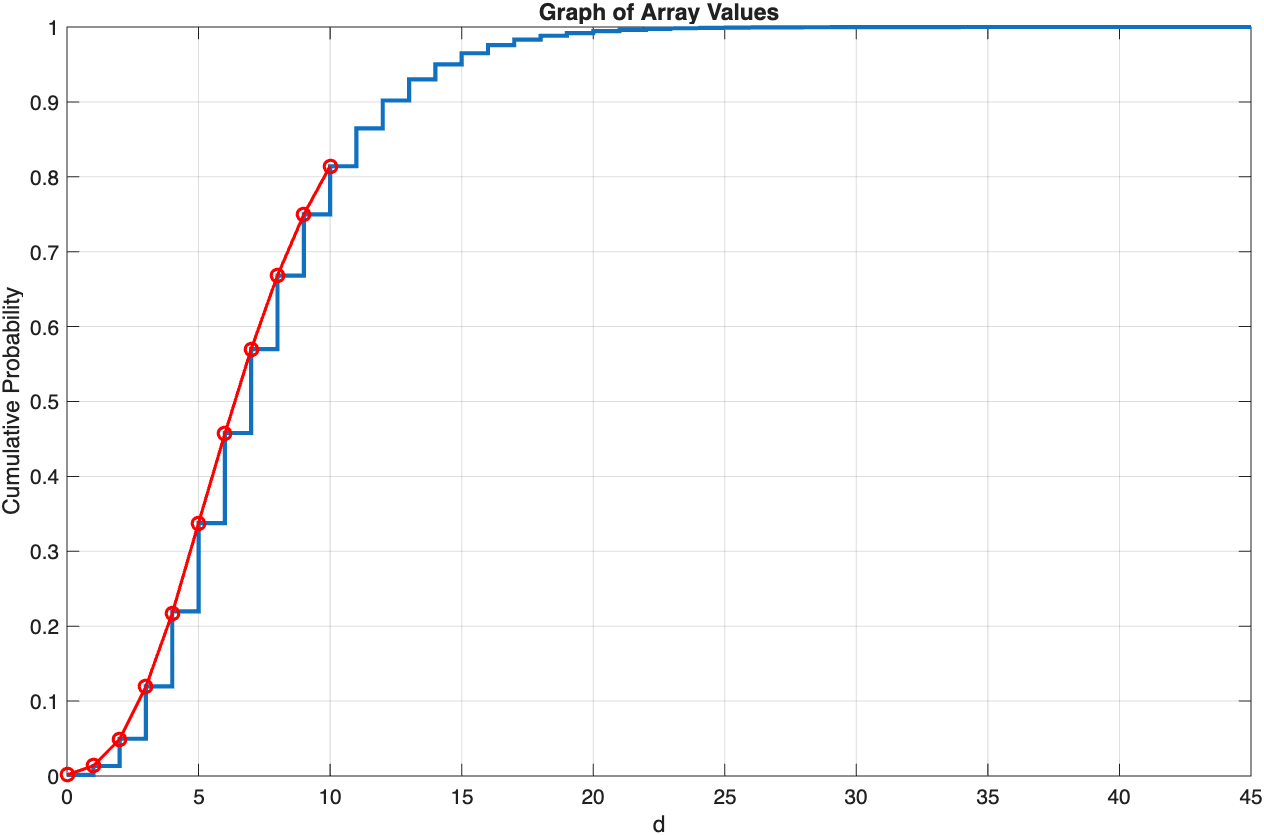}
    \caption{The red curve represents the approximate numerical evaluation of our formula with parameters $N=5$, $s=0.8$, $r=0.4$, $q=0.25.$ The horizontal axis corresponds to $d= \{0,1,2\dots,10\}$, and the vertical axis shows the exact probability $\Pb(G_{r,s}^{\text{stat}}(5,5) \leq d).$ The blue staircase plot is the empirical CDF obtained from simulating the half-space geometric LPP. We generated 50,000 samples of the model 20 times, recorded the LPP values at $(5,5)$ and averaged the results. Again, the CDF was constructed by counting the proportion of samples less than or equal to each 
    $d$.}
    \label{(5,5)}
\end{figure}

We provide a list of numbers from Matlab estimation to show that the error is within $0.001$.
\begin{equation}
    \begin{aligned}
        &\psi(0,5) = 8.360566982592577e-04, \quad \psi(1,5) = 0.008001718755237,\\ &\psi(2,5) = 0.036091649689770, \quad
        \psi(3,5) = 0.109783132398338\\
        & \psi(4,5) = 0.255788772227839,\quad
        \psi(5,5) = 0.496704192344632,\\
        &\psi(6,5) = 0.845600964274872, \quad \psi(7,5) = 1.304995854772828,\\
        &\psi(8,5) = 1.868854964827310, \quad \psi(9,5) = 2.525613136190934,\\
        &\psi(10,5) = 3.261146168402837.
    \end{aligned}
\end{equation}

Then the CDF calculated by applying shift argument to $\psi$ is
\begin{equation}
    \begin{aligned}
        &\Pb(G_{r,s}^{\text{stat}}(5,5) \leq 0) = 2\psi(0,5) = 0.001672113396519,\\
        & \Pb(G_{r,s}^{\text{stat}}(5,5) \leq 1) = 2\psi(1,5) - 3\psi(0,5) = 0.013495267415696,\\
        & \Pb(G_{r,s}^{\text{stat}}(5,5) \leq 2) = 2\psi(2,5) - 3\psi(1,5) + \psi(0,5) = 0.049014199812088,\\
        &\Pb(G_{r,s}^{\text{stat}}(5,5) \leq 3) = 2\psi(3,5) - 3\psi(2,5) + \psi(1,5) =0.119293034482603,\\
        &\Pb(G_{r,s}^{\text{stat}}(5,5) \leq 4) = 2\psi(4,5) - 3\psi(3,5) + \psi(2,5) = 0.218319796950434,\\
        &\Pb(G_{r,s}^{\text{stat}}(5,5) \leq 5) = 2\psi(5,5) - 3\psi(4,5) + \psi(3,5) = 0.335825200404085,\\
     &\Pb(G_{r,s}^{\text{stat}}(5,5) \leq 6) = 2\psi(6,5) - 3\psi(5,5) + \psi(4,5) = 0.456878123743687,\\
     \end{aligned}
\end{equation}
\begin{equation}
\begin{aligned}
   &\Pb(G_{r,s}^{\text{stat}}(5,5) \leq 7) = 2\psi(7,5) - 3\psi(6,5) + \psi(5,5) =0.569893009065672,\\
        &\Pb(G_{r,s}^{\text{stat}}(5,5) \leq 8) = 2\psi(8,5) - 3\psi(7,5) + \psi(6,5) =0.668323329611008,\\
&\Pb(G_{r,s}^{\text{stat}}(5,5) \leq 9) = 2\psi(9,5) - 3\psi(8,5) + \psi(7,5) =0.749657232672766,\\
        &\Pb(G_{r,s}^{\text{stat}}(5,5) \leq 10) = 2\psi(10,5) - 3\psi(9,5) + \psi(8,5) = 0.814307893060182,
    \end{aligned}
\end{equation}
and the CDF obtained from the LPP simulation is
\begin{equation}
\begin{aligned}
   &\Pb(G_{r,s}^{\text{stat}}(5,5) \leq 0) = 0.0016726,\\
   &\Pb(G_{r,s}^{\text{stat}}(5,5) \leq 1) = 0.0134631,\\
   &\Pb(G_{r,s}^{\text{stat}}(5,5) \leq 2) =0.0495106,\\
   &\Pb(G_{r,s}^{\text{stat}}(5,5) \leq 3) =0.1186529,\\
   &\Pb(G_{r,s}^{\text{stat}}(5,5) \leq 4) =0.2181491,\\
   &\Pb(G_{r,s}^{\text{stat}}(5,5) \leq 5) =0.3356810,\\
   &\Pb(G_{r,s}^{\text{stat}}(5,5) \leq 6) =0.4566831,\\
   &\Pb(G_{r,s}^{\text{stat}}(5,5) \leq 7) =0.5699657,\\
   &\Pb(G_{r,s}^{\text{stat}}(5,5) \leq 8) =0.6684062,\\
   &\Pb(G_{r,s}^{\text{stat}}(5,5) \leq 9) =0.7496881,\\
   &\Pb(G_{r,s}^{\text{stat}}(5,5) \leq 10) =0.8142741.\\
\end{aligned}
\end{equation}
This empirical CDF was obtained by generating a sample of size $500,000$ from the LPP model, repeating the simulation $20 $ times, and averaging the results.

\section{Numerical evaluations for Maximal current phase}
In this section, we provide Matlab estimation of the finite time distribution formula for Maximal current phase.
In the following figure, we evaluate $F_{r,1}^{MC}(d,3)$ for $d \in \{0,1,\dots, 15\}$. The red dots show the direct evaluation of \eqref{MaxCDF} using Matlab. The blue staircase shows the simulation of half-space geometric LPP under the Maximal current phase \ref{MaximalModel}.

\begin{figure}
    \centering
    \includegraphics[width=0.6\textwidth]{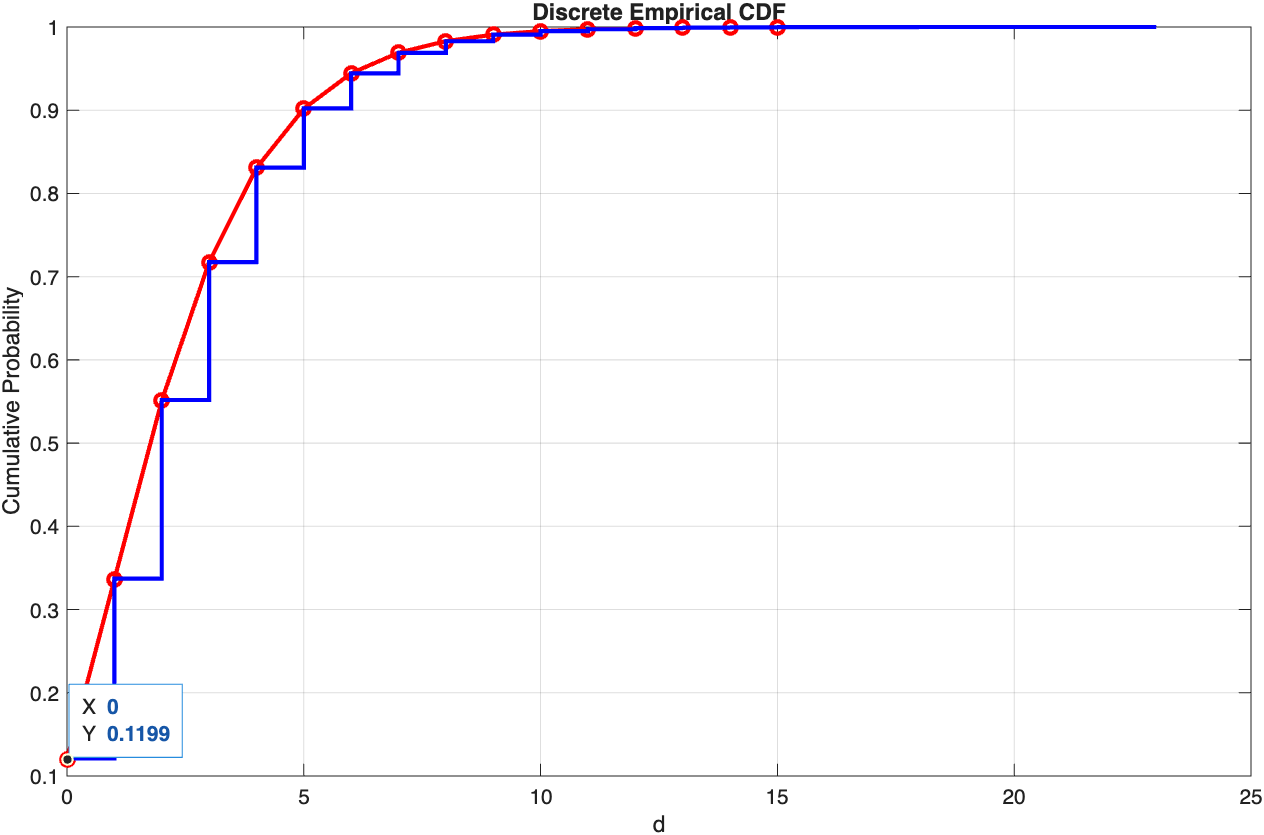}
    \caption{The red curve represents the approximate numerical evaluation of our formula $F_{r,1}^{MC}$ with parameters $N=3$, $r=0.4$, $q=0.25.$ The horizontal axis corresponds to $d= \{0,1,2\dots,15\}$, and the vertical axis shows the exact probability $\Pb(G_{r,1}^{\text{stat}}(3,3) \leq d).$ The blue staircase plot is the empirical CDF obtained from simulating the model \ref{MaximalModel}. We generated 50,000 samples of the model, recorded the LPP values at $(3,3)$ and averaged the results. Again, the CDF was constructed by counting the proportion of samples less than or equal to each 
    $d$. The label in the graph show the probability when $d = 0,$ that is $F_{r,1}^{MC}(0,3) = (1-\sqrt{q})^2 (1-r\sqrt{q})(1-r) = (0.5)^2*0.8*0.6 = 0.12.$}
\end{figure}

\bibliographystyle{plain} 
\bibliography{reference}

@article{gueudre2012,
  title={Directed polymer near a hard wall and {KPZ} equation in the half-space},
  author={Gueudr{\'e}, Thomas and Le Doussal, Pierre},
  journal={Europhysics Letters},
  volume={100},
  number={2},
  pages={26006},
  year={2012},
  publisher={IOP Publishing}
}

@article{das2025convergence,
  title={Convergence to stationary measures for the half-space log-gamma polymer},
  author={Das, Sayan and Serio, Christian},
  journal={Journal of Functional Analysis},
  volume={289},
  number={4},
  pages={110982},
  year={2025},
  publisher={Elsevier}
}

@article{borodin2016directed,
  title={Directed random polymers via nested contour integrals},
  author={Borodin, Alexei and Bufetov, Alexey and Corwin, Ivan},
  journal={Annals of Physics},
  volume={368},
  pages={191--247},
  year={2016},
  publisher={Elsevier}
}

@article{imamura2022solvable,
  title={Solvable models in the {KPZ} class: approach through periodic and free boundary {Schur} measures},
  author={Imamura, Takashi and Mucciconi, Matteo and Sasamoto, Tomohiro},
  journal = {The Annals of Probability},
  year = {to appear},
}

@article{KPZBarraquand,
   title={Steady state of the {KPZ} equation on an interval and {Liouville} quantum mechanics},
   volume={137},
   ISSN={1286-4854},
   url={http://dx.doi.org/10.1209/0295-5075/ac25a9},
   DOI={10.1209/0295-5075/ac25a9},
   number={6},
   journal={Europhysics Letters},
   publisher={IOP Publishing},
   author={Barraquand, Guillaume and Le Doussal, Pierre},
   year={2022},
   month=mar, pages={61003} }

@article{MacdonaldProcess,
   title={Half-space {Macdonald} processes},
   volume={8},
   ISSN={2050-5086},
   url={http://dx.doi.org/10.1017/fmp.2020.3},
   DOI={10.1017/fmp.2020.3},
   journal={Forum of Mathematics, Pi},
   publisher={Cambridge University Press (CUP)},
   author={Barraquand, Guillaume and Borodin, Alexei and Corwin, Ivan},
   year={2020} }

@article{Sasamoto_2004,
   title={Fluctuations of the One-Dimensional Polynuclear Growth Model in Half-Space},
   volume={115},
   ISSN={0022-4715},
   url={http://dx.doi.org/10.1023/B:JOSS.0000022374.73462.85},
   DOI={10.1023/b:joss.0000022374.73462.85},
   number={3/4},
   journal={Journal of Statistical Physics},
   publisher={Springer Science and Business Media LLC},
   author={Sasamoto, T. and Imamura, T.},
   year={2004},
   month=may, pages={749–803} }

@article{Krajenbrink_2020,
   title={Replica {Bethe} {Ansatz} solution to the {Kardar-Parisi-Zhang} equation on the  half-line},
   volume={8},
   ISSN={2542-4653},
   url={http://dx.doi.org/10.21468/SciPostPhys.8.3.035},
   DOI={10.21468/scipostphys.8.3.035},
   number={3},
   journal={SciPost Physics},
   publisher={Stichting SciPost},
   author={Krajenbrink, Alexandre and Le Doussal, Pierre},
   year={2020},
   month=mar }

@article{OpenKPZStationary,
  title={Stationary measure for the open {KPZ} equation},
  author={Corwin, Ivan and Knizel, Alisa},
  journal={Communications on Pure and Applied Mathematics},
  volume={77},
  number={4},
  pages={2183--2267},
  year={2024},
  publisher={Wiley Online Library}
}

@article{baik2001algebraic,
  title={Algebraic aspects of increasing subsequences},
  author={Baik, Jinho and Rains, Eric M},
  journal = {Duke Mathematical Journal},
  year={2001},
  volume  = {109},
  pages   = {1--65}
}

@article{stripStationary,
  title={Stationary measures for integrable polymers on a strip},
  author={Barraquand, Guillaume and Corwin, Ivan and Yang, Zongrui},
  journal={Inventiones mathematicae},
  volume={237},
  number={3},
  pages={1567--1641},
  year={2024},
  publisher={Springer}
}

@article{Barraquand_2023,
   title={Stationary measures of the {KPZ} equation on an interval from {Enaud–Derrida’s} matrix product ansatz representation},
   volume={56},
   ISSN={1751-8121},
   url={http://dx.doi.org/10.1088/1751-8121/acc0eb},
   DOI={10.1088/1751-8121/acc0eb},
   number={14},
   journal={Journal of Physics A: Mathematical and Theoretical},
   publisher={IOP Publishing},
   author={Barraquand, Guillaume and Le Doussal, Pierre},
   year={2023},
   month=mar, pages={144003} }

@article{Betea_2020,
   title={Stationary Half-Space Last Passage Percolation},
   volume={377},
   ISSN={1432-0916},
   url={http://dx.doi.org/10.1007/s00220-020-03712-5},
   DOI={10.1007/s00220-020-03712-5},
   number={1},
   journal={Communications in Mathematical Physics},
   publisher={Springer Science and Business Media LLC},
   author={Betea, Dan and Ferrari, Patrik L. and Occelli, Alessandra},
   year={2020},
   month=mar, pages={421–467} }

@article{KPZ,
  title={The {Kardar-Parisi-Zhang} equation and universality class},
  author={Corwin, Ivan},
  journal={Random matrices: Theory and applications},
  volume={1},
  number={01},
  pages={1130001},
  year={2012},
  publisher={World Scientific}
}

@article{Ferrari_2006TASEP,
   title={Scaling Limit for the Space-Time Covariance of the Stationary Totally Asymmetric Simple Exclusion Process},
   volume={265},
   ISSN={1432-0916},
   url={http://dx.doi.org/10.1007/s00220-006-1549-0},
   DOI={10.1007/s00220-006-1549-0},
   number={1},
   journal={Communications in Mathematical Physics},
   publisher={Springer Science and Business Media LLC},
   author={Ferrari, Patrik L. and Spohn, Herbert},
   year={2006},
   month=mar, pages={1–44} }

@article{BI23,
  title={Stationary measures for the log-gamma polymer and {KPZ} equation in half-space},
  author={Barraquand, Guillaume and Corwin, Ivan},
  journal={The Annals of Probability},
  volume={51},
  number={5},
  pages={1830--1869},
  year={2023},
  publisher={Institute of Mathematical Statistics}
}

@article{Baik_2018,
   title={Pfaffian {Schur} processes and last passage percolation in a half-quadrant},
   volume={46},
   ISSN={0091-1798},
   url={http://dx.doi.org/10.1214/17-AOP1226},
   DOI={10.1214/17-aop1226},
   number={6},
   journal={The Annals of Probability},
   publisher={Institute of Mathematical Statistics},
   author={Baik, Jinho and Barraquand, Guillaume and Corwin, Ivan and Suidan, Toufic},
   year={2018},
   month=nov }

@article{Baik_2010,
   title={Limit process of stationary {TASEP} near the characteristic line},
   volume={63},
   ISSN={1097-0312},
   url={http://dx.doi.org/10.1002/cpa.20316},
   DOI={10.1002/cpa.20316},
   number={8},
   journal={Communications on Pure and Applied Mathematics},
   publisher={Wiley},
   author={Baik, Jinho and Ferrari, Patrik L. and Péché, Sandrine},
   year={2010},
   month=mar, pages={1017–1070} }

@article{StationaryKPZ,
   title={Half-Space Stationary {Kardar–Parisi–Zhang} Equation},
   volume={181},
   ISSN={1572-9613},
   url={http://dx.doi.org/10.1007/s10955-020-02622-z},
   DOI={10.1007/s10955-020-02622-z},
   number={4},
   journal={Journal of Statistical Physics},
   publisher={Springer Science and Business Media LLC},
   author={Barraquand, Guillaume and Krajenbrink, Alexandre and Le Doussal, Pierre},
   year={2020},
   month=aug, pages={1149–1203} }

@article{KPZbeyondBrownian,
   title={Half-space stationary {Kardar–Parisi–Zhang} equation beyond the {Brownian} case},
   volume={55},
   ISSN={1751-8121},
   url={http://dx.doi.org/10.1088/1751-8121/ac761d},
   DOI={10.1088/1751-8121/ac761d},
   number={27},
   journal={Journal of Physics A: Mathematical and Theoretical},
   publisher={IOP Publishing},
   author={Barraquand, Guillaume and Krajenbrink, Alexandre and Le Doussal, Pierre},
   year={2022},
   month=jun, pages={275004} }

@article{StatAiryprocess,
  title={The half-space {Airy} stat process},
  author={Betea, Dan and Ferrari, PL and Occelli, Alessandra},
  journal={Stochastic Processes and their Applications},
  volume={146},
  pages={207--263},
  year={2022},
  publisher={Elsevier}
}

@article{Xincheng2024,
  title={{TASEP} in half-space},
  author={Zhang, Xincheng},
  journal={arXiv preprint arXiv:2409.09974},
  year={2024}
}

@article{FixedPoint,
  title={The {KPZ} fixed point},
  author={Matetski, Konstantin and Quastel, Jeremy and Remenik, Daniel},
  journal={Acta Mathematica},
  volume={227},
  number={1},
  pages={115--203},
  year={2021},
  publisher={Lehigh University Bethlehem, Penn., USA}
}

@article{nestoridi2024approximating,
  title={Approximating the stationary distribution of the {ASEP} with open boundaries},
  author={Nestoridi, Evita and Schmid, Dominik},
  journal={Communications in Mathematical Physics},
  volume={405},
  number={8},
  pages={176},
  year={2024},
  publisher={Springer}
}

@article{hegde2024large,
  title={Large deviation principle for the stationary measures of open asymmetric simple exclusion processes},
  author={Hegde, Milind and Yang, Zongrui},
  journal={arXiv preprint arXiv:2412.12026},
  year={2024}
}

@article{yang2025limits,
  title={Limits of open {ASEP} stationary measures near a boundary},
  author={Yang, Zongrui},
  journal={The Annals of Applied Probability},
  volume={35},
  number={5},
  pages={3242--3270},
  year={2025},
  publisher={Institute of Mathematical Statistics}
}

@article{wang2025asymmetric,
  title={From asymmetric simple exclusion processes with open boundaries to stationary measures of open {KPZ} fixed point: the shock region},
  author={Wang, Yizao and Yang, Zongrui},
  journal={Electronic Journal of Probability},
  volume={30},
  pages={1--38},
  year={2025},
  publisher={The Institute of Mathematical Statistics and the Bernoulli Society}
}

\end{document}